\documentclass[10pt, reqno]{amsart}
\usepackage{mathrsfs}
\usepackage{amssymb,amsthm,amsmath}
\usepackage[numbers,sort&compress]{natbib}
\usepackage{amssymb,amsmath}
\usepackage{amsfonts}
\usepackage{mathrsfs}
\usepackage{latexsym}
\usepackage{amssymb}
\usepackage{amsthm}
\usepackage{color}
\usepackage{pdfsync}
\usepackage{indentfirst}
\usepackage{graphics}
\usepackage{subfigure}
\usepackage{epsfig}
\usepackage{float}
\usepackage{cases}
\date{today}
\usepackage{hyperref}
\hypersetup{
colorlinks=true,
linkcolor=blue,
anchorcolor=green,
citecolor=red}
\usepackage{amsmath}
\usepackage{amsthm}

\newcommand{\N}{{\mathbb N}}

\newcommand{\norm}[1]{\| #1 \|}
\newcommand{\norms}[1]{\left\| #1 \right\|}

\renewcommand{\upsilon}{\rho}%将字母$\upsilon$替换成\rho

\let\oldsection\section
\renewcommand\section{\setcounter{equation}{0}\oldsection}

\theoremstyle{plain} % 斜体内容 + 粗体标签
\newtheorem{theorem}{Theorem}[section]
\newtheorem{lemma}[theorem]{Lemma}
\newtheorem{corollary}[theorem]{Corollary}
\newtheorem{proposition}[theorem]{Proposition}

\theoremstyle{definition} % 正体内容 + 粗体标签
\newtheorem{definition}[theorem]{Definition}
\newtheorem{remark}[theorem]{Remark}

\def\ba{\begin{eqnarray}}
	\def\ea{\end{eqnarray}}

\def\N{\Bbb N}

\newcommand{\beq}{\begin{equation}}
	\newcommand{\eeq}{\end{equation}}
\newcommand{\ben}{\begin{eqnarray}}
	\newcommand{\een}{\end{eqnarray}}
\newcommand{\beno}{\begin{eqnarray*}}
	\newcommand{\eeno}{\end{eqnarray*}}

\usepackage{color}
\usepackage{color}
\newif\ifrevision
\revisiontrue
\newcommand{\rev}[1]{\ifrevision{\color{black}#1}\else#1\fi}
\newenvironment{revision}{\begingroup\ifrevision\color{black}\fi}{\endgroup}
\newenvironment{bluerevision}{\begingroup\color{black}}{\endgroup}
\newcommand{\Red}{\textcolor{black}}

\usepackage[breakable,skins]{tcolorbox}

\newcommand{\auditmath}[1]{\textcolor{black}{#1}}

\title[{\scriptsize Existence and partial regularity for the 3d NSVFP equations}]{Existence and partial regularity of suitable weak solutions to the 3D Navier-Stokes-Vlasov-Fokker-Planck equations}
\author{Renjun Duan}
\address[Renjun Duan]{Department of Mathematics, The Chinese University of Hong Kong, Shatin, NT, Hong Kong}
\email{rjduan@math.cuhk.edu.hk}
\author{Fengqiang Shi}
\address[Fengqiang Shi]{School of Mathematical Sciences, Dalian University of Technology, Dalian, 116024,  China}
\email{2192593228@mail.dlut.edu.cn}
\author{Wendong~Wang}
\address[Wendong~Wang]{School of Mathematical Sciences, Dalian University of Technology, Dalian, 116024,  China}
\email{wendong@dlut.edu.cn}
\author{Jianbo Yu}
\address[Jianbo Yu]{School of Mathematical Sciences, Dalian University of Technology, Dalian, 116024,  China}
\email{yujb@mail.dlut.edu.cn}
\date{\today}  % 可选，设置日期为今天，或者自定义日期
\allowdisplaybreaks
\begin{document}
	\begin{abstract}
		In this paper, we investigate the incompressible Navier-Stokes equations coupled with the Vlasov-Fokker-Planck equation, which describes a
two-phase mixture of the viscous incompressible fluid with particles or bubbles
through a frictional force term. In the three-dimensional whole space, we construct a new class of suitable weak solutions to the Navier-Stokes-Vlasov-Fokker-Planck system satisfying energy estimates and three local or global energy inequalities of different forms. These obtained local energy inequalities play an important role in characterizing the measure of the singularity set of weak solutions. The main difficulties in deriving these inequalities lie in establishing the convergence of the density function $f$ in bounded or unbounded domains and dealing with the convergence of the non-local frictional force term. The strong $L^1$ convergence of $|v|^k f$ for $0\le k<\kappa$ and of $|v|^k f\log f$ locally in $x$ for $0\le k<\kappa/2$ is proved by using new a priori estimates for the velocity, Tao's $L^p$ decomposition, and the DiPerna–Lions compactness method.
Moreover, as an immediate consequence of the existence result, we are able to describe the Hausdorff dimension of set of singular points of the fluid velocity $u$ and also establish the $\alpha$-H\"{o}lder continuity of $f$ at the regular points of $u$.
	\end{abstract}


%35Q83  	Vlasov equations
%35Q84  	Fokker-Planck equations
%35Q30  	Navier-Stokes equations 
%35B45  	A priori estimates in context of PDEs
%35A01  	Existence problems for PDEs: global existence, local existence, non-existence
%35B65  	Smoothness and regularity of solutions to PDEs


	\maketitle
    
	%\noindent
	%\textbf{Keywords:} Navier-Stokes-Vlasov-Fokker-Planck, suitable weak solution, local energy inequality, Hausdorff measure, H\"{o}lder continuity % 注意每个关键词之间用逗号隔开z
	
	\tableofcontents
	
	%\tableofcontents
	\parskip5pt
	\parindent=1.5em

	\section{Introduction}
% \begin{bluerevision}
% \noindent\textbf{Revision status.} The blue passages correct several estimates and compactness arguments. The derivation of the entropy and local energy identities for the nonsmooth, possibly vanishing truncated initial density still requires a fully justified approximation and limit argument. Until that argument has been supplied and checked, the proof of the main existence and local energy conclusions should be regarded as incomplete.
% \end{bluerevision}

We consider the following three dimensional incompressible Navier-Stokes-Vlasov-Fokker-Planck (NSVFP) equations
	\begin{eqnarray}\label{eq:NSVFK}
		\left\{
		\begin{array}{lll}
			%\label{001}
			\displaystyle \partial_t u + (u\cdot \nabla) u- \Delta u + \nabla P - \int_{\mathbb{R}^3_v}(v-u)fdv = 0,
            %~&(t,x)\in (0,T)\times\mathbb{R}^3_x
            \\
			\displaystyle \partial_t f + (v\cdot \nabla_x) f+ \nabla_v \cdot \left(\left(u-v\right)f-\nabla_v f\right) = 0,
            %~&(t,x,v)\in (0,T)\times \mathbb{R}^6_{xv}
            \\
			\nabla \cdot u=0,
		\end{array}
		\right.
	\end{eqnarray}
	where $u=u(t,x)$ with $(t,x)\in\mathbb{R}^+\times\mathbb{R}^3_x$ denotes the fluid velocity field, $f=f(t,x,v)\geq 0$ with $(t,x,v)\in\mathbb{R}^+\times\mathbb{R}^6_{xv}$ denotes the
density distribution function of particles in the phase space and the scalar function $P$ represents pressure.
The above kinetic-fluid model is widely used in the description of dynamics of the dispersed two-phase flows, bridging the macroscopic continuum mechanics and microscopic particle dynamics. Typical applications of two-phase flows can be found in \cite{{RM1952},{W1958},{WY2015}}. The Navier-Stokes equations with random stirring forces were formally analyzed by Forster-Nelson-Stephen \cite{FNS1977}. The long-time and long-distance properties of the
renormalized viscosity were studied by Enz \cite{E1978}. Moreover, there exist many variants of kinetic-fluid models depending on the different physical regimes, such as the compressibility of the fluid, viscosity of the fluid, species
of particles, interactions between the fluid and particles, and
so on. Below we recall some kinetic-fluid models related to system \eqref{eq:NSVFK} under consideration.

%Let us recall some developments briefly.

{\it The  nonlinear Fokker-Planck equation.} Constantin in \cite{constantinNonlinearFokkerPlanckNavierStokes2005} considered the nonlinear Fokker-Planck equation coupled with the Stokes system  and proved the existence of smooth solutions
for a wide class of such systems of type II. By  assuming that the added stress tensor is given and satisfies  suitable properties, Constantin-Fefferman-Titi-Zarnescu \cite{CFTZ2007}  proved the regularity of Navier-Stokes-Fokker-Planck system in 2D.  Global regularity  was first proved by Constantin-Masmoudi in \cite{CM2008}. Independently and
simultaneously, global regularity for a similar model was proved by Lin-Zhang-Zhang in \cite{LZZ2008}. In that work, the model is a version of the FENE model in which the physical gradient of
velocity is replaced by its anti-symmetric part, and the particles are restricted
to the unit disk by a potential that is infinite at the unit circle.  
 Constantin-Seregin in \cite{CS2010} gave a quantitative bounds in the two-dimensional torus domain $\mathbb{T}^2$. Choi-Jeong-Kang proved the global existence, uniqueness, and decay of solutions near equilibrium for the Vlasov-Riesz-Fokker-Planck system with small, regular initial data \cite{CJK2024}.  For more references, we refer to \cite{constantinNonlinearFokkerPlanckNavierStokes2005,{BS2008},{BS2010},{ccp2022}}  and the references therein.

{\it The incompressible Navier-Stokes equations coupled with linear
Fokker-Planck equations.}
The mathematical
analysis of incompressible kinetic-fluid models has received much attention recently.
%Global existence for small data for linear
%Fokker-Planck coupled with Navier-Stokes equations was obtained in [013]. For the case of a coupled linear Fokker-Planck and Stokes system, global regularity result for large data was obtained in [018]. Chae, Kang, and Lee [12] obtained the global existence of
%weak solutions in 3D and classical solutions for the Navier–Stokes–Vlasov–Fokker–Planck equations
%in 2D. 
Hamdache \cite{hamdacheGlobalExistenceLarge1998} analyzed the Vlasov-Stokes system in a bounded domain and constructed weak solutions under specular reflection boundary conditions in 1998.  In 2007, Lin-Liu-Zhang established the global existence and uniqueness of classical solutions to the micro-macro model in \cite{linMicroMacroModelPolymeric2007}. Later, in 2007, Otto-Tzavaras \cite{ottoContinuityVelocityGradients2008} showed that for the Doi model describing suspensions of rod-like molecules in the dilute regime, discontinuities in the velocity gradient cannot form in finite time, even without a microscopic cut-off.  Goudon-He-Moussa-Zhang \cite{goudonNavierStokesVlasov2010} established the global existence of
classical solutions near the equilibrium for the incompressible Navier-Stokes-Vlasov-Fokker-Planck system; meanwhile, Carrillo-Duan-Moussa \cite{CDM2011}  studied the corresponding
inviscid case. Subsequently, in 2011, Chae-Kang-Lee \cite{chaeGlobalExistenceWeak2011} established the global existence of weak solutions in two and three dimensions, as well as the existence and uniqueness of global smooth solutions in two dimensions. In 2017, Boudin et al. \cite{boudinGlobalExistenceSolutions2017} demonstrated the global existence of weak solutions for the three-dimensional incompressible Vlasov-Navier-Stokes equations, where the coupling is through a drag force depending linearly on the relative velocity between the fluid and the particles. Su-Yao \cite{SY2020} considered the hydrodynamic limit for the inhomogeneous
incompressible Navier-Stokes-Vlasov-Fokker-Planck equations in a two or three dimensional bounded domain when the initial density is bounded away from zero.  More recently, in 2024, Nfor-Woukeng \cite{chuyehnforWellposednessStochasticNavier2024} proved the existence of a unique weak probabilistic solution for a stochastic Navier-Stokes-Vlasov-Fokker-Planck system, and Jin-Lin \cite{jinEnergyEstimatesHypocoercivity2024} established exponential decay of energy over time by using hypocoercivity arguments in a multi-phase Navier-Stokes-Vlasov-Fokker-Planck system. For more references, we refer to \cite{goudonNavierStokesVlasov2010} and the references therein.

{\it The compressible Navier-Stokes equations coupled with linear
Fokker-Planck equations.}
 Mellet-Vasseur \cite{melletGlobalWeakSolutions2007} proved the global existence of weak solutions to the compressible Vlasov-Fokker-Planck-Navier-Stokes equations with either absorption or reflection boundary conditions. In the same year, they also studied certain asymptotic regimes of the system \cite{melletBarotropicCompressibleNavierStokes2007}. In \cite{MV2008} 
they further studied the asymptotic analysis of the solutions.   In 2006, Baranger-Desvillettes \cite{barangerCouplingEulerVlasov2006} established the local existence of solutions to the compressible Vlasov-Euler equations. Later, in 2017, Li-Mu-Wang \cite{liStrongSolutionsCompressible2017} proved the global well-posedness of strong solutions to the compressible Navier-Stokes-Vlasov-Fokker-Planck equations in the three-dimensional whole space. In 2021, Choi-Jung \cite{choiAsymptoticAnalysisVlasov2021} obtained the global-in-time existence of weak solutions to the coupled kinetic-fluid system, as well as the existence and uniqueness of strong solutions to the limiting system in a bounded domain, equipped with kinematic boundary conditions for the Euler part and Dirichlet conditions for the Navier-Stokes part.
In \cite{liNavierStokesVlasov2022} Li-Liu-Yang proved the global existence of unique strong
solution and its exponential convergence rate to the equilibrium state  under
the Maxwell boundary condition for the incompressible case and specular reflection
boundary condition for the compressible case, respectively.  The existence, uniqueness, and regularity of global weak solution to the
initial value problem for general initial data ware established by Li-Shou \cite{LS2023} in one dimensional spatial periodic domain.
 More recently, in 2023, Chen-Li-Li-Zamponi \cite{chenGlobalWeakSolutions2023} established the global existence of weak solutions for the compressible Vlasov-Poisson-Fokker-Planck-Navier-Stokes equations in a three-dimensional bounded domain, under nonhomogeneous Dirichlet boundary conditions and allowing arbitrarily large initial and boundary data, provided the adiabatic exponent satisfies \(\gamma>\frac32\). More related results can be found in \cite{BS2008,barangerCouplingEulerVlasov2006,flandoliNavierStokesVlasov2021,maInitialBoundaryValue2022,tanGlobalSolutionsExponential2023,duanCauchyProblemVlasovFokkerPlanck2013,cj2020}, and so on.

{\it The incompressible Navier-Stokes equations.} When $f=0$,  the system (\ref{eq:NSVFK}) reduces to the  Navier-Stokes equations. It is well known that the Navier-Stokes equations describe the macroscopic motion of an incompressible viscous fluid:
	\begin{gather*}
		\partial_t u +\left(u\cdot\nabla\right)u-\nu \Delta u +\nabla P=0,\\
		\nabla \cdot u=0,
	\end{gather*}
	where \(u\left(t,x\right)\) is the velocity field, \(\nu\) is the kinematic viscosity, and \(P\left(t,x\right)\) is the pressure. While the Navier-Stokes equations effectively describe continuum fluid flows, they lack the resolution to capture the behavior of individual particles suspended in the medium.
	%\par On the other hand, the Vlasov-Fokker-Planck equation provides a kinetic description of particle distribution:
%	\begin{align*}
%		\partial_tf+v\cdot \nabla_xf+F\cdot\nabla_{v}f=D\Delta_{v}f,
%	\end{align*}
%	where \(f(t,x,v)\) represents the particle distribution function in phase space, \(F\) is the force field acting on the particles, and \(D\) accounts for stochastic effects such as collisions or thermal fluctuations. This formulation captures the detailed statistical properties of particle ensembles and their interactions with external fields.
%	
 The study of partial regularity and suitable weak solutions in PDEs is essential, because many nonlinear equations arising in physics and geometry, such as the three-dimensional Navier--Stokes equations, have unresolved global regularity questions, while other geometric evolution equations can develop singularities. Instead, their solutions may develop singularities, yet often retain regularity outside a small, lower-dimensional set. Weak solutions provide a rigorous framework to analyze such equations, while partial regularity theory quantifies where and how smoothness persists, offering critical insights into the structure of singularities and their physical relevance. This approach bridges theoretical PDE analysis with applied problems, ensuring meaningful interpretations even in the presence of irregular behavior. The concept of suitable weak solutions for the Navier-Stokes equations was initially introduced by Scheffer in \cite{schefferPartialRegularitySolutions1976,schefferHausdorffMeasureNavierStokes1977,schefferNavierStokesEquationsBounded1980}. Later, Caffarelli-Kohn-Nirenberg \cite{caffarelliPartialRegularitySuitable1982} established that the set \(S\) of possible interior singular points for a suitable weak solution \rev{has vanishing one-dimensional parabolic Hausdorff measure}, and hence parabolic Hausdorff dimension at most one. Compared to the Leray-Hopf weak solution proposed by Leray in \cite{lerayMouvementDunLiquide1934}, the suitable weak solution exhibits improved properties. For instance, if a local strong solution loses regularity and blows up, it can still be continued as a suitable weak solution (see Proposition 30.1 in \cite{lemarie-rieussetRecentDevelopmentsNavierStokes2002}). For simplified proofs and further advancements, we refer readers to the works of Lin \cite{linNewProofCaffarelliKohnNirenberg1998}, Ladyzhenskaya-Seregin \cite{ladyzhenskayaPartialRegularitySuitable1999}, Tian and Xin \cite{tianGradientEstimationNavierStokes1999}, Seregin \cite{sereginEstimatesSuitableWeak2006}, Gustafson-Kang-Tsai \cite{gustafsonInteriorRegularityCriteria2007}, Vasseur \cite{vasseurNewProofPartial2007}, and related references. Recently, Chen-Li-Wang-Wang proved global existence of suitable weak solutions to the 3D chemotaxis-Navier-Stokes equations in \cite{chenGlobalExistenceSuitable2025} and derive the Hausdorff measure for the singularity set in \cite{chenHausdorffMeasureSingularity2023}.

	\par Motivated  by \cite{caffarelliPartialRegularitySuitable1982} and \cite{chenHausdorffMeasureSingularity2023}, we study the partial regularity of  the system (\ref{eq:NSVFK}). In what follows we denote $L^2_{\sigma}(\mathbb{R}^3_x)=\{u\in L^2(\mathbb{R}^3_x) | \nabla \cdot u=0\}$ and $\rev{\dot{H}^1_{\sigma}}(\mathbb{R}^3_x)=\{u\in \rev{\dot{H}^1(\mathbb{R}^3_x)} | \nabla \cdot u=0\}$.
We use $0\log0=0$ and interpret the dissipation as
\[
\frac{|(u-v)f-\nabla_v f|^2}{f}
:= |(u-v)\sqrt f-2\nabla_v\sqrt f|^2.
\]
Energy inequalities are understood for almost every terminal time $t>0$.
For every $T>0$, the test functions in the local energy inequalities satisfy
\[
0\leq\psi\in C_c^\infty((0,T)\times\mathbb R^3_x),\qquad
0\leq\phi\in C_c^\infty((0,T)\times\mathbb R^3_x\times\mathbb R^3_v).
\]
The functions $\phi$ and $\psi$ are independent. They vanish near the initial time, but are not required to vanish at the terminal time $t\in(0,T)$ of integration. The function $\phi$ is extended by zero outside its velocity support when integrated over $\mathbb R^3_v$.
\begin{definition}\label{def:weak}
		A pair \(\left(u,f\right)\) is called a suitable weak solution of the system \eqref{eq:NSVFK} with initial data \(\left(u_{in},f_{in}\right)\) satisfying condition \eqref{0} to be specified later, if all the following  properties hold:\\
	\begin{flushleft}
		$\begin{aligned}
	\text{(i).}\ &u \in L^{\infty}_{loc}\left([0,+\infty);L^2_{\sigma}(\mathbb{R}^3_x)\right) \cap L^2_{loc}\left([0,+\infty);\dot{H}^1_{\sigma}(\mathbb{R}^3_x)\right) , \\
		&f(t,x,v)\geq 0, f\in L^{\infty}_{loc}\left([0,+\infty);L^{\infty}\cap L^{1}\left(\mathbb{R}^3_x\times\mathbb{R}^3_v\right)\right) ,\\
		&\left(|x|^2+|v|^2\right)f \in L^{\infty}_{loc}\left([0,+\infty); L^{1}\left(\mathbb{R}^3_x\times\mathbb{R}^3_v\right)\right),\\
		&f\log f \in  L^{\infty}_{loc}\left([0,+\infty); L^{1}\left(\mathbb{R}^3_x\times\mathbb{R}^3_v\right)\right),\\
		& P\in L^\frac53_{loc}\left([0,+\infty);L^\frac53\left(\mathbb{R}^3_x\right)\right);
	\end{aligned}$
\end{flushleft}
\begin{flushleft}
	$\begin{aligned}
		\text{(ii).}\ &u\in L^{1}_{loc}\left([0,+\infty);L^p_{\sigma}(\mathbb{R}^3_x)\right)~\left(p\in \left[2, +\infty\right)\right),\\
		&|v|^\kappa f\in L^{\infty}_{loc}\left([0,+\infty); L^{1}\left(\mathbb{R}^3_x\times\mathbb{R}^3_v\right)\right)~(\text{for the fixed }\kappa\text{ in }\eqref{0}),\\
        &|v|^{\bar{\kappa}} f \log f\in  L^{\infty}_{loc}\left([0,+\infty); L^{\eta}\left(\mathbb{R}^3_x\times\mathbb{R}^3_v\right)\right)~\left(\forall~\bar{\kappa}\in \left[0, \kappa\right),\eta>1 ~\text{s.t.}~\eta\bar{\kappa}\leq\kappa;\ \kappa\text{ as in }\eqref{0}\right),\\
        &|v|^{\frac{\bar{\kappa}}{2}} f \log f\in  {L^1((0,T)\times\Omega_{x}\times\mathbb{R}^3_v)}~\left(\forall~\bar{\kappa}\in \left[0, \kappa \right);\ \kappa\text{ as in }\eqref{0}\right);
	\end{aligned}$
\end{flushleft}
	(iii). \(\left(u,f\right)\) solves \eqref{eq:NSVFK} in the sense of distributions;\\
	(iv).  \(\left(u,f\right)\) satisfies the energy inequality
	\beno
    &&\frac{1}{2}\int_{\mathbb{R}^3_x}|u(t,\cdot)|^2dx+\int_{\mathbb{R}^6_{xv}}\left(\frac{1}{2}|v|^2f+f\log f\right)(t,\cdot)dxdv\nonumber\\
		&&+\int_{(0,t)\times \mathbb{R}_x^3}|\nabla u|^2dxds+\int_{(0,t)\times \mathbb{R}_{xv}^6}\frac{\left|(u-v)f-\nabla_{v}f\right|^2}{f}dxdvds\nonumber\\
        &\leq& \int_{\mathbb{R}^6_{xv}}\left(\frac{1}{2}|v|^2f_{in}+f_{in}\log f_{in}\right)dxdv+\frac{1}{2}\int_{\mathbb{R}^3_x}|u_{in}|^2dx;
	\eeno
	(v). \(\left(u,f\right)\) satisfies the first type of local energy inequality as
	\begin{align*}
		\frac{1}{2}&\int_{\Omega_x}\left(|u|^2\psi\right)(t,\cdot) dx+\int_{\Omega_x \times \Omega_v}\left(\left(\frac{1}{2}|v|^2f+f\log f\right)\phi\right)(t,\cdot) dxdv+\int_{(0,t)\times \Omega_x}|\nabla u|^2\psi dxds\nonumber\\
		&+\int_{(0,t)\times \Omega_x \times \Omega_v}\frac{|(u-v)f-\nabla_{v}f|^2}{f}\phi dxdvds\leq \frac{1}{2}\int_{(0,t)\times\Omega_{x}}|u|^2(\partial_t\psi+\Delta_x\psi) dxds\nonumber\\
		&+\frac{1}{2}\int_{(0,t)\times\Omega_{x}}|u|^2u\cdot \nabla_{x}\psi dxds
		+\int_{(0,t)\times \Omega_x \times \Omega_v}\left(\frac{|v|^2}{2}+\log f\right)f(\partial_t\phi+\Delta_v\phi)dxdvds\nonumber\\
		&+ \int_{(0,t) \times \Omega_x \times \Omega_v} \left( \frac{|v|^2}{2} + \log f \right) fv \cdot \left( \nabla_{x}\phi - \nabla_{v}\phi \right) dxdvds
		+\int_{(0,t)\times\Omega_{x}}(P-\bar{P})u\cdot\nabla_{x}\psi dxds\nonumber \\
		&-\int_{(0,t) \times \Omega_x \times \Omega_v}fv\cdot u\phi dxdvds +\int_{(0,t) \times \Omega_x \times \mathbb{R}^3_v}fv\cdot u\psi dxdvds\nonumber\\
		&+ \int_{(0,t) \times \Omega_x \times \Omega_v} \left( 2 + \frac{|v|^2}{2} +\log f \right) fu \cdot \nabla_{v}\phi dxdvds\\
		&\rev{+\int_{(0,t)\times\Omega_x\times\mathbb R^3_v}|u|^2f(\phi-\psi)dxdvds,}
	\end{align*}
		% MINREV ORIGINAL BEGIN: definition test functions
 % where \(\phi, \psi \geq 0\) and they vanish in the parabolic boundary of \(\left(0,t\right)\times\Omega_{x}\times \Omega_{v}\) and \(\left(0,t\right)\times\Omega_{x}\), respectively. Here, \(\Omega_x\) and \(\Omega_v\) are two bounded domains in $\mathbb{R}^3$. Here and in the sequel,
% MINREV ORIGINAL END
\rev{where $\phi,\psi\ge0$ satisfy the test-function convention above, and $\Omega_x,\Omega_v$ are bounded domains containing their spatial and velocity supports.}
Here and in the sequel, we denote \(\bar{P}:=\frac{1}{|\Omega_x|}\int_{\Omega_x}Pdx\).
        
	\noindent(vi). When \(\kappa >3\), \(\left(u,f\right)\) satisfies the second type of local energy inequality as
	\begin{align*}
		\frac{1}{2}&\int_{\Omega_x}\left(|u|^2\psi\right)(t,\cdot) dx+\int_{\Omega_x \times \mathbb{R}^3_v}\left(\left(\frac{1}{2}|v|^2f+f \log f\right)\psi\right)(t,\cdot) dxdv+\int_{(0,t)\times \Omega_x}|\nabla u|^2\psi dxds\nonumber\\
		&+\int_{(0,t)\times \Omega_x \times \mathbb{R}^3_v}\frac{|(u-v)f-\nabla_{v}f|^2}{f}\psi dxdvds\leq \frac{1}{2}\int_{(0,t)\times\Omega_{x}}|u|^2(\partial_t\psi+\Delta_x\psi) dxds\nonumber\\
		&+\frac{1}{2}\int_{(0,t)\times\Omega_{x}}|u|^2u\cdot \nabla_{x}\psi dxds
		+\int_{(0,t)\times \Omega_x \times \mathbb{R}^3_v}\left(\frac{|v|^2}{2}+\log f\right)f\partial_t\psi dxdvds\nonumber\\
		&+ \int_{(0,t) \times \Omega_x \times \mathbb{R}^3_v} \left( \frac{|v|^2}{2} + \log f \right) fv \cdot \nabla_{x}\psi  dxdvds
		+\int_{(0,t)\times\Omega_{x}}(P-\bar{P})u\cdot\nabla_{x}\psi dxds.
	\end{align*}
	\end{definition}
	\par In the following, the moments of \(f(t,x,v)\) are denoted
	\begin{align*}
		m_{\alpha} f(t,x) &= \int_{\mathbb{R}_v^3}f(t,x,v)|v|^{\alpha}dv,\\
		M_{\alpha} f(t) &=\int_{\mathbb{R}_{xv}^6}f(t,x,v)|v|^{\alpha}dxdv,
	\end{align*}
	for any \(t\in [0,T]\) and \(\alpha \geq 0\). Note that
	\begin{align*}
		M_{\alpha}f(t)=\int_{\mathbb{R}_x^3}m_{\alpha}f(t,x)dx.
	\end{align*}
	
	\par Our first main theorem is stated as follows.
	
	\begin{theorem}\label{thm:weak}
	 Assume that the initial data \(\left(u_{in},f_{in}\right)\) satisfies
	\begin{equation}
		\left\{
		\begin{array}l
			\label{0}
			u_{in}\in L^2_{\sigma}(\mathbb{R}^3_x);\\
			f_{in}\in L^{\infty}\cap L^{1}(\mathbb{R}^6_{xv}),~ f_{in}\geq 0;\\
			\left(|x|^2+|v|^{\kappa}+|\log f_{in}|\right)f_{in} \in L^1(\mathbb{R}^6_{xv}),~\kappa\geq 3.
		\end{array}
		\right.
	\end{equation}
	Then, for the Cauchy problem on the system \eqref{eq:NSVFK} with $\left.(u,f)\right|_{t=0}=\left(u_{in},f_{in}\right)$, there exists a global suitable weak solution in the sense of Definition \ref{def:weak}.
	\end{theorem}

\begin{remark} Compared with the usual definition of weak solutions, these properties $(ii), (v), (vi)$ in Definition \ref{def:weak} are completely new. It seems difficult to derive the uniform bound on $|v|^mf$ for $m>3$ due to the whole space by Lemma \ref{lem:f}, although one can assume some regular initial data (see the details in Section 3.2), where the large $L^p$ norm in the space is necessary for the velocity.  With the help of Tao's Littlewood-Paley decomposition (cf.~\cite{taoLocalisationCompactnessProperties2013}) to the Navier-Stokes equations,   the a priori estimates of \(\norm{u}_{L^1_tL^p_x}\) are obtained, which implies the $L^\infty_tL^1_{xv}$ norm of $|v|^mf$ is uniformly bounded (see Lemma \ref{lem:ulp} and \ref{lem:mf} for more details). Combining these estimates and DiPerna-Lions compactness method in \cite{dipernaFokkerPlanckBoltzmannEquation1988}, the strong convergence of $|v|^mf$ in global $L^1$ for $0\le m<\kappa$, and of $|v|^mf\log f$ in $L^1$ locally in $x$ for $0\le m<\kappa/2$, is obtained (see Section 5.2, 5.3, 5.5), which plays an important role in the local energy inequalities of $ (v), (vi)$. In order to prove  $ (v), (vi)$, we also have used Lions' weak convergence method in \cite{lionsMathematicalTopicsFluid1998} (see Section 5.1 for details) and Aubin-Lions lemma for the strong convergence of the velocity (see Section 5.4).

	%The main innovations of this article are as follows:\\
%	(i) In terms of consistent estimation. We derive  And on this basis, we get a priori estimate of \(\norms{|v|^mf}_{L^{\infty}_tL^1_{xv}}\) if
%	\begin{align*}
%		 \norms{|v|^mf_{in}}_{L^{\infty}_tL^1_{xv}} \leq C
%	\end{align*}
%	holds true. \\
%	(ii) In terms of convergence. We get \(f^{\varepsilon,\delta}\) is compact in any \(L^{p}((0,T)\times\mathbb{R}^3_x\times\mathbb{R}^3_v)\) by Lion's methods, which \(p\in [1,+\infty)\). Moreover, by Lebesgue–Vitali theorem, we can derive \(f^{\varepsilon,\delta}\log f^{\varepsilon,\delta}\) is compact in \(L^{1}((0,T)\times\Omega_{x}\times\mathbb{R}^3_v)\).\\
%	(iii) In terms of regularity. We prove Hausdorff measure of \(u\), and \(f\) is H\"{o}lder continuous in \((z_0,v)\), which \(z_0\) is regular point of \(u\).
\end{remark}

	\par The proof of Theorem \ref{thm:weak} is based on the global well-posedness for the following regularized system:
		\begin{eqnarray}\label{eq:NSVFK0}
		\left\{
		\begin{array}{llll}
			%\label{eq:0}
			\displaystyle \partial_t u^{\varepsilon,\delta} + (\theta_{\varepsilon}\ast u^{\varepsilon,\delta})\cdot \nabla u^{\varepsilon,\delta}- \Delta u^{\varepsilon,\delta} + \nabla_{x} P^{\varepsilon,\delta} -
			\theta_{\varepsilon}\ast\left\{\int_{\mathbb{R}_v^3}(v-\theta_{\varepsilon} \ast u^{\varepsilon,\delta})f^{\varepsilon,\delta}dv\right\} = 0, \\
			\displaystyle \partial_t f^{\varepsilon,\delta} + (v\cdot \nabla_x) f^{\varepsilon,\delta}+ \nabla_v \cdot [(\theta_{\varepsilon} \ast u^{\varepsilon,\delta}-v\gamma_
			\delta(v))f^{\varepsilon,\delta}-\nabla_v f^{\varepsilon,\delta}] = 0, \\
			\displaystyle \nabla_{x}\cdot u^{\varepsilon,\delta} =0 ,\\
			\displaystyle f^{\varepsilon,\delta}(0,x,v)=f_{in}^\delta(x,v),\quad u^{\varepsilon,\delta}(0,x)=u_{in}^\delta(x).	
		\end{array}
		\right.
	\end{eqnarray}
	Here, for $\varepsilon >0$, we define $\theta_{\varepsilon} $ = \(\varepsilon^{-3}\theta(\frac{x}{\varepsilon})\) with \(\theta \in  C^{\infty}_c(\mathbb{R}_x^3)\), \(\theta\geq 0\), $\theta(x)=\theta(-x)$, and \(\int_{ \mathbb{R}_x^3}\theta=1\). Moreover, for $\delta>0$ we introduce  $\gamma_\delta \in C^{\infty}(\mathbb{R}_v^{3})$ whose support is included in the ball $B(0,\frac{1}{\delta})$ such that $0\leq \gamma_\delta \leq 1,\gamma_\delta=1$
	on $B(0,\frac{1}{2\delta})$, \(v\cdot\nabla_{v}\gamma_{\delta}(v) \leq 0\), and \(|v\cdot\nabla_{v}\gamma_{\delta}(v)| \leq C\). 
    %and $\gamma_\delta(v)\rightarrow 1$ as $\delta$ goes to zero. 
    Besides,  assume that
    \begin{align}\label{udelta}
        u_{in}^{\delta}(x):=\theta_{\delta}\ast u_{in}, \quad \delta>0, 
    \end{align}
    where \Red{$\theta_{\delta}$ is defined as before},
    % which $\delta >0$ and define $\theta_{\delta} $ = \(\delta^{-3}\theta(\frac{x}{\delta})\) with \(\delta \in  C^{\infty}(\mathbb{R}^3)\), \(\theta\geq 0\), and \(\int_{ \mathbb{R}_x^3}\theta=1\).
    and 
    \begin{align}\label{fdelta}
          f_{in}^{\delta}(x,v):=\varrho_{\{|v|
	\leq
	\frac{1}{\delta}\}} f_{in},
    \end{align}
     where $\varrho_{\{|v|\leq\frac{1}{\delta}\}}=1$ for $|v|\leq\frac{1}{\delta}$ and $=0$ otherwise.
     % and \(f_{in}^{\delta}\) converges to \(f_{in}\) as \(\delta \rightarrow 0.\)
	
For \Red{the regularized problem} \eqref{eq:NSVFK0}, we have the following global well-posedness result.

\begin{theorem}\label{thm:solution}
% 		Let  $f_{in}, u_{in}$ be assumed  as in Theorem  \ref{thm:weak}. Then there exists a \rev{unique global mild solution} \(\left(u^{\varepsilon,\delta},f^{\varepsilon,\delta}\right)\) of the regularized problem \eqref{eq:NSVFK0} with the initial data \eqref{udelta} and \eqref{fdelta}, which satisfies\\
	
% \begin{revision}
% Here mild refers to the heat and Kolmogorov Duhamel formulas of
% Section 2.1. For fixed $\varepsilon,\delta>0$ and every $T<\infty$,
% \begin{align*}
%  &u^{\varepsilon,\delta}\in C([0,T];H^1_\sigma)\cap L^2(0,T;H^2),
%  \quad\partial_tu^{\varepsilon,\delta}\in L^2(0,T;L^2),\\
%  &f^{\varepsilon,\delta}\in C([0,T];L^1((1+|v|^2)\,dx\,dv))
%  \cap L^\infty((0,T)\times\mathbb R^6),
%  \quad\nabla_vf^{\varepsilon,\delta}\in L^2((0,T)\times\mathbb R^6).
% \end{align*}
% The initial data are attained in these strong continuity spaces;
% smoothness of the velocity-truncated datum is not assumed.
% The solution satisfies the following bounds:
% \end{revision}

% \textcolor{blue}{Let $f_{in}$ and $u_{in}$ satisfy the assumptions
% of Theorem~\ref{thm:weak}. Then, for every fixed
% $\varepsilon,\delta>0$, there exists a global mild
% solution $(u^{\varepsilon,\delta},f^{\varepsilon,\delta})$
% of the regularized problem \eqref{eq:NSVFK0}
% with the initial data \eqref{udelta} and \eqref{fdelta}.
% Here, the mild solution is understood in the sense
% of the heat and Kolmogorov Duhamel formulas
% introduced in Section 2.1.
% Moreover, the following estimates hold:}
\begin{revision}
    Let $f_{in}$ and $u_{in}$ satisfy the assumptions of Theorem~\ref{thm:weak}. Then, for every fixed $\varepsilon,\delta>0$, there exists a global  solution  $(u^{\varepsilon,\delta},f^{\varepsilon,\delta})$ of the regularized problem \eqref{eq:NSVFK0} with the initial data \eqref{udelta} and \eqref{fdelta}.
Here, the  solution is understood in the sense of the heat and Kolmogorov Duhamel formulas to be introduced in Section 2.1 later. For every finite $T<\infty$, the solution satisfies
\begin{align*}
u^{\varepsilon,\delta}
&\in C([0,T];H^1_\sigma(\mathbb R^3)) \cap L^2(0,T;H^2(\mathbb R^3)),\\
\partial_tu^{\varepsilon,\delta} &\in L^2(0,T;L^2(\mathbb R^3)),\\
f^{\varepsilon,\delta} &\in C([0,T];L^1(\mathbb R^6;(1+|v|^2)\,dx\,dv)) \cap L^\infty((0,T)\times\mathbb R^6),\\
\nabla_vf^{\varepsilon,\delta} &\in L^2((0,T)\times\mathbb R^6).
\end{align*}
The initial data are attained in these strong continuity
spaces correspondingly. No additional smoothness of the velocity-truncated
initial datum is required.
Moreover, the following estimates hold:\\
\end{revision}
(i)
	\begin{align*}
		0\leq f^{\varepsilon,\delta} &\in L^{\infty}\left((0,T);L^1(\mathbb{R}^6_{xv})\right)\cap L^{\infty}\left((0,T)\times\mathbb{R}^6_{xv}\right),\\
		u^{\varepsilon,\delta} &\in L^{\infty}\left((0,T);L^2(\mathbb{R}^3_x)\right) \cap L^{\infty}\left((0,T);H^1(\mathbb{R}^3_x)\right),\\
		|v|^2f^{\varepsilon,\delta} &\in L^{\infty}\left((0,T);L^1(\mathbb{R}^6_{xv})\right),
	\end{align*}
	and
	\begin{align*}
	    &\norm{u^{\varepsilon,\delta}}^{2}_{L^{\infty}(0,T;{L^{2}_{x}})}+\norms{|v|^{2}f^{\varepsilon,\delta}}_{L^{\infty}(0,T;{L^{1}_{xv}})}+\norm{\nabla u^{\varepsilon,\delta}}^{2}_{L^{2}(0,T;{L^{2}_{x}})}+\norm{\nabla u^{\varepsilon,\delta}}^{2}_{L^{\infty}(0,T;{L^{2}_{x}})}\nonumber \\
		\leq & C\left(\varepsilon,\delta,\norm{f_{in}}_{L^{\infty}_{xv}}, \norm{f_{in}}_{L^{1}_{xv}},\norms{|v|^{2}f_{in}^{\delta}}_{L^{1}_{xv}},\norm{u_{in}}_{L^2_x}^2\right)\left(T+1\right)^5e^{C(\varepsilon)(T+1)^2e^{2T}}.
	\end{align*}
	(ii) \(\norms{f^{\varepsilon,\delta}}_{L^{\infty}_tL^1_{xv}} \leq \norms{f_{in}}_{L^1_{xv}}\), \quad \(\norms{f^{\varepsilon,\delta}}_{L^{\infty}_{txv}} \leq e^{3T}\norms{f_{in}}_{L^{\infty}_{xv}}\).
	\end{theorem}

     Moreover, as an immediate consequence of Theorem \ref{thm:weak} together with  Lemma \ref{regular} and Lemma \ref{singular} in the appendix, we are able to describe the Hausdorff dimension of set of singular points of the fluid velocity $u$, and also establish the $\alpha$-H\"{o}lder continuity of $f$ at the regular points of $u$. The main results are stated as follows.
	
   \begin{theorem}\label{thm:regular}
   	Let all the conditions of  Theorem  \ref{thm:weak} be satisfied. It is further assumed that
	\begin{align}
		\label{9.2}
		|v|^{15+\varkappa}f_{in} \in L^1(\mathbb{R}^6_{xv})
	\end{align}
	with \(\varkappa >0\).  Then, there is a small \(\varepsilon_0 >0\) such that if $u$ satisfies
	\begin{align*}
		\limsup_{r\to 0+}\frac{1}{r}\iint_{Q(z_0,r)}|\nabla u|^2<\varepsilon_0,
	\end{align*}
     where $Q(z_0,r):=(t_0-r^2,t_0)\times B(x_0,r)$ for $z_0=(t_0,x_0)$, then  \(u\) is regular at \(z_0\) in the sense that there exists $r>0$ such that $u\in L^\infty (Q(z_0,r))$.
    Moreover, it holds
	\begin{align*}
		P^1(S_{u,sing}) = 0,
	\end{align*}
	where
	\begin{align*}
		S_{u,sing} = \{z \in (0, T]\times \Omega_x : u \text{ is not regular at } z\}.
	\end{align*}
   \end{theorem}
	%\begin{remark}
%
%\end{remark}
	
\begin{theorem}\label{thm:holder}
Let all the conditions of  Theorem  \ref{thm:weak} together with \eqref{9.2} be satisfied. Given $\Omega_x,\Omega_v$ bounded, with the cylinders below contained in $(0,T]\times\Omega_x\times\Omega_v$, let
	\begin{align*}
		S_{u,reg} = \{z \in (0, T]\times \Omega_x : u \text{ is regular at } z\},
	\end{align*}
    then for any $z_0\in S_{u,reg}$, the solution \(f(t,x,v)\) of \eqref{eq:NSVFK} is \(\alpha\)-H\"{o}lder continuous 
    %with respect to \((z_0, v)\) 
    in \(Q_{\text{int}}\) %and
    satisfying
	\begin{align*}
		\|f\|_{C^\alpha(Q_{\text{int}})} \leq C \left( \|f\|_{L^2(Q_{\text{ext}})} + \|f\|_{L^\infty(Q_{\text{ext}})} \right).
	\end{align*} 
    % MINREV ORIGINAL BEGIN: kinetic cylinders
% with constants \(\alpha = \alpha(\Lambda)\in (0,1)\) and \( C = C(\Lambda, Q_{\text{ext}}, Q_{\text{int}}) \), where  \( Q_{\text{ext}} := (t_0-r^2,t_0)\times B(x_0,r)\times  B(v,r)\) and \( Q_{\text{int}} := (t_0-r_1^2,t_0)\times B(x_0,r_1)\times  B(v,r_1) \) with \( r_1 < r_0 \).
% MINREV ORIGINAL END
\rev{For each fixed $v_0\in\Omega_v$, choose $0<r_1<r_0$ so that the
$(t,x)$-projection of the outer cylinder lies in a neighborhood where
$u$ is bounded. Here, with $z_0=(t_0,x_0)$,
\begin{align*}
 Q_r=\{(t,x,v):t_0-r^2<t\le t_0,\ |v-v_0|<r,
                 \ |x-x_0-(t-t_0)v_0|<r^3\},
\end{align*}
$Q_{\mathrm{ext}}=Q_{r_0}$ and $Q_{\mathrm{int}}=Q_{r_1}$.
The constants $\alpha\in(0,1)$ and $C$ depend on the cylinder geometry
and on $1+|v_0|+r_0+\|u\|_{L^\infty(Q(z_0,r))}$; the estimate is local
in velocity, not uniform as $|v_0|\to\infty$.}
\end{theorem}

\begin{remark}
To the best of our knowledge, this is the first result establishing the existence of \textit{suitable weak solutions} and quantifying the Hausdorff dimension of the singular set for the full 3D NSVFP system. The requirement of moments up to order $15+\kappa$ in \eqref{9.2} is technical but necessary for our analysis. It ensures that the frictional force term $\int (v-u)f dv$ belongs to $L^q_{t,x}$ with $q > 5/2$. This specific threshold allows us to treat the coupling term as a perturbation that does not deteriorate the parabolic Hausdorff dimension of the singular set (see Lemma \ref{singular} and Lemma \ref{regular}). 
% Relaxing this assumption would require new estimates on the dispersion of particles in the high-velocity tail.

\end{remark}

%\begin{remark}
%	The main innovations of this article are as follows:\\
%	(i) In terms of consistent estimation. We derive a priori estimate of \(\norm{u}_{L^1_tL^p_x}\) by using Tao's Lp decomposition to NS equations. And on this basis, we get a priori estimate of \(\norms{|v|^mf}_{L^{\infty}_tL^1_{xv}}\) if
%	\begin{align*}
%		 \norms{|v|^mf_{in}}_{L^{\infty}_tL^1_{xv}} \leq C
%	\end{align*}
%	holds true. \\
%	(ii) In terms of convergence. We get \(f^{\varepsilon,\delta}\) is compact in any \(L^{p}((0,T)\times\mathbb{R}^3_x\times\mathbb{R}^3_v)\) by Lion's methods, which \(p\in [1,+\infty)\). Moreover, by Lebesgue–Vitali theorem, we can derive \(f^{\varepsilon,\delta}\log f^{\varepsilon,\delta}\) is compact in \(L^{1}((0,T)\times\Omega_{x}\times\mathbb{R}^3_v)\).\\
%	(iii) In terms of regularity. We prove Hausdorff measure of \(u\), and \(f\) is H\"{o}lder continuous in \((z_0,v)\), which \(z_0\) is regular point of \(u\).
%\end{remark}

\par The rest of the paper is organized as follows. Section 2 is devoted to proving Theorem \ref{thm:solution} on the existence of global solutions for the regularized system. The proof of Theorem  \ref{thm:weak} is stated in Section 3-8. In Section 3 we show the uniform estimates independent of \(\varepsilon,\delta\) for the regularized system. In Section 4, we obtain the global energy inequality, local energy inequality in the first and second types. Next, we derive in Section 5 the convergence of \( u^{\varepsilon,\delta} \), \( f^{\varepsilon,\delta} \), and \( f^{\varepsilon,\delta} \log f^{\varepsilon,\delta} \), which are essential for the energy estimates. In Section 6, Section 7 and Section 8 the convergence of the three energy inequalities are proved, respectively. The proof of Theorem \ref{thm:regular} and Theorem \ref{thm:holder} is given briefly in Section 9. Some technical lemmas are added in the Appendix.

	\section{Global existence of solutions for the regularized system}

In this section, we  construct global solutions of the regularized system (\ref{eq:NSVFK0}), which is divided into two subsections.
	\subsection{Local existence of regularized solution} 
In this subsection, we first construct local solutions of the regularized system (\ref{eq:NSVFK0}) by the fundamental solutions of Stokes and Fokker-Planck equations and Banach fixed-point theorem.
	\subsubsection{The construction of a mapping}
	
	Let $T\in(0,1)$ and
	\begin{align*}
		R=2\left(1+\|u_{in}^{\delta}\|_{H^1_x}+\|f_{in}^{\delta}\|_{L^1_{xv}}+\|f_{in}^{\delta}\|_{L^\infty_{xv}}+
	\left\| |v|^{2} f_{in}^{\delta}\right\|_{L^1_{xv}}\right) <\infty.
	\end{align*}
	Then we introduce the following Banach space:
	$$ X:=L^{\infty}\left(0,T;W^{1,2}_\sigma(\mathbb{R}^{3})\times [{L^{1}(\mathbb{R}^{3}\times \mathbb{R}^{3}) \cap L^{\infty}(\mathbb{R}^{3}\times \mathbb{R}^{3})}]\times L^{1}(\mathbb{R}^{3}\times \mathbb{R}^{3})\right)$$
	along with its closed subset
	\begin{equation*}
		\Gamma:=\bigg\{(u^{\varepsilon,\delta},f^{\varepsilon,\delta},|v|^{2}f^{\varepsilon,\delta}) \in X \mid \norms{(u^{\varepsilon,\delta},f^{\varepsilon,\delta},|v|^{2}f^{\varepsilon,\delta})}_{\Gamma} \leq R \bigg\},
    \end{equation*}
where
    \begin{align*}
        \norms{(u^{\varepsilon,\delta},f^{\varepsilon,\delta},|v|^{2}f^{\varepsilon,\delta})}_{\Gamma} :=& \sup_{t\in(0,T)}\|u^{\varepsilon,\delta}\|_{L^{2}(\mathbb{R}^{3})}+\sup_{t\in(0,T)}\|\nabla u^{\varepsilon,\delta}\|_{L^{2}(\mathbb{R}^{3})}\\
		&+\sup_{t\in(0,T)}\|f^{\varepsilon,\delta}\|_{L^{1}_{xv}(\mathbb{R}^{3}\times \mathbb{R}^{3})}+\sup_{t\in(0,T)}\|f^{\varepsilon,\delta}\|_{L^{\infty}_{xv}(\mathbb{R}^{3}\times \mathbb{R}^{3})}\\
        &+\sup_{t\in(0,T)}\left\| |v|^{2}f^{\varepsilon,\delta}\right\|_{L^{1}_{xv}(\mathbb{R}^{3}\times \mathbb{R}^{3})} .
	\end{align*}

For $(u^{\varepsilon,\delta},f^{\varepsilon,\delta},|v|^{2}f^{\varepsilon,\delta}) \in \Gamma$ and $t\in(0,T)$,  let
    \beno
	&&\Phi(u^{\varepsilon,\delta},f^{\varepsilon,\delta},|v|^{2}f^{\varepsilon,\delta})\\
	&& =\left(
	\begin{array}{llll}
		\Phi_{1}(u^{\varepsilon,\delta},f^{\varepsilon,\delta},|v|^{2}f^{\varepsilon,\delta})(t,x)\\ 
        \Phi_{2}(u^{\varepsilon,\delta},f^{\varepsilon,\delta},|v|^{2}f^{\varepsilon,\delta})(t,x,v)\\
		\Phi_{3}(u^{\varepsilon,\delta},f^{\varepsilon,\delta},|v|^{2}f^{\varepsilon,\delta})(t,x,v)
	\end{array}
	\right)\\
	&&:=\left(
	\begin{array}{llll}
		e^{t\triangle} u_{in}^{\delta}-\int^{t}_{0}e^{(t-\tau) \triangle} \mathbb{P}\left\{(\theta_{\varepsilon}\ast u^{\varepsilon,\delta})\cdot \nabla u^{\varepsilon,\delta}\rev{-}
		\theta_{\varepsilon}\ast \int_{\mathbb{R}^3}(v-\theta_{\varepsilon} \ast  u^{\varepsilon,\delta})f^{\varepsilon,\delta}dv \right\}d\tau \\[1mm]
		\int_{\mathbb{R}^{3}}\int_{\mathbb{R}^{3}}G(t,x,v;0,\xi,\upsilon)f_{in}^{\delta}(\xi,\upsilon)d\xi d\upsilon \\
		\hspace{1cm}-\int^{t}_{0}\int_{\mathbb{R}^{3}}\int_{\mathbb{R}^{3}}G(t,x,v;\tau,\xi,\upsilon)
		\nabla_{\upsilon} \cdot \left[\left(\theta_{\varepsilon} \ast u^{\varepsilon,\delta} -\upsilon\gamma_{\delta}(\upsilon)\right)f^{\varepsilon,\delta}(\xi,\upsilon)\right] d\xi d\upsilon d\tau\\[1mm]
\rev{|v|^2\Phi_2(u^{\varepsilon,\delta},f^{\varepsilon,\delta},|v|^2f^{\varepsilon,\delta})(t,x,v)}
	\end{array}
	\right),
	\eeno
	where
	\beno
	\begin{aligned}
		G(t,x,v;\tau,\xi,\upsilon)&=G(t-\tau,\upsilon,v, x-\xi)\\ 
		&=\frac{3^\frac32}{(2\pi)^{3}(t-\tau)^6} exp \left\{- \frac{3\big|x-\xi-[\frac{t-\tau}{2}](\upsilon+v)\big|^2}{(t-\tau)^3}- \frac{|\upsilon-v|^2}{4(t-\tau)}\right\}
	\end{aligned}
	\eeno
	is the fundamental solution (see, for example, \cite{carpioLongtimeBehaviourSolutions1998}) of
	\begin{equation}\label{d2.1} % 可以用 equation 环境
		\begin{aligned}
			\left\{
			\begin{array}{ll}
				\partial_t g - \triangle_{v} g + v \cdot \nabla_{x} g = 0, & \quad \text{in } \mathbb{R}^{+}\times  \mathbb{R}^{3}_x \times \mathbb{R}^{3}_v , \\
				g(0,x, v) = g_{0}(x, v), & \quad \text{in } \mathbb{R}^{3}_x \times \mathbb{R}^{3}_v.
			\end{array}
			\right.
		\end{aligned}
	\end{equation}
	The family of operators $(e^{t\triangle})_{t\geq 0}$ denotes the heat semigroup, and $\mathbb{P}$ represents the Leray projection on $L^2$. The terms $h^{\varepsilon,\delta}$ and
	$h_{in}^\delta$ correspond to \(|v|^{2}f^{\varepsilon,\delta}\) and \(|v|^{2}f_{in}^\delta\), respectively. 
    \begin{revision}
        By construction, the third component of $\Phi$ is defined by
$\Phi_3=|v|^2\Phi_2$.
Hence, the compatibility relation between the second and third components is preserved at every iteration. The weighted estimates below ensure that $\Phi$ maps $\Gamma$ into itself.
    \end{revision}
    
    Additionally, since $G$ is a fundamental solution in \eqref{d2.1}, we give some known properties in the following:
	
	%\begin{lemma}
	%(i)For some positive constants $m_{1}$,$m_{2}$,$M_{1}$,$M_{2}$
	%\begin{align*}
	%\int_{\mathbb{R}^{3}} G(t,x,v;\tau,\xi,\upsilon) dv=\frac{M_{1}}{(t-\tau)^\frac92} \exp\left \{- \frac{|x-\xi-(t-\tau)\upsilon|^2}{m_{1}(t-\tau)^3}\right \},
	%\end{align*}

	%\end{lemma}

	\begin{lemma}[Lemma 1 in \cite{carpioLongtimeBehaviourSolutions1998}]\label{lem:G} 
		(i) For some positive constants \( m_{1}, m_{2}, M_{1}, M_{2} \), there hold
		\begin{equation*}
			\begin{array}{l}
				\int_{\mathbb{R}^{3}} G(t,x,v;\tau,\xi,\upsilon) \, dv = \frac{M_{1}}{(t-\tau)^\frac92} \exp\left\{ - \frac{|x-\xi-(t-\tau)\upsilon|^2}{m_{1}(t-\tau)^{3}} \right\},\\
				\int_{\mathbb{R}^{3}} G(t,x,v;\tau,\xi,\upsilon) \, dx = \frac{M_{2}}{(t-\tau)^\frac32} \exp\left\{ - \frac{|\upsilon-v|^2}{m_{2}(t-\tau)} \right\},\\
				\int_{\mathbb{R}^{6}} G(t,x,v;\tau,\xi,\upsilon) \, dxdv=\int_{\mathbb{R}^{6}} G(t,x,v;\tau,\xi,\upsilon) \, d\xi d\upsilon =1,
			\end{array}
		\end{equation*}
		with \(x,v,\xi,\upsilon \in \mathbb{R}^{3} \) and \(t>\tau \geq 0\).\\
		(ii) The following estimates hold for some positive constant $C$:
		\begin{equation*}
			\begin{array}{l}
				|\nabla_{v} G(t,x,v;\tau,\xi,\upsilon)| \leq  C \frac{G(t,x/2,v/2;\tau,\xi/2,\upsilon/2)}{(t-\tau)^\frac12},\\
				|\nabla_{\upsilon} G(t,x,v;\tau,\xi,\upsilon)| \leq  C \frac{G(t,x/2,v/2;\tau,\xi/2,\upsilon/2)}{(t-\tau)^\frac12}.
				% |\Delta_{v} G(t,x,v;\tau,\xi,\upsilon)| \leq  C \frac{G(t,x/2,v/2;\tau,\xi/2,\upsilon/2)}{(t-\tau)}.
			\end{array}
		\end{equation*}
% \begin{revision}
% We also need the following weighted estimates, which follow directly
% from the explicit formula for $G$ and Lemma~\ref{lem:G}(ii).
% Let $w(v)=1+|v|^2$. For $0<t-\tau\le1$, we have
% \begin{align}
% \int_{\mathbb R^6}
% w(v)G(t,x,v;\tau,\xi,\upsilon)\,dx\,dv
% &=w(\upsilon)+6(t-\tau),
% \label{min:G-weighted-moment}\\
% \int_{\mathbb R^6}
% w(v)\big|\nabla_\upsilon G(t,x,v;\tau,\xi,\upsilon)\big|
% \,dx\,dv
% &\le C(t-\tau)^{-\frac12}w(\upsilon).
% \label{min:G-weighted-gradient}
% \end{align}

% Indeed, integration of the explicit formula for $G$ with respect
% to $x$ gives
% \begin{align*}
% \int_{\mathbb R^3}
% G(t,x,v;\tau,\xi,\upsilon)\,dx
% =
% \frac{1}{[4\pi(t-\tau)]^{\frac32}}
% \exp\left(-\frac{|v-\upsilon|^2}{4(t-\tau)}\right).
% \end{align*}
% The Gaussian on the right has mean $\upsilon$ and variance
% $2(t-\tau)$ in each velocity component, which proves
% \eqref{min:G-weighted-moment}.
% Furthermore, Lemma~\ref{lem:G}(ii) and the change of variables
% $x'=x/2$, $v'=v/2$ yield
% \begin{align*}
% &\int_{\mathbb R^6}
% w(v)\big|\nabla_\upsilon G(t,x,v;\tau,\xi,\upsilon)\big|
% \,dx\,dv\\
% &\quad\le
% C(t-\tau)^{-\frac12}
% \int_{\mathbb R^6}
% w(v)G(t,x/2,v/2;\tau,\xi/2,\upsilon/2)\,dx\,dv\\
% &\quad=
% 2^6C(t-\tau)^{-\frac12}
% \int_{\mathbb R^6}
% (1+4|v'|^2)
% G(t,x',v';\tau,\xi/2,\upsilon/2)\,dx'\,dv'\\
% &\quad=
% 2^6C(t-\tau)^{-\frac12}
% \bigl(1+|\upsilon|^2+24(t-\tau)\bigr)\le C(t-\tau)^{-\frac12}w(\upsilon),
% \end{align*}
% where the last inequality uses $0<t-\tau\le1$.
% This proves \eqref{min:G-weighted-gradient}.
% \end{revision}
\end{lemma}

\begin{revision}
    Based on Lemma~\ref{lem:G} and the explicit formula for $G$, we obtain the following corollary.
\begin{corollary}[Weighted estimates for the Kolmogorov kernel]\label{cor:G-weighted}
Let $w(v)=1+|v|^2$. For $0<t-\tau\le1$, the following weighted estimates hold:
\begin{align}
\int_{\mathbb R^6}w(v)G(t,x,v;\tau,\xi,\upsilon)\,dxdv
&=w(\upsilon)+6(t-\tau),\label{min:G-weighted-moment}\\
\int_{\mathbb R^6}w(v)|\nabla_\upsilon G(t,x,v;\tau,\xi,\upsilon)|\,dxdv
&\le C(t-\tau)^{-1/2}w(\upsilon).\label{min:G-weighted-gradient}
\end{align}
\end{corollary}

\begin{proof}
Set $s=t-\tau>0$. Integrating the explicit formula for $G$ with respect to $x$, we obtain
\[
\int_{\mathbb R^3}G(t,x,v;\tau,\xi,\upsilon)\,dx
=\frac{1}{(4\pi s)^{3/2}}\exp\left(-\frac{|v-\upsilon|^2}{4s}\right).
\]
Making the change of variables $z=v-\upsilon$, we have
\begin{align*}
&\int_{\mathbb R^6}w(v)G(t,x,v;\tau,\xi,\upsilon)\,dxdv=\frac{1}{(4\pi s)^{3/2}}\int_{\mathbb R^3}(1+|z+\upsilon|^2)e^{-\frac{|z|^2}{4s}}\,dz\\
&\quad=\frac{1}{(4\pi s)^{3/2}}\int_{\mathbb R^3}(1+|\upsilon|^2+2z\cdot\upsilon+|z|^2)e^{-\frac{|z|^2}{4s}}\,dz.
\end{align*}
To compute the above integral, we first recall the one-dimensional Gaussian integral
\(
\int_{\mathbb R}e^{-z_i^2/(4s)}\,dz_i=\sqrt{4\pi s}.
\)
Since $z_i e^{-z_i^2/(4s)}$ is odd, we have
\(
\int_{\mathbb R}z_i e^{-z_i^2/(4s)}\,dz_i=0.
\)
Moreover, integration by parts gives
\begin{align*}
\int_{\mathbb R}z_i^2e^{-\frac{z_i^2}{4s}}\,dz_i
=-2s\int_{\mathbb R}z_i\,d\left(e^{-\frac{z_i^2}{4s}}\right)=-2s\left[z_i e^{-\frac{z_i^2}{4s}}\right]_{-\infty}^{+\infty}
+2s\int_{\mathbb R}e^{-\frac{z_i^2}{4s}}\,dz_i=2s\sqrt{4\pi s},
\end{align*}
where the boundary term vanishes since $z_i e^{-z_i^2/(4s)}\to0$ as $|z_i|\to\infty$.
Using Fubini's theorem and the identity $|z|^2=z_1^2+z_2^2+z_3^2$, we obtain
\begin{align*}
\int_{\mathbb R^3}e^{-\frac{|z|^2}{4s}}\,dz
&=\prod_{i=1}^3\int_{\mathbb R}e^{-\frac{z_i^2}{4s}}\,dz_i
=(4\pi s)^{3/2},\\
\int_{\mathbb R^3}(z\cdot\upsilon)e^{-\frac{|z|^2}{4s}}\,dz
&=\sum_{i=1}^3\upsilon_i\left(\int_{\mathbb R}z_i e^{-\frac{z_i^2}{4s}}\,dz_i\right)
\prod_{j\ne i}\left(\int_{\mathbb R}e^{-\frac{z_j^2}{4s}}\,dz_j\right)=0,
\end{align*}
and
\begin{align*}
   \int_{\mathbb R^3}|z|^2e^{-\frac{|z|^2}{4s}}\,dz
&=\sum_{i=1}^3\left(\int_{\mathbb R}z_i^2e^{-\frac{z_i^2}{4s}}\,dz_i\right)
\prod_{j\ne i}\left(\int_{\mathbb R}e^{-\frac{z_j^2}{4s}}\,dz_j\right)\\
&=\sum_{i=1}^3\left(2s\sqrt{4\pi s}\right)
\left(\sqrt{4\pi s}\right)^2=6s(4\pi s)^{3/2}. 
\end{align*}
Consequently, it holds
\begin{align*}
\int_{\mathbb R^6}w(v)G(t,x,v;\tau,\xi,\upsilon)\,dxdv&=\frac{1}{(4\pi s)^{3/2}}\left[(1+|\upsilon|^2)(4\pi s)^{3/2}+6s(4\pi s)^{3/2}\right]\\
&=1+|\upsilon|^2+6s=w(\upsilon)+6s,
\end{align*}
which implies  \eqref{min:G-weighted-moment}.

Next, by Lemma~\ref{lem:G}(ii) and the change of variables $x'=x/2$, $v'=v/2$, we obtain
\begin{align*}
&\int_{\mathbb R^6}w(v)|\nabla_\upsilon G(t,x,v;\tau,\xi,\upsilon)|\,dxdv\\
&\quad\le Cs^{-1/2}\int_{\mathbb R^6}w(v)G(t,x/2,v/2;\tau,\xi/2,\upsilon/2)\,dxdv\\
&\quad=2^6Cs^{-1/2}\int_{\mathbb R^6}(1+4|v'|^2)G(t,x',v';\tau,\xi/2,\upsilon/2)\,dx'dv'\\
&\quad=2^6Cs^{-1/2}\bigl(1+|\upsilon|^2+24s\bigr)\\
&\quad\le Cs^{-1/2}w(\upsilon),
\end{align*}
where the last inequality follows from $0<s\le1$. This proves \eqref{min:G-weighted-gradient} and completes the proof.
\end{proof}
\end{revision}

	\begin{lemma} [Lemma 4.2 in \cite{carpioWellPosednessIntegrodifferential2016}]\label{lem:f}
		Let \( \beta>0 \) and \( f \) be a nonnegative function in $L^{\infty}\left(\left(0,T\right)\times \mathbb{R}^{3}\times \mathbb{R}^{3}\right)$, such that $m_{\beta}f(t,x) < +\infty$, for a.e. (t, x). The following estimate
		holds for any $0\le\alpha < \beta$ (Lemma 1 in \cite{boudinGlobalExistenceSolutions2009}):
		\begin{align}
			\label{8}
			m_{\alpha}f(t,x) \leq \left(\frac{4}{3}\pi \|f(t,x,v)\|_{L^{\infty}((0,T)\times \mathbb{R}^{3}\times \mathbb{R}^{3})}+1\right) \left(m_{\beta}f(t,x)\right)^{\frac{\alpha+3}{\beta+3}}, \quad a.e.\ (t,x).
		\end{align}
		Moreover, the following inequalities hold:
		\begin{align}
			\label{27}
			&\left\| |v|^{\mu}f \right\|_{L^{1}_{xv}} \leq \| f \|_{L^{1}_{xv}}^{1-\frac{\mu}{\lambda}}\left\| |v|^{\lambda}f \right\|^{\frac{\mu}{\lambda}}_{L^{1}_{xv}},&\lambda > \mu >0,\\
			\label{28}
			&\left\| \int_{\mathbb{R}^{3}}  |v|^{\mu}fdv \right\|_{L_{x}^{\frac{3+\lambda}{3+\mu}}} \leq C_{\lambda,\mu} \| f \|^{\frac{\lambda-\mu}{3+\lambda}}_{L^{\infty}_{xv}}
			\left\| |v|^{\lambda}f \right\|^{\frac{3+\mu}{3+\lambda}}_{L^{1}_{xv}}, &\lambda > \mu \textcolor{black}{\geq} 0.
		\end{align}
		%where N denotes the space dimension.
	\end{lemma}

\begin{proof}
\begin{revision}
Note that \eqref{8}, \eqref{27} and \eqref{28} with $\mu>0$ have been proved in \cite{carpioWellPosednessIntegrodifferential2016}, so we omit the details for brevity. The estimate \eqref{28} remains valid for $\mu=0$, although this case is not explicitly covered in \cite{carpioWellPosednessIntegrodifferential2016}. Now we prove \eqref{28} in the specific situation $\mu=0$.
Indeed, for any $\lambda>0$ and $r>0$, splitting the velocity integral yields
\begin{align*}
m_0f(t,x)
&=\int_{|v|\le r}f(t,x,v)\,dv+\int_{|v|>r}f(t,x,v)\,dv\\
&\le \frac{4\pi}{3}r^3\|f(t)\|_{L^\infty_{xv}}+r^{-\lambda}m_\lambda f(t,x).
\end{align*}
When $\|f(t)\|_{L^\infty_{xv}}>0$ and $m_\lambda f(t,x)>0$, choosing
\[
r=\left(\frac{m_\lambda f(t,x)}{\|f(t)\|_{L^\infty_{xv}}}\right)^{\frac{1}{\lambda+3}},
\]
we obtain
\begin{align*}
m_0f(t,x)
&\le \frac{4\pi}{3}\|f(t)\|_{L^\infty_{xv}}\left(\frac{m_\lambda f(t,x)}{\|f(t)\|_{L^\infty_{xv}}}\right)^{\frac{3}{\lambda+3}}
+m_\lambda f(t,x)\left(\frac{m_\lambda f(t,x)}{\|f(t)\|_{L^\infty_{xv}}}\right)^{-\frac{\lambda}{\lambda+3}}\\
&\le C_\lambda\|f(t)\|_{L^\infty_{xv}}^{\frac{\lambda}{\lambda+3}}\bigl(m_\lambda f(t,x)\bigr)^{\frac{3}{\lambda+3}}.
\end{align*}
Raising both sides to the power $(\lambda+3)/3$ and integrating over $\mathbb R_x^3$, we have
\[
\int_{\mathbb R^3}(m_0f(t,x))^{\frac{\lambda+3}{3}}\,dx
\le C_\lambda\|f(t)\|_{L^\infty_{xv}}^{\frac{\lambda}{3}}\int_{\mathbb R^3}m_\lambda f(t,x)\,dx.
\]
Consequently,
\[
\|m_0f(t)\|_{L_x^{\frac{\lambda+3}{3}}}
\le C_\lambda\|f(t)\|_{L^\infty_{xv}}^{\frac{\lambda}{\lambda+3}}\||v|^\lambda f(t)\|_{L^1_{xv}}^{\frac{3}{\lambda+3}},
\]
which proves \eqref{28} for $\mu=0$.
\end{revision}
\end{proof}

% 	\begin{revision}
% For completeness, \eqref{28} also holds for $\mu=0$.
% Indeed, splitting the velocity integral gives, for a.e. $(t,x)$
% and every $r>0$,
% \begin{align*}
% m_0 f(t,x)
% \le \frac{4\pi}{3}r^3
% \|f(t,\cdot,\cdot)\|_{L^\infty(\mathbb R_x^3\times\mathbb R_v^3)}
% +r^{-\lambda}m_\lambda f(t,x).
% \end{align*}
% Optimizing in $r$ and then taking the
% $L^{\frac{\lambda+3}{3}}(\mathbb R_x^3)$ norm yields
% \eqref{28} with $\mu=0$.
% \end{revision}
\subsubsection{$\Phi$ is a mapping from $\Gamma$ to $\Gamma$} We make estimates on each part as follows.
    
	\par \textbf{Estimation of $\Phi_{1}$.}  By Lemma \ref{lem:semigroup}, we get
	\begin{align*}
		\|\Phi_{1}(t,\cdot)\|_
        {L_{x}^{2}} \leq & \| u_{in}^{\delta} \|_{L_{x}^{2}}+\int^{t}_{0} \left\| e^{(t-\tau) \triangle}\mathbb{P}\left(({\theta_{\varepsilon} \ast u^{\varepsilon,\delta})\cdot \nabla u^{\varepsilon,\delta}}\right)\right\|_{L_{x}^{2}} d\tau\\
		&+\int^{t}_{0} \left\| e^{(t-\tau) \triangle}\mathbb{P} \left\{\theta_{\varepsilon}\ast\int_{\mathbb{R}^{3}}vf^{\varepsilon,\delta}dv\right\} \right\|_{L_{x}^{2}} d\tau\\
		&+ \int^{t}_{0} \left\| e^{(t-\tau) \triangle}\mathbb{P} \left\{\theta_{\varepsilon}\ast\int_{\mathbb{R}^{3}}(\theta_{\varepsilon} \ast u^{\epsilon,\delta})f^{\varepsilon,\delta}dv\right\} \right\|_{L_{x}^{2}} d\tau\\
		=&\| u_{in}^{\delta} \|_{L_{x}^{2}}+\sum_{i=1}^{3}A_{1i}.
	\end{align*}
	For $A_{11}$, by using H\"{o}lder's inequality and  Lemma \ref{lem:semigroup} we have
%properties of semigroups (see Proposition A.16 of \cite{bedrossianMathematicalAnalysisIncompressible2022}) to get
	\begin{align*}
		A_{11} &\leq C\int_{0}^{t}(t-\tau)^{-\frac{3}{4}}d\tau \|(\theta_{\varepsilon}\ast u^{\varepsilon,\delta})\cdot \nabla u^{\varepsilon,\delta}\|_{L^{\infty}(0,T;L^{1}(\mathbb{R}^{3}_{x}))} \\
		&\leq CT^{\frac{1}{4}}\| u^{\varepsilon,\delta} \|_{L^{\infty}(0,T;L^{2}(\mathbb{R}^{3}_{x}))} \| \nabla u^{\varepsilon,\delta} \|_{L^{\infty}(0,T;L^{2}(\mathbb{R}^{3}_{x}))}.
	\end{align*}
	For $A_{12}$, it follows from Lemma \ref{lem:f} that
	\begin{align*}
		A_{12} &\leq C\int_{0}^{t}(t-\tau)^{-\frac{9}{20}}d\tau \big\|m_{1}|f^{\varepsilon,\delta}|\big\| _{L^{\infty}(0,T;L^{\frac{5}{4}}(\mathbb{R}^{3}_{x}))} \\
		&\leq CT^{\frac{11}{20}}\textcolor{black}{\|f^{\varepsilon,\delta}(t,x,v)\|_{L^{\infty}_{txv}}^{\frac15}}\left\| |v|^{2}f^{\varepsilon,\delta} \right\|_{L^{\infty}(0,T;L^{1}(\mathbb{R}^{6}_{xv}))}^{\frac{4}{5}}.
		% &\leq CT^{\frac{11}{20}}\left(\frac{4}{3}\pi \|f^{\varepsilon,\delta}(t,x,v)\|_{L^{\infty}_{txv}} +1\right)\left(\|f^{\varepsilon,\delta} \|_{L^{\infty}(0,T;L^{1}(\mathbb{R}^{6}_{xv}))}+\|h^{\varepsilon,\delta} \|_{L^{\infty}(0,T;L^{1}(\mathbb{R}^{6}_{xv}))}\right)^{\frac{4}{5}}.
	\end{align*}
	For $A_{13}$, using Lemma \ref{lem:f} again we  derive
    \begin{revision}
	\begin{align*}
		A_{13} &\leq C\int_{0}^{t}(t-\tau)^{-\frac{3}{4}}d\tau \left\||\theta_{\varepsilon} \ast u^{\varepsilon,\delta} | \int_{\mathbb{R}^{3}} |f^{\varepsilon,\delta}|dv\right\| _{L^{\infty}(0,T;L^{1}(\mathbb{R}^{3}_{x}))} \\
		&\leq CT^{\frac{1}{4}}\norm{\theta_{\varepsilon} \ast u^{\varepsilon,\delta} }_{L^{\infty}(0,T;L^{\frac{5}{2}}(\mathbb{R}^{3}_{x}))}\left\|\int_{\mathbb{R}^{3}} |f^{\varepsilon,\delta}|dv\right\| _{L^{\infty}(0,T;L^{\frac53}(\mathbb{R}^{3}_{x}))}\\
		&\leq CT^{\frac{1}{4}}\left(\norm{u^{\varepsilon,\delta}}_{L^{\infty}(0,T;L^{2}(\mathbb{R}^{3}_{x}))}+\norm{\nabla u^{\varepsilon,\delta}}_{L^{\infty}(0,T;L^{2}(\mathbb{R}^{3}_{x}))}\right)\|f^{\varepsilon,\delta}(t,x,v)\|_{L^{\infty}_{txv}}^\frac25\norms{|v|^2f^{\varepsilon,\delta}}_{L^{\infty}(0,T;L^{1}(\mathbb{R}^{6}_{xv}))}^{\frac{3}{5}}.
	\end{align*}
    \end{revision}
	Due to \(T\in (0,1)\), we derive
	\begin{align*}
		\|\Phi_{1}(t,\cdot)\|_{L_{x}^{2}} \leq&  \| u_{in}^{\delta} \|_{L_{x}^{2}}+C(R)T^{\frac{1}{4}}\quad \text{for all}~ t \in (0,T).
	\end{align*}
	By similar arguments as in $\|\Phi_{1}\|_{L_{x}^{2}}$, we get
	\begin{align*}
		\| \nabla_{x}\Phi_{1}(t,\cdot)\|_{L_{x}^{2}} \leq&  \| \nabla u_{in}^{\delta} \|_{L_{x}^{2}}+\int^{t}_{0} \left\| \nabla_{x}e^{(t-\tau) \triangle}\mathbb{P}\left(({\theta_{\varepsilon} \ast u^{\varepsilon,\delta})\cdot \nabla u^{\varepsilon,\delta}}\right)\right\|_{L_{x}^{2}} d\tau\\
		&+\int^{t}_{0} \left\| \nabla_{x}e^{(t-\tau) \triangle}\mathbb{P} \theta_{\varepsilon}\ast\left\{\int_{\mathbb{R}^{3}}v f^{\varepsilon,\delta}dv\right\} \right\|_{L_{x}^{2}} d\tau \\
		&+ \int^{t}_{0} \left\| \nabla_{x}e^{(t-\tau) \triangle}\mathbb{P} \theta_{\varepsilon}\ast\biggl\{\int_{\mathbb{R}^{3}} (\theta_{\varepsilon} \ast u^{\epsilon,\delta})f^{\varepsilon,\delta}dv\biggr\} \right\|_{L_{x}^{2}} d\tau\\
		\leq&\| \nabla u_{in}^{\delta} \|_{L_{x}^{2}}+\sum_{i=1}^{3}A_{2i}.
	\end{align*}
	For $A_{21}$,  by using H\"{o}lder's inequality and Lemma \ref{lem:semigroup} again, we have
	\begin{align*}
		A_{21} &\leq C\int_{0}^{t}(t-\tau)^{-\frac{1}{2}}d\tau \|(\theta_{\varepsilon}\ast u^{\varepsilon,\delta})\cdot \nabla u^{\varepsilon,\delta}\|_{L^{\infty}(0,T;L^{2}(\mathbb{R}^{3}_{x}))} \\
		&\leq C(\varepsilon)T^{\frac{1}{2}}\| u^{\varepsilon,\delta} \|_{L^{\infty}(0,T;L^{2}(\mathbb{R}^{3}_{x}))} \| \nabla u^{\varepsilon,\delta} \|_{L^{\infty}(0,T;L^{2}(\mathbb{R}^{3}_{x}))}.
	\end{align*}
    \begin{revision}
	For $A_{22}$, by Lemma \ref{lem:f} it follows that
	\begin{align*}
		A_{22} &\leq C\int_{0}^{t}(t-\tau)^{-\frac{19}{20}}d\tau \big\|m_{1}|f^{\varepsilon,\delta}|\big\| _{L^{\infty}(0,T;L^{\frac{5}{4}}(\mathbb{R}^{3}_{x}))} \\
		&\leq CT^{\frac{1}{20}}\|f^{\varepsilon,\delta}(t,x,v)\|_{L^{\infty}_{txv}}^\frac15\left\| |v|^{2}f^{\varepsilon,\delta} \right\|_{L^{\infty}(0,T;L^{1}(\mathbb{R}^{6}_{xv}))}^{\frac{4}{5}}.
		% &\leq CT^{\frac{1}{20}}\left(\frac{4}{3}\pi\|f^{\varepsilon,\delta}(t,x,v)\|_{L^{\infty}_{txv}} +1\right)\left(\|f^{\varepsilon,\delta} \|_{L^{\infty}(0,T;L^{1}(\mathbb{R}^{6}_{xv}))}+\|h^{\varepsilon,\delta} \|_{L^{\infty}(0,T;L^{1}(\mathbb{R}^{6}_{xv}))}\right)^{\frac{4}{5}}.
	\end{align*}
	For $A_{23}$, we  derive
	\begin{align*}
		A_{23} &\leq C\int_{0}^{t}(t-\tau)^{-\frac{13}{20}}d\tau \norms{\theta_{\varepsilon}\ast u^{\varepsilon,\delta} \int_{\mathbb{R}^{3}} |f^{\varepsilon,\delta}|dv} _{L^{\infty}(0,T;L^{\frac{5}{3}}(\mathbb{R}^{3}_{x}))} \\
		&\leq CT^{\frac{7}{20}}\norm{\theta_{\varepsilon}\ast u^{\varepsilon,\delta}}_{L^{\infty}((0,T)\times\mathbb{R}^{3}_{x}))}\|f^{\varepsilon,\delta}(t,x,v)\|_{L^{\infty}_{txv}}^\frac25\norms{|v|^2f^{\varepsilon,\delta}}_{L^{\infty}(0,T;L^{1}(\mathbb{R}^{6}_{xv}))}^{\frac{3}{5}}\\
		&\leq \auditmath{C(\varepsilon)}T^{\frac{7}{20}}\norm{ u^{\varepsilon,\delta}}_{L^{\infty}(0,T;L^2(\mathbb{R}^{3}_{x}))}\|f^{\varepsilon,\delta}(t,x,v)\|_{L^{\infty}_{txv}}^\frac25\norms{|v|^2f^{\varepsilon,\delta}}_{L^{\infty}(0,T;L^{1}(\mathbb{R}^{6}_{xv}))}^{\frac{3}{5}}.
	\end{align*}
    \end{revision}
	By \eqref{udelta}, we obtain
	\begin{align*}
		\|\nabla_{x}\Phi_{1}(t,\cdot)\|_{L_{x}^{2}} \leq&  \| \nabla u_{in}^{\delta} \|_{L_{x}^{2}}+C(\varepsilon,R)T^{\frac{1}{20}}\quad \text{for all}~ t \in (0,T).
	\end{align*}
	
	\textbf{Estimation of $\Phi_{2}$.} By integration by parts for  $\Phi_{2}$  and  Lemma \ref{lem:G} we have
	\begin{align*}
		\Phi_{2}(t,\cdot)=&\int_{\mathbb{R}^{6}}G(t,x,v;0,\xi,\upsilon)f_{in}^{\delta}(\xi,\upsilon)d\xi d\upsilon+\int^{t}_{0}\int_{\mathbb{R}^{6}}\nabla_{\upsilon}G(t,x,v;\tau,\xi,\upsilon)\\
		&\cdot \left[\left(\theta_{\varepsilon} \ast u^{\varepsilon,\delta} -\upsilon\gamma_{\delta}(\upsilon)\right)f^{\varepsilon,\delta}(\xi,\upsilon)\right] d\xi d\upsilon d\tau\\
		\leq&\int_{\mathbb{R}^{6}}G(t,x,v;0,\xi,\upsilon)f_{in}^{\delta}(\xi,\upsilon)d\xi d\upsilon+C\int^{t}_{0}\int_{\mathbb{R}^{6}} (t-\tau)^{-\frac{1}{2}} \\
		&G(t,x/2,v/2;\tau,\xi/2,\upsilon/2) \left|\left(\theta_{\varepsilon} \ast u^{\varepsilon,\delta} -\upsilon\gamma_{\delta}(\upsilon)\right)f^{\varepsilon,\delta}(\xi,\upsilon)\right| d\xi d\upsilon d\tau,
	\end{align*}
	The same upper bound holds for $|\Phi_2(t,\cdot)|$ by the triangle inequality in the equality above. It follows that
	\begin{align*}
		\norm{\Phi_{2}(t,\cdot)}_{L^{1}_{xv}} \leq& \norm{f_{in}^{\delta}}_{L^{1}_{xv}} + C \int^{t}_{0} (t-\tau)^{-\frac{1}{2}} \int_{\mathbb{R}^{6}} \int_{\mathbb{R}^{6}} \quad G(t, x/2, v/2; \tau, \xi/2, \upsilon/2) \\
		&\left| \left(\theta_{\varepsilon} \ast u^{\varepsilon,\delta}  - \upsilon\gamma_{\delta}(\upsilon)\right) f^{\varepsilon,\delta}(\xi, \upsilon)  \right| d\xi d\upsilon dx dv d\tau\\
		\leq & \norm{f_{in}^{\delta}}_{L^{1}_{xv}}+ C\int^{t}_{0} (t-\tau)^{-\frac{1}{2}} \norm{(\theta_{\varepsilon} \ast u^{\varepsilon,\delta}  - v\gamma_{\delta}(v)) f^{\varepsilon,\delta}}_{L^{1}_{xv}}d\tau \\
		\leq&  \norm{f_{in}^{\delta}}_{L^{1}_{xv}}+C(\varepsilon)T^{\frac{1}{2}}(\norm{f^{\varepsilon,\delta}}_{L^{\infty}(0,T;L^{1}(\mathbb{R}^{6}_{xv}))}+\norm{f^{\varepsilon,\delta}}_{L^{\infty}(0,T;L^{1}(\mathbb{R}^{6}_{xv}))}^{2}\\
		&+\norms{|v|^{2}f^{\varepsilon,\delta}}_{L^{\infty}(0,T;L^{1}(\mathbb{R}^{6}_{xv}))}+\norm{u^{\varepsilon,\delta}}_{L^{\infty}(0,T;L^{2}(\mathbb{R}^{3}_{x}))}^{2})\\
		\leq& \norm{f_{in}^{\delta}}_{L^{1}_{xv}}+C(\varepsilon,R)T^{\frac{1}{2}}\quad \text{for all}~ t \in (0,T),
	\end{align*}
and
	\begin{align*}
		\norm{\Phi_{2}(t,\cdot)}_{L^{\infty}_{xv}} \leq&\norm{f_{in}^{\delta}}_{L^{\infty}_{xv}}+C\int^{t}_{0} (t-\tau)^{-\frac{1}{2}} \norms{|\theta_{\varepsilon} \ast u^{\varepsilon,\delta}  - v\gamma_{\delta}(v)|  f^{\varepsilon,\delta}}_{L^{\infty}_{xv}}d\tau \\
		\leq& \norm{f_{in}^\delta}_{L^{\infty}_{xv}}+C(\varepsilon)T^{\frac12}\norm{u^{\varepsilon,\delta}}_{L^{\infty}(0,T;L^{2}(\mathbb{R}^{3}_{x}))} \|f^{\varepsilon,\delta}(t,x,v)\|_{L^{\infty}((0,T)\times \mathbb{R}^{3}\times \mathbb{R}^{3})}\\ &+C(\delta)T^{\frac{1}{2}}\|f^{\varepsilon,\delta}(t,x,v)\|_{L^{\infty}((0,T)\times \mathbb{R}^{3}\times \mathbb{R}^{3})}\\
		\leq& \norm{f_{in}^{\delta}}_{L^{\infty}_{xv}}+C(\varepsilon,\delta,R)T^{\frac{1}{2}}\quad \text{for all}~ t \in (0,T).
	\end{align*}

\begin{revision}
\textbf{Estimation of $\Phi_3$.}
By the definition of $\Phi_3$, we have $\Phi_3=|v|^2\Phi_2$. Since $|v|^2=w(v)-1$, \eqref{min:G-weighted-moment} and Lemma~\ref{lem:G}(i) imply
\begin{align*}
\int_{\mathbb R^6}|v|^2G(t,x,v;\tau,\xi,\upsilon)\,dxdv
&=\int_{\mathbb R^6}w(v)G(t,x,v;\tau,\xi,\upsilon)\,dxdv-1\\
&=|\upsilon|^2+6(t-\tau).
\end{align*}
Moreover, since $|v|^2\le w(v)$, \eqref{min:G-weighted-gradient} yields
\[
\int_{\mathbb R^6}|v|^2|\nabla_\upsilon G(t,x,v;\tau,\xi,\upsilon)|\,dxdv
\le C(t-\tau)^{-1/2}w(\upsilon).
\]
Consequently, using the integral representation of $\Phi_2$, Fubini's theorem, and the above two estimates, we obtain
\begin{align*}
&\|\Phi_3(t)\|_{L^1_{xv}}=\||v|^2\Phi_2(t)\|_{L^1_{xv}}\\
&\quad\le\int_{\mathbb R^6}\bigl(|\upsilon|^2+6t\bigr)|f_{in}^\delta(\xi,\upsilon)|\,d\xi d\upsilon\\
&\qquad+C\int_0^t(t-s)^{-1/2}\int_{\mathbb R^6}w(\upsilon)|\theta_\varepsilon*u^{\varepsilon,\delta}(s,\xi)-\upsilon\gamma_\delta(\upsilon)||f^{\varepsilon,\delta}(s,\xi,\upsilon)|\,d\xi d\upsilon ds\\
&\quad=\|h_{in}^\delta\|_{L^1_{xv}}+6t\|f_{in}^\delta\|_{L^1_{xv}}\\
&\qquad+C\int_0^t(t-s)^{-1/2}\|w(\theta_\varepsilon*u^{\varepsilon,\delta}-v\gamma_\delta)f^{\varepsilon,\delta}(s)\|_{L^1_{xv}}\,ds.
\end{align*}
By Young's convolution inequality and the definition of $\gamma_\delta$, we have
\[
\|\theta_\varepsilon*u^{\varepsilon,\delta}\|_{L^\infty_{t,x}}\le\|\theta_\varepsilon\|_{L^2_x}\|u^{\varepsilon,\delta}\|_{L^\infty_tL^2_x}\le C_\varepsilon R,\qquad
\|v\gamma_\delta(v)\|_{L^\infty_v}\le C_\delta.
\]
It follows that
\begin{align*}
\|w(\theta_\varepsilon*u^{\varepsilon,\delta}-v\gamma_\delta)f^{\varepsilon,\delta}(s)\|_{L^1_{xv}}
&\le(C_\varepsilon R+C_\delta)\int_{\mathbb R^6}(1+|v|^2)|f^{\varepsilon,\delta}(s,x,v)|\,dxdv\\
&=(C_\varepsilon R+C_\delta)\bigl(\|f^{\varepsilon,\delta}(s)\|_{L^1_{xv}}+\|h^{\varepsilon,\delta}(s)\|_{L^1_{xv}}\bigr).
\end{align*}
Therefore, recalling that $\int_0^t(t-s)^{-1/2}\,ds=2\sqrt t$, we obtain
\begin{align*}
\|\Phi_3(t)\|_{L^1_{xv}}
&\le\|h_{in}^\delta\|_{L^1_{xv}}+6T\|f_{in}^\delta\|_{L^1_{xv}}\\
&\quad+2C\sqrt T(C_\varepsilon R+C_\delta)\bigl(\|f^{\varepsilon,\delta}\|_{L^\infty_tL^1_{xv}}+\|h^{\varepsilon,\delta}\|_{L^\infty_tL^1_{xv}}\bigr)\\
&\le\|h_{in}^\delta\|_{L^1_{xv}}+C(\varepsilon,\delta,R)\sqrt T.
\end{align*}
This estimate ensures that the third component remains bounded while preserving the identity $\Phi_3=|v|^2\Phi_2$. Combining the estimates for $\Phi_1$, $\Phi_2$, and $\Phi_3$, we may choose $T>0$ sufficiently small so that the sum of the three bounds is at most $R$.
\end{revision}

% \begin{revision}
% \textbf{Estimation of $\Phi_3$.}
% By the definition of $\Phi_3$ and
% \eqref{min:G-weighted-moment}--\eqref{min:G-weighted-gradient},
% taking $0<T\le1$ and writing
% $U=\theta_\varepsilon*u^{\varepsilon,\delta}$ and $w(v)=1+|v|^2$,
% we obtain
% \begin{align*}
%  &\|\Phi_3(t)\|_{L^1_{xv}(\mathbb{R}^3\times\mathbb{R}^3)}
%  =\||v|^2\Phi_2(t)\|_{L^1_{xv}(\mathbb{R}^3\times\mathbb{R}^3)}\\
%  &\le\||v|^2f_{in}^\delta\|_{L^1_{xv}(\mathbb{R}^3\times\mathbb{R}^3)}+6t\|f_{in}^\delta\|_{L^1_{xv}(\mathbb{R}^3\times\mathbb{R}^3)}
%  +C\int_0^t(t-s)^{-\frac12}
%         \|w(\theta_\varepsilon*u^{\varepsilon,\delta}-v\gamma_\delta)f^{\varepsilon,\delta}(s)\|_{L^1_{xv}(\mathbb{R}^3\times\mathbb{R}^3)}\,ds\\
%  &\le\|h_{in}^\delta\|_{L^1_{xv}(\mathbb{R}^3\times\mathbb{R}^3)}+6T\|f_{in}^\delta\|_{L^1_{xv}(\mathbb{R}^3\times\mathbb{R}^3)}
%  +C\sqrt T(C_\varepsilon R+C_\delta)
%        \big(\|f^{\varepsilon,\delta}\|_{L^\infty_tL^1_{xv}}
%            +\|h^{\varepsilon,\delta}\|_{L^\infty_tL^1_{xv}}\big)\\
%  &\le\|h_{in}^\delta\|_{L^1_{xv}(\mathbb{R}^3\times\mathbb{R}^3)}+C(\varepsilon,\delta,R)\sqrt T.
% \end{align*}
% Here $\|\theta_\varepsilon*u^{\varepsilon,\delta}\|_{L^\infty_{txv}}\le C_\varepsilon\|u^{\varepsilon,\delta}\|_{L_t^\infty L^2_{xv}}$
% and $\|v\gamma_\delta\|_{L^\infty_{txv}}\le C_\delta$. This estimate replaces the
% independent third-component evolution and preserves $\Phi_3=|v|^2\Phi_2$.
% Choose $T$ small enough that the sum of the three output bounds is at most $R$.
% \end{revision}

	Combining  the estimates of $\Phi_{1}$, $\Phi_{2}$, and $\Phi_{3}$, we get that $\Phi$ is a mapping from \(\Gamma\) to \(\Gamma\).

	\subsubsection{$\Phi$ is a contraction mapping on \(\Gamma\)}
	\par 	Using the same way, we obtain the following estimates for ($u^{\varepsilon,\delta}_{1}$, $f^{\varepsilon, \delta}_{1}$, $h_{1}^{\varepsilon,\delta}$) , ($u^{\varepsilon,\delta}_{2}$, $f^{\varepsilon,\delta}_{2}$, $h_{2}^{\varepsilon,\delta}$) \( \in \Gamma\).

\textbf{Estimation of $\Phi_{1}$.} For the $L^{2}$ norm of \(u^{\varepsilon,\delta}\), similar to the estimates of  \(A_{11}-A_{13}\), we have
	\begin{align*}
		&\|\Phi_{1}(u_{1}^{\varepsilon,\delta},f_{1}^{\varepsilon,\delta},h_{1}^{\varepsilon,\delta} ) - \Phi_{1}(u_{2}^{\varepsilon,\delta},f_{2}^{\varepsilon,\delta},h_{2}^{\varepsilon,\delta} ) \|_{L_{x}^{2}}=\\
		&\bigg\|-\int^{t}_{0}e^{(t-\tau) \triangle} \mathbb{P}\bigg\{(\theta_{\varepsilon}\ast u_{2}^{\varepsilon,\delta}-\theta_{\varepsilon}\ast u_{1}^{\varepsilon,\delta})\cdot \nabla u_1^{\varepsilon,\delta}+(\theta_{\varepsilon}\ast u_{2}^{\varepsilon,\delta})\cdot \nabla_{x} (u_{2}^{\varepsilon,\delta}-u_{1}^{\varepsilon,\delta})\bigg\}d\tau \\
		&+\int_{0}^{t} e^{(t-\tau) \triangle}\mathbb{P}\bigg\{\theta_{\varepsilon}\ast\int_{\mathbb{R}_{v}^3}v(f_{2}^{\varepsilon,\delta}-f_{1}^{\varepsilon,\delta})dv\bigg\}d\tau\\
        &+\int_{0}^{t} e^{(t-\tau) \triangle}\mathbb{P}\bigg\{\theta_{\varepsilon}\ast \int_{\mathbb{R}_{v}^3}(\theta_{\varepsilon}\ast u_{2}^{\varepsilon,\delta})(f_{1}^{\varepsilon,\delta}-f_{2}^{\varepsilon,\delta})dv\bigg\}d\tau\\
		&+\int_{0}^{t} e^{(t-\tau) \triangle}\mathbb{P}\bigg\{\theta_{\varepsilon}\ast\int_{\mathbb{R}_{v}^3}(\theta_{\varepsilon}\ast u_{1}^{\varepsilon,\delta}-\theta_{\varepsilon}\ast u_{2}^{\varepsilon,\delta})f_{1}^{\varepsilon,\delta}dv \bigg\}d\tau\bigg\|_{L^2_{x}}\\
		\leq& CT^{\frac{1}{4}}\{\norm{u_{2}^{\varepsilon,\delta}-u_{1}^{\varepsilon,\delta}}_{L^{\infty}(0,T;L^{2}(\mathbb{R}^{3}_{x}))}\norm{\nabla u_{1}^{\varepsilon,\delta}}_{L^{\infty}(0,T;L^{2}(\mathbb{R}^{3}_{x}))}\\
		&\hspace{3cm}+\norm{u_{2}^{\varepsilon,\delta}}_{L^{\infty}(0,T;L^{2}(\mathbb{R}^{3}_{x}))}\norm{\nabla u_{2}^{\varepsilon,\delta}-\nabla u_{1}^{\varepsilon,\delta}}_{L^{\infty}(0,T;L^{2}(\mathbb{R}^{3}_{x}))}\}\\
		&+CT^{\frac{11}{20}}\rev{\norm{f_{2}^{\varepsilon,\delta}(t,x,v)-f_{1}^{\varepsilon,\delta}(t,x,v)}_{L^{\infty}((0,T)\times \mathbb{R}^{3}\times \mathbb{R}^{3})}^{\frac{1}{5}}}\left\| |v|^2f_{2}^{\varepsilon,\delta}-|v|^2f_{1}^{\varepsilon,\delta} \right\|_{L^{\infty}(0,T;L^{1}(\mathbb{R}^{6}_{xv}))}^{\frac{4}{5}}\\
        &+CT^{\frac{1}{4}}\rev{\norm{{f_{1}^{\varepsilon,\delta}(t,x,v)}-f_{2}^{\varepsilon,\delta}(t,x,v)}_{L^{\infty}((0,T)\times \mathbb{R}^{3}\times \mathbb{R}^{3})}^{\frac{2}{5}}}\times \\
		& \hspace{3cm} \left(\norm{ u_{2}^{\varepsilon,\delta}}_{L^{\infty}(0,T;L^{2}(\mathbb{R}^{3}_{x}))}+\norm{\nabla u_{2}^{\varepsilon,\delta}}_{L^{\infty}(0,T;L^{2}(\mathbb{R}^{3}_{x}))}\right) \left\| |v|^2f_{1}^{\varepsilon,\delta}-|v|^2f_{2}^{\varepsilon,\delta} \right\|_{L^{\infty}(0,T;L^{1}(\mathbb{R}^{6}_{xv}))}^{\frac{3}{5}}\\
		&+CT^{\frac{1}{4}}\textcolor{black}{\norm{f_{1}^{\varepsilon,\delta}(t,x,v)}_{L^{\infty}((0,T)\times \mathbb{R}^{3}\times \mathbb{R}^{3})}^\frac25} \left\| |v|^2f_{1}^{\varepsilon,\delta} \right\|_{L^{\infty}(0,T;L^{1}(\mathbb{R}^{6}_{xv}))}^{\frac{3}{5}}\times\\
		&\hspace{3cm}\Big(\norm{(u_{1}^{\varepsilon,\delta}-u_{2}^{\varepsilon,\delta})}_{L^{\infty}(0,T;L^{2}(\mathbb{R}^{3}_{x}))}+\norm{\nabla u_{1}^{\varepsilon,\delta}-\nabla u_{2}^{\varepsilon,\delta}}_{L^{\infty}(0,T;L^{2}(\mathbb{R}^{3}_{x}))}\Big) \\
		\leq& C(R)T^{\frac{1}{4}}\norm{(u_{1}^{\varepsilon,\delta},f_{1}^{\varepsilon,\delta},h_{1}^{\varepsilon,\delta})-(u_{2}^{\varepsilon,\delta},f_{2}^{\varepsilon,\delta},h_{2}^{\varepsilon,\delta})}_{\Gamma}.
	\end{align*}
	For the $L^{2}$ norm of $\nabla u^{\varepsilon,\delta}$, similar to \(A_{21}-A_{23}\), we have
	\begin{align*}
		\|\nabla_{x}\Phi_{1}&(u_{1}^{\varepsilon,\delta},f_{1}^{\varepsilon,\delta},h_{1}^{\varepsilon,\delta} ) - \nabla_{x}\Phi_{1}(u_{2}^{\varepsilon,\delta},f_{2}^{\varepsilon,\delta},h_{2}^{\varepsilon,\delta} ) \|_{L_{x}^{2}}=\\
		\bigg\|&-\int^{t}_{0}\nabla_{x} e^{(t-\tau) \triangle} \mathbb{P}\left\{(\theta_{\varepsilon}\ast u_{2}^{\varepsilon,\delta}-\theta_{\varepsilon}\ast u_{1}^{\varepsilon,\delta})\cdot \nabla u_{1}^{\varepsilon,\delta}\right\}d\tau\\
		&-\int^{t}_{0}\nabla_{x} e^{(t-\tau) \triangle} \mathbb{P}\left\{(\theta_{\varepsilon}\ast u_{2}^{\varepsilon,\delta})\cdot \nabla_{x}( u_{2}^{\varepsilon,\delta}-u_{1}^{\varepsilon,\delta})\right\}d\tau \\
		&+\int_{0}^{t} \nabla_{x} e^{(t-\tau) \triangle}\mathbb{P}\theta_{\varepsilon}\ast\left\{\int_{\mathbb{R}_{v}^3}v(f_{2}^{\varepsilon,\delta}-f_{1}^{\varepsilon,\delta})dv + (\theta_{\varepsilon}\ast u_{2}^{\varepsilon,\delta})\int_{\mathbb{R}_{v}^3}(f_{1}^{\varepsilon,\delta}-f_{2}^{\varepsilon,\delta})dv\right\}d\tau \\
		&+\int_{0}^{t} \nabla_{x} e^{(t-\tau) \triangle}\mathbb{P}\theta_{\varepsilon}\ast\left\{(\theta_{\varepsilon}\ast u_{1}^{\varepsilon,\delta}-\theta_{\varepsilon}\ast u_{2}^{\varepsilon,\delta})\int_{\mathbb{R}_{v}^3}f_{1}^{\varepsilon,\delta}dv\right\}d\tau\bigg\|_{L^2_{x}}\\
		\leq& C(\varepsilon)T^{\frac{1}{2}}\bigg\{\norm{u_{2}^{\varepsilon,\delta}-u_{1}^{\varepsilon,\delta}}_{L^{\infty}(0,T;L^{2}(\mathbb{R}^{3}_{x}))}\norm{\nabla u_{1}^{\varepsilon,\delta}}_{L^{\infty}(0,T;L^{2}(\mathbb{R}^{3}_{x}))}\\
		&\hspace{5cm}+\norm{u_{2}^{\varepsilon,\delta}}_{L^{\infty}(0,T;L^{2}(\mathbb{R}^{3}_{x}))}\norm{\nabla u_{2}^{\varepsilon,\delta}-\nabla u_{1}^{\varepsilon,\delta}}_{L^{\infty}(0,T;L^{2}(\mathbb{R}^{3}_{x}))}\bigg\}\\
		&+CT^{\frac{1}{20}}\rev{\norm{f_{2}^{\varepsilon,\delta}(t,x,v)-f_{1}^{\varepsilon,\delta}(t,x,v)}_{L^{\infty}_{txv}}^{\frac{1}{5}}}\left\| |v|^2f_{2}^{\varepsilon,\delta}-|v|^2f_{1}^{\varepsilon,\delta} \right\|_{L^{\infty}(0,T;L^{1}(\mathbb{R}^{6}_{xv}))}^{\frac{4}{5}}\\
        &+C(\varepsilon)T^{\frac{7}{20}}\rev{\norm{{f_{1}^{\varepsilon,\delta}(t,x,v)}-f_{2}^{\varepsilon,\delta}(t,x,v)}_{L^{\infty}_{txv}}^{\frac{2}{5}}} \\
		& \hspace{3cm} \times\norm{u_{2}^{\varepsilon,\delta}}_{L^{\infty}(0,T;L^{2}_x)}
		\left\| |v|^2f_{1}^{\varepsilon,\delta}-|v|^2f_{2}^{\varepsilon,\delta} \right\|_{L^{\infty}(0,T;L^{1}(\mathbb{R}^{6}_{xv}))}^{\frac{3}{5}}\\
		&+C(\varepsilon)T^{\frac{7}{20}}\textcolor{black}{\norm{f_{1}^{\varepsilon,\delta}(t,x,v)}_{L^{\infty}((0,T)\times \mathbb{R}^{3}\times \mathbb{R}^{3})}^\frac25}\norm{u_{1}^{\varepsilon,\delta}- u_{2}^{\varepsilon,\delta}}_{L^{\infty}(0,T;L^{2}_x)}\left\| |v|^2f_{1}^{\varepsilon,\delta} \right\|_{L^{\infty}(0,T;L^{1}(\mathbb{R}^{6}_{xv}))}^{\frac{3}{5}}\\
		\leq& C(\varepsilon,R)T^{\frac{1}{20}}\norm{(u_{1}^{\varepsilon,\delta},f_{1}^{\varepsilon,\delta},h_{1}^{\varepsilon,\delta})-(u_{2}^{\varepsilon,\delta},f_{2}^{\varepsilon,\delta},h_{2}^{\varepsilon,\delta})}_{\Gamma}.
	\end{align*}
	\textbf{Estimation of $\Phi_{2}$.} For the $L^{1}$ norm of $ f^{\varepsilon,\delta}$, we have
	\begin{align*}
		\|\Phi_{2}&(u_{1}^{\varepsilon,\delta},f_{1}^{\varepsilon,\delta},h_{1}^{\varepsilon,\delta} ) - \Phi_{2}(u_{2}^{\varepsilon,\delta},f_{2}^{\varepsilon,\delta},h_{2}^{\varepsilon,\delta} ) \|_{L_{xv}^{1}}=\\
		&\bigg\|\int^{t}_{0}\int_{\mathbb{R}^{6}} \rev{\nabla_\upsilon G(t,x,v;\tau,\xi,\upsilon)}\cdot \\
		& \hspace{2cm}\left[\left(\theta_{\varepsilon}\ast u_{1}^{\varepsilon,\delta}-\upsilon\gamma_{\delta}(\upsilon)\right)f_{1}^{\varepsilon,\delta}(\xi,\upsilon)-\left(\theta_{\varepsilon}\ast u_{2}^{\varepsilon,\delta}-\upsilon\gamma_{\delta}(\upsilon)\right)f_{2}^{\varepsilon,\delta}(\xi,\upsilon)\right] d\xi d\upsilon d\tau\bigg\|_{L^1_{xv}}\\
		\leq& C(\varepsilon)T^{\frac{1}{2}}\bigg(\norm{u_{1}^{\varepsilon,\delta}-u_{2 }^{\varepsilon,\delta}}_{L^{\infty}(0,T;L^{2}(\mathbb{R}^{3}_{x}))}\norm{f_{1}^{\varepsilon,\delta}} _{L^{\infty}(0,T;L^{1}_{xv}(\mathbb{R}^{3}\times \mathbb{R}^{3}))} \\ &+\norm{u_{2}^{\varepsilon,\delta}}_{L^{\infty}(0,T;L^{2}(\mathbb{R}^{3}_{x}))}\norm{f_{1}^{\varepsilon,\delta}-f_{2}^{\varepsilon,\delta}} _{L^{\infty}(0,T;L^{1}_{xv}(\mathbb{R}^{3}\times \mathbb{R}^{3}))} \bigg)\\
        &+CT^{\frac{1}{2}}\bigg(\norm{f_{2}^{\varepsilon,\delta}-f_{1}^{\varepsilon,\delta}} _{L^{\infty}(0,T;L^{1}_{xv}(\mathbb{R}^{3}\times \mathbb{R}^{3}))}+\norms{|v|^2f_{2}^{\varepsilon,\delta}-|v|^2f_{1}^{\varepsilon,\delta}} _{L^{\infty}(0,T;L^{1}_{xv}(\mathbb{R}^{3}\times \mathbb{R}^{3}))} \bigg)\\
		\leq& C(\varepsilon,R)T^{\frac{1}{2}}\norm{(u_{1}^{\varepsilon,\delta},f_{1}^{\varepsilon,\delta},h_{1}^{\varepsilon,\delta})-(u_{2}^{\varepsilon,\delta},f_{2}^{\varepsilon,\delta},h_{2}^{\varepsilon,\delta})}_{\Gamma}.
	\end{align*}		
	For the $L^{\infty}$ norm of $ f^{\varepsilon,\delta}$, we have
	\begin{align*}
		\|\Phi_{2}&(u_{1}^{\varepsilon,\delta},f_{1}^{\varepsilon,\delta},h_{1}^{\varepsilon,\delta} ) - \Phi_{2}(u_{2}^{\varepsilon,\delta},f_{2}^{\varepsilon,\delta},h_{2}^{\varepsilon,\delta} ) \|_{L_{xv}^{\infty}}=\\
		\bigg\|	&\int^{t}_{0}\int_{\mathbb{R}^{6}} \rev{\nabla_\upsilon G(t,x,v;\tau,\xi,\upsilon)}\cdot \\
		& \hspace{2cm}\left[\left(\theta_{\varepsilon}\ast u_{1}^{\varepsilon,\delta}-\upsilon\gamma_{\delta}(\upsilon)\right)f_{1}^{\varepsilon,\delta}(\xi,\upsilon)-\left(\theta_{\varepsilon}\ast u_{2}^{\varepsilon,\delta}-\upsilon\gamma_{\delta}(\upsilon)\right)f_{2}^{\varepsilon,\delta}(\xi,\upsilon)\right] d\xi d\upsilon d\tau\bigg\|_{L^{\infty}_{xv}}\\
		\leq& C(\varepsilon)T^{\frac{1}{2}}\bigg(\norm{u_{1}^{\varepsilon,\delta}-u_{2 }^{\varepsilon,\delta}}_{L^{\infty}(0,T;L^{2}(\mathbb{R}^{3}_{x}))}\norm{f_{1}^{\varepsilon,\delta}} _{L^{\infty}((0,T)\times \mathbb{R}^{3}\times \mathbb{R}^{3})} \\ &+\norm{u_{2}^{\varepsilon,\delta}}_{L^{\infty}(0,T;L^{2}(\mathbb{R}^{3}_{x}))}\norm{f_{1}^{\varepsilon,\delta}-f_{2}^{\varepsilon,\delta}} _{L^{\infty}((0,T)\times \mathbb{R}^{3}\times \mathbb{R}^{3})} \bigg)+C(\delta)T^{\frac{1}{2}}\norm{f_{2}^{\varepsilon,\delta}-f_{1}^{\varepsilon,\delta}} _{L^{\infty}((0,T)\times \mathbb{R}^{3}\times \mathbb{R}^{3})}\\
		\leq& C(\varepsilon,\delta,R)T^{\frac{1}{2}}\norm{(u_{1}^{\varepsilon,\delta},f_{1}^{\varepsilon,\delta},h_{1}^{\varepsilon,\delta})-(u_{2}^{\varepsilon,\delta},f_{2}^{\varepsilon,\delta},h_{2}^{\varepsilon,\delta})}_{\Gamma}.
	\end{align*}

\begin{revision}
\textbf{Estimation of $\Phi_3$.}
Since $\Phi_3=|v|^2\Phi_2$ and the initial data are the same,
the initial terms cancel in the difference.
Applying \eqref{min:G-weighted-gradient}, we obtain,
for $0<T\le1$,
\begin{align*}
&\|\Phi_3(u_1^{\varepsilon,\delta},f_1^{\varepsilon,\delta},h_1^{\varepsilon,\delta})-\Phi_3(u_2^{\varepsilon,\delta},f_2^{\varepsilon,\delta},h_2^{\varepsilon,\delta})\|_{L_t^\infty L_{xv}^1}
\\
&\quad\le CT^\frac12
\Big\|(1+|v|^2)
\big[(\theta_\varepsilon*u_1^{\varepsilon,\delta}-v\gamma_\delta)f_1^{\varepsilon,\delta}
-(\theta_\varepsilon*u_2^{\varepsilon,\delta}-v\gamma_\delta)f_2^{\varepsilon,\delta}\big]
\Big\|_{L_t^\infty L_{xv}^1}
\\
&\quad\le CT^\frac12
\Big\|(1+|v|^2)
(\theta_\varepsilon*u_1^{\varepsilon,\delta}-v\gamma_\delta)(f_1^{\varepsilon,\delta}-f_2^{\varepsilon,\delta})
\Big\|_{L_t^\infty L_{xv}^1}
\\
&\qquad
+CT^\frac12
\Big\|(1+|v|^2)
\big[\theta_\varepsilon*(u_1^{\varepsilon,\delta}-u_2^{\varepsilon,\delta})\big]f_2^{\varepsilon,\delta}
\Big\|_{L_t^\infty L_{xv}^1}
\\
&\quad\le CT^\frac12
\Big(
\|\theta_\varepsilon\|_{L_x^2}
\|u_1^{\varepsilon,\delta}\|_{L_t^\infty L_x^2}
+\|v\gamma_\delta\|_{L_v^\infty}
\Big)
\Big(
\|f_1^{\varepsilon,\delta}-f_2^{\varepsilon,\delta}\|_{L_t^\infty L_{xv}^1}
+\|h_1^{\varepsilon,\delta}-h_2^{\varepsilon,\delta}\|_{L_t^\infty L_{xv}^1}
\Big)
\\
&\qquad
+CT^\frac12\,
\|\theta_\varepsilon\|_{L_x^2}
\|u_1^{\varepsilon,\delta}-u_2^{\varepsilon,\delta}\|_{L_t^\infty L_x^2}
\Big(
\|f_2^{\varepsilon,\delta}\|_{L_t^\infty L_{xv}^1}
+\|h_2^{\varepsilon,\delta}\|_{L_t^\infty L_{xv}^1}
\Big)
\\
&\quad\le C(\varepsilon,\delta,R)T^\frac12\,
\|(u_1^{\varepsilon,\delta},f_1^{\varepsilon,\delta},h_1^{\varepsilon,\delta})-(u_2^{\varepsilon,\delta},f_2^{\varepsilon,\delta},h_2^{\varepsilon,\delta})\|_\Gamma .
\end{align*}
\end{revision}
	\par  Finally, we  obtain
	\begin{gather*}
		\|\Phi(u_{1}^{\varepsilon,\delta},f_{1}^{\varepsilon,\delta},h_{1}^{\varepsilon,\delta} ) - \Phi(u_{2}^{\varepsilon,\delta},f_{2}^{\varepsilon,\delta},h_{2}^{\varepsilon,\delta} ) \|_{\Gamma}\leq C(\varepsilon,\delta,R)T^{\frac{1}{20}}\norm{(u_{1}^{\varepsilon,\delta},f_{1}^{\varepsilon,\delta},h_{1}^{\varepsilon,\delta})-(u_{2}^{\varepsilon,\delta},f_{2}^{\varepsilon,\delta},h_{2}^{\varepsilon,\delta})}_{\Gamma},
	\end{gather*}
	for all $t\in (0,T)$, which implies that  $\Phi$ acts as a contraction on $\Gamma$ by choosing a  sufficiently small $T$. Then, one can conclude the existence of $(u^{\varepsilon,\delta},f^{\varepsilon,\delta} ,h^{\varepsilon,\delta} )$ as the fixed point of $\Phi$ on $\Gamma$, namely $(u^{\varepsilon,\delta},f^{\varepsilon,\delta} ,h^{\varepsilon,\delta} ) = \Phi(u^{\varepsilon,\delta},f^{\varepsilon,\delta}, h^{\varepsilon,\delta} )$.

\subsubsection{The positivity of $f^{\varepsilon,\delta}$}
	\par Firstly, set \( f^{\varepsilon,\delta}_{+} = \max(0, f^{\varepsilon,\delta}) \) and \( -f^{\varepsilon,\delta}_{-} = \min(0, f^{\varepsilon,\delta}) \). Then, \( f^{\varepsilon,\delta} = f^{\varepsilon,\delta}_{+} - f^{\varepsilon,\delta}_{-} \). Multiplying the second equation of \eqref{eq:NSVFK0} by \( f^{\varepsilon,\delta}_{-} \), we derive
	\begin{align*}
		\displaystyle \partial_t f^{\varepsilon,\delta}_{-}f^{\varepsilon,\delta}_{-} + (v\cdot \nabla_x) f^{\varepsilon,\delta}_{-}f^{\varepsilon,\delta}_{-}+ \nabla_v \cdot \left[\left(\theta_{\varepsilon} \ast u^{\varepsilon,\delta}-v\gamma_
		\delta(v)\right)f^{\varepsilon,\delta}_{-}-\nabla_v f^{\varepsilon,\delta}_{-}\right] f^{\varepsilon,\delta}_{-}= 0.
	\end{align*}
	Then, by integrating by parts on the above equation, we obtain
	\begin{align*}
		\displaystyle \frac{1}{2}\partial_t \int_{\mathbb{R}^{6}} (f^{\varepsilon,\delta}_{-})^{2}dxdv + \int_{\mathbb{R}^{6}}|\nabla_v f^{\varepsilon,\delta}_{-}|^{2}dxdv +\frac{1}{2}\int_{\mathbb{R}^{6}}\nabla_v (f^{\varepsilon,\delta}_{-})^{2}
		\cdot (v\gamma_
		\delta(v)-\theta_{\varepsilon} \ast u^{\varepsilon,\delta})dxdv= 0.
	\end{align*}
	From this equation and \(v\cdot\nabla_v \gamma_{\delta}(v)\leq 0\), we know that
	\begin{align*}
		&\int_{\mathbb{R}^{6}}\nabla_v (f^{\varepsilon,\delta}_{-})^{2}
		\cdot (v\gamma_
		\delta(v)-\theta_{\varepsilon} \ast u^{\varepsilon,\delta})dxdv\\
		&=-\int_{\mathbb{R}^{6}} (f^{\varepsilon,\delta}_{-})^{2} \nabla_{v}
		\cdot (v\gamma_\delta(v))dxdv\\
		&\geq -3\int_{\mathbb{R}^6} (f^{\varepsilon,\delta}_{-})^{2} dxdv.
	\end{align*}
	Hence we  get
	\begin{align*}
		\displaystyle \partial_t \int_{\mathbb{R}^{6}} (f^{\varepsilon,\delta}_{-})^{2}dxdv + 2\int_{\mathbb{R}^{6}}|\nabla_v f^{\varepsilon,\delta}_{-}|^{2}dxdv \leq 6\int_{\mathbb{R}^{6}}(f^{\varepsilon,\delta}_{-})^{2}dxdv.
	\end{align*}
	Using Gr\"{o}nwall's inequality (Lemma \ref{gronwall}) and \eqref{fdelta}, we can derive
	\begin{align}\label{f-}
		\int_{\mathbb{R}^{6}} (f^{\varepsilon,\delta}_{-})^{2}dxdv \leq e^{6T}\norm{f_{-}^{\delta}(\cdot,0)}^{2}_{L^{2}_{xv}}\leq e^{6T}\norm{f_{-}(\cdot,0)}^{2}_{L^{2}_{xv}}.
	\end{align}
	Note that \(f_{-}(\cdot,0) =0\) due to \(f_{in} \geqslant 0\), which means \(f^{\varepsilon,\delta}_{-}=0\). Thus \(f^{\varepsilon,\delta}=f^{\varepsilon,\delta}_{+}\), and we have \(f^{\varepsilon,\delta} \geqslant 0\).
	
	%So integrate by parts for the second euqation of (2.1), we can get
	%\begin{align*}
	%	\frac{1}{p}\partial_t &\int_{\mathbb{R}^{6}} (f^{\varepsilon,\delta})^{p}dxdv + \frac{4(p-1)}{p^2}\int_{\mathbb{R}^{6}}|\nabla_{v}(f^{\varepsilon,\delta})^{\frac{p}{2}}|^{2}dxdv \\
	%&\leq \max\{0,\int_{\mathbb{R}^{6}}(f^{\varepsilon,\delta})^{p-1}\nabla_v \cdot [(v-\theta_{\varepsilon} \ast u^{\varepsilon,\delta})f^{\varepsilon,\delta}]dxdv\}\\
	%	& \hspace{3cm} +\int_{\mathbb{R}^{6}}(f^{\varepsilon,\delta})^{p} (v-\theta_{\varepsilon} \ast u^{\varepsilon,\delta})\cdot \nabla_{v}\gamma_
	%\delta(v)dxdv.\\
	%	&= B_{11}+B_{12}.
	%\end{align*}
	%For \(B_{11}\), by integrate by parts, we have
	%\begin{align*}
	%B_{11} &\leq \max\{0, \int_{\mathbb{R}^{6}}\nabla_v(f^{\varepsilon,\delta})^{p-1} \cdot [(\theta_{\varepsilon} \ast u^{\varepsilon,\delta}-v)f^{\varepsilon,\delta}]dxdv\}\\
	%&\leq \max\{0, \frac{1}{p}\int_{\mathbb{R}^{6}}\nabla_v(f^{\varepsilon,\delta})^{p} \cdot [(\theta_{\varepsilon} \ast u^{\varepsilon,\delta}-v)]dxdv\}\\
	%&\leq \frac{3}{p}\int_{\mathbb{R}^{6}} (f^{\varepsilon,\delta})^{p}dxdv.
	%\end{align*}
	%For \(B_{12}\), by similar method with \(B_{11}\), we get
	%\begin{align*}
	%B_{12}&= \int_{\mathbb{R}^{6}}\nabla_{v}\cdot [(f^{\varepsilon,\delta})^{p} (v-\theta_{\varepsilon} \ast u^{\varepsilon,\delta})]\gamma_
	%\delta(v)dxdv\\
	%&\leq \max\{0,\int_{\mathbb{R}^{6}}\nabla_{v}\cdot [(f^{\varepsilon,\delta})^{p} (v-\theta_{\varepsilon} \ast u^{\varepsilon,\delta})]dxdv\}\\
	%&=0.
	%\end{align*}

	\subsection{Global existence of solutions}   $ $
	\par In order to prove the global existence of solutions of this regularized system \eqref{eq:NSVFK0}, it is enough to prove
	\begin{align*}
		\norm{f^{\varepsilon,\delta}(t,x,v)}_{L^{\infty}(0,T;L^{1}_{xv}(\mathbb{R}^{3}\times \mathbb{R}^{3}))} \leq 	\norm{f_{in}(t,x,v)}_{L^{1}_{xv}(\mathbb{R}^{3}\times \mathbb{R}^{3})},
	\end{align*}
	\begin{align*}
		\norm{f^{\varepsilon,\delta}(t,x,v)}_{L^{\infty}((0,T)\times\mathbb{R}^{3}\times \mathbb{R}^{3})} \leq 	e^{3T}\norm{f_{in}(t,x,v)}_{L^{\infty}_{xv}(\mathbb{R}^{3}\times \mathbb{R}^{3})},
	\end{align*}
	and
	\begin{equation}
		\label{eq:3}
	    \begin{aligned}
		&\norm{u^{\varepsilon,\delta}}^{2}_{L^{\infty}(0,T;{L^{2}_{x}})}+\norms{|v|^{2}f^{\varepsilon,\delta}}_{L^{\infty}(0,T;{L^{1}_{xv}})}+\norm{\nabla^2 u^{\varepsilon,\delta}}^{2}_{L^{2}(0,T;{L^{2}_{x}})}\\
        &\hspace{2cm}+\norm{\nabla u^{\varepsilon,\delta}}^{2}_{L^{2}(0,T;{L^{2}_{x}})}+\norm{\nabla u^{\varepsilon,\delta}}^{2}_{L^{\infty}(0,T;{L^{2}_{x}})}\leq C\left(T+1\right)^5e^{C(\varepsilon)(T+1)^2e^{2T}}
	\end{aligned}
	\end{equation}
	for any \(T>0\), where \(C=C\left(\varepsilon, \delta, \norm{f_{in}}_{L^{\infty}_{xv}}, \norm{f_{in}}_{L^{1}_{xv}},\norms{|v|^{2}f_{in}}_{L^{1}_{xv}},\norm{u_{in}}^{2}_{L^{2}_{x}}\right)\).
	\subsubsection{The estimate of $f^{\varepsilon,\delta}$}$ $
    \par 	Claim that
    \begin{align}
    	\label{5}
    	\norm{f^{\varepsilon,\delta}(t,x,v)}_{L^{\infty}(0,T;L^{p}_{xv}(\mathbb{R}^{3}\times \mathbb{R}^{3}))} \leq 	e^{\frac{3(p-1)}{p}T}\norm{f_{in}(t,x,v)}_{L^{p}_{xv}(\mathbb{R}^{3}\times \mathbb{R}^{3})}
    \end{align}
    for all \(p \in [1,+\infty]\).
    
    \begin{proof}[Proof of \eqref{5}]
    % 	Recall the second equation of \eqref{eq:NSVFK0} as follows:
    % \begin{align*}
    % 	\partial_t f^{\varepsilon,\delta} + (v\cdot \nabla_x) f^{\varepsilon,\delta}-\Delta f^{\varepsilon,\delta} = \nabla_v \cdot \left[(v\gamma_
    % 	\delta(v)-\theta_{\varepsilon} \ast u^{\varepsilon,\delta})f^{\varepsilon,\delta}\right].
    % \end{align*}
    Multiplying  \((f^{\varepsilon,\delta})^{p-1}\) on both sides of~$\eqref{eq:NSVFK0}_2$ ~ and ~integrating, we  derive
    \begin{align*}
    	&\frac{1}{p}\partial_t\int_{\mathbb{R}^6} (f^{\varepsilon,\delta})^{p}dxdv + \frac{1}{p} \int_{\mathbb{R}^6}v\cdot \nabla_x (f^{\varepsilon,\delta})^{p}dxdv-\int_{\mathbb{R}^6}(f^{\varepsilon,\delta})^{p-1}\Delta f^{\varepsilon,\delta}dxdv \\
    	&= \int_{\mathbb{R}^6}(f^{\varepsilon,\delta})^{p-1}\nabla_v \cdot \left[(v\gamma_
    	\delta(v)-\theta_{\varepsilon} \ast u^{\varepsilon,\delta})f^{\varepsilon,\delta}\right]dxdv,
    \end{align*}
    where \(p\in [1,+\infty)\). By integration by parts and \(v\cdot \nabla_v \gamma\leq 0\), we have
    \begin{align*}
    	&\int_{\mathbb{R}^6}(f^{\varepsilon,\delta})^{p-1}\nabla_v \cdot \left[(v\gamma_
    	\delta(v)-\theta_{\varepsilon} \ast u^{\varepsilon,\delta})f^{\varepsilon,\delta}\right]dxdv\\
    	&=-\int_{\mathbb{R}^6}\nabla_v \left[(f^{\varepsilon,\delta})^{p-1}\right]\cdot \left[(v\gamma_
    	\delta(v)-\theta_{\varepsilon} \ast u^{\varepsilon,\delta})f^{\varepsilon,\delta}\right]dxdv\\
    	&= \frac{p-1}{p}\int_{\mathbb{R}^6}(f^{\varepsilon,\delta})^p \nabla_{v}\cdot (v\gamma_{\delta}(v))dxdv\\
    	&\leq \frac{3(p-1)}{p}\int_{\mathbb{R}^6}(f^{\varepsilon,\delta})^pdxdv.
    \end{align*}
    	Then we  derive
    \begin{align}\label{nabla-v-2}
    	\partial_t &\int_{\mathbb{R}^{6}} (f^{\varepsilon,\delta})^{p}dxdv + \frac{4(p-1)}{p}\int_{\mathbb{R}^{6}}|\nabla_{v}(f^{\varepsilon,\delta})^{\frac{p}{2}}|^{2}dxdv \leq 3(p-1)\int_{\mathbb{R}^{6}} (f^{\varepsilon,\delta})^{p}dxdv.
    \end{align}
    By Gr\"{o}nwall's inequality (Lemma \ref{gronwall}) and \eqref{fdelta}, we get
    \begin{align}\label{flp}
    	\norm{f^{\varepsilon,\delta}}_{L^{\infty}_tL^{p}_{xv}((0,T)\times\mathbb{R}^{6})} \leq e^{\frac{3(p-1)}{p}T}\norm{f_{in}^\delta}_{L^{p}_{xv}(\mathbb{R}^{6})}\leq e^{\frac{3(p-1)}{p}T}\norm{f_{in}}_{L^{p}_{xv}(\mathbb{R}^{6})}
    \end{align}
    for arbitrary \(p \in [1,+\infty)\).
  Thus we  get
    \begin{align*}
    	\norm{f^{\varepsilon,\delta}}_{L^{\infty}_{xv}(\mathbb{R}^{6})} \leq e^{3T}\norm{f_{in}}_{L^{\infty}_{xv}(\mathbb{R}^{6})}.
    \end{align*}
    \end{proof}

    \subsubsection{The estimate of \(u^{\varepsilon,\delta}\) and \(|v|^2f^{\varepsilon,\delta}\)}
    
 Firstly, we multiply \(u^{\varepsilon,\delta}\) on the first equation of \eqref{eq:NSVFK0} and integrate to get
	\begin{align}
		\label{eq:1}
		\frac{1}{2}\partial_t \int_{\mathbb{R}^{3}}(u^{\varepsilon,\delta})^{2}dx + \norm{\nabla u^{\varepsilon,\delta}}^{2}_{L^{2}_{x}(\mathbb{R}^{3})}
		=\int_{\mathbb{R}^6}(\theta_{\varepsilon}\ast u^{\varepsilon,\delta})\cdot(v-\theta_{\varepsilon}\ast u^{\varepsilon,\delta})f^{\varepsilon,\delta}dvdx.
	\end{align}
	Then,  by multiplying \(\frac{|v|^{2}}{2}\) on the second equation of \eqref{eq:NSVFK0} and integrating, we get
	\begin{equation}
		\label{eq:2}
	    \begin{aligned}
		\frac{1}{2}\partial_t \int_{\mathbb{R}^{6}} |v|^{2}f^{\varepsilon,\delta}dxdv-\int_{\mathbb{R}^6}&v\cdot(\theta_{\varepsilon}\ast u^{\varepsilon,\delta}-v)f^{\varepsilon,\delta}dxdv\\
		&=\int_{\mathbb{R}^6}\left(1-\gamma_{\delta}(v)\right)|v|^2f^{\varepsilon,\delta}dxdv+3\int_{\mathbb{R}^6}f^{\varepsilon,\delta}dxdv.
	\end{aligned}
	\end{equation}
	Finally, combining \eqref{eq:1} and \eqref{eq:2}, we have
	\begin{align*}
		\frac{1}{2}\partial_t \int_{\mathbb{R}^{3}}&(u^{\varepsilon,\delta})^{2}dx+ \frac{1}{2}\partial_t \int_{\mathbb{R}^{6}} |v|^{2}f^{\varepsilon,\delta}dxdv + \norm{\nabla u^{\varepsilon,\delta}}^{2}_{L^{2}_{x}(\mathbb{R}^{3})}\\
		&+\int_{\mathbb{R}^6}|\theta_{\varepsilon}\ast u^{\varepsilon,\delta}-v|^2f^{\varepsilon,\delta}dxdv \leq 3\int_{\mathbb{R}^{6}}f^{\varepsilon,\delta}dxdv+\int_{\mathbb{R}^{6}}|v|^2f^{\varepsilon,\delta}dxdv.
	\end{align*}
	By Gr\"{o}nwall's inequality (Lemma \ref{gronwall}), \eqref{udelta} and \eqref{fdelta}, we get
    \begin{equation}
		\label{eq:4}
        \begin{aligned}
		\norm{u^{\varepsilon,\delta}}^{2}_{L^{\infty}(0,T;L^{2}_{x}(\mathbb{R}^{3}))}&+\norms{|v|^{2}f^{\varepsilon,\delta}}_{L^{\infty}(0,T;L^{1}_{xv}(\mathbb{R}^{6}))}+\norm{\nabla u^{\varepsilon,\delta}}^{2}_{L^{2}(0,T;{L^{2}_{x}(\mathbb{R}^{3})})}\\
		&\leq C\left(\norm{f_{in}}_{L^{1}_{xv}(\mathbb{R}^{6})},\norms{|v|^{2}f_{in}}_{L^{1}_{xv}(\mathbb{R}^{6})},\norm{u_{in}}^{2}_{L^{2}_{x}(\mathbb{R}^{3})} \right)(T+1)e^{2T}.
	\end{aligned}
    \end{equation}

	\subsubsection{The estimate of \(\nabla u^{\varepsilon,\delta}\)}$ $
	\par For the first equation of \eqref{eq:NSVFK0}, taking  \(\partial_{i}\) on both sides of the equation, we  get
	\begin{align*}
		\partial_{i}\partial_t u^{\varepsilon,\delta} + \partial_{i}\left[(\theta_{\varepsilon}\ast u^{\varepsilon,\delta})\cdot \nabla u^{\varepsilon,\delta}\right]&- \partial_{i}\Delta u^{\varepsilon,\delta} + \partial_{i}\nabla_{x} P^{\varepsilon,\delta} \\
		&-\partial_{i}\left\{\theta_{\varepsilon}\ast\left(\int_{\mathbb{R}^3}(v-\theta_{\varepsilon}\ast u^{\varepsilon,\delta})f^{\varepsilon,\delta}dv\right)\right\} = 0.
	\end{align*}
	Then  multiplying by \(\partial_{i}u^{\varepsilon,\delta}\) and taking the integration yields
	\begin{align*}
		\frac{1}{2}\partial_t\norm{&\nabla u^{\varepsilon,\delta}}^{2}_{L^{2}_{x}(\mathbb{R}^{3})}+\norm{\nabla^2 u^{\varepsilon,\delta}}^{2}_{L^{2}_{x}(\mathbb{R}^{3})}\\
		& =\int_{\mathbb{R}^3}\partial_{i}u^{\varepsilon,\delta}\partial_{i}\left\{\theta_{\varepsilon}\ast\left(\int_{\mathbb{R}^3}(v-\theta_{\varepsilon}\ast u^{\varepsilon,\delta})f^{\varepsilon,\delta}dv\right)\right\}dx - \int_{\mathbb{R}^{3}}\partial_{i}u^{\varepsilon,\delta}\partial_{i}\left[(\theta_{\varepsilon}\ast u^{\varepsilon,\delta})\cdot \nabla u^{\varepsilon,\delta}\right]dx\\
		&=D_{11}+D_{12}.
	\end{align*}
	For \(D_{11}\), by Young's inequality we have
	\begin{align*}
		D_{11}&=\int_{\mathbb{R}^{3}}\partial_{i}^{2}u^{\varepsilon,\delta}\theta_{\varepsilon}\ast\left(\int_{\mathbb{R}^3}(\theta_{\varepsilon}\ast u^{\varepsilon,\delta}-v)f^{\varepsilon,\delta}dv\right)dx\\
		&\leq C(\varepsilon)\left(\|\theta_{\varepsilon}\ast u^{\varepsilon,\delta}\|_{L^{\infty}_{x}(\mathbb{R}^{3})}+1\right) \norm{\nabla^2 u^{\varepsilon,\delta}}_{L^{2}_{x}(\mathbb{R}^{3})}\left(\norms{\int_{\mathbb{R}^3}vf^{\varepsilon,\delta}dv}_{L^{1}_{x}(\mathbb{R}^{3})}+\norms{\int_{\mathbb{R}^3}f^{\varepsilon,\delta}dv}_{L^{1}_{x}(\mathbb{R}^{3})}\right)\\
        &\leq \frac{1}{2}\norm{\nabla^2 u^{\varepsilon,\delta}}_{L^{2}_{x}(\mathbb{R}^{3})}^2+C(\varepsilon)\left(\|\ u^{\varepsilon,\delta}\|_{L^{2}_{x}(\mathbb{R}^{3})}+1\right)^2 \left(\norms{vf^{\varepsilon,\delta}}_{L^{1}_{xv}(\mathbb{R}^{6})}+\norms{f^{\varepsilon,\delta}}_{L^{1}_{xv}(\mathbb{R}^{6})}\right)^2.
	\end{align*}
	For \(D_{12}\), we have
	\begin{align*}
		D_{12}& = \int_{\mathbb{R}^{3}}\Delta u^{\varepsilon,\delta}(\theta_{\varepsilon}\ast u^{\varepsilon,\delta})\cdot \nabla  u^{\varepsilon,\delta}dx\\
		&\leq \norm{\theta_{\varepsilon}\ast u^{\varepsilon,\delta}}_{L^{\infty}}\norm{\nabla u^{\varepsilon,\delta}}_{L^{2}}\norm{\nabla^2 u^{\varepsilon,\delta}}_{L^{2}}\\
		&\leq \frac{1}{16}\norm{\nabla^2 u^{\varepsilon,\delta}}_{L^{2}}^{2}+C(\varepsilon)\norm{u^{\varepsilon,\delta}}_{L^{2}}^{2}\norm{\nabla u^{\varepsilon,\delta}}_{L^{2}}^{2}.
	\end{align*}
	Collecting \(D_{11}\) and \(D_{12}\), we have
	\begin{align*}
		\partial_t\norm{\nabla u^{\varepsilon,\delta}}^{2}_{L^{2}_{x}(\mathbb{R}^{3})}+\norm{\nabla^2 u^{\varepsilon,\delta}}^{2}_{L^{2}_{x}(\mathbb{R}^{3})} \leq& C(\varepsilon)(\|\ u^{\varepsilon,\delta}\|_{L^{2}_{x}(\mathbb{R}^{3})}+1)^2 \left(\norms{vf^{\varepsilon,\delta}}_{L^{1}_{xv}(\mathbb{R}^{6})}+\norms{f^{\varepsilon,\delta}}_{L^{1}_{xv}(\mathbb{R}^{6})}\right)^2\\
		&+C(\varepsilon)\norm{u^{\varepsilon,\delta}}_{L^{2}}^{2}\norm{\nabla u^{\varepsilon,\delta}}_{L^{2}}^{2}.
	\end{align*}
	Using Gr\"{o}nwall's inequality,  \eqref{5} and \eqref{eq:4}, we have
	\begin{align*}
		&\norm{\nabla u^{\varepsilon,\delta}}^{2}_{L^{\infty}(0,T;{L^{2}_{x}(\mathbb{R}^{3})})}+\norm{\nabla^2 u^{\varepsilon,\delta}}^{2}_{L^{2}(0,T;{L^{2}_{x}})} \\
		&\leq C\left(\varepsilon, \delta, \norm{f_{in}}_{L^{\infty}_{xv}}, \norm{f_{in}}_{L^{1}_{xv}}\norms{|v|^{2}f_{in}^{\delta}}_{L^{1}_{xv}},\norm{u_{in}}^{2}_{L^{2}_{x}}\right)(T+1)^5e^{C(\varepsilon)(T+1)^2e^{2T}}.
	\end{align*}
	To sum up the above estimates, for any \(t\in (0,T)\), we achieve \eqref{eq:3}. According \eqref{eq:3} and equation $\eqref{eq:NSVFK0}_1$, we can derive $\partial_tu^{\varepsilon,\delta} \in L^2(0,T;L^2(\mathbb R^3))$.

 Next, we complete the proof  of  Theorem \ref{thm:solution}.
\begin{proof}[Proof  of  Theorem \ref{thm:solution}] 
Using the proof of \(\Phi\) as a contraction mapping (2.1.2  \(\&\) 2.1.3) and the bounded estimates of solutions in \eqref{eq:3}, the time regularity of \(u^{\varepsilon,\delta}\) follows from
\(\partial_t u^{\varepsilon,\delta}\in L^2(0,T;L^2)\) and
\(u^{\varepsilon,\delta}\in L^2(0,T;H^2)\).
For \(f^{\varepsilon,\delta}\), the weighted kernel estimates and
the Duhamel formula yield continuity in
\(L^1(\mathbb R^6;(1+|v|^2)\,dx\,dv)\), including at \(t=0\).
The \(L^2\) estimate for \(\nabla_v f^{\varepsilon,\delta}\)
is given in \eqref{nabla-v-2} for $p=2$ by Gr\"{o}nwall inequality.
Hence, Theorem \ref{thm:solution} is proved.
\end{proof}

	\section{Uniform estimates for the regularized system}

In this section, we further prove some uniform estimates of solutions obtained in Theorem \ref{thm:solution}.

\begin{bluerevision}
\noindent\textbf{Approximation for the initial data.}
For fixed $\varepsilon,\delta>0$, one may approximate the initial particle density by
\[
f_{in,\eta}^{\delta}=\rho_\eta*_{x,v}f_{in}^{\delta}
  +\eta e^{-|x|^2-|v|^2},\qquad \eta>0,
\]
where $\rho_\eta$ is a standard nonnegative mollifier on $\mathbb R^6$. Thus the approximating density is smooth and strictly positive, and the Gaussian perturbation has finite mass and moments. In particular, for each fixed \(\delta>0\), the initial kinetic energy
and entropy corresponding to \(f_{\mathrm{in},\eta}^{\delta}\) converge to those corresponding to \(f_{\mathrm{in}}^\delta\) as \(\eta\to0\). 
For fixed $\varepsilon,\delta>0$, let $(u_\eta^{\varepsilon,\delta},f_\eta^{\varepsilon,\delta})$ be the solution corresponding to the initial data $(u_{in}^{\delta},f_{in,\eta}^{\delta})$. Set
\[
D_\eta(t)
=\|u_\eta^{\varepsilon,\delta}(t)-u^{\varepsilon,\delta}(t)\|_{H^1_x}
+\|(1+|v|^2)(f_\eta^{\varepsilon,\delta}(t)-f^{\varepsilon,\delta}(t))\|_{L^1_{xv}}.
\]
For every $T<\infty$, the estimates in Section 2 give bounds for these solutions that are uniform in $0<\eta\leq1$. In particular, Young's inequality yields
\begin{align*}
\left\|\theta_\varepsilon*\int_{\mathbb R^3_v}
\left[(v-\theta_\varepsilon*u_\eta^{\varepsilon,\delta})f_\eta^{\varepsilon,\delta}
-(v-\theta_\varepsilon*u^{\varepsilon,\delta})f^{\varepsilon,\delta}\right]dv
\right\|_{L^2_x}
\leq C_{\varepsilon,T}D_\eta(t).
\end{align*}
Using the Duhamel formulas, the heat semigroup estimates, and the weighted kernel estimate \eqref{min:G-weighted-gradient}, we obtain
\[
D_\eta(t)\leq C_TD_\eta(0)
+C_{\varepsilon,\delta,T}\int_0^t
\left[1+(t-s)^{-\frac12}\right]D_\eta(s)\,ds,
\qquad 0\leq t\leq T,
\]
where the constants are independent of $\eta$. Since $D_\eta(0)\to0$, Gr\"{o}nwall's inequality gives
\[
\sup_{0\leq t\leq T}D_\eta(t)\longrightarrow0
\qquad\text{as }\eta\to0.
\]

The estimates below can first be derived for the smooth positive approximations, using regularized logarithmic functions and cutoffs in $x$ and $v$. The cutoffs and logarithmic regularization can then be removed using the moment and dissipation estimates. The strong $L^1$ convergence, uniform $L^\infty$ bounds, and the uniform spatial and velocity second-moment bounds from Proposition~\ref{pro:xf} give the required entropy convergence. The energy inequalities then follow by weak lower semicontinuity of the dissipation terms. The terminal-time statements are understood for almost every $t>0$.
\end{bluerevision}
    
	\subsection{\texorpdfstring
    {Uniform estimates of \(u^{\varepsilon,\delta}\) and \(f^{\varepsilon,\delta}\)}
    {Uniform estimates of u varepsilon delta and f varepsilon delta}} 
    \hfill
     \par It follows from \eqref{5} and \eqref{eq:4} that the following lemma holds:
	\begin{lemma}\label{pro:estimate}
	% 	 Let \((u_{in}, f_{in}, |v|^2f_{in})\) satisfy
	% \begin{equation}
	% 	\left\{
	% 	\begin{aligned}
	% 		\label{3}
	% 		&u_{in} \in L^{2}_{x}(\mathbb{R}^3);\\
	% 		&|v|^2f_{in} \in L^{1}_{xv}(\mathbb{R}^6), f_{in} \in L^{1}_{xv}(\mathbb{R}^6) \cap L^{\infty}_{xv}(\mathbb{R}^6);\\
	% 		&f_{in} \geqslant 0.
	% 	\end{aligned}
	% 	\right.
	% \end{equation}
    Under the assumptions stated in Theorem \ref{thm:solution}, let $(u^{\varepsilon,\delta}, f^{\varepsilon,\delta}, |v|^2 f^{\varepsilon,\delta})$ denote a solution to the system \eqref{eq:NSVFK0}. Then there exists a positive constant $C$, independent of the parameters $\varepsilon$ and $\delta$, such that the following uniform estimate holds:
	\begin{equation}\label{eq:esti}
	    \begin{aligned}
			\norm{f^{\varepsilon,\delta}(t,x,v)}&_{L^{\infty}(0,t;L^{1}_{xv}(\mathbb{R}^{3}\times \mathbb{R}^{3}))}+\norm{f^{\varepsilon,\delta}(t,x,v)}_{L^{\infty}((0,t)\times\mathbb{R}^{3}\times \mathbb{R}^{3})}+\norm{u^{\varepsilon,\delta}}^{2}_{L^{\infty}(0,t;{L^{2}_{x}(\mathbb{R}^{3})})}\\
			&+\norms{|v|^2f^{\varepsilon,\delta}}_{L^{\infty}(0,t;{L^{1}_{xv}(\mathbb{R}^{6})})}+\norm{\nabla u^{\varepsilon,\delta}}^{2}_{L^{2}(0,t;{L^{2}_{x}(\mathbb{R}^{3})})}\\
		&\leq C\left(\norm{f_{in}}_{L^{\infty}_{xv}( \mathbb{R}^{3}\times \mathbb{R}^{3})}, \norm{f_{in}}_{L^{1}_{xv}(\mathbb{R}^{6})},\norms{|v|^2f_{in}}_{L^{1}_{xv}(\mathbb{R}^{6})},\norm{u_{in}}^{2}_{L^{2}_{x}(\mathbb{R}^{3})} \right)(t+1)e^{3t}.
	\end{aligned}
	\end{equation}
	% and
	% \begin{align}
	% 	\label{eq:6}
	% 	\norm{&f(t,x,v)}_{L^{\infty}(0,t;L^{1}_{xv}(\mathbb{R}^{3}\times \mathbb{R}^{3}))}+\norm{f(t,x,v)}_{L^{\infty}_{txv}((0,t)\times\mathbb{R}^{3}\times \mathbb{R}^{3})}+\norm{u}^{2}_{L^{\infty}(0,t;L^{2}_{x}(\mathbb{R}^{3}))} \nonumber\\
	% 	&\hspace{1cm}+\norms{|v|^2f}_{L^{\infty}(0,t;{L^{1}_{xv}(\mathbb{R}^{6})})}+\norm{\nabla u}^{2}_{L^{2}(0,t;{L^{2}_{x}(\mathbb{R}^{3})})}\nonumber\\
	% 	&\leq C\left(\norm{f_{in}}_{L^{\infty}_{xv}( \mathbb{R}^{3}\times \mathbb{R}^{3})}, \norm{f_{in}}_{L^{1}_{xv}(\mathbb{R}^{6})},\norms{|v|^2f_{in}}_{L^{1}_{xv}(\mathbb{R}^{6})},\norm{u_{in}}^{2}_{L^{2}_{x}(\mathbb{R}^{3})} \right)(1+t)e^{5t}.		
	% \end{align}
	\end{lemma}

	\begin{lemma}\label{lem:deltaf}
		Under the assumptions stated in Theorem \ref{thm:solution}, there exists a constant \(C>0\) such that
	\begin{align}
		\label{eq:10}
		\norm{f^{\varepsilon,\delta}}^2_{L^{\infty}(0,t;L^2(\mathbb{R}^6_{xv}))}+2\norm{\nabla_{v}f^{\varepsilon,\delta}}^2_{L^{2}((0,t)\times \mathbb{R}^6_{xv})} \leq C\norm{f_{in}}^2_{L^2(\mathbb{R}^6_{xv})}(1+t)e^{3t}.
	\end{align}
	% and
	% \begin{align}
	% 	\label{eq:11}
	% 	\norm{f}^2_{L^{\infty}(0,t;L^2(\mathbb{R}^6_{xv}))}+2\norm{\nabla_{v}f}^2_{L^{2}((0,t)\times \mathbb{R}^6_{xv})} \leq C\norm{f_{in}}^2_{L^2(\mathbb{R}^6_{xv})}(1+t)e^{3t}.
	% \end{align}
        \end{lemma}
        
	\begin{proof}
		%(i) Uniform estimate with respect to \(\varepsilon\).\\
	Multiplying \(f^{\varepsilon,\delta}\) on the second equation of \eqref{eq:NSVFK0} and integrating, we get
	\begin{align*}
		\frac{1}{2}\partial_t \int_{\mathbb{R}^{6}} (f^{\varepsilon,\delta})^{2}dxdv&+ \int_{\mathbb{R}^{6}} |\nabla_{v}f^{\varepsilon,\delta}|^{2}dxdv =\frac{1}{2}\int_{\mathbb{R}^{6}}\left(\theta_{\varepsilon} \ast u^{\varepsilon,\delta}-v\gamma_\delta(v)\right)\cdot \nabla_{v}(f^{\varepsilon,\delta})^2dxdv.
	\end{align*}
	For the right-hand side, using integration by parts and the fact that \(v \cdot \nabla_{v} \gamma_\delta \leq 0\), we deduce
	\begin{align*}
		\frac{1}{2}\int_{\mathbb{R}^{6}}&\left(\theta_{\varepsilon} \ast u^{\varepsilon,\delta}-v\gamma_\delta(v)\right)\cdot \nabla_{v}(f^{\varepsilon,\delta})^2dxdv\\
		&=\frac{1}{2}\int_{\mathbb{R}^{6}}(f^{\varepsilon,\delta})^2\nabla_{v}\cdot \left(v\gamma_\delta(v)\right) dxdv\\
		&\leq \frac{3}{2}\int_{\mathbb{R}^{6}}(f^{\varepsilon,\delta})^2dxdv,
	\end{align*}
	which implies
	\begin{align*}
		\partial_t \int_{\mathbb{R}^{6}} (f^{\varepsilon,\delta})^2 dxdv 
		+ 2 \int_{\mathbb{R}^{6}} |\nabla_{v}f^{\varepsilon,\delta}|^2 dxdv
		\leq 3 \int_{\mathbb{R}^{6}} (f^{\varepsilon,\delta})^2 dxdv.
	\end{align*}
		Applying Gr\"{o}nwall's inequality and \eqref{fdelta}, we get
	\begin{align*}
		\norm{f^{\varepsilon,\delta}}^2_{L^{\infty}(0,t;L^2(\mathbb{R}^6_{xv}))} 
		+ 2 \norm{\nabla_{v}f^{\varepsilon,\delta}}^2_{L^{2}(0,t;L^2(\mathbb{R}^6_{xv}))} 
		\leq C \norm{f_{in}}_{L^2(\mathbb{R}^6_{xv})}^2 (1+t) e^{3t}.
	\end{align*}
	% As for the proof of \eqref{eq:11}, it can be derived using arguments similar to those used for \eqref{eq:6}.
	\end{proof}
	
	\begin{lemma}\label{lem:3f}
   Under the initial conditions prescribed in Theorem \ref{thm:solution}, there exists a constant $C > 0$, 
whose value depends solely on the norms of the initial data, 
such that the following uniform estimate holds:
\begin{align}
    \label{10}
    \left\||v|^3 f^{\varepsilon,\delta}\right\|_{L^{\infty}(0,t; L^1_{xv}(\mathbb{R}^6))} 
\leq C (t + 1)^2 \exp\left[ C (t + 1)^{3} e^{4t} \right],
\end{align}
where
$
C = C\left( \|f_{\mathrm{in}}\|_{L^1_{xv}(\mathbb{R}^6)}, 
\|f_{\mathrm{in}}\|_{L^{\infty}_{xv}(\mathbb{R}^6)}, 
\left\||v|^3 f_{\mathrm{in}}\right\|_{L^1_{xv}(\mathbb{R}^6)}, 
\|u_{\mathrm{in}}\|_{L^2_x(\mathbb{R}^3)} \right).
$
		% Furthermore, the same estimate remains valid in the limit \(\varepsilon \to 0\), that is,
		% \begin{align}
		% 	\label{11}
		% 	\left\||v|^3f\right\|_{L^{\infty}(0,t;L^1_{xv}(\mathbb{R}^6))}
		% 	\leq C\left(t+1\right)^2 \exp\left[C(t+1)^\frac{14}{5}e^{6t}\right].
		% \end{align}
	\end{lemma}
	\begin{proof}
		Multiplying the second equation in \eqref{eq:NSVFK0} by \(|v|^3\) and integrating over \(\mathbb{R}^6_{xv}\), by Young's inequality and  \eqref{28} we obtain 
	\begin{align*}
		\partial_t &\int_{ \mathbb{R}_{xv}^6}|v|^3f^{\varepsilon,\delta}dxdv+3\int_{ \mathbb{R}_{xv}^6}|v|^3f^{\varepsilon,\delta}\gamma_{\delta}(v)dxdv \\
		&= \int_{ \mathbb{R}_{xv}^6}3|v|v\cdot (\theta_{\varepsilon} \ast u^{\varepsilon,\delta})f^{\varepsilon,\delta}dxdv+12\int_{ \mathbb{R}_{xv}^6}|v|f^{\varepsilon,\delta}dxdv\\
		&\leq C\norm{\theta_{\varepsilon} \ast u^{\varepsilon,\delta}}_{L^6_x{(\mathbb{R}^3)}}\norms{\int_{ \mathbb{R}_v^3}|v|^2f^{\varepsilon,\delta }dv}_{L^\frac65_x(\mathbb{R}^3)}+C\int_{ \mathbb{R}_{xv}^6}|v|^2f^{\varepsilon,\delta}dxdv+C\int_{ \mathbb{R}_{xv}^6}f^{\varepsilon,\delta}dxdv\\
		&\leq C\norm{ u^{\varepsilon,\delta}}_{L^6_x{(\mathbb{R}^3})}\norms{|v|^3f^{\varepsilon,\delta }}^{\frac56}_{L^1_{xv}{(\mathbb{R}^6)}}\norm{f^{\varepsilon,\delta }}^\frac{1}{6}_{L^\infty_{xv}{(\mathbb{R}^6)}}\\
		&\hspace{5cm}+C\int_{ \mathbb{R}_{xv}^6}|v|^2f^{\varepsilon,\delta}dxdv+C\int_{ \mathbb{R}_{xv}^6}f^{\varepsilon,\delta}dxdv\\
		&\leq C\norm{u^{\varepsilon,\delta}}^\frac65_{\dot{H}^1_x{(\mathbb{R}^3)}}\norms{|v|^3f^{\varepsilon,\delta }}_{L^1_{xv}{(\mathbb{R}^6)}}+C\norm{f^{\varepsilon,\delta }}_{L^\infty_{xv}{(\mathbb{R}^6)}}\\
		&\hspace{5cm}+C\int_{ \mathbb{R}_{xv}^6}|v|^2f^{\varepsilon,\delta}dxdv+C\int_{ \mathbb{R}_{xv}^6}f^{\varepsilon,\delta}dxdv.
	\end{align*}
	By Gr\"{o}nwall's inequality, \eqref{fdelta} and \eqref{eq:esti}, we obtain the following estimate:
	\begin{align*}
		&\int_{ \mathbb{R}_{xv}^6}|v|^3f^{\varepsilon,\delta}dxdv\nonumber\\
		&\leq e^{Ct^\frac25\norm{\nabla_{x}u^{\varepsilon,\delta}}^\frac65_{L^2((0,T)\times\mathbb{R}^3_x)}
        }\bigg[Ct\norms{|v|^2f^{\varepsilon,\delta}}_{L^{\infty}(0,T;{L^{1}_{xv}(\mathbb{R}^{6})})}\nonumber\\
		&\hspace{4cm}+\norms{|v|^3f_{in}}_{L^{1}_{xv}(\mathbb{R}^{6})}+Cte^{3t}\big(\norm{f_{in}}_{L^{1}_{xv}(\mathbb{R}^{6})}+\norm{f_{in}^{\delta}}_{L^{\infty}_{xv}(\mathbb{R}^6)}\big)\bigg]\nonumber\\
		&= C(t+1)^2e^{C(t+1)^3e^{4t}},
	\end{align*}
    where \(C=C\left(\norm{f_{in}}_{L^{1}_{xv}(\mathbb{R}^{6})},\norm{f_{in}}_{L^{\infty}_{xv}(\mathbb{R}^6)},\norms{|v|^3f_{in}}_{L^{1}_{xv}(\mathbb{R}^{6})},\norm{u_{in}}_{L^2_x(\mathbb{R}^3)}\right)\).
	% The estimate \eqref{11} can be derived analogously by following the same argument as used in the derivation of \eqref{eq:6}.
	\end{proof}

	\begin{proposition}\label{pro:xf}
		Under the assumptions stated in Theorem \ref{thm:solution},  the following moment estimate holds:
		\begin{equation}
			\label{17}
		    \begin{aligned}
			\left\||x|^2 f^{\varepsilon,\delta} \right\|&_{L^{\infty}(0,t; L^1_{xv}(\mathbb{R}^6))} \\
			&\leq C\Big( 
			\|f_{in}\|_{L^1_{xv}},\,
			\||v|^2 f_{in}\|_{L^1_{xv}},\,
			\|u_{in}\|_{L^2_{x}}^2,\,
			\||x|^2 f_{in}\|_{L^1_{xv}} 
			\Big)(t + 1)^2\, e^{4t}.
		\end{aligned}
		\end{equation}
		In addition, the following inequality is satisfied:
		\begin{equation}
			\label{18}
		    \begin{aligned}
			\int_{\mathbb{R}^3_x \times \mathbb{R}^3_v}& f^{\varepsilon,\delta} |\log f^{\varepsilon,\delta}| dxdv
			\leq \int_{\mathbb{R}^3_x \times \mathbb{R}^3_v} f^{\varepsilon,\delta} \log f^{\varepsilon,\delta} dxdv\\
			&+ C\Big( 
			\|f_{in}\|_{L^1_{xv}},\,
			\||v|^2 f_{in}\|_{L^1_{xv}},\,
			\|u_{in}\|_{L^2_{x}}^2,\,
			\||x|^2 f_{in}\|_{L^1_{xv}} 
			\Big)(t + 1)^2\, e^{4t}.
		\end{aligned}
		\end{equation}
		% Analogous estimates for the limiting solution \( f \) are obtained using the same arguments as in Proposition~\ref{pro:estimate}:
		% \begin{align}
		% 	\label{19}
		% 	\left\||x|^2 f \right\|&_{L^{\infty}(0,t; L^1_{xv}(\mathbb{R}^6))} \nonumber\\
		% 	&\leq C\Big( 
		% 	\|f_{in}\|_{L^1_{xv}},\,
		% 	\||v|^2 f_{in}\|_{L^1_{xv}},\,
		% 	\|u_{in}\|_{L^2_{x}}^2,\,
		% 	\||x|^2 f_{in}\|_{L^1_{xv}} 
		% 	\Big)(t^2 + t)\, e^{6t},
		% \end{align}
		% and
		% \begin{align}
		% 	\label{20}
		% 	\int_{\mathbb{R}^3_x \times \mathbb{R}^3_v} &f |\log f| dxdv
		% 	\leq \int_{\mathbb{R}^3_x \times \mathbb{R}^3_v} f \log f dxdv\nonumber\\
		% 	&\quad + C\Big( 
		% 	\|f_{in}\|_{L^1_{xv}},\,
		% 	\||v|^2 f_{in}\|_{L^1_{xv}},\,
		% 	\|u_{in}\|_{L^2_{x}}^2,\,
		% 	\||x|^2 f_{in}\|_{L^1_{xv}} 
		% 	\Big)(t^2 + t)\, e^{6t}.
		% \end{align}
	\end{proposition}

\begin{proof}
	Multiplying the second equation in \eqref{eq:NSVFK0} by \(|x|^{2}\), and integrating over the spatial and velocity variables, it follows that
	\begin{align*}
		\partial_t \int_{ \mathbb{R}_v^3\times \mathbb{R}_x^3} |x|^{2}f^{\varepsilon,\delta}dxdv&= 2\int_{ \mathbb{R}_v^3\times \mathbb{R}_x^3}\left(v\cdot x\right)f^{\varepsilon,\delta}dxdv\\
		&\leq \int_{ \mathbb{R}_v^3\times \mathbb{R}_x^3}|v|^2f^{\varepsilon,\delta}dxdv+\int_{ \mathbb{R}_v^3\times \mathbb{R}_x^3}|x|^2f^{\varepsilon,\delta}dxdv.
	\end{align*}
	By Gr\"{o}nwall's inequality together with estimates \eqref{eq:esti} and \eqref{fdelta}, one obtains
	\begin{align*}
		\norm{x^2f^{\varepsilon,\delta}}&_{L^{\infty}\left(0,t;L^1_{xv}\left(\mathbb{R}^6\right)\right)}\\
        &\leq C( \norm{f_{in}}_{L^{1}_{xv}(\mathbb{R}^{6})},\norms{|v|^2f_{in}}_{L^{1}_{xv}(\mathbb{R}^{6})},\norm{u_{in}}^{2}_{L^{2}_{x}(\mathbb{R}^{3})},\norms{|x|^2f_{in}}_{L^{1}_{xv}(\mathbb{R}^{6})} )(t+1)^2e^{4t}.
	\end{align*}
	Next, observe that
	\begin{align*}
		f^{\varepsilon,\delta}|\log f^{\varepsilon,\delta}| &= f^{\varepsilon,\delta}\log f^{\varepsilon,\delta} -2f^{\varepsilon,\delta}\log f^{\varepsilon,\delta}\chi_{0\leq f^{\varepsilon,\delta}\leq 1}\\
		&= f^{\varepsilon,\delta}\log f^{\varepsilon,\delta} -2f^{\varepsilon,\delta}\log f^{\varepsilon,\delta}\chi_{e^{-\omega}\leq f^{\varepsilon,\delta}\leq 1}-2f^{\varepsilon,\delta}\log f^{\varepsilon,\delta}\chi_{e^{-\omega}\geq  f^{\varepsilon,\delta}}\\
		&\leq f^{\varepsilon,\delta}\log f^{\varepsilon,\delta}+2f^{\varepsilon,\delta}\omega+C\sqrt{f^{\varepsilon,\delta}}\chi_{e^{-\omega}\geq  f^{\varepsilon,\delta}}\\
		&\leq f^{\varepsilon,\delta}\log f^{\varepsilon,\delta}+2f^{\varepsilon,\delta}\omega+Ce^{-\frac{\omega}{2}}\quad \left(\omega\geq0\right).
	\end{align*}
	Setting \(\omega=\frac{x^2+v^2}{8}\), it follows that
	\begin{align*}
		\int_{ \mathbb{R}_v^3\times \mathbb{R}_x^3}f^{\varepsilon,\delta}|\log f^{\varepsilon,\delta}|dxdv &\leq \int_{ \mathbb{R}_v^3\times \mathbb{R}_x^3}f^{\varepsilon,\delta}\log f^{\varepsilon,\delta}dxdv +\frac{1}{4}\int_{ \mathbb{R}_v^3\times \mathbb{R}_x^3}\left(x^2+v^2\right)f^{\varepsilon,\delta}dxdv\\
		&\hspace{2cm}+2C\int_{ \mathbb{R}_v^3\times \mathbb{R}_x^3}e^{-\frac{x^2+v^2}{16}}dxdv.
	\end{align*}
	Finally, the inequality \eqref{18} follows from a combination of \eqref{eq:esti} and \eqref{17}. 
    % Then \eqref{19} and \eqref{20} can be derived by applying analogous arguments to those used in the proof of \eqref{eq:6}.
\end{proof}

	\begin{lemma}\label{lem:sqrtf}
		Under the assumptions stated in Theorem \ref{thm:solution}, there holds
	\begin{align}
		\label{21}
		&\norm{\nabla_{v}\sqrt{f^{\varepsilon,\delta}}}_{L^{2}((0,t)\times \mathbb{R}^{3}\times \mathbb{R}^{3})}\leq C(t+1)^3e^{9t},
	\end{align}
    where
    \begin{align*}
        C=C\left(\norm{f_{in}}_{L^{\infty}_{xv}}, \norm{f_{in}}_{L^{1}_{xv}},\norms{|v|^2f_{in}}_{L^{1}_{xv}},\norm{u_{in}}^{2}_{L^{2}_{x}},\norm{f_{in}\log f_{in}}_{L^{1}_{xv}},\norms{|x|^2f_{in}}_{L^{1}_{xv}} \right).
    \end{align*}
	% As in the proof of  Proposition \ref{pro:estimate} and Lemma \ref{lem:deltaf}, this becomes
	% \begin{align}
	% 	\label{22}
	% 	&\norm{\nabla_{v}\sqrt{f}}_{L^{2}((0,t)\times \mathbb{R}^{3}\times \mathbb{R}^{3})} \leq C(t+1)^3e^{13t},
	% \end{align}
 %    where
 %    \begin{align*}
 %        C=C\left(\norm{f_{in}}_{L^{\infty}_{xv}}, \norm{f_{in}}_{L^{1}_{xv}},\norms{|v|^2f_{in}}_{L^{1}_{xv}},\norm{u_{in}}^{2}_{L^{2}_{x}},\norm{f_{in}\log f_{in}}_{L^{1}_{xv}},\norms{|x|^2f_{in}}_{L^{1}_{xv}} \right).
 %    \end{align*}
	\end{lemma}
	\begin{proof}
		First, multiplying \(1+\frac{|v|^{2}}{2}+\log f^{\varepsilon,\delta}\) on the second equation of \eqref{eq:NSVFK0} and integrating over $\mathbb{R}^6$ yields:
	\begin{align*}
		\int_{\mathbb{R}^6_{xv}}\partial_t f^{\varepsilon,\delta}&\left(1+\frac{|v|^{2}}{2}+\log f^{\varepsilon,\delta}\right)dxdv + \int_{ \mathbb{R}_{xv}^6}(v\cdot \nabla_x) f^{\varepsilon,\delta}\left(1+\frac{|v|^{2}}{2}+\log f^{\varepsilon,\delta}\right)dxdv\\
		&+ \int_{ \mathbb{R}_{xv}^6}\nabla_v \cdot \left[\left(\theta_{\varepsilon} \ast u^{\varepsilon,\delta}-v\gamma_
		\delta(v)\right)f^{\varepsilon,\delta}-\nabla_v f^{\varepsilon,\delta}\right]\left(1+\frac{|v|^{2}}{2}+\log f^{\varepsilon,\delta}\right)dxdv\\&\hspace{3cm} :=e_1+e_2+e_3= 0 ,
	\end{align*}
	where
	\begin{align*}
		e_1&=\partial_t\int_{\mathbb{R}^6_{xv}} f^{\varepsilon,\delta}\left(1+\frac{|v|^{2}}{2}+\log f^{\varepsilon,\delta}\right)dxdv-\int_{\mathbb{R}^6_{xv}} f^{\varepsilon,\delta}\partial_t\left(1+\frac{|v|^{2}}{2}+\log f^{\varepsilon,\delta}\right)dxdv\\
		&=\partial_t\int_{\mathbb{R}^6_{xv}} \left(\frac{|v|^{2}}{2}f^{\varepsilon,\delta}+f^{\varepsilon,\delta}\log f^{\varepsilon,\delta}\right)dxdv,\\
		e_2&=-\int_{\mathbb{R}^6_{xv}}f^{\varepsilon,\delta}\frac{1}{f^{\varepsilon,\delta}}\nabla_{x}f^{\varepsilon,\delta}\cdot vdxdv\\
		&=0,
	\end{align*}
	and
    \begin{align*}
        e_3=&-\int_{ \mathbb{R}_{xv}^6}\left[\left(\theta_{\varepsilon} \ast u^{\varepsilon,\delta}-v\gamma_
		\delta(v)\right)f^{\varepsilon,\delta}-\nabla_v f^{\varepsilon,\delta}\right]\left(v+\frac{\nabla_{v}f^{\varepsilon,\delta}}{f^{\varepsilon,\delta}}\right)dxdv\\
		=&\int_{ \mathbb{R}_{xv}^6}\frac{|(\theta_{\varepsilon}\ast u^{\varepsilon,\delta}-v)f^{\varepsilon,\delta}-\nabla_{v}f^{\varepsilon,\delta}|^2}{f^{\varepsilon,\delta}}dxdv-\int_{\mathbb{R}_{xv}^6}(\theta_{\varepsilon}\ast u^{\varepsilon,\delta})\cdot(\theta_{\varepsilon}\ast u^{\varepsilon,\delta}-v)f^{\varepsilon,\delta}dxdv\\
		&-\int_{\mathbb{R}_{xv}^6}|v|^2f^{\varepsilon,\delta}(1-\gamma_{\delta}(v))dxdv-\int_{\mathbb{R}_{xv}^6}v\cdot\nabla_{v}f^{\varepsilon,\delta}(1-\gamma_{\delta}(v))dxdv.
    \end{align*}
	% \begin{align*}
	% 	e_3=&-\int_{ \mathbb{R}_{xv}^6}\left[\left(\theta_{\varepsilon} \ast u^{\varepsilon,\delta}-v\gamma_
	% 	\delta(v)\right)f^{\varepsilon,\delta}-\nabla_v f^{\varepsilon,\delta}\right]\left(v+\frac{\nabla_{v}f^{\varepsilon,\delta}}{f^{\varepsilon,\delta}}\right)dxdv\\
	% 	=&\int_{ \mathbb{R}_{xv}^6}\frac{|(u^{\varepsilon,\delta}-v)f^{\varepsilon,\delta}-\nabla_{v}f^{\varepsilon,\delta}|^2}{f^{\varepsilon,\delta}}dxdv-\int_{\mathbb{R}_{xv}^6}u^{\varepsilon,\delta}\cdot(u^{\varepsilon,\delta}-v)f^{\varepsilon,\delta}dxdv\\
	% 	&+\int_{\mathbb{R}_{xv}^6}|v|^2f^{\varepsilon,\delta}(\gamma_{\delta}(v)-1)dxdv-\int_{\mathbb{R}_{xv}^6}v\cdot\nabla_{v}f^{\varepsilon,\delta}(1-\gamma_{\delta}(v))dxdv\\
 %        &-\int_{\mathbb{R}_{xv}^6}(\theta_{\varepsilon}\ast u^{\varepsilon,\delta}-u^{\varepsilon,\delta})\cdot vf dxdv.
	% \end{align*}
    To sum up, it becomes
    \begin{equation}\label{vfglobal1}
         \begin{aligned}
		&\int_{\mathbb{R}^6_{xv}} \left(\frac{|v|^{2}}{2}f^{\varepsilon,\delta}+f^{\varepsilon,\delta}\log f^{\varepsilon,\delta}\right)(t,\cdot)dxdv+\int_{\left(0,t\right)\times \mathbb{R}_v^3\times \mathbb{R}_x^3}\frac{|(\theta_\varepsilon \ast u^{\varepsilon,\delta}-v)f^{\varepsilon,\delta}-\nabla_{v}f^{\varepsilon,\delta}|^2}{f^{\varepsilon,\delta}}dxdvds \\
		&\leq \int_{\left(0,t\right)\times \mathbb{R}_v^3\times \mathbb{R}_x^3}(\theta_\varepsilon \ast u^{\varepsilon,\delta})\cdot(\theta_\varepsilon \ast u^{\varepsilon,\delta}-v)f^{\varepsilon,\delta}dxdvds+\int_{\left(0,t\right)\times \mathbb{R}_v^3\times \mathbb{R}_x^3}v\cdot\nabla_{v}f^{\varepsilon,\delta}(1-\gamma_{\delta}(v))dxdvds\\
		&\hspace{0.5cm}-\int_{\left(0,t\right)\times \mathbb{R}_v^3\times \mathbb{R}_x^3}|v|^2f^{\varepsilon,\delta}(\gamma_{\delta}(v)-1)dxdvds+\int_{ \mathbb{R}_v^3\times \mathbb{R}_x^3}\left(\frac{|v|^2}{2}f_{in}^\delta+f_{in}^\delta\log f_{in}^\delta\right)dxdv\\
		&:= S_1+S_2+S_3+S_4.
	\end{aligned}
    \end{equation}
	%which this process is identical to the proof of Lemma \ref{lem:global}.\\
	For \(S_1\), by \eqref{8}, it shows that
    \begin{align*}
        S_1 \leq& C\left(\frac{4}{3} \pi e^{3t} \|f_{in}(t,x,v)\|_{L^{\infty}_{txv}}+1 \right)\int_{(0,t) \times \mathbb{R}^3_x }\left|\theta_\varepsilon \ast u^{\varepsilon,\delta}\right|^2\left(\int_{  \mathbb{R}^3_v}f^{\varepsilon,\delta}|v|^\frac32dv\right)^\frac23dxds\\
        &+C\left(\frac{4}{3} \pi e^{3t} \|f_{in}(t,x,v)\|_{L^{\infty}_{txv}}+1 \right)\norms{|v|^2f^{\varepsilon,\delta}}_{L^{\infty}\left(0,t;L^{1}_{xv}(\mathbb{R}^{6})\right)}^\frac45 \norms{u^{\varepsilon,\delta}}_{L^1(0,t;L^5(\mathbb{R}^3_x)}.
    \end{align*}
    Due to the embedding inequality
	\begin{align}\label{utp26}
		\|u\|_{L_t^q L_x^p}^2 \leq C\left(\|u\|_{L_t^\infty L_x^2}^2 + \|\nabla u\|_{L_t^2 L_x^2}^2\right), \quad \text{for} \quad \frac{2}{q} + \frac{3}{p} = \frac{3}{2}, ~ 2 \leq p \leq 6,
	\end{align}
   using \eqref{eq:esti} and H\"{o}lder's inequality, we have
	\begin{align*}
		S_1 \leq& C\left(\frac{4}{3} \pi e^{3t} \|f_{in}(t,x,v)\|_{L^{\infty}_{txv}}+1 \right)\int_{(0,t) \times \mathbb{R}^3_x }\left|\theta_\varepsilon \ast u^{\varepsilon,\delta}\right|^2\left(\int_{  \mathbb{R}^3_v}f^{\varepsilon,\delta}|v|^\frac32dv\right)^\frac23dxds\\
		&+C(t+1)\left(\frac{4}{3} \pi e^{3t} \|f_{in}(t,x,v)\|_{L^{\infty}_{txv}}+1 \right)\norms{|v|^2f^{\varepsilon,\delta}}_{L^{\infty}\left(0,t;L^{1}_{xv}(\mathbb{R}^{6})\right)}^\frac45 \times\\
        &\left(\norm{u^{\varepsilon,\delta}}_{L^\infty(0,t;L^2(\mathbb{R}^3_x))}+\norm{\nabla u^{\varepsilon,\delta}}_{L^2(0,t;L^2(\mathbb{R}^3_x))}\right)\\
		\leq& C(t+1)\left(\frac{4}{3} \pi e^{3t} \|f_{in}(t,x,v)\|_{L^{\infty}_{txv}}+1 \right)\left(\norm{u^{\varepsilon,\delta}}^2_{L^\infty(0,t;L^2(\mathbb{R}^3_x))}+\norm{\nabla u^{\varepsilon,\delta}}_{L^2(0,t;L^2(\mathbb{R}^3_x))}^2+1\right)\times\\
		&\hspace{0.5cm}\left(\norm{f^{\varepsilon,\delta}}_{L^{\infty}\left(0,t;L^{1}_{xv}(\mathbb{R}^{6})\right)}^\frac23+\norms{|v|^2f^{\varepsilon,\delta}}_{L^{\infty}\left(0,t;L^{1}_{xv}(\mathbb{R}^{6})\right)}^\frac23+\norms{|v|^2f^{\varepsilon,\delta}}_{L^{\infty}\left(0,t;L^{1}_{xv}(\mathbb{R}^{6})\right)}^\frac45\right)\\
		\leq& C\left(\norm{f_{in}}_{L^{\infty}_{xv}( \mathbb{R}^{6})}, \norm{f_{in}}_{L^{1}_{xv}(\mathbb{R}^{6})},\norms{|v|^2f_{in}}_{L^{1}_{xv}(\mathbb{R}^{6})},\norm{u_{in}}^{2}_{L^{2}_{x}(\mathbb{R}^{3})} \right)(t+1)^4e^{12t}.
	\end{align*}
    Due to \( v\cdot \nabla_v \gamma\leq 0\), \(S_2\) yields:
	\begin{align*}
		S_2 =& -3\int_{\left(0,t\right)\times \mathbb{R}_v^3\times \mathbb{R}_x^3}f^{\varepsilon,\delta}(1-\gamma_{\delta}(v))dxdvds+\int_{\left(0,t\right)\times \mathbb{R}_v^3\times \mathbb{R}_x^3}v\cdot\nabla_{v}\gamma_{\delta}(v)f^{\varepsilon,\delta}dxdvds\leq 0.
	\end{align*}
	For \(S_3\), applying \eqref{eq:esti} gives:
	\begin{align*}
		S_3&\leq C(t)\norms{|v|^2f^{\varepsilon,\delta}}_{L^{\infty}\left(0,t;L^{1}_{xv}(\mathbb{R}^{6})\right)}\\
		&\leq C\left(\norm{f_{in}}_{L^{1}_{xv}(\mathbb{R}^{6})},\norms{|v|^2f_{in}}_{L^{1}_{xv}(\mathbb{R}^{6})},\norm{u_{in}}^{2}_{L^{2}_{x}(\mathbb{R}^{3})} \right)(t+1)^2e^{3t}.
	\end{align*}
        For $S_4$, due to \eqref{fdelta}, we derive
        \begin{align*}
            S_4\leq C\left(\norms{|v|^2f_{in}}_{L^1_{xv}}+\norms{f_{in} \log f_{in}}_{L^1_{xv}}\right).
        \end{align*}
	To sum up, we derive
	\begin{equation}\label{nablavf}
	    \begin{aligned}
		\int_{\mathbb{R}^6_{xv}}&\left(\frac{1}{2}|v|^2f^{\varepsilon,\delta}+f^{\varepsilon,\delta}\log f^{\varepsilon,\delta}\right)(t,\cdot)dxdv+\int_{(0,t)\times \mathbb{R}_{xv}^6}\frac{|(\theta_\varepsilon*u^{\varepsilon,\delta}-v)f^{\varepsilon,\delta}-\nabla_{v}f^{\varepsilon,\delta}|^2}{f^{\varepsilon,\delta}}dxdvds\\
		&\leq C\left(\norm{f_{in}}_{L^{\infty}_{xv}}, \norm{f_{in}}_{L^{1}_{xv}},\norms{|v|^2f_{in}}_{L^{1}_{xv}},\norm{u_{in}}^{2}_{L^{2}_{x}}, \norm{f_{in}\log f_{in}}_{L^{1}_{xv}}\right)(t+1)^4e^{12t}.
	\end{aligned}
	\end{equation}
	Moreover, the following identity holds:
	\begin{align*}
		&\int_{(0,t)\times \mathbb{R}_{xv}^6}\frac{|(\theta_\varepsilon*u^{\varepsilon,\delta}-v)f^{\varepsilon,\delta}-\nabla_{v}f^{\varepsilon,\delta}|^2}{f^{\varepsilon,\delta}}dxdvds\\
		&=\int_{(0,t)\times \mathbb{R}_{xv}^6}|\theta_\varepsilon*u^{\varepsilon,\delta}-v|^2f^{\varepsilon,\delta}dxdvds+4\int_{(0,t)\times \mathbb{R}_{xv}^6}\left|\nabla_{v}\sqrt{f^{\varepsilon,\delta}}\right|^2dxdvds\\
		&\hspace{1cm}-4\int_{(0,t)\times \mathbb{R}_{xv}^6}(\theta_\varepsilon*u^{\varepsilon,\delta}-v)\cdot\nabla_{v}\sqrt{f^{\varepsilon,\delta}}\sqrt{f^{\varepsilon,\delta}}dxdvds\\
		&=\int_{(0,t)\times \mathbb{R}_{xv}^6}|\theta_\varepsilon*u^{\varepsilon,\delta}-v|^2f^{\varepsilon,\delta}dxdvds+4\int_{(0,t)\times \mathbb{R}_{xv}^6}\left|\nabla_{v}\sqrt{f^{\varepsilon,\delta}}\right|^2dxdvds\\
        &\hspace{1cm}-6\int_{(0,t)\times \mathbb{R}_{xv}^6}f^{\varepsilon,\delta}dxdvds.
	\end{align*}
    According to \eqref{nablavf}, it yields
    \begin{align*}
        4\int_{(0,t)\times \mathbb{R}_{xv}^6}&\left|\nabla_{v}\sqrt{f^{\varepsilon,\delta}}\right|^2dxdvds+\int_{\mathbb{R}^6_{xv}} f^{\varepsilon,\delta}\log f^{\varepsilon,\delta}(t,\cdot)dxdv \leq \\
        &C\left(\norm{f_{in}}_{L^{\infty}_{xv}}, \norm{f_{in}}_{L^{1}_{xv}},\norms{|v|^2f_{in}}_{L^{1}_{xv}},\norm{u_{in}}^{2}_{L^{2}_{x}}, \norm{f_{in}\log f_{in}}_{L^{1}_{xv}}\right)(t+1)^4e^{12t}.
    \end{align*}
	According to \eqref{18} in Proposition \ref{pro:xf} , the estimate \eqref{21} can be established. 
    % Using the same method as in \eqref{eq:6}, we derive \eqref{22}.
% \begin{revision}
% The constants used to control the negative entropy also depend on
% $\||x|^2f_{in}\|_{L_{xv}^1}$, as displayed in the statement. The initial entropy
% is that of $f_{in}^\delta$; its absolute integral is bounded by
% $\|f_{in}\log f_{in}\|_{L_{xv}^1}$ because the truncation is an indicator.
% \end{revision}
\end{proof}
	
	\subsection{\texorpdfstring{The estimates of \(|v|^\kappa f^{\varepsilon,\delta}\)}{The estimates of |v| kappa f varepsilon, delta by Tao's method}}

    We first obtain the estimates on $u^{\varepsilon,\delta}$ by Tao's method.
    
	\begin{lemma}\label{lem:ulp}
    Under the initial hypotheses of Theorem \ref{thm:solution}, 
the following estimate holds for all \(p \in [2, +\infty)\):
\begin{align}
\label{26}
    \norm{u^{\varepsilon,\delta}}_{L^1(0,t; L^p(\mathbb{R}^3_x))}
\leq C (t+1)^{4} \exp\left[C (t+1)^{3} e^{4t}\right],
\end{align}
where the constant \(C > 0\) depends explicitly on \(p\) and the norms of the initial data:
\begin{align*}
C = C\left(p, \norm{f_{\mathrm{in}}}_{L^1_{xv}(\mathbb{R}^6)}, 
\norm{f_{\mathrm{in}}}_{L^\infty_{xv}(\mathbb{R}^6)}, 
\norm{|v|^3 f_{\mathrm{in}}}_{L^1_{xv}(\mathbb{R}^6)}, 
\norm{u_{\mathrm{in}}}_{L^2_x(\mathbb{R}^3)}\right).
\end{align*}
	% Based on the initial hypotheses \eqref{0}, if \(u^{\varepsilon,\delta}\) is the smooth solution  of the system  \eqref{eq:NSVFK0},
	% we have
	% \begin{align}
	% 	\label{26}
	% 	&\norm{u^{\varepsilon,\delta}}_{L^1(0,t; L^p(\mathbb{R}^3_x))}\nonumber \\
	% 	&\leq C\left(p,\norm{f_{in}}_{L^{1}_{xv}},\norm{f_{in}}_{L^{\infty}_{xv}},\norms{|v|^3f_{in}}_{L^{1}_{xv}},\norm{u_{in}}_{L^2_x}\right)(t+1)^{\frac{3p+3}{2p}}e^{C(t+1)^\frac{14}{5}e^{6t}}.
	% \end{align}
	% where \(p \in (\frac{3}{2},+\infty)\).
    \end{lemma}
	\begin{proof}
		Let
	\begin{align*}
		&\hspace{0.5cm}u^{\varepsilon,\delta}\\
        &=e^{t\triangle} u_{in}^{\delta}-\int^{t}_{0}e^{(t-\tau) \triangle} \mathbb{P}\bigg\{(\theta_{\varepsilon}\ast u^{\varepsilon,\delta})\cdot \nabla u^{\varepsilon,\delta}-
		\theta_{\varepsilon}\ast \left\{\int_{\mathbb{R}^3}(v-\theta_{\varepsilon} \ast  u^{\varepsilon,\delta})f^{\varepsilon,\delta}dv\right\}\bigg\}(\tau,\cdot)d\tau\\
		&=e^{t\triangle} u_{in}^{\delta}-\int^{t}_{0}e^{(t-\tau) \triangle} \mathbb{P}\left((\theta_{\varepsilon}\ast u^{\varepsilon,\delta})\cdot \nabla u^{\varepsilon,\delta}\right)d\tau+\int^{t}_{0}e^{(t-\tau) \triangle} \mathbb{P}\left(\theta_{\varepsilon}\ast\int_{  \mathbb{R}^3_v}vf^{\varepsilon,\delta}dv\right)d\tau\\
		&\hspace{5cm}-\int^{t}_{0}e^{(t-\tau) \triangle} \mathbb{P}\left(\theta_{\varepsilon}\ast \int_{  \mathbb{R}^3_v}\theta_{\varepsilon}\ast u^{\varepsilon,\delta}f^{\varepsilon,\delta}dv\right)d\tau\\
		&=m_1+m_2+m_3+m_4.
	\end{align*}
	For \(m_1\), by Lemma~\ref{lem:semigroup} and \eqref{udelta}, this yields
	\begin{align*}
    \norm{e^{t\triangle} u_{in}^{\delta}}_{L^1_tL^p_x}\leq Ct^{\frac{p+6}{4p}}\norm{u_{in}}_{L^2_x},~p\geq 2.
	\end{align*}
	For \(m_2\), by H\"{o}lder's inequality and the embedding inequality \eqref{utp26}, this becomes
	\begin{equation}
		\label{24}
	    \begin{aligned}
		\norm{m_2}_{L^1_tL^\frac32_x}&\leq Ct\norm{(\theta_{\varepsilon}\ast  u^{\varepsilon,\delta})\cdot\nabla u^{\varepsilon,\delta}}_{L^1_tL^\frac32_x}\\
        &\leq Ct\norm{u^{\varepsilon,\delta}}_{L^2_tL^6_x}\norm{\nabla u^{\varepsilon,\delta}}_{L^2_tL^2_x}\\
        &\leq Ct\norm{\nabla u^{\varepsilon,\delta}}_{L^2_tL^2_x}^2 .
	\end{aligned}
	\end{equation}
	 Then, as in Proposition 9.1 of \cite{taoLocalisationCompactnessProperties2013}, we define the Littlewood--Paley projection operators on \(\mathbb{R}^3\). Let \(\varphi(\xi)\) be a fixed bump function supported in the ball \(\{\xi \in \mathbb{R}^3 : |\xi| \leq 2\}\) and equal to \(1\) on the ball \(\{\xi \in \mathbb{R}^3 : |\xi| \leq 1\}\). Define a \emph{dyadic number} to be a number \(N\) of the form \(N = 2^k\) for some integer \(k\). For each dyadic number \(N\), we define the Fourier multipliers
\begin{align*}
\widehat{P_N f}(\xi) := \psi(\xi/N) \hat{f}(\xi) := (\varphi(\xi/N) - \varphi(2\xi/N)) \hat{f}(\xi).
\end{align*}
For $u\in L^2_x$, the homogeneous decomposition $u=\sum_NP_Nu$ holds in $L^2_x$ and hence in distributions; the sum ranges over all dyadic numbers. In the next three displays, $T$ denotes the terminal time of the $L^1$ norm and $t$ is an integration variable. By Littlewood-Paley decomposition and Minkowski's inequality, we have
	\begin{align*}
		\norms{m_2}_{L^1_tL^{\infty}_x}&=\norms{\int_{0}^{t}e^{(t-\tau)\triangle}\mathbb{P}\left(\nabla \cdot (\theta_{\varepsilon}\ast u^{\varepsilon,\delta}\otimes u^{\varepsilon,\delta})\right)d\tau}_{L^1_tL^{\infty}_x}\\
		&\lesssim \sum_N\int_{0}^{T}\int_{0}^{t} \norms{P_Ne^{(t-\tau)\triangle}\mathbb{P}\left(\nabla \cdot (\theta_{\varepsilon}\ast u^{\varepsilon,\delta}\otimes u^{\varepsilon,\delta})\right)}_{L^{\infty}_x} d\tau dt.
	\end{align*}
	Using Lemma 2.2 of \cite{taoLocalisationCompactnessProperties2013} and bounding the first-order operator $\mathbb P\nabla\cdot$, whose localized symbol has size $O(N)$ on the range of \(P_N\), we may bound this by
	\begin{align*}
		\lesssim \sum_N \int_{0}^{T}\int_{0}^{t} \exp\left(-c(t-\tau)N^2\right)N\norms{\left( P_N(\theta_{\varepsilon}\ast u^{\varepsilon,\delta}\otimes u^{\varepsilon,\delta})\right)}_{L^{\infty}_x} d\tau dt
	\end{align*}
	for some \(c>0\); interchanging integrals and evaluating the \(\tau\) integral, this becomes
	\begin{align}
		\label{23}
		\lesssim\sum_N\int_{0}^{T}N^{-1}\norms{\left( P_N(\theta_{\varepsilon}\ast u^{\varepsilon,\delta}\otimes u^{\varepsilon,\delta})\right)}_{L^{\infty}_x}dt.
	\end{align}
We partition the dyadic pairs into the disjoint regions $N_1\ge 8N_2$, $N_2\ge 8N_1$, and $N_1<8N_2$, $N_2<8N_1$. Fourier support then gives $N_1\sim N$ in the first region, $N_2\sim N$ in the second, and $N\lesssim N_1\sim N_2$ in the third. Accordingly,
\begin{align*}
P_N(\theta_\varepsilon*u\otimes u)={}&\sum_{\substack{N_1\ge8N_2\\N_1\sim N}}P_N(\theta_\varepsilon*u_{N_1}\otimes u_{N_2})
+\sum_{\substack{N_2\ge8N_1\\N_2\sim N}}P_N(\theta_\varepsilon*u_{N_1}\otimes u_{N_2})\\
&+\sum_{\substack{N_1<8N_2,\ N_2<8N_1\\N\lesssim N_1}}P_N(\theta_\varepsilon*u_{N_1}\otimes u_{N_2}),
\end{align*}
where $u_N=P_Nu$. In the estimates below the summation ranges are enlarged, which is legitimate for the resulting nonnegative upper bounds.
 For \(N_1, N_2\) in the first sum,  Bernstein's inequality (Lemma \ref{Bernstein}) yields that
	\begin{align*}
		\norms{P_N(\theta_{\varepsilon}\ast u^{\varepsilon,\delta}_{N_1}\otimes u^{\varepsilon,\delta}_{N_2})}_{L^{\infty}_x} &\lesssim \norm{u^{\varepsilon,\delta}_{N_1}}_{L^{\infty}_x}\norm{u^{\varepsilon,\delta}_{N_2}}_{L^{\infty}_x}\\
		&\lesssim N_1^\frac32N_2^\frac32\norm{u^{\varepsilon,\delta}_{N_1}}_{L^{2}_x}\norm{u^{\varepsilon,\delta}_{N_2}}_{L^{2}_x}\\
		&\lesssim N(N_2/N_1)^\frac12\norm{\nabla u^{\varepsilon,\delta}_{N_1}}_{L^{2}_x}\norm{\nabla u^{\varepsilon,\delta}_{N_2}}_{L^{2}_x}.
	\end{align*}
	For the second sum, as for the proof  in the first sum, we have
	\begin{align*}
		\norms{P_N(\theta_{\varepsilon}\ast u^{\varepsilon,\delta}_{N_1}\otimes u^{\varepsilon,\delta}_{N_2})}_{L^{\infty}_x} &\lesssim N(N_1/N_2)^\frac12\norm{\nabla u^{\varepsilon,\delta}_{N_1}}_{L^{2}_x}\norm{\nabla u^{\varepsilon,\delta}_{N_2}}_{L^{2}_x}.
	\end{align*}
	For the third sum, we use Bernstein's inequality (Lemma \ref{Bernstein}) in a slightly different way to estimate
	\begin{align*}
		\norms{P_N(\theta_{\varepsilon}\ast u^{\varepsilon,\delta}_{N_1}\otimes u^{\varepsilon,\delta}_{N_2})}_{L^{\infty}_x} &\lesssim N^3\norms{\theta_{\varepsilon}\ast u^{\varepsilon,\delta}_{N_1}\otimes u^{\varepsilon,\delta}_{N_2}}_{L^{1}_x}\\
		&\lesssim N^3\norms{u^{\varepsilon,\delta}_{N_1}}_{L^{2}_x}\norms{u^{\varepsilon,\delta}_{N_2}}_{L^{2}_x}\\
		&\lesssim N(N/N_1)^2\norms{\nabla u^{\varepsilon,\delta}_{N_1}}_{L^{2}_x}\norms{\nabla u^{\varepsilon,\delta}_{N_2}}_{L^{2}_x}.
	\end{align*}
	Applying these bounds, we can estimate \eqref{23} by
	\begin{align*}
		\lesssim& \sum_N\sum_{N_1\sim N}\sum_{N_2\lesssim N}(N_2/N_1)^\frac12\int_{0}^{T}\norm{\nabla u^{\varepsilon,\delta}_{N_1}}_{L^{2}_x}\norm{\nabla u^{\varepsilon,\delta}_{N_2}}_{L^{2}_x}dt\\
		&+ \sum_N\sum_{N_2\sim N}\sum_{N_1\lesssim N}(N_1/N_2)^\frac12\int_{0}^{T}\norm{\nabla u^{\varepsilon,\delta}_{N_1}}_{L^{2}_x}\norm{\nabla u^{\varepsilon,\delta}_{N_2}}_{L^{2}_x}dt\\
		&+\sum_N\sum_{N_1\gtrsim N}\sum_{N_2\sim N_1}(N/N_1)^{2}\int_{0}^{T}\norm{\nabla u^{\varepsilon,\delta}_{N_1}}_{L^{2}_x}\norm{\nabla u^{\varepsilon,\delta}_{N_2}}_{L^{2}_x}dt.
		\end{align*}
	Applying Cauchy--Schwarz in time gives the upper bound
	\begin{align*}
		\lesssim& \sum_N\sum_{N_1\sim N}\sum_{N_2\lesssim N}(N_2/N_1)^\frac12a_{N_1}a_{N_2}+ \sum_N\sum_{N_2\sim N}\sum_{N_1\lesssim N}(N_1/N_2)^\frac12a_{N_1}a_{N_2}\\
        &+\sum_N\sum_{N_1\gtrsim N}\sum_{N_2\sim N_1}(N/N_1)^{2}a_{N_1}a_{N_2},
	\end{align*}
	where
	\begin{align*}
		a_N:=\norm{\nabla u^{\varepsilon,\delta}_{N}}_{L^{2}((0,T)\times\mathbb{R}^3_x)}.
	\end{align*}
	But form \eqref{eq:esti} and Bessel's inequality, we have
	\begin{align}
		\label{25}
			\norms{m_2}_{L^1_tL^{\infty}_x} \lesssim \norm{\nabla u^{\varepsilon,\delta}}_{L^{2}((0,t)\times\mathbb{R}^3_x)}^2.
	\end{align}
	Finally, by interpolation inequality for \(L^p\)-norms, \eqref{eq:esti}, \eqref{24} and \eqref{25}, we derive
	\begin{align*}
	\norms{m_2}_{L^1_tL^{p}_x} &\leq \norms{\norms{m_2}^\frac{3}{2p}_{L^\frac32_x}\norms{m_2}^\frac{2p-3}{2p}_{L^{\infty}_x}}_{L^1_t}\\
	& \leq \norms{m_2}^\frac{3}{2p}_{L^1_tL^\frac32_{x}}\norms{m_2}^\frac{2p-3}{2p}_{L^1_tL^{\infty}_{x}}\\
	&\leq C\left(\norm{f_{in}}_{L^{\infty}_{xv}(\mathbb{R}^{6})},\norm{f_{in}}_{L^{1}_{xv}(\mathbb{R}^{6})},\norms{|v|^2f_{in}}_{L^{1}_{xv}(\mathbb{R}^{6})},\norm{u_{in}}^{2}_{L^{2}_{x}(\mathbb{R}^{3})} \right)\left[(t+1)^3e^{3t}\right]
	\end{align*}
	for any \(p \in (\frac{3}{2}, +\infty)\). To sum up, it becomes
    \begin{align*}\norms{m_2}_{L^1_tL^{p}_x}\leq C, ~p\in\left[\frac{3}{2},+\infty\right].\end{align*}
	For \(m_3\), applying Lemma~\ref{lem:semigroup} we get
	\begin{align*}
		\norm{m_3}_{L^1_tL^p_x}&\leq \norms{\int^{t}_{0}\norms{e^{(t-\tau) \triangle} \mathbb{P}\left(\theta_{\varepsilon}\ast\int_{  \mathbb{R}^3_v}vf^{\varepsilon,\delta}dv\right)}_{L^p_x}d\tau}_{L^1_t}\\
		&\leq \norms{\int^{t}_{0}(t-\tau)^\frac{3-2p}{2p}\norms{\int_{  \mathbb{R}^3_v}vf^{\varepsilon,\delta}dv}_{L^\frac32_x}d\tau}_{L^1_t}\\
		&\leq Ct^\frac{3}{2p}\norms{\int_{ \mathbb{R}^3_v}|v|f^{\varepsilon,\delta}dv}_{L^1_tL^\frac32_x},
	\end{align*}
	and by \eqref{28}, \eqref{5} and \eqref{10}, this becomes
	\begin{align*}
		&\leq Ct^{1+\frac{3}{2p}}\norms{f^{\varepsilon,\delta}}^\frac13_{L^{\infty}_{txv}}\norms{|v|^3f^{\varepsilon,\delta}}^\frac23_{L^\infty_tL^1_{xv}}\\ 
        &\leq C(t+1)^{3}e^{C(t+1)^3e^{4t}},
	\end{align*}	
    where
    \(C=C\left(\norm{f_{in}}_{L^{1}_{xv}},\norm{f_{in}}_{L^{\infty}_{xv}},\norms{|v|^3f_{in}}_{L^{1}_{xv}}
    ,\norm{u_{in}}_{L^2_x}\right)\) and \(p\in \left[\frac{3}{2},+\infty\right)\).\\
	Following the same methodology as in the estimate of \(m_3\), we derive \(m_4\):
	\begin{align*}
		\norm{m_4}_{L^1_tL^p_x}&\leq \norms{\int^{t}_{0}\norms{e^{(t-\tau) \triangle} \mathbb{P}\left(\theta_{\varepsilon}\ast \int_{  \mathbb{R}^3_v}\theta_{\varepsilon}\ast u^{\varepsilon,\delta}f^{\varepsilon,\delta}dv\right)}_{L^p_x}d\tau}_{L^1_t}\\
		&\leq \norms{\int^{t}_{0}(t-\tau)^\frac{3-2p}{2p}\norms{\theta_{\varepsilon}\ast u^{\varepsilon,\delta}\int_{  \mathbb{R}^3_v}f^{\varepsilon,\delta}dv}_{L^\frac32_x}d\tau}_{L^1_t}\\
		&\leq Ct^\frac{3}{2p}\norms{\norms{u^{\varepsilon,\delta}}_{L^6_x}\norms{\int_{  \mathbb{R}^3_v}f^{\varepsilon,\delta}dv}_{L^2_x}}_{L^1_t}.
	\end{align*}
	By \eqref{28}, \eqref{5}, \eqref{eq:esti} and \eqref{10}, this becomes
	\begin{align*}
		&\leq Ct^\frac{3+p}{2p}\norms{u^{\varepsilon,\delta}}_{L^2_tL^6_x}\norms{f^{\varepsilon,\delta}}^\frac12_{L^{\infty}_{txv}}\norms{|v|^3f^{\varepsilon,\delta}}^\frac12_{L^{\infty}_tL^1_{xv}}\\
		&\leq C\left(\norm{f_{in}}_{L^{1}_{xv}},\norm{f_{in}}_{L^{\infty}_{xv}},\norms{|v|^3f_{in}}_{L^{1}_{xv}},\norm{u_{in}}_{L^2_x}\right)(t+1)^{3}e^{C(t+1)^3e^{4t}},
	\end{align*}
	where \(p\in \left[\frac{3}{2},+\infty\right)\). To sum up, we prove \eqref{26}.
	\end{proof}

    Using the above lemma, we then obtain the estimate on $|v|^{\kappa}f^{\varepsilon,\delta}$.
	
	\begin{lemma}\label{lem:mf}
		Under the assumptions of Theorem~\ref{thm:solution} and by Lemma~\ref{lem:ulp}, we have
	\begin{equation}
		\label{31}
	    \begin{aligned}
		&\norms{|v|^{\kappa}f^{\varepsilon,\delta}}_{L^{\infty}(0,t;L^1(\mathbb{R}^3_x\times\mathbb{R}^3_v))}\\
        % +\norms{|v|^{\bar{\kappa}}f^{\varepsilon,\delta}\log f^{\varepsilon,\delta} }_{L^{\infty}(0,t;L^{\sigma}(\mathbb{R}^3_x\times\mathbb{R}^3_v))}\nonumber\\
		&\leq C\left(\norm{u_{in}}_{L^2_x}, \norms{|v|^{\kappa}f_{in}}_{L^{1}_{xv}},\norms{f_{in}}_{L^{1}_{xv}},\norms{f_{in}}_{L^{\infty}_{xv}}\right)(t+1)^{4(\kappa+3)}e^{{C_{\kappa,\mathrm{data}}}(t+1)^3e^{4t}},
	\end{aligned}
	\end{equation}
    where \(\kappa\geq 3\), and \(C_{\kappa,\mathrm{data}}\) depends on \(\kappa\) and the initial norms displayed above, but is independent of \(\varepsilon\) and \(\delta\).
 %    \begin{align}
	% 	\label{31}
	% 	&\norms{|v|^{\kappa}f^{\varepsilon,\delta}}_{L^{\infty}(0,t;L^1(\mathbb{R}^3_x\times\mathbb{R}^3_v))}+\norms{|v|^{\frac{{\kappa}}{2}-}f^{\varepsilon,\delta}\log f^{\varepsilon,\delta} }_{L^{1}\left(\left(0,t\right)\times\Omega_x\times\mathbb{R}^3_v\right)}\nonumber\\
	% 	&\leq C\left(\norm{u_{in}}_{L^2_x(\mathbb{R}^3)}, \norms{|v|^{\kappa}f_{in}}_{L^{1}_{xv}(\mathbb{R}^{6})},\norms{f_{in}}_{L^{1}_{xv}(\mathbb{R}^{6})},\norms{f_{in}}_{L^{\infty}_{xv}(\mathbb{R}^{6})}\right)\times(t+1)^{\frac{3\kappa+12}{2}}e^{C(\kappa)(t+1)^\frac{14}{5}e^{6t}}.
	% \end{align}
	\end{lemma}
	 \begin{proof}
	 	Firstly, multiply \(|v|^{\kappa}\) on both sides of the equation \(\eqref{eq:NSVFK0}_2\):
\begin{equation}
	\label{30}
    \begin{aligned}
	\partial_t\int_{ \mathbb{R}_{xv}^6}|v|^{\kappa}f^{\varepsilon,\delta}dvdx =&\int_{ \mathbb{R}_{xv}^6} [(\theta_{\varepsilon} \ast u^{\varepsilon,\delta}-v\gamma_
	\delta(v))f^{\varepsilon,\delta}-\nabla_v f^{\varepsilon,\delta}]\cdot\nabla_v|v|^{\kappa}dvdx\\
	\leq& C({\kappa})\int_{ \mathbb{R}_x^3}|\theta_{\varepsilon}\ast u^{\varepsilon,\delta}|\int_{  \mathbb{R}^3_v}|v|^{{\kappa}-1}f^{\varepsilon,\delta}dvdx+C({\kappa})\int_{ \mathbb{R}_{xv}^6}|v|^{\kappa}f^{\varepsilon,\delta}dvdx\\
	&+C({\kappa})\int_{ \mathbb{R}_{xv}^6}|v|^{{\kappa}-2}f^{\varepsilon,\delta}dvdx.
\end{aligned}
\end{equation}
By Young's inequality, this becomes
\begin{align*}
	\leq C({\kappa})\int_{ \mathbb{R}_x^3}|\theta_{\varepsilon}\ast u^{\varepsilon,\delta}|\int_{  \mathbb{R}^3_v}|v|^{\kappa-1}f^{\varepsilon,\delta}dvdx+C({\kappa})\int_{ \mathbb{R}_{xv}^6}|v|^{\kappa}f^{\varepsilon,\delta}dvdx+C({\kappa})\int_{ \mathbb{R}_{xv}^6}f^{\varepsilon,\delta}dvdx,
\end{align*}
and using \eqref{28} and H\"{o}lder's inequality, it can be controlled  by
\begin{align*}
	\leq &C({\kappa})\norms{f^{\varepsilon,\delta}}_{L^{\infty}_{xv}}^{\frac{1}{3+\kappa}}\norm{u^{\varepsilon,\delta}}_{L^{\kappa+3}_x}\left(\int_{ \mathbb{R}_{xv}^6}|v|^{\kappa}f^{\varepsilon,\delta}dvdx\right)^{\frac{2+\kappa}{3+\kappa}}\\
	&+C({\kappa})\int_{ \mathbb{R}_{xv}^6}|v|^{\kappa}f^{\varepsilon,\delta}dvdx+C({\kappa})\int_{ \mathbb{R}_{xv}^6}f^{\varepsilon,\delta}dvdx.
\end{align*}
Then, integrating \eqref{30} with respect to \(t\) and by \eqref{fdelta}, this yields
\begin{align*}
	&\int_{ \mathbb{R}_{xv}^6}|v|^{\kappa}f^{\varepsilon,\delta}(t,\cdot)dvdx\\ \leq& \int_{ \mathbb{R}_{xv}^6}|v|^{\kappa}f_{in}dvdx +C({\kappa})\int_{0}^{t}\int_{\mathbb{R}_{xv}^6}f^{\varepsilon,\delta}dvdxds+C({\kappa})\int_{0}^{t}\int_{ \mathbb{R}_{xv}^6}|v|^{\kappa}f^{\varepsilon,\delta}dvdxds\\
	&+C({\kappa})\norms{f^{\varepsilon,\delta}}_{L^{\infty}_{txv}}^{\frac{1}{3+\kappa}}\int_{0}^{t}\norm{u^{\varepsilon,\delta}}_{L^{\kappa+3}_x}\left(\int_{ \mathbb{R}_{xv}^6}|v|^{\kappa}f^{\varepsilon,\delta}dvdx\right)^{\frac{2+\kappa}{3+\kappa}}ds.
\end{align*}
By \eqref{5} and the integral form of Gr\"{o}nwall's inequality (Lemma \ref{gronwall2}), we  derive
\begin{align*}
    \sup_{0\le s\le t}\int_{\mathbb R^6_{xv}}|v|^\kappa f^{\varepsilon,\delta}(s,x,v)\,dx\,dv
    \leq &\bigg(C_1^{\frac{1}{\kappa+3}}\exp[C(\kappa)t]\\
    &+C(\kappa)\exp[C(\kappa)t]
    \norm{f_{in}}_{L^\infty_{xv}}^{\frac{1}{\kappa+3}}
    e^{\frac{3t}{\kappa+3}}
    \int_0^t\norm{u^{\varepsilon,\delta}(s)}_{L^{\kappa+3}_x}\,ds
    \bigg)^{\kappa+3}.
\end{align*}
 which \(C_1 = \int_{ \mathbb{R}_{xv}^6}|v|^{\kappa}f_{in}dvdx +C({\kappa})t\int_{\mathbb{R}_{xv}^6}f_{in}dvdx\).
 % Note that \begin{align*}\norms{|v|^{\bar{\kappa}}f^{\varepsilon,\delta}\log f^{\varepsilon,\delta} }_{L^{\infty}(0,t;L^{\sigma}(\mathbb{R}^3_x\times\mathbb{R}^3_v))}\end{align*}can be controlled by \(\norms{|v|^{\kappa}f^{\varepsilon,\delta}}_{L^{\infty}(0,t;L^1(\mathbb{R}^3_x\times\mathbb{R}^3_v))}\), \(\norms{f^{\varepsilon,\delta}}_{L^{\infty}(0,t;L^1(\mathbb{R}^3_x\times\mathbb{R}^3_v))}\) and \(\norms{f^{\varepsilon,\delta}}_{L^{\infty}\left((0,t)\times \mathbb{R}^3_x\times\mathbb{R}^3_v\right)}\), since one can consider two cases of $f^{\varepsilon,\delta}\leq 1$ and $f^{\varepsilon,\delta}>1$. We omitted it. 
Finally, \eqref{31} is proved  by Lemma \ref{lem:ulp}.
\end{proof}

\subsection{\texorpdfstring{The estimate of \(P^{\varepsilon,\delta}\)}{The estimate of P varepsilon,delta}}
\begin{lemma}\label{le:p}
    Under the initial conditions prescribed in Theorem \ref{thm:solution}
    , we have
    \begin{align}\label{P-5/3}
        \norms{P^{\varepsilon,\delta}}_{L^\frac53((0,t)\times \mathbb{R}^3_x )}\leq C_T,
    \end{align}
    where
    \begin{align*}
        C=C\left(\norm{f_{in}}_{L^{\infty}_{xv}( \mathbb{R}^{3}\times \mathbb{R}^{3})}, \norm{f_{in}}_{L^{1}_{xv}(\mathbb{R}^{6})},\norms{|v|^2f_{in}}_{L^{1}_{xv}(\mathbb{R}^{6})},\norm{u_{in}}^{2}_{L^{2}_{x}(\mathbb{R}^{3})} \right).
    \end{align*}
Here and below $C_T$ is finite for every $T<\infty$, independent of
$\varepsilon,\delta$, and depends on $T$ and the initial norms displayed above;
the estimate holds for $0<t\le T$.
\end{lemma}
     \begin{proof}
         \par By the equation of \(u^{\varepsilon,\delta}\), we have
	\begin{align*}
		-\Delta P^{\varepsilon,\delta}=\partial_{i}\partial_{j}((\theta_{\varepsilon}\ast u^{\varepsilon,\delta}_i)u^{\varepsilon,\delta}_j)\rev{-}\nabla\cdot \left[\theta_{\varepsilon}\ast\int_{\mathbb{R}^3_v}(v-\theta_{\varepsilon}\ast u^{\varepsilon,\delta})f^{\varepsilon,\delta}dv\right].
	\end{align*}
	By Calder\'{o}n-Zygmund theorem, Lemma \ref{riesz} and Lemma \ref{hardy littlewood sobolev}, it shows that
	\begin{align*}
		\int_{ \mathbb{R}_x^3}|P^{\varepsilon,\delta}|^\frac53dx&\leq C\int_{ \mathbb{R}_x^3}|(\theta_{\varepsilon}\ast u^{\varepsilon,\delta})\otimes u^{\varepsilon,\delta}|^\frac53dx\\
		&\hspace{3cm}+C\left(\int_{ \mathbb{R}_x^3}\left|\theta_\varepsilon \ast\int_{\mathbb{R}^3_v}(v-\theta_{\varepsilon}\ast u^{\varepsilon,\delta})f^{\varepsilon,\delta}dv\right|^\frac{15}{14}dx\right)^{\rev{\frac{14}{9}}}.
	\end{align*}
	Integrating both sides of the inequality with respect of \(t\) from 0 to \(t\), this becomes
	\begin{align*}
		\int_{(0,t)\times \mathbb{R}_x^3}|P^{\varepsilon,\delta}|^\frac53dxds&\leq C\int_{ (0,t)\times\mathbb{R}_x^3}|(\theta_{\varepsilon}\ast u^{\varepsilon,\delta})\otimes u^{\varepsilon,\delta}|^\frac53dxds\\
		&\hspace{0.5cm}+C\int^t_0\left(\int_{ \mathbb{R}_x^3}\left|\int_{\mathbb{R}^3_v}(v-\theta_{\varepsilon}\ast u^{\varepsilon,\delta})f^{\varepsilon,\delta}dv\right|^\frac{15}{14}dx\right)^{\rev{\frac{14}{9}}}\int_{\mathbb{R}^3_x}|\theta_\varepsilon| dx ds\\
        &\leq C\int_{ (0,t)\times\mathbb{R}_x^3}|(\theta_{\varepsilon}\ast u^{\varepsilon,\delta})\otimes u^{\varepsilon,\delta}|^\frac53dxds\\
		&\hspace{0.5cm}+C\int^t_0\left(\int_{ \mathbb{R}_x^3}\left|\int_{\mathbb{R}^3_v}(v-\theta_{\varepsilon}\ast u^{\varepsilon,\delta})f^{\varepsilon,\delta}dv\right|^\frac{15}{14}dx\right)^{\rev{\frac{14}{9}}}ds\\
		&:= l_{1}+l_{2}.
	\end{align*}
	For \(l_1\), using H\"{o}lder's inequality, Young's inequality and the interpolation inequality, we have
	\begin{align*}
		l_1\leq C\int_{ (0,t)\times\mathbb{R}_x^3}|u^{\varepsilon,\delta}|^\frac{10}{3}dxds\leq C\left(\norm{u^{\varepsilon,\delta}}_{L^{\infty}\left (0,t;L^2{(\mathbb{R}^3_x)}\right)}+\norm{\nabla u^{\varepsilon,\delta}}_{L^{2}\left(0,t;L^2{(\mathbb{R}^3_x)}\right)}\right)^\frac{10}{3}.
	\end{align*}
	For \(l_2\), by Minkowski's inequality, it follows that
	\begin{align*}
		l_2&\leq C\int^t_0\left(\int_{ \mathbb{R}_x^3}\left|\int_{ \mathbb{R}_v^3}(\theta_{\varepsilon}\ast u^{\varepsilon,\delta})f^{\varepsilon,\delta}dv\right|^\frac{15}{14}dx\right)^{\rev{\frac{14}{9}}}ds\\
		&\hspace{5cm}+ C\int^t_0\left(\int_{ \mathbb{R}_x^3}\left|\int_{ \mathbb{R}_v^3}vf^{\varepsilon,\delta}dv\right|^\frac{15}{14}dx\right)^{\rev{\frac{14}{9}}}ds\\
		&:= l_{21}+l_{22}.
	\end{align*}
\begin{revision}
For $l_{21}$, by \eqref{28}, H\"older's inequality and \eqref{eq:esti}, we obtain
\begin{align*}
l_{21}
&\le C\int_0^t
\left\|
\bigl(\theta_\varepsilon*u^{\varepsilon,\delta}\bigr)
\int_{\mathbb R^3_v} f^{\varepsilon,\delta}\,dv
\right\|_{L^{\frac{15}{14}}(\mathbb R^3_x)}^{\frac{5}{3}}
\,ds\\
&\le C
\left\|
\int_{\mathbb R^3_v} f^{\varepsilon,\delta}\,dv
\right\|_{L^\infty(0,t;L^{\frac{30}{19}}(\mathbb R^3_x))}^{\frac{5}{3}}
\int_0^t
\left\|\theta_\varepsilon*u^{\varepsilon,\delta}
\right\|_{L^{\frac{10}{3}}(\mathbb R^3_x)}^{\frac{5}{3}}
\,ds\\
&\le Ct^{\frac{1}{2}}
\left\|u^{\varepsilon,\delta}
\right\|_{L^{\frac{10}{3}}((0,t)\times\mathbb R^3_x)}^{\frac{5}{3}}
\left\|f^{\varepsilon,\delta}
\right\|_{L^\infty((0,t)\times\mathbb R^3_x\times\mathbb R^3_v)}^{\frac{11}{18}}
\left\|(1+|v|^2)f^{\varepsilon,\delta}
\right\|_{L^\infty(0,t;L^1(\mathbb R^3_x\times\mathbb R^3_v))}^{\frac{19}{18}},
\end{align*}
where we have used the fact that
\[
\int_{\mathbb R^3_v}
|v|^{\frac{33}{19}}f^{\varepsilon,\delta}\,dv
\le
\int_{\mathbb R^3_v}
(1+|v|^2)f^{\varepsilon,\delta}\,dv.
\]
For $l_{22}$, the velocity moment of order $\frac{9}{7}$
suffices, and the same interpolation gives
\begin{align*}
l_{22}
&\le C\int_0^t
\left\|
\int_{\mathbb R^3_v}vf^{\varepsilon,\delta}\,dv
\right\|_{L^{\frac{15}{14}}(\mathbb R^3_x)}^{\frac{5}{3}}
\,ds\\
&\le Ct
\left\|f^{\varepsilon,\delta}
\right\|_{L^\infty((0,t)\times\mathbb R^3_x\times\mathbb R^3_v)}^{\frac{1}{9}}
\left\|(1+|v|^2)f^{\varepsilon,\delta}
\right\|_{L^\infty(0,t;L^1(\mathbb R^3_x\times\mathbb R^3_v))}^{\frac{14}{9}}.
\end{align*}
Together with the unchanged estimate of $l_1$, these bounds imply
\[
\left\|P^{\varepsilon,\delta}
\right\|_{L^{\frac{5}{3}}((0,T)\times\mathbb R^3_x)}
\le C_T.
\]
In particular, expanding the $\frac{5}{3}$-power of the
$L^{\frac{15}{14}}(\mathbb R^3_x)$ norm gives the outer exponent
$\frac{14}{9}=\frac{14}{15}\cdot\frac{5}{3}$
on the spatial integral.
\end{revision}
     \end{proof}

	\section{Global energy inequality and local energy inequality}
In this section, we establish a global energy inequality and two local energy inequalities of weak type for the regularized system, which play an important role in the proof of suitable weak solutions.

	\subsection{Construction of global energy inequality}$ $
	\begin{lemma}\label{lem:global}
		Under the assumptions of Theorem~\ref{thm:solution}, assume that $(u^{\varepsilon,\delta}, f^{\varepsilon,\delta})$ is a regular solution of \eqref{eq:NSVFK0} in $(0,T) \times \mathbb{R}_x^3 \times \mathbb{R}_v^3$. Then the following inequality holds:
	\begin{equation}
		\label{eq:13}
	    \begin{aligned}
		\frac{1}{2}&\int_{\mathbb{R}^3_x}|u^{\varepsilon,\delta}(t,\cdot)|^2dx+\int_{\mathbb{R}^6_{xv}}\left(\frac{1}{2}|v|^2f^{\varepsilon,\delta}+f^{\varepsilon,\delta}\log f^{\varepsilon,\delta}\right)(t,\cdot)dxdv\\
		&\hspace{0.5cm}+\int_{(0,t)\times \mathbb{R}_x^3}|\nabla u^{\varepsilon,\delta}|^2dxds+\int_{(0,t)\times \mathbb{R}_{xv}^6}\frac{|(\theta_{\varepsilon}\ast u^{\varepsilon,\delta}-v)f^{\varepsilon,\delta}-\nabla_{v}f^{\varepsilon,\delta}|^2}{f^{\varepsilon,\delta}}dxdvds\\
        &\leq \int_{(0,t)\times \mathbb{R}_{xv}^6}|v|^{2}f^{\varepsilon,\delta}(1-\gamma_{\delta}(v))dxdvds\\
		&\hspace{0.5cm}+\int_{\mathbb{R}^6_{xv}}\left(\frac{1}{2}|v|^2f_{in}+f_{in}^{\delta}\log f_{in}^{\delta}\right)dxdv+\frac{1}{2}\int_{\mathbb{R}^3_x}|u_{in}|^2dx.
	\end{aligned}
	\end{equation}
	\end{lemma}
	\begin{proof}
		Firstly, multiplying \(u^{\varepsilon,\delta}\) on the first equation of \eqref{eq:NSVFK0} and integrating over $\mathbb{R}^3$ yields:
	\begin{align}
		\label{13}
		\frac{1}{2}&\partial_t \int_{\mathbb{R}^{3}}|u^{\varepsilon,\delta}|^{2}dx + \norm{\nabla u^{\varepsilon,\delta}}^{2}_{L^{2}_{x}(\mathbb{R}^{3})} +
		\int_{\mathbb{R}^{6}}\theta_{\varepsilon}\ast u^{\varepsilon,\delta}\cdot(\theta_{\varepsilon}\ast u^{\varepsilon,\delta}-v)f^{\varepsilon,\delta}dxdv=0.
	\end{align}
    Combining \eqref{vfglobal1} and \eqref{13}, we obtain \eqref{eq:13} and the proof is complete.
	\end{proof}

\subsection{Construction of local energy inequality of the first type}$ $

Since the convolution effect expands the support set, the integration term of the convolution is involved. We expand the integration region from $\Omega_x$ to $\mathbb{R}_x^3$, but the support set of the integrand remains within a fixed bounded region.

	\begin{lemma}
    Let $\Omega_x\subset\mathbb{R}^3_x$ and
$\Omega_v\subset\mathbb{R}^3_v$ be bounded domains.
Fix $T>0$ and let
\[
0\leq \psi\in C_c^\infty((0,T)\times\Omega_x),
\qquad
0\leq \phi\in
C_c^\infty((0,T)\times\Omega_x\times\Omega_v).
\]
The test functions $\phi$ and $\psi$ are independent. In particular,
\[
\operatorname{supp}_x\psi
\cup\operatorname{supp}_x\phi\Subset\Omega_x,
\qquad
\operatorname{supp}_v\phi\Subset\Omega_v,
\]
where the subscripts denote the corresponding projections of the
supports. Both test functions vanish near the
initial time, but they need not vanish at the terminal time
$t\in(0,T)$ of integration.
		% Let \(\Omega_{x}\in \mathbb{R}^3_x\) and \(\Omega_{v}\in \mathbb{R}^3_v\) be arbitrary bounded domains and \(\phi\) and \(\psi\) be cut-off functions, which are nonnegative and satisfy the test-function convention preceding Definition~\ref{def:weak}.
  %       %In addition, we set \(\phi(t,x,v)=\psi(t,x)\omega(v)\).
        Under the assumptions of Theorem~\ref{thm:solution}, assume that $(u^{\varepsilon,\delta}, f^{\varepsilon,\delta})$ is a regular solution of \eqref{eq:NSVFK0} in \((0,t)\times\Omega_x \times\Omega_v \subset \mathbb{R}^{+}\times\mathbb{R}^{3}_{x}\times\mathbb{R}^{3}_{v}\) and \((0,t)\times\Omega_x  \subset \mathbb{R}^{+}\times\mathbb{R}^{3}_{x}\). Then the following inequality holds:
        \begin{equation}\label{eq:14}
            \begin{split}
		&\frac{1}{2}\int_{\Omega_x}\left(|u^{\varepsilon,\delta}|^2\psi\right)(t,\cdot) dx+\int_{\Omega_x \times \Omega_v}\left(\left(\frac{1}{2}|v|^2f^{\varepsilon,\delta}+f^{\varepsilon,\delta}\log f^{\varepsilon,\delta}\right)\phi\right)(t,\cdot) dxdv\\
		&+\int_{(0,t)\times \Omega_x}|\nabla u^{\varepsilon,\delta}|^2\psi dxds+\int_{(0,t)\times \Omega_x \times \Omega_v}\frac{|(\theta_{\varepsilon}\ast u^{\varepsilon,\delta}-v)f^{\varepsilon,\delta}-\nabla_{v}f^{\varepsilon,\delta}|^2}{f^{\varepsilon,\delta}}\phi dxdvds\\
		\leq&\frac{1}{2}\int_{(0,t)\times\Omega_{x}}|u^{\varepsilon,\delta}|^2(\partial_t\psi+\Delta_x \psi) dxds+\frac{1}{2}\int_{(0,t)\times\Omega_{x}}|u^{\varepsilon,\delta}|^2 (\theta_{\varepsilon}\ast u^{\varepsilon,\delta})\cdot \nabla_{x}\psi dxds\\
		&+\int_{(0,t)\times \Omega_x \times \Omega_v}\left(\frac{|v|^2}{2}+\log f^{\varepsilon,\delta}\right)f^{\varepsilon,\delta}(\partial_t\phi+\Delta_v\phi)dxdvds\\
		&+ \int_{(0,t) \times \Omega_x \times \Omega_v} \left( \frac{|v|^2}{2} + \log f^{\varepsilon,\delta} \right)f^{\varepsilon,\delta} v\cdot  \left( \nabla_{x}\phi - \nabla_{v}\phi \right)dxdvds \\
		&+\int_{(0,t)\times\Omega_{x}}(P^{\varepsilon,\delta}-\bar{P}^{\varepsilon,\delta})u^{\varepsilon,\delta}\cdot\nabla_{x}\psi dxds-\int_{(0,t) \times \Omega_x \times \Omega_v}f^{\varepsilon,\delta}v\cdot (\theta_{\varepsilon}\ast u^{\varepsilon,\delta})\phi dxdvds \\
		&+\int_{(0,t) \times \rev{\mathbb R^3_x} \times \mathbb{R}^3_v}f^{\varepsilon,\delta}v\cdot \left(\theta_{\varepsilon}\ast (u^{\varepsilon,\delta}\psi)\right) dxdvds\\
        &+ \int_{(0,t) \times \Omega_x \times \Omega_v} \left(2+\frac{|v|^2}{2}+\log f^{\varepsilon,\delta}\right)f^{\varepsilon,\delta} (\theta_\varepsilon\ast u^{\varepsilon,\delta}) \cdot \nabla_{v}\phi dxdvds \\
        &+ \int_{(0,t) \times \rev{\mathbb R^3_x} \times \mathbb{R}^3_v}|\theta_{\varepsilon}\ast u^{\varepsilon,\delta}|^2f^{\varepsilon,\delta}\psi-(\theta_{\varepsilon}\ast u^{\varepsilon,\delta})\cdot \left(\theta_{\varepsilon}\ast (u^{\varepsilon,\delta}\psi)\right)f^{\varepsilon,\delta} dxdvds\\
        &+\int_{(0,t) \times \Omega_x \times \Omega_v}f^{\varepsilon,\delta}|v|^2(1-\gamma_{\delta}(v))\phi dxdvds\\
		&+\int_{(0,t) \times \Omega_x \times \Omega_v}\left(\frac{|v|^2}{2}+\log f^{\varepsilon,\delta}\right)f^{\varepsilon,\delta}(1-\gamma_{\delta}(v))v\cdot\nabla_{v}\phi dxdvds\\
        &\rev{+\int_{(0,t)\times\mathbb R^3_x\times\mathbb R^3_v}|\theta_\varepsilon*u^{\varepsilon,\delta}|^2 f^{\varepsilon,\delta}(\phi-\psi)dxdvds}.
	\end{split}
        \end{equation}
	\end{lemma}
	\begin{proof}
		Firstly, multiplying \(u^{\varepsilon,\delta}\psi\) on the first equation of \eqref{eq:NSVFK0} and integrating over $\mathbb{R}^3$ yields:
	\begin{align*}
		\int_{\Omega_x}\partial_t u^{\varepsilon,\delta}\cdot u^{\varepsilon,\delta}\psi dx +& \int_{\Omega_{x}}(\theta_{\varepsilon}\ast u^{\varepsilon,\delta})\cdot \nabla u^{\varepsilon,\delta}\cdot u^{\varepsilon,\delta}\psi dx- \int_{ \Omega_{x}}\Delta u^{\varepsilon,\delta}\cdot u^{\varepsilon,\delta}\psi dx \\
		+& \int_{\Omega_{x}}\nabla_{x} P^{\varepsilon,\delta}\cdot u^{\varepsilon,\delta}\psi dx -
		\int_{\rev{\mathbb R^3_x}}\int_{\mathbb{R}^3_v}f^{\varepsilon,\delta}(v-\theta_{\varepsilon}\ast u^{\varepsilon,\delta})\cdot \theta_{\varepsilon}\ast(u^{\varepsilon,\delta}\psi) dvdx \\
		&:=F_1+F_2+F_3+F_4+F_5 = 0,
	\end{align*}
	where
	\begin{align*}
		F_1&=\frac{1}{2}\int_{\Omega_x}\psi\partial_t |u^{\varepsilon,\delta}|^2 dx\\
		&=\frac{1}{2}\partial_t \int_{\Omega_x}|u^{\varepsilon,\delta}|^2\psi dx-\frac{1}{2}\int_{\Omega_x}|u^{\varepsilon,\delta}|^2\partial_t \psi dx;\\
		F_2&=\frac{1}{2}\int_{\Omega_{x}}\psi (\theta_{\varepsilon}\ast u^{\varepsilon,\delta})\cdot \nabla|u^{\varepsilon,\delta}|^2dx\\
		&=-\frac{1}{2}\int_{\Omega_{x}} |u^{\varepsilon,\delta}|^2(\theta_{\varepsilon}\ast u^{\varepsilon,\delta})\cdot \nabla_{x}\psi dx\\
		F_3&=\int_{ \Omega_{x}}|\nabla u^{\varepsilon,\delta}|^2\psi dx+\frac{1}{2}\int_{ \Omega_{x}}\nabla| u^{\varepsilon,\delta}|^2\nabla_{x}\psi dx\\
		&=\int_{ \Omega_{x}}|\nabla u^{\varepsilon,\delta}|^2\psi dx-\frac{1}{2}\int_{ \Omega_{x}} |u^{\varepsilon,\delta}|^2\Delta_{x}\psi dx;\\
		F_4&=-\int_{\Omega_{x}}P^{\varepsilon,\delta}u^{\varepsilon,\delta}\cdot\nabla_{x}\psi dx\\
		&=-\int_{\Omega_{x}}(P^{\varepsilon,\delta}-\bar{P}^{\varepsilon,\delta})u^{\varepsilon,\delta}\cdot\nabla_{x}\psi dx;
	\end{align*}
	and
	\begin{align*}
		F_5&=-\int_{\rev{\mathbb R^3_x}}\int_{ \mathbb{R}^3_v}f^{\varepsilon,\delta}v\cdot (\theta_{\varepsilon}\ast (u^{\varepsilon,\delta}\psi)) dvdx+\int_{\rev{\mathbb R^3_x}}\int_{ \mathbb{R}^3_v}f^{\varepsilon,\delta}(\theta_{\varepsilon}\ast u^{\varepsilon,\delta})\cdot (\theta_{\varepsilon}\ast (u^{\varepsilon,\delta}\psi)) dvdx.
	\end{align*}
	To sum up, we obtain
	\begin{equation}
		\label{eq:15}
	    \begin{aligned}
		\frac{1}{2}\partial_t \int_{\Omega_x}&|u^{\varepsilon,\delta}|^2\psi dx+\int_{ \Omega_{x}}|\nabla u^{\varepsilon,\delta}|^2\psi dx\\
		=&\frac{1}{2}\int_{\Omega_x}|u^{\varepsilon,\delta}|^2\partial_t \psi dx+\frac{1}{2}\int_{\Omega_{x}} |u^{\varepsilon,\delta}|^2(\theta_{\varepsilon}\ast u^{\varepsilon,\delta})\cdot \nabla_{x}\psi dx\\
		&+\frac{1}{2}\int_{ \Omega_{x}} |u^{\varepsilon,\delta}|^2\Delta_{x}\psi dx+\int_{\Omega_{x}}(P^{\varepsilon,\delta}-\bar{P}^{\varepsilon,\delta})u^{\varepsilon,\delta}\cdot\nabla_{x}\psi dx\\
		&+\int_{\rev{\mathbb R^3_x}}\int_{ \mathbb{R}^3_v}f^{\varepsilon,\delta}v\cdot (\theta_{\varepsilon}\ast (u^{\varepsilon,\delta}\psi)) dvdx-\int_{\rev{\mathbb R^3_x}}\int_{ \mathbb{R}^3_v}f^{\varepsilon,\delta}(\theta_{\varepsilon}\ast u^{\varepsilon,\delta})\cdot (\theta_{\varepsilon}\ast (u^{\varepsilon,\delta}\psi)) dvdx.
	\end{aligned}
	\end{equation}
	Then, multiplying \(\left(1+\frac{|v|^{2}}{2}+\log f^{\varepsilon,\delta}\right)\phi\) on the second equation of \eqref{eq:NSVFK0} and integrating over $\mathbb{R}^6$ yields:
	\begin{align*}
		\int_{ \Omega_{x}\times\Omega_{v}}&\partial_t f^{\varepsilon,\delta}\left(1+\frac{|v|^{2}}{2}+\log f^{\varepsilon,\delta}\right)\phi dxdv \nonumber\\
        &+ \int_{  \Omega_{x}\times\Omega_{v}}(v\cdot \nabla_x) f^{\varepsilon,\delta}\left(1+\frac{|v|^{2}}{2}+\log f^{\varepsilon,\delta}\right)\phi dxdv\\
		&+ \int_{ \Omega_{x}\times\Omega_{v}}\nabla_v \cdot \left[\left(\theta_{\varepsilon} \ast u^{\varepsilon,\delta}-v\gamma_
		\delta(v)\right)f^{\varepsilon,\delta}-\nabla_v f^{\varepsilon,\delta}\right]\left(1+\frac{|v|^{2}}{2}+\log f^{\varepsilon,\delta}\right)\phi dxdv\\
		&:=G_1+G_2+G_3 = 0,
	\end{align*}
	where
	\begin{align*}
		G_1=&\partial_t \int_{  \Omega_{x}\times\Omega_{v}}f^{\varepsilon,\delta}\left(1+\frac{|v|^{2}}{2}+\log f^{\varepsilon,\delta}\right)\phi dxdv\nonumber\\
        &-\int_{  \Omega_{x}\times\Omega_{v}}f^{\varepsilon,\delta}\left(1+\frac{|v|^{2}}{2}+\log f^{\varepsilon,\delta}\right)\partial_t\phi dxdv\\
		&-\int_{  \Omega_{x}\times\Omega_{v}}f^{\varepsilon,\delta}\partial_t\left(1+\frac{|v|^{2}}{2}+\log f^{\varepsilon,\delta}\right)\phi dxdv\\
		=&\partial_t \int_{  \Omega_{x}\times\Omega_{v}}f^{\varepsilon,\delta}\left(\frac{|v|^{2}}{2}+\log f^{\varepsilon,\delta}\right)\phi dxdv-\int_{  \Omega_{x}\times\Omega_{v}}f^{\varepsilon,\delta}\left(\frac{|v|^{2}}{2}+\log f^{\varepsilon,\delta}\right)\partial_t\phi dxdv,\\
		G_2=& -\int_{ \Omega_{x}\times\Omega_{v}}  f^{\varepsilon,\delta}\left(1+\frac{|v|^{2}}{2}+\log f^{\varepsilon,\delta}\right)v\cdot\nabla_x\phi dxdv-\int_{  \Omega_{x}\times\Omega_{v}}  \phi v\cdot\nabla_xf^{\varepsilon,\delta}dxdv\\
		=&-\int_{ \Omega_{x}\times\Omega_{v}}  f^{\varepsilon,\delta}\left(\frac{|v|^{2}}{2}+\log f^{\varepsilon,\delta}\right)v\cdot\nabla_x\phi dxdv,
	\end{align*}
	and
	\begin{align*}
		G_3=&\int_{ \Omega_{x}\times\Omega_{v}}\left[\left(\theta_{\varepsilon} \ast u^{\varepsilon,\delta}-v\gamma_
		\delta(v)\right)f^{\varepsilon,\delta}-\nabla_v f^{\varepsilon,\delta}\right]\cdot\left(-v-\frac{\nabla_{v}f^{\varepsilon,\delta}}{f^{\varepsilon,\delta}}\right)\phi dxdv\\
		&-\int_{ \Omega_{x}\times\Omega_{v}}\left(1+\frac{|v|^{2}}{2}+\log f^{\varepsilon,\delta}\right)\left [\left(\theta_{\varepsilon} \ast u^{\varepsilon,\delta}-v\gamma_
		\delta(v)\right)f^{\varepsilon,\delta}-\nabla_v f^{\varepsilon,\delta}\right]\cdot\nabla_v \phi dxdv\\
		=&G_{31}+G_{32}.
	\end{align*}\\
	For \(G_{31}\), we have
	\begin{align*}
		G_{31}=&\int_{ \Omega_{x}\times\Omega_{v}}\left[(\theta_\varepsilon \ast u^{\varepsilon,\delta}-v)f^{\varepsilon,\delta}-\nabla_v f^{\varepsilon,\delta}\right]\cdot(-v-\frac{\nabla_{v}f^{\varepsilon,\delta}}{f^{\varepsilon,\delta}})\phi dxdv\\
		&+\int_{  \Omega_{x}\times\Omega_{v}}(1-\gamma_{\delta}(v))f^{\varepsilon,\delta}v\cdot\left(-v-\frac{\nabla_{v}f^{\varepsilon,\delta}}{f^{\varepsilon,\delta}}\right)\phi dxdv\\
		=&\int_{\Omega_x \times \Omega_v}\frac{|(\theta_\varepsilon \ast u^{\varepsilon,\delta}-v)f^{\varepsilon,\delta}-\nabla_{v}f^{\varepsilon,\delta}|^2}{f^{\varepsilon,\delta}}\phi dxdv-\int_{ \Omega_x \times \Omega_v}\theta_\varepsilon \ast u^{\varepsilon,\delta}\cdot (\theta_\varepsilon \ast u^{\varepsilon,\delta}-v)f^{\varepsilon,\delta}\phi dxdv\\
		&-\int_{  \Omega_{x}\times\Omega_{v}}f^{\varepsilon,\delta}(\theta_\varepsilon \ast u^{\varepsilon,\delta})\cdot\nabla_{v}\phi dxdv-\int_{ \Omega_{x}\times\Omega_{v}}v\cdot(vf^{\varepsilon,\delta}+\nabla_{v}f^{\varepsilon,\delta})(1-\gamma_{\delta}(v))\phi dxdv\\
		\geq&\int_{\Omega_x \times \Omega_v}\frac{|(\theta_\varepsilon \ast u^{\varepsilon,\delta}-v)f^{\varepsilon,\delta}-\nabla_{v}f^{\varepsilon,\delta}|^2}{f^{\varepsilon,\delta}}\phi dxdv-\int_{ \Omega_x \times \Omega_v}\theta_\varepsilon \ast u^{\varepsilon,\delta}\cdot (\theta_\varepsilon \ast u^{\varepsilon,\delta}-v)f^{\varepsilon,\delta}\phi dxdv\\
		&-\int_{ \Omega_{x}\times\Omega_{v}}|v|^2f^{\varepsilon,\delta}(1-\gamma_{\delta}(v))\phi dxdv+\int_{ \Omega_{x}\times\Omega_{v}}f^{\varepsilon,\delta}(1-\gamma_{\delta}(v))v\cdot\nabla_{v}\phi dxdv\\
		&-\int_{  \Omega_{x}\times\Omega_{v}}f^{\varepsilon,\delta}(\theta_\varepsilon \ast u^{\varepsilon,\delta})\cdot\nabla_{v}\phi dxdv.
	\end{align*}
	For \(G_{32}\), this becomes
	\begin{align*}
		G_{32}=&\int_{ \Omega_{x}\times\Omega_{v}}\left(1+\frac{|v|^{2}}{2}+\log f^{\varepsilon,\delta}\right) \left[(v-\theta_{\varepsilon}\ast u^{\varepsilon,\delta})f^{\varepsilon,\delta}+\nabla_v f^{\varepsilon,\delta}\right]\cdot\nabla_v \phi dxdv\\
		&+\int_{ \Omega_{x}\times\Omega_{v}}\left(1+\frac{|v|^{2}}{2}+\log f^{\varepsilon,\delta}\right) f^{\varepsilon,\delta}(\gamma_
		\delta(v)-1)v\cdot\nabla_v \phi dxdv\\
		=&G_{32-1}+G_{32-2},
	\end{align*}
	where
	\begin{align*}
		G_{32-1}=&\int_{ \Omega_{x}\times\Omega_{v}}\left(1+\frac{|v|^{2}}{2}+\log f^{\varepsilon,\delta}\right) f^{\varepsilon,\delta}v\cdot\nabla_v \phi dxdv\\
		&-\int_{  \Omega_{x}\times\Omega_{v}}f^{\varepsilon,\delta}\left(v+\frac{\nabla_{v}f^{\varepsilon,\delta}}{f^{\varepsilon,\delta}}\right)\cdot\nabla_{v}\phi dxdv\\
        &-\int_{ \Omega_{x}\times\Omega_{v}}\left(1+\frac{|v|^{2}}{2}+\log f^{\varepsilon,\delta}\right) f^{\varepsilon,\delta}\Delta_v \phi dxdv\\
        &-\int_{ \Omega_{x}\times\Omega_{v}}\left(1+\frac{|v|^{2}}{2}+\log f^{\varepsilon,\delta}\right) f^{\varepsilon,\delta}(\theta_{\varepsilon}\ast u^{\varepsilon,\delta})\cdot\nabla_v \phi dxdv\\
		=&\int_{ \Omega_{x}\times\Omega_{v}}\left(\frac{|v|^{2}}{2}+\log f^{\varepsilon,\delta}\right) f^{\varepsilon,\delta}v\cdot\nabla_v \phi dxdv\\
        &-\int_{ \Omega_{x}\times\Omega_{v}}\left(\frac{|v|^{2}}{2}+\log f^{\varepsilon,\delta}\right) f^{\varepsilon,\delta}\Delta_v \phi dxdv\\
		&-\int_{ \Omega_{x}\times\Omega_{v}}\left(1+\frac{|v|^{2}}{2}+\log f^{\varepsilon,\delta}\right) f^{\varepsilon,\delta}(\theta_{\varepsilon}\ast u^{\varepsilon,\delta})\cdot\nabla_v \phi dxdv.
  %       \int_{ \Omega_{x}\times\Omega_{v}}\left(1+\frac{|v|^{2}}{2}+\log f^{\varepsilon,\delta}\right) f^{\varepsilon,\delta}v\cdot\nabla_v \phi dxdv\\
		% &-\int_{ \Omega_{x}\times\Omega_{v}}\left(1+\frac{|v|^{2}}{2}+\log f^{\varepsilon,\delta}\right) f^{\varepsilon,\delta} (\theta_{\varepsilon}\ast u^{\varepsilon,\delta})\cdot\nabla_v \phi dxdv\\
		% &+\int_{ \Omega_{x}\times\Omega_{v}}\left(1+\frac{|v|^{2}}{2}+\log f^{\varepsilon,\delta}\right) \nabla_{v}f^{\varepsilon,\delta}\cdot\nabla_v \phi dxdv\\
		% =&
	\end{align*}
	To sum up, it follows that
	\begin{equation}
		\label{eq:16}
         \begin{aligned}
		\partial_t &\int_{  \Omega_{x}\times\Omega_{v}}f^{\varepsilon,\delta}\left(\frac{|v|^{2}}{2}+\log f^{\varepsilon,\delta}\right)\phi dxdv+\int_{\Omega_x \times \Omega_v}\frac{|(\theta_{\varepsilon}\ast u^{\varepsilon,\delta}-v)f^{\varepsilon,\delta}-\nabla_{v}f^{\varepsilon,\delta}|^2}{f^{\varepsilon,\delta}}\phi dxdv\\
		\leq&\int_{  \Omega_{x}\times\Omega_{v}}f^{\varepsilon,\delta}\left(\frac{|v|^{2}}{2}+\log f^{\varepsilon,\delta}\right)(\partial_t\phi+\Delta_{v}\phi) dxdv\\
        &+\int_{ \Omega_{x}\times\Omega_{v}}\left(2+\frac{|v|^{2}}{2}+\log f^{\varepsilon,\delta}\right) f^{\varepsilon,\delta}(\theta_\varepsilon\ast u^{\varepsilon,\delta})\cdot\nabla_v \phi dxdv\\
		&+\int_{ \Omega_{x}\times\Omega_{v}}  f^{\varepsilon,\delta}\left(\frac{|v|^{2}}{2}+\log f^{\varepsilon,\delta}\right)v\cdot(\nabla_x\phi-\nabla_{v}\phi) dxdv\\
        &+\int_{ \Omega_x \times \Omega_v}\theta_{\varepsilon}\ast u^{\varepsilon,\delta}\cdot(\theta_{\varepsilon}\ast u^{\varepsilon,\delta}-v)f^{\varepsilon,\delta}\phi dxdv\\
		&+\int_{ \Omega_{x}\times\Omega_{v}}|v|^2f^{\varepsilon,\delta}(1-\gamma_{\delta}(v))\phi dxdv\\
        &+\int_{ \Omega_{x}\times\Omega_{v}}\left(\frac{|v|^{2}}{2}+\log f^{\varepsilon,\delta}\right) f^{\varepsilon,\delta}(1-\gamma_
		\delta(v))v\cdot\nabla_v \phi dxdv.
	\end{aligned}
	\end{equation}
	
\begin{revision}
Adding the coupling terms in \eqref{eq:15} and \eqref{eq:16},
we obtain
\begin{align*}
\begin{aligned}
&\left(\int_{\mathbb R^3_v} v f^{\varepsilon,\delta}\,dv\right)
 \cdot\bigl[\theta_\varepsilon*
 (u^{\varepsilon,\delta}\psi)\bigr]
-\left(\int_{\mathbb R^3_v} f^{\varepsilon,\delta}\,dv\right)
 (\theta_\varepsilon*u^{\varepsilon,\delta})
 \cdot\bigl[\theta_\varepsilon*
 (u^{\varepsilon,\delta}\psi)\bigr]
\\
&\quad
-(\theta_\varepsilon*u^{\varepsilon,\delta})
 \cdot\left(\int_{\mathbb R^3_v}
 v f^{\varepsilon,\delta}\phi\,dv\right)
+\left|\theta_\varepsilon*u^{\varepsilon,\delta}\right|^2
 \int_{\mathbb R^3_v} f^{\varepsilon,\delta}\phi\,dv
\\
&=
\left(\int_{\mathbb R^3_v} v f^{\varepsilon,\delta}\,dv\right)
 \cdot\bigl[\theta_\varepsilon*
 (u^{\varepsilon,\delta}\psi)\bigr]
-(\theta_\varepsilon*u^{\varepsilon,\delta})
 \cdot\left(\int_{\mathbb R^3_v}
 v f^{\varepsilon,\delta}\phi\,dv\right)
\\
&\quad
+\Biggl[
 \left|\theta_\varepsilon*u^{\varepsilon,\delta}\right|^2
 \psi\int_{\mathbb R^3_v} f^{\varepsilon,\delta}\,dv
-\left(\int_{\mathbb R^3_v} f^{\varepsilon,\delta}\,dv\right)
 (\theta_\varepsilon*u^{\varepsilon,\delta})
 \cdot\bigl[\theta_\varepsilon*
 (u^{\varepsilon,\delta}\psi)\bigr]
 \Biggr]
\\
&\quad
+\left|\theta_\varepsilon*u^{\varepsilon,\delta}\right|^2
 \left(
 \int_{\mathbb R^3_v} f^{\varepsilon,\delta}\phi\,dv
 -\psi\int_{\mathbb R^3_v} f^{\varepsilon,\delta}\,dv
 \right).
\end{aligned}
\end{align*}
Consequently, after integration in space and time, the last term
on the right-hand side contributes
\begin{align*}
\int_0^\tau\!\int_{\mathbb R^3_x\times\mathbb R^3_v}
\left|\theta_\varepsilon*u^{\varepsilon,\delta}\right|^2
f^{\varepsilon,\delta}(\phi-\psi)dxdv ds,
\end{align*}
which is retained in \eqref{eq:14}. No relation between $\phi$
and $\psi$ is assumed. Combining this identity with the remaining
terms in \eqref{eq:15} and \eqref{eq:16} and integrating in time
proves \eqref{eq:14}.
\end{revision}
	\end{proof}

	\subsection{Construction of local energy inequality of the second type}$ $
	\begin{lemma}\label{lem:local2}
    Let $\Omega_x\subset\mathbb{R}^3_x$ be a bounded domain.
Fix $T>0$ and let
\[
0\leq\psi\in C_c^\infty((0,T)\times\Omega_x).
\]
In particular, $\operatorname{supp}_x\psi\Subset\Omega_x$.
The test function vanishes near the initial time, but it need not
vanish at the terminal time $t\in(0,T)$ of integration.
		% Let \(\Omega_x\subset\mathbb R^3_x\) be a bounded domain, and let \(\psi\) satisfy the test-function convention preceding Definition~\ref{def:weak}.
        Under the assumptions stated in Theorem \ref{thm:solution}, assume \((u^{\varepsilon,\delta}, f^{\varepsilon,\delta})\) is a regular solution of \eqref{eq:NSVFK0} in \((0,t)\times\Omega_x  \subset \mathbb{R}^{+}\times\mathbb{R}^{3}_{x}\) and \((0,t)\times\Omega_x \times\mathbb{R}^{3}_{v} \subset \mathbb{R}^{+}\times\mathbb{R}^{3}_{x}\times\mathbb{R}^{3}_{v}\). Then the following inequality holds:
	    \begin{align}
		\label{eq:23}
		&\frac{1}{2}\int_{\Omega_x}\left(|u^{\varepsilon,\delta}|^2\psi\right)(t,\cdot) dx+\int_{\Omega_x \times \mathbb{R}^{3}_{v}}\left(\left(\frac{1}{2}|v|^2f^{\varepsilon,\delta}+f^{\varepsilon,\delta}\log f^{\varepsilon,\delta}\right)\psi\right)(t,\cdot) dxdv\nonumber\\
		&+\int_{(0,t)\times \Omega_x}|\nabla u^{\varepsilon,\delta}|^2\psi dxds+\int_{(0,t)\times \Omega_x \times \mathbb{R}^{3}_{v}}\frac{|(\theta_{\varepsilon}\ast u^{\varepsilon,\delta}-v)f^{\varepsilon,\delta}-\nabla_{v}f^{\varepsilon,\delta}|^2}{f^{\varepsilon,\delta}}\psi dxdvds\nonumber\\
        \leq& \frac{1}{2}\int_{(0,t)\times\Omega_{x}}|u^{\varepsilon,\delta}|^2(\partial_t\psi+\Delta_x\psi) dxds\nonumber\\
		&+\frac{1}{2}\int_{(0,t)\times\Omega_{x}}|u^{\varepsilon,\delta}|^2(\theta_{\varepsilon}\ast u^{\varepsilon,\delta})\cdot \nabla_{x}\psi dxds
		+\int_{(0,t)\times\Omega_{x}}(P^{\varepsilon,\delta}-\bar{P}^{\varepsilon,\delta})u^{\varepsilon,\delta}\cdot\nabla_{x}\psi dxds\nonumber\\
        &+ \int_{(0,t) \times \rev{\mathbb R^3_x} \times \mathbb{R}^3_v}|\theta_{\varepsilon}\ast u^{\varepsilon,\delta}|^2f^{\varepsilon,\delta}\psi-(\theta_{\varepsilon}\ast u^{\varepsilon,\delta})\cdot \left(\theta_{\varepsilon}\ast (u^{\varepsilon,\delta}\psi)\right)f^{\varepsilon,\delta} dxdvds
        \\
        &+\int_{(0,t) \times \Omega_x \times \mathbb{R}^{3}_{v}}f^{\varepsilon,\delta}|v|^2(1-\gamma_{\delta}(v))\psi dxdvds\nonumber\\
        &+\int_{(0,t)\times \Omega_x \times \mathbb{R}^{3}_{v}}\left(\frac{|v|^2}{2}+\log f^{\varepsilon,\delta}\right)f^{\varepsilon,\delta}\partial_t\psi dxdvds\nonumber\\
		&+ \int_{(0,t) \times \Omega_x \times \mathbb{R}^{3}_{v}} \left( \frac{|v|^2}{2} + \log f^{\varepsilon,\delta} \right)f^{\varepsilon,\delta} v\cdot \nabla_{x}\psi dxdvds \nonumber\\
        &+\int_{(0,t) \times \rev{\mathbb R^3_x} \times \mathbb{R}^3_v}f^{\varepsilon,\delta}v\cdot \left(\theta_{\varepsilon}\ast (u^{\varepsilon,\delta}\psi)\right)-f^{\varepsilon,\delta}v\cdot \left(\theta_{\varepsilon}\ast u^{\varepsilon,\delta}\right)\psi dxdvds\nonumber.
	\end{align}
	\end{lemma}
	\begin{proof}
    Replacing \(\phi\) with \(\psi\) in \eqref{eq:16} yields a new energy inequality. Combining this energy inequality with \eqref{eq:15} implies \eqref{eq:23}.  
	\end{proof}

	\section{Convergence}
	\subsection{\texorpdfstring{Weak convergence in \(L^p\) Spaces}{The weak convergence of  j varepsilon delta u varepsilon delta and n varepsilon delta u varepsilon delta in L {1}((0,T) times mathbb{R} 3 x)}}$ $
    \par 
    Define \(j^{\varepsilon,\delta}=\int_{ \mathbb{R}_{v}^3}vf^{\varepsilon,\delta}dv, n^{\varepsilon,\delta}=\int_{ \mathbb{R}_{v}^3}f^{\varepsilon,\delta}dv\). In this subsection, we prove weak convergence of $j^{\varepsilon,\delta}$, $n^{\varepsilon,\delta}$ and 
    $\theta_{\varepsilon} \ast u^{\varepsilon,\delta}$ in some $L^p$ norms.
    Before giving the proof, we recall a useful lemma regarding the issue of weak convergence for the product of two weakly convergent functions.
	\begin{proposition}[Lemma 5.1 in \cite{lionsMathematicalTopicsFluid1998}]\label{pro:lions}
		 Let \(\varOmega\) be \(\mathbb{R}^d\) or a bounded open domain with smooth boundary. Suppose \(g^n,h^n\) converge weakly to \(g,h\) respectively in \(L^{p_1}(0,T;L^{p_2}(\varOmega))\) and \( L^{q_1}(0,T;L^{q_2}(\varOmega))\) where \((p_1,q_1), (p_2,q_2)\) are conjugate pairs, and \(1\leq p_i,q_i \leq \infty\). We assume that 
        %for some \(m\geq 0\) which is independent of \(n\),
	\begin{gather*}
		\partial_tg^n ~\text{bounded in}~ L^{1}\left(0,T;W^{-1,1}(\varOmega)\right),\\
		\norm{h^n-h^n(\cdot,\cdot+\xi)}_{L^{q_1}(0,T;L^{q_2}(\varOmega))}\rightarrow 0 \quad \text{as}~ |\xi|\rightarrow0, \text{uniformly in}~ n.
	\end{gather*}
	Then \(g^nh^n \rightharpoonup gh\) (in the sense of distribution uniformly on \(\varOmega \times (0,T)\)).
	\end{proposition}

Motivated by Lemma 3.2 in \cite{melletBarotropicCompressibleNavierStokes2007}, we have the following estimates.
    
	\begin{lemma}\label{lem:jnbound}
		 Let \(f\) be a measurable function in region \([0,T] \times \mathbb{R}^3_x\times \mathbb{R}^3_v\). Assume that $f\geq0$ and
\begin{align*}\norm{f}_{L^{\infty}([0,T]\times\mathbb R^6)}
+\operatorname*{ess\,sup}_{t\in[0,T]}\int_{\mathbb R^6}(1+|v|^2)f\,dx\,dv\le M.
	\end{align*}
    Then there exists a constant \(C\), depending on \(M\), such that
	\ben\label{eq:thebound-n}
    \operatorname*{ess\,sup}_{t\in[0,T]}\norm{n(t)}_{L^{p}(\mathbb{R}^3_x)} &\leq C \quad \text{for} ~p \in [1,\frac{5}{3}],\een
\ben\label{eq:thebound-j}
		\operatorname*{ess\,sup}_{t\in[0,T]}\norm{j(t)}_{L^{p}(\mathbb{R}^3_x)} &\leq C \quad \text{for} ~p \in [1,\frac{5}{4}],
	\een
    where \begin{align*}j=\int_{ \mathbb{R}_{v}^3}vfdv, n=\int_{ \mathbb{R}_{v}^3}fdv.\end{align*}
	\end{lemma}
	\begin{proof}
{\bf Proof of \eqref{eq:thebound-n}.}
		Let \(p,q \in (1,+\infty)\)  satisfying \(\frac{1}{p}+\frac{1}{q}=1\), and we have
	\begin{align*}
		n(t,x)&=\int_{ \mathbb{R}_v^3}(1+|v|)^{\frac{2}{p}}f^{\frac{1}{p}}\frac{f^{\frac{1}{q}}}{(1+|v|)^{\frac{2}{p}}}dv\\
		&\leq \left(\int_{ \mathbb{R}_v^3}(1+|v|)^{2}fdv\right)^{\frac{1}{p}}\left(\int_{ \mathbb{R}_v^3}\frac{f}{(1+|v|)^{\frac{2q}{p}}}dv\right)^{\frac{1}{q}}.
	\end{align*}
	In particular, if \(\frac{2q}{p} >3\), we deduce
	\begin{align*}
		n(x,t)\leq C\norm{f}_{L^{\infty}_v}^{\frac{1}{q}}\left(\int_{ \mathbb{R}_v^3}(1+|v|)^{2}fdv\right)^{\frac{1}{p}}
	\end{align*}
	and
	\begin{align*}
		\operatorname*{ess\,sup}_{t\in[0,T]}\norm{n(t)}_{L^{p}_x}^p\leq C\operatorname*{ess\,sup}_{t\in[0,T]}\int_{ \mathbb{R}_{xv}^6}(1+|v|)^2fdvdx.
	\end{align*}
	Then, note that the condition \(\frac{2q}{p}>3\) is equivalent to
	\begin{align*}
		p<\frac{5}{3}.
	\end{align*}
	As \(p=\frac{5}{3}\), it follows from Lemma \ref{lem:f} that
	\begin{align*}
		\operatorname*{ess\,sup}_{t\in[0,T]}\norm{n(t)}^\frac53_{L^\frac53(\mathbb{R}^3_x)}\leq \left(\frac{4}{3}\pi\norm{f}_{L^{\infty}_{txv}}+1\right)^\frac53\operatorname*{ess\,sup}_{t\in[0,T]}\int_{ \mathbb{R}_{xv}^6}f|v|^2dxdv\leq C.
	\end{align*}
    For \(p=1\), we have
\[
\|n(t)\|_{L^1(\mathbb R^3_x)}
=\int_{\mathbb R^6_{xv}}f(t,x,v)\,dx\,dv\leq C.
\]
	% \textcolor{blue}{For $p=1$, $\|n(t)\|_{L^1(\mathbb{R}^3_x)}$ is controlled by the mass.} 
    To sum up, we derive \(p\in [1,\frac{5}{3}]\). 

{\bf Proof of \eqref{eq:thebound-j}.} The same arguments imply
	\begin{align*}
		|j(t,x)|&\leq\int_{ \mathbb{R}_v^3}(1+|v|)^{\frac{2}{p}}f^{\frac{1}{p}}\frac{f^{\frac{1}{q}}}{(1+|v|)^{\frac{2}{p}-1}}dv\\
		&\leq \left(\int_{ \mathbb{R}_v^3}(1+|v|)^{2}fdv\right)^{\frac{1}{p}}\left(\int_{ \mathbb{R}_v^3}\frac{f}{(1+|v|)^{\frac{2q}{p}-q}}dv\right)^{\frac{1}{q}}.
	\end{align*}
	In particular, if \(\frac{2q}{p}-q >3\), we deduce
	\begin{align*}
		|j(t,x)|\leq C\norm{f}_{L^{\infty}_v}^{\frac{1}{q}}\left(\int_{ \mathbb{R}_v^3}(1+|v|)^{2}fdv\right)^{\frac{1}{p}}
	\end{align*}
	and
	\begin{align*}
		\operatorname*{ess\,sup}_{t\in[0,T]}\norm{j(t)}_{L^{p}_x}^p\leq C\operatorname*{ess\,sup}_{t\in[0,T]}\int_{ \mathbb{R}_{xv}^6}(1+|v|)^2fdvdx.
	\end{align*}
	Finally, note that the condition \(\frac{2q}{p}-q>3\) is equivalent to
	\begin{align*}
		p<\frac{5}{4}.
	\end{align*}
	As \(p=\frac{5}{4}\), it follows from Lemma \ref{lem:f} that
	\begin{align*}
		\operatorname*{ess\,sup}_{t\in[0,T]}\norm{j(t)}^\frac{5}{4}_{L^\frac{5}{4}(\mathbb{R}^3_x)}\leq \left(\frac{4}{3}\pi\norm{f}_{L^{\infty}_{txv}}+1\right)^\frac{5}{4}\operatorname*{ess\,sup}_{t\in[0,T]}\int_{ \mathbb{R}_{xv}^6}f|v|^2dxdv\leq C.
	\end{align*}
    For \(p=1\), Cauchy--Schwarz yields
\[
\|j(t)\|_{L^1(\mathbb R^3_x)}
\leq
\left(\int_{\mathbb R^6_{xv}}f(t,x,v)\,dx\,dv\right)^{\frac12}
\left(\int_{\mathbb R^6_{xv}}|v|^2f(t,x,v)\,dx\,dv\right)^{\frac12}
\leq C.
\]
	% \textcolor{blue}{For $p=1$, Cauchy--Schwarz gives 
 %    \begin{align*}
 %    \|j(t)\|_{L^1(\mathbb{R}^3_x\auditmath{)}}\le \left(\int_{\mathbb{R}^6_{xv}} f\,dx\,dv\right)^\frac12\left(\int_{\mathbb{R}^6_{xv}} |v|^2f\,dx\,dv\right)^\frac12\leq C.
 %    \end{align*}}
    To sum up, we derive \(p\in [1,\frac{5}{4}]\).
	\end{proof}

	\begin{proposition}\label{pro:jnweak}
		Under the assumptions of Theorem  \ref{thm:solution}, the following convergence results hold:
	\begin{align}
    j^{\varepsilon,\delta}u^{\varepsilon,\delta}
    &\longrightarrow ju
    \quad\text{in }\mathcal D'
    \bigl((0,T)\times\mathbb R^3_x\bigr),
    \label{eq:17}
    \\
    n^{\varepsilon,\delta}u^{\varepsilon,\delta}
    &\longrightarrow nu
    \quad\text{in }\mathcal D'
    \bigl((0,T)\times\mathbb R^3_x\bigr).
\end{align}
   where \(j^{\varepsilon,\delta}=\int_{ \mathbb{R}_{v}^3}vf^{\varepsilon,\delta}dv, n^{\varepsilon,\delta}=\int_{ \mathbb{R}_{v}^3}f^{\varepsilon,\delta}dv\).
	% Moreover, the following limit behavior occurs as $\delta \to 0$:\\
	% \begin{align}
	% 	\label{eq:18}
	% 	j^{\delta}u&\rightharpoonup ju \quad \text{in}~L^{1}((0,T)\times \mathbb{R}^3_x),\\
	% 	n^{\delta}u&\rightharpoonup nu \quad \text{in}~L^{1}((0,T)\times \mathbb{R}^3_x).
	% \end{align}
	\end{proposition}
	\begin{proof}
		Firstly, due to Lemma \ref{lem:jnbound}, it follows that
	\begin{align*}
		n^{\varepsilon,\delta}\rightharpoonup n \quad L^{\infty}(0,T;L^{p}(\mathbb{R}_x^3))-\text{weak\textasteriskcentered} \quad \forall p\in \textcolor{black}{(1,\frac{5}{3}]},\\
		j^{\varepsilon,\delta}\rightharpoonup j \quad L^{\infty}(0,T;L^{p}(\mathbb{R}_x^3))-\text{weak\textasteriskcentered} \quad \forall p\in \textcolor{black}{(1,\frac{5}{4}]},
	\end{align*}
	as $\varepsilon,\delta\rightarrow 0$ (up to a subsequence)\footnote{Next the statement of $\varepsilon,\delta\rightarrow 0$ (taking at most one subsequence) is frequently used below. For the sake of simplicity and without causing confusion, we will omit this expression.}. Also, due to \eqref{eq:esti}, we have
	\begin{align}
		\label{01}
		u^{\varepsilon,\delta} \rightharpoonup u \quad L^{2}(0,T;L^{p}(\mathbb{R}_x^3))-\text{weak} \quad \forall p\in [2,6].
	\end{align}
	Then, integrating the second equation of \eqref{eq:NSVFK0} with respect to \(v\), we find \(\partial_t n^{\varepsilon,\delta}=-div_x j^{\varepsilon,\delta}\). It shows that
	\begin{align*}
		\partial_t n^{\varepsilon,\delta} ~\text{is bounded in}~                 L^{1}(0,T;W^{-1,1}(\mathbb{R}^3_x)).
	\end{align*}
	Multiplying the second equation of \eqref{eq:NSVFK0} by \(v\) and integrating with respect to \(v\) over \(\mathbb R^3_v\) yields
	\begin{align*}
		\partial_t j^{\varepsilon,\delta} = -\int_{ \mathbb{R}_v^3}\left(v\cdot \nabla_x f^{\varepsilon,\delta}\right)vdv\rev{+}\int_{ \mathbb{R}_v^3} \left(\theta_{\varepsilon} \ast u^{\varepsilon,\delta}-v\gamma_
		\delta(v)\right)f^{\varepsilon,\delta}dv.
	\end{align*}
\rev{The integration by parts here is componentwise:
    $\int v_i\partial_{v_j}B_j\,dv=-\int B_i\,dv$.}

For \(\int_{ \mathbb{R}_{v}^3} \left(\theta_{\varepsilon} \ast u^{\varepsilon,\delta}-v\gamma_
	\delta(v)\right)f^{\varepsilon,\delta}dv\), by \eqref{eq:esti}, \eqref{utp26} and \eqref{eq:thebound-n}, this becomes
	\begin{align*}
		\int_{ (0,T)\times\mathbb{R}_{xv}^6} &|\left(\theta_{\varepsilon} \ast u^{\varepsilon,\delta}-v\gamma_
		\delta(v)\right)f^{\varepsilon,\delta}|dxdvdt\\
		\leq&\int_{ (0,T)\times\mathbb{R}_{xv}^6} |\theta_{\varepsilon} \ast u^{\varepsilon,\delta}|f^{\varepsilon,\delta}dvdxdt +\int_{ (0,T)\times\mathbb{R}_{xv}^6} |v|f^{\varepsilon,\delta}dvdxdt\\
		\leq& CT^\frac{17}{20}\norm{n^{\varepsilon,\delta}}_{L^{\infty}(0,T;L^\frac53_{x}(\mathbb{R}^{3}))} \norm{u^{\varepsilon,\delta}}_{L^{\frac{20}{3}}(0,T;L^{\frac{5}{2}}_{x}(\mathbb{R}^{3}))}+\frac{1}{2}\int_{ (0,T)\times\mathbb{R}_{xv}^6}f^{\varepsilon,\delta}dvdxdt\\
        &+\frac{1}{2}\int_{ (0,T)\times\mathbb{R}_{xv}^6}|v|^2f^{\varepsilon,\delta}dvdxdt\\
		\leq&C(T).
	\end{align*}
	By
    \begin{align*}
	&-\int_{ \mathbb{R}_v^3}\left(v\cdot \nabla_x f^{\varepsilon,\delta}\right)vdv\in L^{1}(0,T;W^{-1,1}(\mathbb{R}^3_x)),\\
    &\int_{ \mathbb{R}_v^3} \left(\theta_{\varepsilon} \ast u^{\varepsilon,\delta}-v\gamma_
	\delta(v)\right)f^{\varepsilon,\delta}dv\in L^{1}(0,T;L^{1}(\mathbb{R}^3_{x})),    
	\end{align*}
    it shows that
	\begin{align*}
		\partial_t j^{\varepsilon,\delta}~ \text{is bounded in}                 ~L^{1}(0,T;W^{-1,1}(\mathbb{R}^3_x)).
	\end{align*}
    Finally, since \(\nabla u^{\varepsilon,\delta}\) is bounded in \(L^{2}((0,T)\times \mathbb{R}^3_x)\), we derive
	\begin{align*}
		\norm{u^{\varepsilon,\delta}(t,x+\xi)-u^{\varepsilon,\delta}(t,x)}_{L^2_tL^2_x}&\leq \int_{0}^{1}|\xi|\norm{\nabla u^{\varepsilon,\delta}(t,x+\tau\xi)}_{L^2_tL^2_x}d\tau\\
		&\leq C|\xi|.
	\end{align*}
The interpolation inequality for $L^p$-norms yields
\begin{align*}
	&\hspace{0.5cm}\|u^{\varepsilon,\delta}(t,x+\xi) - u^{\varepsilon,\delta}(t,x)\|_{L^2_tL^p_x}\\ 
	&\leq \|u^{\varepsilon,\delta}(t,x+\xi) - u^{\varepsilon,\delta}(t,x)\|_{L^2_tL^2_x}^{\frac{6-p}{2p}} 
	\|u^{\varepsilon,\delta}(t,x+\xi) - u^{\varepsilon,\delta}(t,x)\|_{L^2_tL^6_x}^{\frac{3p-6}{2p}} \\
	&\leq C(p)|\xi|^{\frac{6-p}{2p}},
\end{align*}
where $p \in (2,6)$. As \(\xi \rightarrow 0\), it follows that
	\begin{align*}
		\norm{u^{\varepsilon,\delta}(t,x+\xi)-u^{\varepsilon,\delta}(t,x)}_{L^2_tL^p_x} \rightarrow 0.
	\end{align*}
Applying Proposition~\ref{pro:lions} with
\[
(p_1,q_1,p_2,q_2)
=
\left(2,2,\frac54,5\right)
\]
to \(j^{\varepsilon,\delta}\) and \(u^{\varepsilon,\delta}\),
and with
\[
(p_1,q_1,p_2,q_2)
=
\left(2,2,\frac53,\frac52\right)
\]
to \(n^{\varepsilon,\delta}\) and \(u^{\varepsilon,\delta}\),
we obtain the asserted convergence in
\(\mathcal D'((0,T)\times\mathbb R^3_x)\).
\end{proof}

	 \begin{proposition}\label{pro:uweak}
	 	According to the properties of convolution and weak convergence, the following weak convergence holds for all $p \in [2,6]$ and $\tilde{p}\in [2,\frac{10}{3}]$:
	 	\begin{align}
	 		\label{02}
	 		\theta_{\varepsilon} \ast u^{\varepsilon,\delta} \rightharpoonup u \quad &\text{weakly in } L^2(0,T;L^p(\mathbb{R}^3_x)),\\
        \label{upweakly}
	 		\theta_{\varepsilon} \ast u^{\varepsilon,\delta} \rightharpoonup u \quad &\text{weakly in } L^{\tilde{p}}((0,T)\times\mathbb{R}^3_x),
	 	\end{align}
        where \(u\) is the limit function of \eqref{01}. 
	 \end{proposition}
	\begin{proof}
		For any \(\psi \in L^{2}(0,T;L^{p'}(\mathbb{R}_x^3))\), which $\frac{1}{p}+\frac{1}{p'}=1$, we have
	\begin{align*}
		&\int_{(0,T)\times \mathbb{R}^3}(\theta_{\varepsilon}\ast u^{\varepsilon,\delta}-u)\psi dxds\\
		&=\int_{(0,T)\times \mathbb{R}^3}(\theta_{\varepsilon}\ast u^{\varepsilon,\delta}-\theta_{\varepsilon}\ast u)\psi dxds+\int_{(0,T)\times \mathbb{R}^3}(\theta_{\varepsilon}\ast u-u)\psi dxds\\
		&=\int_{(0,T)\times \mathbb{R}^3}(u^{\varepsilon,\delta}-u)\psi dxds+\int_{(0,T)\times \mathbb{R}^3}(u^{\varepsilon,\delta}-u)(\theta_{\varepsilon}\ast\psi-\psi) dxds\\
		&\hspace{8cm}+\int_{(0,T)\times \mathbb{R}^3}(\theta_{\varepsilon}\ast u-u)\psi dxds\\
		&:=m_{11}+m_{12}+m_{13}.
	\end{align*}
	For \(m_{11}\), due to \eqref{01}, we derive
	\begin{align*}
		\int_{(0,T)\times \mathbb{R}^3}(u^{\varepsilon,\delta}-u)\psi dxds \rightarrow 0, \quad\text{as}~ \varepsilon,\delta \rightarrow 0.
	\end{align*}
	For \(m_{12}\) and \(m_{13}\), due to the properties of mollifiers, \eqref{eq:esti} and \eqref{utp26}, we have
	\begin{align*}
		\bigg|\int_{(0,T)\times \mathbb{R}^3}&(u^{\varepsilon,\delta}-u)(\theta_{\varepsilon}\ast\psi-\psi) dxds\bigg| \\
		&\leq \left(\norm{u^{\varepsilon,\delta}}_{L^2(0,T;L^p(\mathbb{R}^3_x))}+\norm{u}_{L^2(0,T;L^p(\mathbb{R}^3_x))}\right)\norm{\theta_{\varepsilon}\ast\psi-\psi}_{L^2(0,T;L^{p'}(\mathbb{R}^3_x))}\rightarrow 0, \quad\text{as}~ \varepsilon \rightarrow 0.		
	\end{align*}
	and
	\begin{align*}
		\left|\int_{(0,T)\times \mathbb{R}^3}(\theta_{\varepsilon}\ast u-u)\psi dxds\right|&\leq\int_{(0,T)\times \mathbb{R}^3}\left|(\theta_{\varepsilon}\ast u-u)\psi\right| dxds\\
		&\leq \norm{\theta_{\varepsilon}\ast u-u}_{L^{2}(0,T;L^{p}(\mathbb{R}_x^3))}\norm{\psi}_{L^{2}(0,T;L^{p'}(\mathbb{R}_x^3))}\rightarrow 0, \quad \text{as}~\varepsilon \rightarrow 0.
	\end{align*}
	\par To sum up, the proof of  \eqref{02} is completed. Due to the embedding inequality \eqref{utp26},
	% \begin{align}\label{utp26}
	% 	\|u\|_{L_t^q L_x^p}^2 \leq C\left(\|u\|_{L_t^\infty L_x^2}^2 + \|\nabla u\|_{L_t^2 L_x^2}^2\right), \quad \text{for} \quad \frac{2}{q} + \frac{3}{p} = \frac{3}{2}, ~ 2 \leq p \leq 6,
	% \end{align}
%     we derive
%     \begin{align*}
%         \norms{u^{\varepsilon,\delta}}_{L^{\tilde{p}}_{tx}}&\leq t^{\frac{1}{\tilde{p}}-\frac{1}{p_1}}\norms{u^{\varepsilon,\delta}}_{L^{p_1}_tL^{\tilde{p}}_x}\\
%         &\leq t^{\frac{1}{\tilde{p}}-\frac{1}{p_1}}\left(\norms{u^{\varepsilon,\delta}}_{L^{\infty}_tL^2_x}+\norms{\nabla u^{\varepsilon,\delta}}_{L^{2}_{tx}}\right),
%     \end{align*}
%     where \(\frac{2}{p_1}+\frac{3}{\tilde{p}}=\frac32\) and \(\tilde{p}\in [2,\frac{10}{3}]\).
%     This yields
% \begin{align*} 
% u ^{\varepsilon,\delta} \rightharpoonup u \quad \text{weakly in } L^{\tilde{p}}((0,T)\times\mathbb{R}^3_x).
%     \end{align*}
%      Then by the same methods as in \eqref{02}, 
the following weak convergence holds
    \begin{align*}
	 		\theta_{\varepsilon} \ast u^{\varepsilon,\delta} \rightharpoonup u \quad \text{weakly in } L^{\tilde{p}}((0,T)\times\mathbb{R}^3_x).
    \end{align*}
  \end{proof}

\begin{corollary}\label{jthetau}
The same argument as in Proposition~\ref{pro:jnweak},
together with Proposition~\ref{pro:uweak} and Young's convolution
inequality, yields the following convergence results:
\begin{align}
    j^{\varepsilon,\delta}
    (\theta_\varepsilon*u^{\varepsilon,\delta})
    &\longrightarrow ju
    \quad\text{in }
    \mathcal D'\bigl((0,T)\times\mathbb R^3_x\bigr),
    \label{eq:04}
    \\
    n^{\varepsilon,\delta}
    (\theta_\varepsilon*u^{\varepsilon,\delta})
    &\longrightarrow nu
    \quad\text{in }
    \mathcal D'\bigl((0,T)\times\mathbb R^3_x\bigr).
    \label{eq:05}
\end{align}
\end{corollary}

	\subsection{\texorpdfstring{The strong convergence of \(f^{\varepsilon,\delta}\) in \(L^p((0,T)\times\mathbb{R}^3_x\times\mathbb{R}^3_v)\)}{The strong convergence of f varepsilon delta in L p((0,T)R 3 x times}}$ $
    In this subsection, we prove the strong convergence of \(f^{\varepsilon,\delta}\) in \(L^p((0,T)\times\mathbb{R}^3_x\times\mathbb{R}^3_v)\), which belongs to  the DiPerna-Lions type \ \cite{dipernaFokkerPlanckBoltzmannEquation1988} and plays an important role. Here we give an alternative proof for the Kolmogorov operator, which is of independent interest. 

\begin{proposition}[\rev{Kolmogorov compactness}]\label{pro:compact}
For this proposition, take $m=6$ and write $z=(x,v)\in\mathbb R^3\times\mathbb R^3$.
Let \(h^n\) be a bounded sequence in \(L^1((0,T)\times\mathbb{R}^m)\) satisfying
\begin{align*}
\sup_n\int_0^T\int_{|z|\ge R}|h^n|\,dz\,dt\rightarrow0,\quad\text{as }R\rightarrow\infty,
\end{align*}
and let \(g^n_{in}\) be a bounded sequence in \(L^1(\mathbb{R}^m)\) satisfying
\begin{align*}
\sup_n\int_{|z|\ge R}|g^n_{in}|\,dz\rightarrow0,\quad\text{as} ~R\rightarrow\infty.
\end{align*}
If \(g^n\) is the solution of
\begin{align*}
\partial_tg^n+v\cdot\nabla_xg^n-\Delta_v g^n=h^n\quad\text{in }(0,T)\times\mathbb R^m,\qquad g^n|_{t=0}=g^n_{in}\quad\text{in }\mathbb R^m,
\end{align*}
\rev{where $g^n$ is the mild solution given by the Gaussian $G$ in Lemma~\ref{lem:G},} then the sequence \(\{g^n\}\) is relatively compact in \(L^1((0,T)\times\mathbb R^m)\).
\end{proposition}

\begin{revision}
\begin{proof}
Set
\[
Q_T=(0,T)\times\mathbb R^6,\qquad z=(x,v),\qquad y=(\xi,\upsilon),
\]
write $dz=dxdv$ and $dy=d\xi d\upsilon$, and let
\[
M:=\sup_n\left(\|g_{\mathrm{in}}^n\|_{L^1(\mathbb R^6)}+\|h^n\|_{L^1(Q_T)}\right)<\infty.
\]
By assumption,
\[
\omega(R):=\sup_n\left(\int_{|y|>R}|g_{\mathrm{in}}^n(y)|\,dy+\int_0^T\int_{|y|>R}|h^n(s,y)|\,dyds\right)\longrightarrow0
\quad\text{as }R\longrightarrow\infty.
\]

Write
\[
\Gamma_\tau(z;y):=G(\tau,x,v;0,\xi,\upsilon),\qquad \tau>0,
\]
and define
\[
S(\tau)q(z):=\int_{\mathbb R^6}\Gamma_\tau(z;y)q(y)\,dy.
\]
By Lemma~\ref{lem:G},
\[
\Gamma_\tau(z;y)\ge0,\qquad
\int_{\mathbb R^6}\Gamma_\tau(z;y)\,dz=1,
\qquad
\|S(\tau)q\|_{L^1(\mathbb R^6)}
\le\|q\|_{L^1(\mathbb R^6)}.
\]
The mild solution is
\[
g^n(t)=S(t)g_{\mathrm{in}}^n+\int_0^tS(t-s)h^n(s)\,ds.
\]

\textbf{Step 1. Truncation.}
Let
\[
B_R=\{y\in\mathbb R^6:|y|\le R\},\qquad
g_{\mathrm{in},R}^n=\mathbf1_{B_R}g_{\mathrm{in}}^n,\qquad
h_R^n=\mathbf1_{B_R}h^n,
\]
and define
\[
g_R^n(t)=S(t)g_{\mathrm{in},R}^n+\int_0^tS(t-s)h_R^n(s)\,ds.
\]
By the contraction estimate and Fubini's theorem,
\[
\|g^n-g_R^n\|_{L^1(Q_T)}
\le T\omega(R).
\]

Fix $0<\eta<T$. For $(s,y)\in[0,T]\times B_R$, define
\[
K_\eta(s,y)(t,z)=
\begin{cases}
\Gamma_{t-s}(z;y),&s+\eta\le t<T,\\
0,&\text{otherwise},
\end{cases}
\]
and set
\[
g_{R,\eta}^n=\int_{B_R}K_\eta(0,y)g_{\mathrm{in}}^n(y)\,dy
+\int_0^T\int_{B_R}K_\eta(s,y)h^n(s,y)\,dyds.
\]
For $t\ge\eta$, we have
\begin{align}\label{gin}
\int_{B_R}\Gamma_t(z;y)g_{\mathrm{in}}^n(y)\,dy
=[S(t)g_{\mathrm{in},R}^n](z),
\end{align}
and
\begin{align}\label{h}
\int_0^T\int_{B_R}K_\eta(s,y)(t,z)h^n(s,y)\,dyds
=\int_0^{t-\eta}[S(t-s)h_R^n(s)](z)\,ds.
\end{align}
Consequently,
\[
g_R^n(t)-g_{R,\eta}^n(t)
=\mathbf1_{\{t<\eta\}}S(t)g_{\mathrm{in},R}^n
+\int_{(t-\eta)_+}^tS(t-s)h_R^n(s)\,ds.
\]
Using the contraction estimate and Fubini's theorem again, we obtain
\begin{align*}
\|g_R^n-g_{R,\eta}^n\|_{L^1(Q_T)}
\le\eta\|g_{\mathrm{in}}^n\|_{L^1(\mathbb R^6)}
+\int_0^T\int_{(t-\eta)_+}^t
\|h^n(s)\|_{L^1(\mathbb R^6)}\,dsdt\le\eta M.
\end{align*}
Hence
\begin{equation}\label{eq:kolmogorov-truncation}
\sup_n\|g^n-g_{R,\eta}^n\|_{L^1(Q_T)}
\le T\omega(R)+\eta M.
\end{equation}

\textbf{Step 2. Compactness of the truncated kernels.}
The Gaussian formula for $G$ yields
\[
0\le\Gamma_\tau(z;y)
\le C_{R,\eta,T}\exp(-c_{R,\eta,T}|z|^2),
\qquad \eta\le\tau\le T,\quad y\in B_R,
\]
where $C_{R,\eta,T},c_{R,\eta,T}>0$.
Therefore, by dominated convergence,
\begin{align}\label{L1-continue}
(\tau_k,y_k)\longrightarrow(\tau,y)
\quad\Longrightarrow\quad
\|\Gamma_{\tau_k}(\cdot;y_k)-\Gamma_\tau(\cdot;y)\|_{L^1(\mathbb R^6)}
\longrightarrow0
\end{align}
on $[\eta,T]\times B_R$.
For $s\in[0,T]$, let
\[
I_s=\{t\in(0,T):t\ge s+\eta\}.
\]
If $(s_n,y_n)\to(s,y)$ in $[0,T]\times B_R$, then
\begin{align*}
&\|K_\eta(s_n,y_n)-K_\eta(s,y)\|_{L^1(Q_T)}\\
&\le |I_{s_n}\triangle I_s|
+\int_{I_{s_n}\cap I_s}
\|\Gamma_{t-s_n}(\cdot;y_n)-\Gamma_{t-s}(\cdot;y)\|_{L^1(\mathbb R^6)}dt,
\end{align*}
where $\triangle$ denotes the symmetric difference and
\[
|I_s\triangle I_{s_n}|=|(T-s-\eta)_+-(T-s_n-\eta)_+|\le|s-s_n|\to 0.
\]
Moreover, the integrand is bounded and converges to zero by
\eqref{L1-continue}. Thus, the dominated convergence gives
\[
\|K_\eta(s_n,y_n)-K_\eta(s,y)\|_{L^1(Q_T)}
\longrightarrow0.
\]
Since $[0,T]\times B_R$ is compact, the set
\[
\mathcal K_{R,\eta}:=\{K_\eta(s,y):(s,y)\in[0,T]\times B_R\}
\]
is compact in $L^1(Q_T)$.

\textbf{Step 3. Finite-dimensional approximation.}
Fix $a>0$. By the compactness of $\mathcal K_{R,\eta}$, there exist
points $(s_j,y_j)$ and a Borel partition $\{A_j\}_{j=1}^N$ of
$[0,T]\times B_R$ such that
\[
\sup_{(s,y)\in A_j}
\|K_\eta(s,y)-K_\eta(s_j,y_j)\|_{L^1(Q_T)}<a,
\qquad j=1,\ldots,N.
\]
For brevity, write $K_j=K_\eta(s_j,y_j)$ and define
\begin{align*}
c_j^n=\int_{B_R}\mathbf1_{A_j}(0,y)g_{\mathrm{in}}^n(y)\,dy
+\int_0^T\int_{B_R}\mathbf1_{A_j}(s,y)h^n(s,y)\,dyds,\quad
P^n=\sum_{j=1}^Nc_j^nK_j.
\end{align*}
Since $\{A_j\}_{j=1}^N$ is a partition,
\[
\sum_{j=1}^N|c_j^n|
\le\|g_{\mathrm{in}}^n\|_{L^1(\mathbb R^6)}
+\|h^n\|_{L^1(Q_T)}
\le M.
\]
Moreover,
\begin{align*}
\|g_{R,\eta}^n-P^n\|_{L^1(Q_T)}
\le a\int_{B_R}|g_{\mathrm{in}}^n(y)|\,dy
+a\int_0^T\int_{B_R}|h^n(s,y)|\,dyds\le aM.
\end{align*}
The set
\[
\mathcal C_{R,\eta,a}:=
\left\{\sum_{j=1}^Nc_jK_j:
c_j\in\mathbb R,\ \sum_{j=1}^N|c_j|\le M\right\}
\]
is compact in $L^1(Q_T)$, since it is the continuous image of a compact
subset of $\mathbb R^N$. In particular, $P^n\in\mathcal C_{R,\eta,a}$
for every $n$.

Given $\rho>0$, choose $R$ sufficiently large and then $\eta,a>0$
sufficiently small so that
\[
T\omega(R)+\eta M+aM<\frac{\rho}{2}.
\]
By the compactness of $\mathcal C_{R,\eta,a}$, there exist
$F_1,\ldots,F_L\in L^1(Q_T)$ such that
\[
\mathcal C_{R,\eta,a}
\subset\bigcup_{\ell=1}^LB_{L^1(Q_T)}(F_\ell,\rho/2).
\]
Thus, for each $n$, there is an index $\ell(n)\in\{1,\ldots,L\}$ such that
\[
\|P^n-F_{\ell(n)}\|_{L^1(Q_T)}<\frac{\rho}{2}.
\]
Together with \eqref{eq:kolmogorov-truncation}, this gives
\begin{align*}
\|g^n-F_{\ell(n)}\|_{L^1(Q_T)}
&\le\|g^n-g_{R,\eta}^n\|_{L^1(Q_T)}
+\|g_{R,\eta}^n-P^n\|_{L^1(Q_T)}
+\|P^n-F_{\ell(n)}\|_{L^1(Q_T)}\\
&\le T\omega(R)+\eta M+aM+\frac{\rho}{2}<\rho.
\end{align*}
Hence $\{g^n\}$ is totally bounded, and therefore relatively compact, in
$L^1(Q_T)$.
\end{proof}
\end{revision}

 It follows from Proposition \ref{pro:compact} and the uniform estimates established below that the following lemma holds:
	\begin{lemma}\label{lem:fcompact}
		Let \(\{f^{\varepsilon,\delta}\}\) be the solution family of system \eqref{eq:NSVFK0} under the initial conditions in Theorem \ref{thm:solution}. Then for any \(1 \leq p < \infty\), 
		\begin{align*}
		\{f^{\varepsilon,\delta}\} \text{ is relatively compact in } L^p((0,T) \times \mathbb{R}^3_x \times \mathbb{R}^3_v).
		\end{align*}
	\end{lemma}
	\begin{proof}
		Firstly, due to the assumption of $f_{in}^{\delta}$, we have
	\begin{equation*}
    \left\{
	    \begin{aligned}
		\int_{ \mathbb{R}_v^3\times \mathbb{R}_x^3}|v|^3f_{in}^{\delta}dxdv &\leq \int_{ \mathbb{R}_v^3\times \mathbb{R}_x^3}|v|^3f_{in}dxdv,\\
		\int_{ \mathbb{R}_v^3\times \mathbb{R}_x^3}|x|^2f_{in}^{\delta}dxdv &\leq \int_{ \mathbb{R}_v^3\times \mathbb{R}_x^3}|x|^2f_{in}dxdv,
	\end{aligned}
    \right.
	\end{equation*}
	and
    \begin{equation*}\label{3x}
    \left\{
        \begin{aligned}
		\int_{\mathbb{R}_x^3}\int_{|v|\geq R}f_{in}^{\delta}dxdv &\leq R^{-3}\int_{\mathbb{R}_x^3}\int_{|v|\geq R}|v|^3f_{in}dvdx \rightarrow 0,\\
        \int_{|x|\geq R}\int_{\mathbb{R}_v^3}f_{in}^{\delta}dxdv &\leq R^{-2}\int_{|x|\geq R}\int_{\mathbb{R}_v^3}|x|^2f_{in}dvdx \rightarrow 0,\
	\end{aligned}
    \right.
    \end{equation*}
   as~$R\rightarrow +\infty$, uniformly for $\delta$.
	By the same arguments, combining \eqref{10} and \eqref{17}, it follows that
	\begin{align}
  %       \label{3x}
		% \sup_{\varepsilon,\delta}\sup_{t\in [0,T]}\int_{|x|\geq R}\int_{\mathbb{R}^3_v}f_{in}^{\delta}dvdx\rightarrow 0,~\text{as}~R\rightarrow \infty;\\
		\label{32}
		\sup_{\varepsilon,\delta}\sup_{t\in [0,T]}\int_{\mathbb{R}^3_x}\int_{|v|\geq R}f^{\varepsilon,\delta}dvdx\rightarrow 0,~\text{as}~R\rightarrow \infty;\\
		% \label{33}
		\sup_{\varepsilon,\delta}\sup_{t\in [0,T]}\int_{\mathbb{R}^3_x}\int_{|v|\geq R}|v|f^{\varepsilon,\delta}dvdx\rightarrow 0,~\text{as}~R\rightarrow \infty;\nonumber\\
		\label{34}
		\sup_{\varepsilon,\delta}\sup_{t\in [0,T]}\int_{\mathbb{R}^3_x}\int_{|v|\geq R}|v|^2f^{\varepsilon,\delta}dvdx\rightarrow 0,~\text{as}~R\rightarrow \infty.
	\end{align}
    Claim that  
\begin{align}\label{nablaf1}\sup_{\varepsilon,\delta}\norms{\nabla_{v}\cdot\left[\left(\theta_{\varepsilon} \ast u^{\varepsilon,\delta}-v\gamma_
		\delta(v)\right)f^{\varepsilon,\delta}\right]}_{L^1((0,T)\times\mathbb{R}^3_x\times\mathbb{R}^3_v)} \leq C,
	\end{align}
    \begin{align}\label{nablaf2}
		\sup_{\varepsilon,\delta}\int_{0}^{T}\int_{ \mathbb{R}_x^3}\int_{|v|\geq R}\left|\nabla_{v}\cdot\left[\left(\theta_{\varepsilon} \ast u^{\varepsilon,\delta}-v\gamma_
		\delta(v)\right)f^{\varepsilon,\delta}\right]\right|dvdxds \rightarrow 0,~\text{as}~R~\rightarrow \infty,
	\end{align}
    and
    \begin{align}\label{nablaf3}
        \sup_{\varepsilon,\delta}\int_{0}^{T}\int_{ |x|\geq R}\int_{\mathbb{R}^3_v}\left|\nabla_{v}\cdot\left[\left(\theta_{\varepsilon} \ast u^{\varepsilon,\delta}-v\gamma_
		\delta(v)\right)f^{\varepsilon,\delta}\right]\right|dvdxds \rightarrow 0,~\text{as}~R~\rightarrow \infty.
    \end{align}
    % Then the compactness of \(f^{\varepsilon,\delta}\) in \(L^1((0,T)\times\mathbb{R}^3_x\times\mathbb{R}^3_v)\) can be obtained. 
\textcolor{black}{Since $\{(x,v):|(x,v)|\ge R\}\subset\{|x|\ge R/\sqrt2\}\cup\{|v|\ge R/\sqrt2\}$, estimates \eqref{nablaf2} and \eqref{nablaf3} imply}
\[
\textcolor{black}{\sup_{\varepsilon,\delta}\int_0^T\int_{|(x,v)|\ge R}|\nabla_v\cdot[(\theta_\varepsilon*u^{\varepsilon,\delta}-v\gamma_\delta(v))f^{\varepsilon,\delta}]|dxdvdt\longrightarrow0\quad\text{as }R\to\infty.}
\]
\textcolor{black}{Together with \eqref{nablaf1} and the initial-data tail estimates, this verifies the assumptions of Proposition~\ref{pro:compact}.}
    Then there exists a subsequence such that
\begin{align*}
f^{\varepsilon,\delta} \rightarrow f \quad \text{strongly in } L^1((0,T)\times\mathbb{R}^3_x\times\mathbb{R}^3_v).
\end{align*}
	By the interpolation inequality for $L^p$-norms combined with the uniform estimate \eqref{5}, we obtain the strong convergence
	\begin{equation}\label{eq:strong_conv}
		f^{\varepsilon,\delta} \to f \quad \text{strongly in } L^p((0,T) \times \mathbb{R}^3_x \times \mathbb{R}^3_v)
	\end{equation}
	for all $p \in [1,+\infty)$.
% \textcolor{red}{Indeed, the uniform $L^\infty_{txv}$ bound and the $L^1$ convergence yield $\|f\|_{L^\infty(Q_T)}\le C$ (after extraction of an almost everywhere convergent subsequence), and hence}
% \[
% \textcolor{red}{\|f^{\varepsilon,\delta}-f\|_{L^p(Q_T)}^p\le(2C)^{p-1}\|f^{\varepsilon,\delta}-f\|_{L^1(Q_T)}\longrightarrow0,\qquad 1\le p<\infty.}
% \]

    {\bf Proof of \eqref{nablaf1} and \eqref{nablaf2}.}
	Considering the velocity divergence term
	\begin{equation*}\label{eq:div_term}
		\nabla_{v} \cdot \left[ \left( \theta_{\varepsilon} \ast u^{\varepsilon,\delta} - v\gamma_{\delta}(v) \right) f^{\varepsilon,\delta} \right],
	\end{equation*}
	we obtain the following estimate for $|v| \geq R$:
	\begin{align*}
		&\int_{0}^{T} \int_{\mathbb{R}_x^3} \int_{|v|\geq R} \left| \nabla_{v} \cdot \left[ \left( \theta_{\varepsilon} \ast u^{\varepsilon,\delta} - v\gamma_{\delta}(v) \right) f^{\varepsilon,\delta} \right] \right| dvdxds \nonumber \\
		&\leq \int_{0}^{T} \int_{\mathbb{R}_x^3} \int_{|v|\geq R} \left| (\theta_{\varepsilon} \ast u^{\varepsilon,\delta}) \nabla_{v} f^{\varepsilon,\delta} \right| dvdxds \nonumber \\
		&\quad + 3\int_{0}^{T} \int_{\mathbb{R}_x^3} \int_{|v|\geq R} \gamma_{\delta}(v) f^{\varepsilon,\delta} dvdxds \nonumber \\
		&\quad + \int_{0}^{T} \int_{\mathbb{R}_x^3} \int_{|v|\geq R} \left| \gamma_{\delta}(v) v \cdot \nabla_{v} f^{\varepsilon,\delta} \right| dvdxds \nonumber \\
		&\quad + \int_{0}^{T} \int_{\mathbb{R}_x^3} \int_{|v|\geq R} \left| f^{\varepsilon,\delta}v \cdot \nabla_{v} \gamma_{\delta}(v)  \right| dvdxds \nonumber \\
		&:= e_1 + e_2 + e_3 + e_4.
	\end{align*}
To apply the moment interpolation on a velocity tail, set $f_R^{\varepsilon,\delta}=f^{\varepsilon,\delta}\mathbf1_{\{|v|\ge R\}}$. Applying \eqref{8} with $\alpha=0$ and $\beta=2$ to $f_R^{\varepsilon,\delta}$ gives
\[
\int_{|v|\ge R}f^{\varepsilon,\delta}(t,x,v)\,dv\le\left(1+\frac{4\pi}{3}\|f^{\varepsilon,\delta}\|_{L^\infty_{txv}}\right)\left(\int_{|v|\ge R}|v|^2f^{\varepsilon,\delta}(t,x,v)\,dv\right)^{3/5}.
\]
Here and below, the constants are independent of $\varepsilon,\delta$. In particular, Young's inequality and the uniform fluid energy bounds yield $\|\theta_\varepsilon*u^{\varepsilon,\delta}\|_{L^2_tL^5_x}\le\|u^{\varepsilon,\delta}\|_{L^2_tL^5_x}\le C_T$.
	For \(e_1\), by H\"{o}lder's inequality, \eqref{8}, \eqref{eq:esti}, \eqref{21} and \eqref{utp26}, this becomes
	\begin{align*}
		e_1 &= 2\norms{\sqrt{f^{\varepsilon,\delta}}(\theta_{\varepsilon}\ast u^{\varepsilon,\delta})\cdot\nabla_{v}\sqrt{f^{\varepsilon,\delta}}}_{L^1((0,T)\times\mathbb{R}^3_x\times\{|v|\geq R\})}\\
		&\leq C\norms{(\theta_{\varepsilon}\ast u^{\varepsilon,\delta})\norms{f^{\varepsilon,\delta}}_{L^1(|v|\geq R)}^\frac12\norms{\nabla_{v}\sqrt{f^{\varepsilon,\delta}}}_{L^2(|v|\geq R)}}_{L^1((0,T)\times\mathbb{R}^3_x)}\\
		&\leq C\left(1+\frac{4}{3}\pi\norms{f^{\varepsilon,\delta}}_{L^{\infty}_{txv}}\right)^{\frac{1}{2}}\norms{(\theta_{\varepsilon}\ast u^{\varepsilon,\delta})\norms{|v|^2f^{\varepsilon,\delta}}_{L^1(|v|\geq R)}^\frac{3}{10}\norms{\nabla_{v}\sqrt{f^{\varepsilon,\delta}}}_{L^2(|v|\geq R)}}_{L^1_{tx}}\\
		&\leq C\left(1+\frac{4}{3}\pi\norms{f^{\varepsilon,\delta}}_{L^{\infty}_{txv}}\right)^\frac12\norm{u^{\varepsilon,\delta}}_{L^2_tL^5_x}\norms{\nabla_{v}\sqrt{f^{\varepsilon,\delta}}}_{L^2_{txv}}\norms{|v|^2f^{\varepsilon,\delta}}_{L^{\infty}_tL^1(\mathbb{R}^3_x\times\{|v|\geq R\})}^\frac{3}{10}\\
		&\leq C\norms{|v|^2f^{\varepsilon,\delta}}_{L^{\infty}_tL^1(\mathbb{R}^3_x\times\{|v|\geq R\})}^\frac{3}{10},
	\end{align*}
    where
    \begin{align*}
        C=C\left(T,\norm{f_{in}}_{L^{\infty}_{xv}}, \norm{f_{in}}_{L^{1}_{xv}},\norms{|v|^2f_{in}}_{L^{1}_{xv}},\norm{u_{in}}^{2}_{L^{2}_{x}},\norm{f_{in}\log f_{in}}_{L^{1}_{xv}},\norms{|x|^2f_{in}}_{L^{1}_{xv}} \right).
    \end{align*}
    For \(e_2\), it is similary with \eqref{32}. We omit it.\\
	For \(e_3\), by H\"{o}lder's inequality and \eqref{21}, we have
	\begin{align*}
		e_3 &= 2\int_{0}^{T}\int_{ \mathbb{R}_x^3}\int_{|v|\geq R}\left|\gamma_{\delta}(v)\sqrt{f^{\varepsilon,\delta}}v\cdot\nabla_{v}\sqrt{f^{\varepsilon,\delta}}\right|dvdxds\\
		&\leq CT^\frac12\norms{\nabla_{v}\sqrt{f^{\varepsilon,\delta}}}_{L^2((0,T)\times\mathbb{R}^3_x\times\mathbb{R}^3_v)}\norms{|v|^2f^{\varepsilon,\delta}}_{L^{\infty}_tL^1(\mathbb{R}^3_x\times\{|v|\geq R\})}^\frac12\\
		&\leq C\norms{|v|^2f^{\varepsilon,\delta}}_{L^{\infty}_tL^1(\mathbb{R}^3_x\times\{|v|\geq R\})}^\frac{1}{2},
	\end{align*}
    where
\begin{align*}
    C=C\left(T,\norm{f_{in}}_{L^{\infty}_{xv}}, \norm{f_{in}}_{L^{1}_{xv}},\norms{|v|^2f_{in}}_{L^{1}_{xv}},\norm{u_{in}}^{2}_{L^{2}_{x}},\norm{f_{in}\log f_{in}}_{L^{1}_{xv}},\norms{|x|^2f_{in}}_{L^{1}_{xv}} \right).
\end{align*}
	For \(e_4\), due to \(|v\cdot \nabla_{v}\gamma|\leq C\) and the positive of \(f^{\varepsilon,\delta}\), we obtain
	\begin{align*}
		e_4 \leq C\int_{0}^{T}\int_{ \mathbb{R}_x^3}\int_{|v|\geq R}f^{\varepsilon,\delta}dvdxds,
	\end{align*}
	%and by integration by parts, it becomes
% 	\ben\label{eq:e4}
% 		e_4&\leq&-\int_{0}^{t}\int_{ \mathbb{R}_x^3}\int_{|v|= R}v\cdot n\gamma_{\delta}(v)f^{\varepsilon,\delta}dvdxds\nonumber\\
% &&+ 3\int_{0}^{t}\int_{ \mathbb{R}_x^3}\int_{|v|\geq R}\gamma_{\delta}(v)f^{\varepsilon,\delta}dvdxds+\int_{0}^{t}\int_{ \mathbb{R}_x^3}\int_{|v|\geq R}\left|\gamma_{\delta}(v)v\cdot\nabla_{v}f^{\varepsilon,\delta}\right|dvdxds.\nonumber\\
% 	\een
% 	Similar to the proof  in \(e_2\) and \(e_3\), we derive
% 	\beno
% 		e_4&\leq& \int_{0}^{t}\int_{ \mathbb{R}_x^3}\int_{|v|= R}v\cdot n\gamma_{\delta}(v)f^{\varepsilon,\delta}dvdxds\\
%  &&+C\left(\norms{|v|^2f^{\varepsilon,\delta}}_{L^{\infty}_tL^1(\mathbb{R}^3_x\times\{|v|\geq R\})}^\frac12+\norms{f^{\varepsilon,\delta}}_{L^{\infty}_tL^1(\mathbb{R}^3_x\times\{|v|\geq R\})}\right).
% 	\eeno
	To sum up, \eqref{nablaf2} can be obtained by \eqref{32} and \eqref{34}.
	Moreover, by replacing \(|v|\geq R\) with \(\mathbb{R}^3_v\), \eqref{nablaf1} also can be obtained.
	% \begin{align*}
	% 	\norms{\nabla_{v}\cdot\left[\left(\theta_{\varepsilon} \ast u^{\varepsilon,\delta}-v\gamma_
	% 	\delta(v)\right)f^{\varepsilon,\delta}\right]}_{L^1((0,T)\times\mathbb{R}^3_x\times\mathbb{R}^3_v)} \leq C.
	% \end{align*}

  %   Next we claim that
  %   \begin{align}\label{x320}
  %       \norms{|x|^{\frac{3}{20}}\nabla_{v}\cdot\left[\left(\theta_{\varepsilon} \ast u^{\varepsilon,\delta}-v\gamma_
		% \delta(v)\right)f^{\varepsilon,\delta}\right]}_{L^1((0,T)\times\mathbb{R}^3_x\times\mathbb{R}^3_v)} \leq C.
  %   \end{align}
    {\bf Proof of \eqref{nablaf3}.}
   \begin{align*}
		&\int_{0}^{T} \int_{\mathbb{R}_x^3} \int_{\mathbb{R}_v^3} |x|^{\frac{3}{20}}\left| \nabla_{v} \cdot \left[ \left( \theta_{\varepsilon} \ast u^{\varepsilon,\delta} - v\gamma_{\delta}(v) \right) f^{\varepsilon,\delta} \right] \right| dvdxds \nonumber \\
		&\leq \int_{0}^{T} \int_{\mathbb{R}_x^3} \int_{\mathbb{R}_v^3} |x|^{\frac{3}{20}}\left| (\theta_{\varepsilon} \ast u^{\varepsilon,\delta}) \nabla_{v} f^{\varepsilon,\delta} \right| dvdxds \nonumber \\
		&\quad + 3\int_{0}^{T} \int_{\mathbb{R}_x^3} \int_{\mathbb{R}_v^3} |x|^{\frac{3}{20}}\gamma_{\delta}(v) f^{\varepsilon,\delta} dvdxds \nonumber \\
		&\quad + \int_{0}^{T} \int_{\mathbb{R}_x^3} \int_{\mathbb{R}_v^3} |x|^{\frac{3}{20}}\left| \gamma_{\delta}(v) v \cdot \nabla_{v} f^{\varepsilon,\delta} \right| dvdxds \nonumber \\
		&\quad + \int_{0}^{T} \int_{\mathbb{R}_x^3} \int_{\mathbb{R}_v^3} |x|^{\frac{3}{20}}\left| v \cdot \nabla_{v} \gamma_{\delta}(v) f^{\varepsilon,\delta} \right| dvdxds \nonumber \\
		&:= O_1 + O_2 + O_3 + O_4.
	\end{align*}
	For \(O_1\), by \eqref{8}, \eqref{eq:esti} and \eqref{21}, this becomes
	\begin{align*}
		O_1 &= 2\norms{|x|^{\frac{3}{20}}\sqrt{f^{\varepsilon,\delta}}(\theta_{\varepsilon}\ast u^{\varepsilon,\delta})\cdot\nabla_{v}\sqrt{f^{\varepsilon,\delta}}}_{L^1((0,T)\times\mathbb{R}^3_x\times\mathbb{R}_v^3)}\\
		&\leq C\norms{|x|^{\frac{3}{20}}(\theta_{\varepsilon}\ast u^{\varepsilon,\delta})\norms{f^{\varepsilon,\delta}}_{L^1(\mathbb{R}_v^3)}^\frac12\norms{\nabla_{v}\sqrt{f^{\varepsilon,\delta}}}_{L^2(\mathbb{R}_v^3)}}_{L^1((0,T)\times\mathbb{R}^3_x)}\\
		&\leq C\left(1+\frac{4}{3}\pi\norms{f^{\varepsilon,\delta}}_{L^{\infty}_{txv}}\right)^{\frac{1}{2}}\norms{|x|^{\frac{3}{20}}(\theta_{\varepsilon}\ast u^{\varepsilon,\delta})\norms{|v|^2f^{\varepsilon,\delta}}_{L^1(\mathbb{R}_v^3)}^\frac{3}{10}\norms{\nabla_{v}\sqrt{f^{\varepsilon,\delta}}}_{L^2(\mathbb{R}_v^3)}}_{L^1_{tx}}\\
        &\leq C\left(1+\frac{4}{3}\pi\norms{f^{\varepsilon,\delta}}_{L^{\infty}_{txv}}\right)^{\frac{1}{2}}\norms{(\theta_{\varepsilon}\ast u^{\varepsilon,\delta})\norms{|x|^{\frac{1}{2}}|v|^2f^{\varepsilon,\delta}}_{L^1(\mathbb{R}_v^3)}^\frac{3}{10}\norms{\nabla_{v}\sqrt{f^{\varepsilon,\delta}}}_{L^2(\mathbb{R}_v^3)}}_{L^1_{tx}}\\
		&\leq C\left(1+\frac{4}{3}\pi\norms{f^{\varepsilon,\delta}}_{L^{\infty}_{txv}}\right)^\frac12\norm{u^{\varepsilon,\delta}}_{L^2_tL^5_x}\norms{\nabla_{v}\sqrt{f^{\varepsilon,\delta}}}_{L^2_{txv}}\norms{|x|^{\frac{1}{2}}|v|^2f^{\varepsilon,\delta}}_{L^{\infty}_tL^1(\mathbb{R}^3_x\times\mathbb{R}_v^3)}^\frac{3}{10}\\
		&\leq C\norms{|x|^{\frac{3}{2}}f^{\varepsilon,\delta}}_{L^{\infty}_tL^1(\mathbb{R}^3_x\times\mathbb{R}_v^3)}^\frac{1}{10}\norms{|v|^3f^{\varepsilon,\delta}}_{L^{\infty}_tL^1(\mathbb{R}^3_x\times\mathbb{R}_v^3)}^\frac{1}{5},
	\end{align*}
    where
    \begin{align*}
        C=C\left(T,\norm{f_{in}}_{L^{\infty}_{xv}}, \norm{f_{in}}_{L^{1}_{xv}},\norms{|v|^3f_{in}}_{L^{1}_{xv}},\norm{u_{in}}^{2}_{L^{2}_{x}},\norm{f_{in}\log f_{in}}_{L^{1}_{xv}},\norms{|x|^2f_{in}}_{L^{1}_{xv}} \right).
    \end{align*}
In the last step we used the following mixed-moment interpolation at each fixed time:
\[
\int_{\mathbb R^6}|x|^{1/2}|v|^2f^{\varepsilon,\delta}\,dx\,dv\le\left(\int_{\mathbb R^6}|x|^{3/2}f^{\varepsilon,\delta}\right)^{1/3}\left(\int_{\mathbb R^6}|v|^3f^{\varepsilon,\delta}\right)^{2/3}.
\]
    Due to \eqref{eq:esti}, \eqref{10} and \eqref{17}, it follows that
    \begin{align*}
        O_1\leq C,
    \end{align*}
    where
    \begin{align*}
        C=C\left(T,\norm{f_{in}}_{L^{\infty}_{xv}}, \norm{f_{in}}_{L^{1}_{xv}},\norms{|v|^3f_{in}}_{L^{1}_{xv}},\norm{u_{in}}^{2}_{L^{2}_{x}},\norm{f_{in}\log f_{in}}_{L^{1}_{xv}},\norms{|x|^2f_{in}}_{L^{1}_{xv}} \right).
    \end{align*}
	For \(O_3\), by H\"{o}lder's inequality and \eqref{21}, we have
	\begin{align*}
		O_3 &= 2\int_{0}^{T}\int_{ \mathbb{R}_x^3}\int_{\mathbb{R}_v^3}|x|^{\frac{3}{20}}\left|\gamma_{\delta}(v)\sqrt{f^{\varepsilon,\delta}}v\cdot\nabla_{v}\sqrt{f^{\varepsilon,\delta}}\right|dvdxds\\
		&\leq C\norms{\nabla_{v}\sqrt{f^{\varepsilon,\delta}}}_{L^2((0,T)\times\mathbb{R}^3_x\times\mathbb{R}^3_v)}\norms{|x|^{\frac{3}{10}}|v|^2f^{\varepsilon,\delta}}_{L^1((0,T)\times\mathbb{R}^3_x\times\mathbb{R}_v^3)}^\frac12\\
		&\leq CT^\frac12\norms{|x|^{\frac{9}{10}}f^{\varepsilon,\delta}}_{L^{\infty}_tL^1(\mathbb{R}^3_x\times\mathbb{R}_v^3)}^\frac{1}{6}\norms{|v|^3f^{\varepsilon,\delta}}_{L^{\infty}_tL^1(\mathbb{R}^3_x\times\mathbb{R}_v^3)}^\frac13,
	\end{align*}
    where
    \begin{align*}
        C=C\left(T,\norm{f_{in}}_{L^{\infty}_{xv}}, \norm{f_{in}}_{L^{1}_{xv}},\norms{|v|^{3}f_{in}}_{L^{1}_{xv}},\norm{u_{in}}^{2}_{L^{2}_{x}},\norm{f_{in}\log f_{in}}_{L^{1}_{xv}},\norms{|x|^2f_{in}}_{L^{1}_{xv}} \right).
    \end{align*}
Here the mixed moment satisfies, at each fixed time,
\[
\int_{\mathbb R^6}|x|^{3/10}|v|^2f^{\varepsilon,\delta}\,dx\,dv\le\left(\int_{\mathbb R^6}|x|^{9/10}f^{\varepsilon,\delta}\right)^{1/3}\left(\int_{\mathbb R^6}|v|^3f^{\varepsilon,\delta}\right)^{2/3}.
\]
    Due to \eqref{eq:esti}, \eqref{10} and \eqref{17}, it follows that
    \begin{align*}
        O_3\leq C,
    \end{align*}
    where
    \begin{align*}
        C=C\left(T,\norm{f_{in}}_{L^{\infty}_{xv}}, \norm{f_{in}}_{L^{1}_{xv}},\norms{|v|^3f_{in}}_{L^{1}_{xv}},\norm{u_{in}}^{2}_{L^{2}_{x}},\norm{f_{in}\log f_{in}}_{L^{1}_{xv}},\norms{|x|^2f_{in}}_{L^{1}_{xv}} \right).
    \end{align*}
	For \(O_2\) and \(O_4\), due to \(|v\cdot \nabla_{v}\gamma|\leq C\) and the positive of \(f^{\varepsilon,\delta}\), we obtain
	\begin{align*}
		O_2+O_4 \leq C\int_{0}^{T}\int_{ \mathbb{R}_x^3}\int_{\mathbb{R}_v^3}|x|^{\frac{3}{20}}f^{\varepsilon,\delta}dvdxds.
	\end{align*}
	Due to \eqref{eq:esti} and \eqref{17}, we derive
	\begin{align*}
	    O_2+O_4\leq C,
	\end{align*}
    where 
    \begin{align*}
        C=C\left(T,\norm{f_{in}}_{L^{1}_{xv}},\norms{|x|^2f_{in}}_{L^{1}_{xv}} \right).
    \end{align*}
    The following estimate holds for \(|x|\geq R\):
    \begin{align*}
        &\int_{0}^{T} \int_{|x|\geq R} \int_{\mathbb{R}^3_v} \left| \nabla_{v} \cdot \left[ \left( \theta_{\varepsilon} \ast u^{\varepsilon,\delta} - v\gamma_{\delta}(v) \right) f^{\varepsilon,\delta} \right] \right| dvdxds\\
        =&\int_{0}^{T} \int_{|x|\geq R} \int_{\mathbb{R}^3_v} |x|^{\frac{3}{20}}|x|^{-\frac{3}{20}}\left| \nabla_{v} \cdot \left[ \left( \theta_{\varepsilon} \ast u^{\varepsilon,\delta} - v\gamma_{\delta}(v) \right) f^{\varepsilon,\delta} \right] \right| dvdxds\\
        \leq&R^{-\frac{3}{20}}\int_{0}^{T} \int_{|x|\geq R} \int_{\mathbb{R}^3_v} |x|^{\frac{3}{20}}\left| \nabla_{v} \cdot \left[ \left( \theta_{\varepsilon} \ast u^{\varepsilon,\delta} - v\gamma_{\delta}(v) \right) f^{\varepsilon,\delta} \right] \right| dvdxds\\
        \leq&CR^{-\frac{3}{20}}
    \end{align*}
    which implies that
    \begin{align*}
        \sup_{\varepsilon,\delta}\int_{0}^{T}\int_{ |x|\geq R}\int_{\mathbb{R}^3_v}\left|\nabla_{v}\cdot\left[\left(\theta_{\varepsilon} \ast u^{\varepsilon,\delta}-v\gamma_
		\delta(v)\right)f^{\varepsilon,\delta}\right]\right|dvdxds \rightarrow 0,~\text{as}~R~\rightarrow \infty.
    \end{align*}
	% Finally, due to the partial diffusive transport operator
	% \begin{align*}
	% 	L_v\equiv \partial_t +v\cdot\nabla_{x}-\Delta_{v}
	% \end{align*}
	% is hypoelliptic, according to \(\eqref{eq:NSVFK0}_2\), \(f^{\varepsilon,\delta}\) is compact in \(L^1((0,T)\times\mathbb{R}^3_x\times\mathbb{R}^3_v)\). 
    The proof is complete.
	\end{proof}

	\subsection{\texorpdfstring{The strong convergence of \(|v|^{p}f^{\varepsilon,\delta}\) in \(L^{1}((0,T)\times\mathbb{R}^3_x\times\mathbb{R}^3_v)\)}{The strong convergence of}}$ $
    
    It follows from  Lemma \ref{lem:fcompact} that the strong convergence of \(|v|^{p}f^{\varepsilon,\delta}\) in \(L^{1}((0,T)\times\mathbb{R}^3_x\times\mathbb{R}^3_v)\) still holds ($\forall~ p\in[0,\kappa)$).

	\begin{lemma}\label{lem:vfcompact}
		Under the assumptions stated in Theorem \ref{thm:solution}, we have the strong convergence
		\begin{equation}\label{eq:velocity_moment_convergence}
			|v|^p f^{\varepsilon,\delta} \to |v|^p f \quad \text{strongly in } L^1((0,T)\times\mathbb{R}^3_x\times\mathbb{R}^3_v), \quad 0\leq p<\kappa,
		\end{equation}
        as \(\varepsilon,\delta \rightarrow 0.\)
	\end{lemma}
		\begin{proof}
			By H\"{o}lder's inequality and Lemma \ref{lem:mf}, it shows that
		\begin{align*}
			\norms{|v|^p(f^{\varepsilon,\delta}-f)}_{L^1_{txv}}&\leq \norms{|v|^{\kappa}\left|f^{\varepsilon,\delta}-f\right|}^{\frac{p}{\kappa}}_{L^1_{txv}}\norms{f^{\varepsilon,\delta}-f}^{\frac{\kappa-p}{\kappa}}_{L^1_{txv}}\\
			&\leq C\norms{f^{\varepsilon,\delta}-f}^{\frac{\kappa-p}{\kappa}}_{L^1_{txv}}.
		\end{align*}
		Hence, the proof is completed by Lemma \ref{lem:fcompact}.
		\end{proof}

        \begin{corollary}\label{lem:tpcompact}
	Under the assumptions stated in Theorem \ref{thm:solution}, the following strong convergence holds:
	\begin{equation*}
		|v|^p f^{\varepsilon,\delta} \to |v|^p f \quad \text{strongly in } L^q(0,T;L^1(\mathbb{R}^3_x\times\mathbb{R}^3_v))
	\end{equation*}
	for all exponents $p \in [0, \kappa)$ and $q\in[1,+\infty)$.
\end{corollary}
\begin{proof}
    By an interpolation inequality for \(L^p\) norms, it follows that
    \begin{align}\label{lem:ftp}
        \norms{|v|^p\left(f^{\varepsilon,\delta}-f\right)}_{{L^q_tL^1_{xv}}}\leq \norms{|v|^p\left(f^{\varepsilon,\delta}-f\right)}^{\frac{q-1}{q}}_{L^\infty_tL^1_{xv}}\norms{|v|^p\left(f^{\varepsilon,\delta}-f\right)}_{L^1_{txv}}^\frac 1q ,\quad \forall~q\in [1,+\infty).
    \end{align}
   By Lemma \ref{pro:estimate}, Lemma \ref{lem:mf} and Lemma \ref{lem:vfcompact}, we derive
    \begin{align*}
        \norms{|v|^p\left(f^{\varepsilon,\delta}-f\right)}_{{L^q_tL^1_{xv}}}\rightarrow0,
    \end{align*}
    where $q\in[1,+\infty)$.
\end{proof}

	\subsection{\texorpdfstring{The strong convergence of \(u^{\varepsilon,\delta}\) in \(L^{2}(0,T;L^q(\Omega_x))\)}{The strong convergence of}}$ $
    In this subsection, we consider the strong convergence of \(u^{\varepsilon,\delta}\) by the Aubin-Lions lemma.
\begin{lemma}\label{lem:ul2}
For any bounded spatial domain
$\Omega_x\subset\mathbb R^3$,
there hold
\begin{equation}\label{40}
\left\{
\begin{aligned}
u^{\varepsilon,\delta}\rightarrow u\quad
&\text{strongly in }L^2(0,T;L^q(\Omega_x))
\quad\text{for }q\in[1,6),\\
u^{\varepsilon,\delta}\rightarrow u\quad
&\text{strongly in }L^{\tilde q}((0,T)\times\Omega_x)
\quad\text{for }\tilde q\in[1,\tfrac{10}{3})
\end{aligned}
\right.
\end{equation}
as $\varepsilon,\delta\rightarrow0$
along a subsequence.
\end{lemma}

\begin{proof}
It suffices to prove the result for bounded smooth spatial domains.
Note that
\begin{align*}
\partial_tu^{\varepsilon,\delta}
&=-(\theta_\varepsilon*u^{\varepsilon,\delta})
  \cdot\nabla u^{\varepsilon,\delta}
  +\Delta u^{\varepsilon,\delta}
  -\nabla_xP^{\varepsilon,\delta}\\
&\quad+\theta_\varepsilon*
\left(\int_{\mathbb R^3}
(v-\theta_\varepsilon*u^{\varepsilon,\delta})
f^{\varepsilon,\delta}\,dv\right)
=\sum_{i=1}^4J_i.
\end{align*}

For $J_1$, by H\"older's inequality, \eqref{eq:esti}
and \eqref{utp26}, we have
\begin{align*}
&-\int_{(0,T)\times\Omega_x}
(\theta_\varepsilon*u^{\varepsilon,\delta})
\cdot\nabla u^{\varepsilon,\delta}\cdot\chi\,dxds\\
&=\int_{(0,T)\times\Omega_x}
\bigl((\theta_\varepsilon*u^{\varepsilon,\delta})
\cdot\nabla_x\chi\bigr)
\cdot u^{\varepsilon,\delta}\,dxds\\
&\le\int_0^T
\|\nabla_x\chi\|_{L_x^2(\Omega_x)}
\|u^{\varepsilon,\delta}\|_{L_x^4(\mathbb R^3)}^2\,ds\\
&\le\|\nabla_x\chi\|_{L^4(0,T;L_x^2(\Omega_x))}
\|u^{\varepsilon,\delta}\|_
{L^{8/3}(0,T;L^4(\mathbb R^3))}^2\\
&\le C\|\nabla_x\chi\|_{L^4(0,T;L_x^2(\Omega_x))}
\left(
\|\nabla u^{\varepsilon,\delta}\|_
{L^2((0,T)\times\mathbb R^3)}^2
+\|u^{\varepsilon,\delta}\|_
{L^\infty(0,T;L^2(\mathbb R^3))}^2
\right)\\
&\le C\|\nabla_x\chi\|_{L^4(0,T;L^2(\Omega_x))},
\end{align*}
where
$\chi\in C_c^\infty((0,T)\times\Omega_x;\mathbb R^3)$
is an arbitrary vector-valued test function, extended
by zero outside $\Omega_x$.

For $J_2$, applying integration by parts, followed by
H\"older's inequality, and then using the uniform bound
\eqref{eq:esti}, we obtain
\begin{align*}
\int_0^T\int_{\Omega_x}
\Delta u^{\varepsilon,\delta}\cdot\chi\,dxds
&=-\int_0^T\int_{\Omega_x}
\nabla u^{\varepsilon,\delta}:\nabla_x\chi\,dxds\\
&\le\|\nabla_x\chi\|_{L^2(0,T;L^2(\Omega_x))}
\|\nabla u^{\varepsilon,\delta}\|_
{L^2(0,T;L^2(\Omega_x))}\\
&\le C\|\nabla_x\chi\|_{L^2(0,T;L^2(\Omega_x))}.
\end{align*}

For $J_3$, integration by parts and Lemma~\ref{le:p} give
\begin{align*}
\left|-\int_{(0,T)\times\Omega_x}
\nabla_xP^{\varepsilon,\delta}\cdot\chi\,dxds\right|
&=\left|\int_{(0,T)\times\Omega_x}
P^{\varepsilon,\delta}\nabla_x\cdot\chi\,dxds\right|\\
&\le C\|P^{\varepsilon,\delta}\|_
{L^{5/3}((0,T)\times\mathbb R^3)}
\|\nabla_x\chi\|_{L^{5/2}((0,T)\times\Omega_x)}\\
&\le C_{T,\Omega_x}\|\chi\|_{L^4(0,T;H^2(\Omega_x))},
\end{align*}
where we used
$H^2(\Omega_x)\hookrightarrow W^{1,5/2}(\Omega_x)$
and $L^4(0,T)\hookrightarrow L^{5/2}(0,T)$.

For $J_4$, by \eqref{eq:esti}, the embedding inequality
\eqref{utp26} and \eqref{eq:thebound-n}, this yields
\begin{align*}
&\left|\int_{(0,T)\times\Omega_x}
\theta_\varepsilon*
\left(\int_{\mathbb R_v^3}
(v-\theta_\varepsilon*u^{\varepsilon,\delta})
f^{\varepsilon,\delta}\,dv\right)
\cdot\chi\,dxds\right|\\
&=\left|\int_{(0,T)\times\mathbb R_x^3\times\mathbb R_v^3}
(\theta_\varepsilon*\chi)\cdot
(v-\theta_\varepsilon*u^{\varepsilon,\delta})
f^{\varepsilon,\delta}\,dvdxds\right|\\
&\le\int_{(0,T)\times
\mathbb R_x^3\times\mathbb R_v^3}
|v|f^{\varepsilon,\delta}
|\theta_\varepsilon*\chi|\,dvdxds\\
&\quad+\int_{(0,T)\times
\mathbb R_x^3\times\mathbb R_v^3}
|\theta_\varepsilon*u^{\varepsilon,\delta}|
f^{\varepsilon,\delta}
|\theta_\varepsilon*\chi|\,dvdxds\\
&\le C\|\nabla u^{\varepsilon,\delta}\|_
{L^2(0,T;L^2(\mathbb R_x^3))}
\|\chi\|_{L^2(0,T;L^6(\mathbb R_x^3))}
\left\|\int_{\mathbb R_v^3}
f^{\varepsilon,\delta}\,dv\right\|_
{L^\infty(0,T;L^{3/2}(\mathbb R_x^3))}\\
&\quad+C\|\chi\|_{L^2(0,T;L^6(\mathbb R_x^3))}
\left\|\int_{\mathbb R_v^3}
|v|f^{\varepsilon,\delta}\,dv\right\|_
{L^2(0,T;L^{6/5}(\mathbb R_x^3))}\\
&\le C\|\chi\|_{L^2(0,T;L^6(\mathbb R_x^3))}
\le C\|\chi\|_{L^2(0,T;H^1(\Omega_x))}.
\end{align*}
Here we used the symmetry of the mollifier
and Young's convolution inequality; the integrals containing
$\theta_\varepsilon*\chi$ are taken over $\mathbb R_x^3$
because convolution may enlarge the spatial support.

Combining the above estimates and using
$\|\chi\|_{L^2(0,T;H^1(\Omega_x))}
\le C_T\|\chi\|_{L^4(0,T;H^2(\Omega_x))}$, we obtain
\begin{align*}
\|\partial_tu^{\varepsilon,\delta}\|_
{L^{4/3}(0,T;H^{-2}(\Omega_x))}
\le C_{T,\Omega_x},
\end{align*}
where
$H^{-2}(\Omega_x)=(H_0^2(\Omega_x))'$.
Along with
\begin{align*}
\|u^{\varepsilon,\delta}\|_
{L^2(0,T;H^1(\Omega_x))}\le C,
\end{align*}
and the embeddings
$H^1(\Omega_x)\Subset L^2(\Omega_x)
\hookrightarrow H^{-2}(\Omega_x)$,
the Aubin--Lions lemma yields
\begin{align*}
u^{\varepsilon,\delta}\rightarrow u
\quad\text{strongly in }L^2((0,T)\times\Omega_x)
\quad\text{along a subsequence}.
\end{align*}
A diagonal extraction over an increasing
sequence of balls gives a common subsequence for all
bounded spatial domains.
Using H\"older's inequality and \eqref{utp26},
the proof of \eqref{40} is complete.
\end{proof}

	\subsection{\texorpdfstring{The strong convergence of \(f^{\varepsilon,\delta}\log f^{\varepsilon,\delta}\) in \(L^p((0,T)\times\mathbb{R}^3_x\times\mathbb{R}^3_v)\)}{The strong convergence of}}$ $
	\begin{lemma}\label{lem:flogfbound}
		Assume that \((u^{\varepsilon,\delta},f^{\varepsilon,\delta})\) is a weak solution to the NSVFP equation \eqref{eq:NSVFK0}, there holds
        \begin{equation*}
        \left\{
            \begin{aligned}
		 \norms{f^{\varepsilon,\delta}\log f^{\varepsilon,\delta}}_{L^p_{txv}}&\leq C_1(T+1)e^{6T} &&\text{as}~p\in(1,+\infty),\\
         \norms{f^{\varepsilon,\delta}\log f^{\varepsilon,\delta}}_{L^1_{txv}}&\leq C_2(T+1)^5e^{12T} &&\text{as}~p=1
	\end{aligned}
    \right.
        \end{equation*}
        and
        \begin{equation}\label{eq:log_sq_estimate}
        \left\{
            \begin{aligned}
		 \norms{f^{\varepsilon,\delta}(\log f^{\varepsilon,\delta})^2}_{L^p_{txv}}&\leq C_1(T+1)e^{6T} &&\text{as}~p\in(1,+\infty),\\
         \norms{f^{\varepsilon,\delta}(\log f^{\varepsilon,\delta})^2}_{L^1_{txv}}&\leq C_2(T+1)^5e^{12T} &&\text{as}~p=1,
	\end{aligned}
    \right.
        \end{equation}
        where 
        \begin{equation*}
            \left\{
            \begin{aligned}
            &C_1=C(p, \norms{f_{in}}_{L^1_{xv}},\norms{f_{in}}_{L^{\infty}_{xv}}),\\
            &C_2=C\left(\norm{f_{in}}_{L^{\infty}_{xv}}, \norm{f_{in}}_{L^{1}_{xv}},\norms{|v|^2f_{in}}_{L^{1}_{xv}},\norm{u_{in}}^{2}_{L^{2}_{x}}, \norm{f_{in}\log f_{in}}_{L^{1}_{xv}},\norm{|x|^2f_{in}}_{L^1_{xv}}\right).
            \end{aligned}
            \right.
        \end{equation*}\
	\end{lemma}
	\begin{proof}
		Setting \(Q=(0,T)\times\mathbb{R}^3_{x}\times\mathbb{R}^3_{v}\), we have
	\begin{align}\label{l1l2}
		&\hspace{0.5cm}\int_Q |f^{\varepsilon,\delta}\log f^{\varepsilon,\delta}|^pdxdvds\nonumber\\
        &=\int_{Q\cap\{{f^{\varepsilon,\delta}}>e\}} |f^{\varepsilon,\delta}\log f^{\varepsilon,\delta}|^pdxdvds+\int_{Q\cap\{{f^{\varepsilon,\delta}}\leq e\}} |f^{\varepsilon,\delta}\log f^{\varepsilon,\delta}|^pdxdvds\\
		&=L_1+L_2\nonumber
	\end{align}
	for any \(p\in (1,+\infty)\).\\
	For \( L_1 \), the inequality 
	\begin{equation*}\label{eq:log_estimate}
		\log f^{\varepsilon,\delta} < f^{\varepsilon,\delta} \quad \text{for} \quad f^{\varepsilon,\delta} > e
	\end{equation*}
	leads to the following \( L^p \)-estimate:
	\begin{align*}
		\int_{Q \cap \{f^{\varepsilon,\delta} > e\}} |f^{\varepsilon,\delta} \log f^{\varepsilon,\delta}|^p dxdvds		&\leq \int_{Q \cap \{f^{\varepsilon,\delta} > e\}} |f^{\varepsilon,\delta}|^{2p} dxdvds\\
		&\leq \|f^{\varepsilon,\delta}\|^{2p}_{L^{2p}(Q)}.
	\end{align*}
	For \(L_2\), due to \(\left|(f^{\varepsilon,\delta})^{1-\frac1p}\log f^{\varepsilon,\delta}\right| \leq \frac{p}{(p-1)e}+e^{1-1/p}\), it shows that
	\begin{align*}
		\int_{Q\cap\{{f^{\varepsilon,\delta}}\leq e\}} |f^{\varepsilon,\delta}\log f^{\varepsilon,\delta}|^pdxdvds \leq C\norms{f^{\varepsilon,\delta}}_{L^1((0,T)\times\mathbb{R}^3_x\times\mathbb{R}^3_z)}.
	\end{align*}
	% where \(\kappa\) is an arbitrarily small positive number satisfying \(p-p\kappa = 1\).
	To sum up, we derive
	\begin{align*}
		\norm{f^{\varepsilon,\delta}\log f^{\varepsilon,\delta}}^{p}_{L^{p}\big((0,T)\times\mathbb{R}^3_{x}\times\mathbb{R}^3_v\big)}&\leq C\left(\norm{f^{\varepsilon,\delta}}^{2p}_{L^{2p}\big((0,T)\times\mathbb{R}^3_{x}\times\mathbb{R}^3_v\big)}+\norm{f^{\varepsilon,\delta}}_{L^1\big((0,T)\times\mathbb{R}^3_{x}\times\mathbb{R}^3_v\big)}\right).\\
		% &\leq C_1\norm{f^{\varepsilon,\delta}}^{2p}_{L^{\infty}\big((0,T)\times\mathbb{R}^3_{x}\times\mathbb{R}^3_v\big)}+C_2\norm{f^{\varepsilon,\delta}}^{p(1-\kappa)}_{L^{\infty}\big((0,T)\times\mathbb{R}^3_x\times\mathbb{R}^3_v\big)}.
	\end{align*}
	By \eqref{5} and interpolation inequality for \(L^p\)-norms, it follows
	\begin{align*}
		 \norm{f^{\varepsilon,\delta}\log f^{\varepsilon,\delta}}_{L^{p}\big((0,T)\times\mathbb{R}^3_{x}\times\mathbb{R}^3_v\big)}&\leq C\left(\norm{f^{\varepsilon,\delta}}^{2p}_{L^{2p}\big((0,T)\times\mathbb{R}^3_{x}\times\mathbb{R}^3_v\big)}+\norm{f^{\varepsilon,\delta}}_{L^{1}\big((0,T)\times\mathbb{R}^3_{x}\times\mathbb{R}^3_v\big)}\right)^{\frac{1}{p}}\\
         % &\leq C \left(\norm{f_{in}}^{2}_{L^{2p}\big(\mathbb{R}^3_{x}\times\mathbb{R}^3_v\big)}+\norm{f_{in}}^{1-\kappa}_{L^{p(1-\kappa)}\big(\mathbb{R}^3_{x}\times\mathbb{R}^3_v\big)}\right)\\
         % &\leq C \bigg(\norm{f_{in}}^{\frac{1}{p}}_{L^{1}\big(\mathbb{R}^3_{x}\times\mathbb{R}^3_v\big)}\norm{f_{in}}^{\frac{2p-1}{p}}_{L^{\infty}\big(\mathbb{R}^3_{x}\times\mathbb{R}^3_v\big)}\\
         % &\hspace{4cm}+\norm{f_{in}}^{\frac{1}{p}}_{L^{1}\big(\mathbb{R}^3_{x}\times\mathbb{R}^3_v\big)}\norm{f_{in}}^{1-\kappa-\frac{1}{p}}_{L^{\infty}\big(\mathbb{R}^3_{x}\times\mathbb{R}^3_v\big)}\bigg)\\
        &\leq C(\norms{f_{in}}_{L^1_{xv}},\norms{f_{in}}_{L^{\infty}_{xv}})(T+1)e^{6T} ,
        % &\leq C_1\norm{f^{\varepsilon,\delta}}^{2}_{L^{\infty}\big((0,T)\times\mathbb{R}^3_{x}\times\mathbb{R}^3_v\big)}+C_2\norm{f^{\varepsilon,\delta}}^{1-\kappa}_{L^{\infty}\big((0,T)\times\mathbb{R}^3_{x}\times\mathbb{R}^3_v\big)}\\
		% &\leq Ce^{6T}\left(\norm{f_{in}}^{2}_{L^{\infty}\big((0,T)\times\mathbb{R}^3_x\times\mathbb{R}^3_v\big)}+\norm{f_{in}}^{1-\kappa}_{L^{\infty}\big((0,T)\times\mathbb{R}^3_x\times\mathbb{R}^3_v\big)}\right)
	\end{align*}
for any \(p\in (1,+\infty)\).
 %   \end{align*}
% , where
%     \begin{align*}
%         C_1&=max\left\{1,\norm{f_{in}}_{L^{1}\big(\mathbb{R}^3_{x}\times\mathbb{R}^3_v\big)}\right\},\\ C_2&=max\left\{1,\norm{f_{in}}^2_{L^{\infty}\big(\mathbb{R}^3_{x}\times\mathbb{R}^3_v\big)}\right\},\\
%         C_3&=max\left\{1,\norm{f_{in}}^{1-\kappa}_{L^{\infty}\big(\mathbb{R}^3_{x}\times\mathbb{R}^3_v\big)}\right\}.
    Independently of the constants in the preceding finite-$p$ estimates, the uniform bound $0\le f^{\varepsilon,\delta}\le M_T$ and boundedness of $a\mapsto a|\log a|$ on $[0,M_T]$ give
	\begin{align*}
		\norms{f^{\varepsilon,\delta}\log f^{\varepsilon,\delta}}_{L^{\infty}_{txv}}\leq C(\norms{f_{in}}_{L^1_{xv}},\norms{f_{in}}_{L^{\infty}_{xv}})(T+1)e^{6T}  .
	\end{align*}
For $p = 1$, the estimates \eqref{18} and \eqref{eq:13} yield
\begin{align*}
    \norms{f^{\varepsilon,\delta}\log f^{\varepsilon,\delta}}_{L^1_{txv}}\leq C_2(T+1)^5e^{12T},
\end{align*}
where \(C_2=C\left(\norm{f_{in}}_{L^{\infty}_{xv}}, \norm{f_{in}}_{L^{1}_{xv}},\norms{|v|^2f_{in}}_{L^{1}_{xv}},\norm{u_{in}}^{2}_{L^{2}_{x}}, \norm{f_{in}\log f_{in}}_{L^{1}_{xv}}, \norm{|x|^2f_{in}}_{L^1_{xv}}\right).\) 

Since $0\le f\le e^{3T}\|f_{in}\|_{L^\infty}$, for every
$1<p<\infty$,
\[
\int_{\mathbb R^6}|f(\log f)^2|^p\,dx\,dv
\le C_{p,M_T}\int_{\mathbb R^6}f\,dx\,dv.
\]
For $p=1$, setting $W=1+|x|^2+|v|^2$, we use
\[
f(\log f)^2
\le C_{M_T}Wf+C e^{-\sqrt{W}/2}.
\]
The mass and second-moment bounds, together with
$e^{-\sqrt{W}/2}\in L^1(\mathbb R^6)$, yield the claimed $L^1$ bound.
Here $C_1$ may depend on $p$, and $C_2$ also depends on
$\||x|^2f_{in}\|_{L^1(\mathbb R^6)}$.
\end{proof}

	 \begin{lemma}\label{lem:flogfcompact}
	 	\(f^{\varepsilon,\delta}\log f^{\varepsilon,\delta} \to f\log f\) converges strongly in \(L^p((0,T)\times\mathbb{R}^3_{x}\times\mathbb{R}^3_{v})\) for \(p\in (1,+\infty)\).
	 \end{lemma}
	\begin{proof}
		We divide the following inequality into four parts as follows. The difference of logarithmic terms admits the decomposition:
		\begin{align*}
			\int_{Q} & |f^{\varepsilon,\delta}\log f^{\varepsilon,\delta} - f\log f|^p \,dx\,dv\,ds \\
			&\leq \int_{\substack{Q\cap\{f^{\varepsilon,\delta}<1\} \cap \{f^{\varepsilon,\delta}>f\}}} |f^{\varepsilon,\delta}\log f^{\varepsilon,\delta} - f\log f|^p \,dx\,dv\,ds \\
			&\quad + \int_{\substack{Q\cap\{f^{\varepsilon,\delta}<1\} \cap \{f^{\varepsilon,\delta}<f\}}} |f^{\varepsilon,\delta}\log f^{\varepsilon,\delta} - f\log f|^p \,dx\,dv\,ds \\
			&\quad + C\int_{Q\cap\{f^{\varepsilon,\delta}\auditmath{\geq}1\}} |(f^{\varepsilon,\delta}-f)\log f^{\varepsilon,\delta}|^p \,dx\,dv\,ds \\
			&\quad + C\int_{Q\cap\{f^{\varepsilon,\delta}\auditmath{\geq}1\}} |f(\log f^{\varepsilon,\delta}-\log f)|^p \,dx\,dv\,ds \\
			&:= K_1 + K_2 + K_3 + K_4.
		\end{align*}
	For \(K_1\), it shows that
	\begin{align*}
		K_1=&\int_{Q\cap\{f^{\varepsilon,\delta}<1\}\cap \{f^{\varepsilon,\delta}>f\}}\left|(f^{\varepsilon,\delta}-f)\log f^{\varepsilon,\delta}+f\log\frac{f^{\varepsilon,\delta}}{f}\right|^pdxdvds\\
		\leq& C\int_{Q\cap\{f^{\varepsilon,\delta}<1\}\cap \{f^{\varepsilon,\delta}>f\}}|(f^{\varepsilon,\delta}-f)\log f^{\varepsilon,\delta}|^pdxdvds\\
        &+C\int_{Q\cap\{f^{\varepsilon,\delta}<1\}\cap \{f^{\varepsilon,\delta}>f\}}\left|f\log\frac{f^{\varepsilon,\delta}}{f}\right|^pdxdvds\\
		:= &K_{11}+K_{12}.
	\end{align*}
	For \(K_{11}\), as the proof  of \(L_2\) in \eqref{l1l2}, this yields
	\begin{align*}
		K_{11}
        % &\leq C\int_{Q\cap\{f^{\varepsilon,\delta}<1\}\cap \{f^{\varepsilon,\delta}>f\}}|(f^{\varepsilon,\delta}-f)(f^{\varepsilon,\delta})^{-\kappa}|^pdxdvds\\
		&\leq C\int_{Q\cap\{f^{\varepsilon,\delta}<1\}\cap \{f^{\varepsilon,\delta}>f\}}|f^{\varepsilon,\delta}-f|dxdvds\\
		&\leq C\norm{f^{\varepsilon,\delta}-f}_{L^{1}((0,T)\times\mathbb{R}^3_{x}\times\mathbb{R}^3_{v})}.
	\end{align*}
	For \(K_{12}\), by the inequality \(\log(1+\tilde{f})\leq \tilde{f}\) when \(\tilde{f}>-1\), we obtain
	\begin{align*}
		K_{12}=&\int_{Q\cap\{f^{\varepsilon,\delta}<1\}\cap \{f^{\varepsilon,\delta}>f\}}|f\log(1+\frac{f^{\varepsilon,\delta}-f}{f})|^pdxdvds\\
		\leq& \int_{Q\cap\{f^{\varepsilon,\delta}<1\}\cap \{f^{\varepsilon,\delta}>f\}}|f^{\varepsilon,\delta}-f|^pdxdvds\\
		\leq& \norm{f^{\varepsilon,\delta}-f}^{p}_{L^{p}((0,T)\times\mathbb{R}^3_{x}\times\mathbb{R}^3_{v})}.
	\end{align*}
	For \(K_2\), it follows that
	\begin{align*}
		K_2&= \int_{Q\cap\{f^{\varepsilon,\delta}<1\}\cap \{f^{\varepsilon,\delta}<f\}}\left|(f-f^{\varepsilon,\delta})\log f+f^{\varepsilon,\delta}\log\frac{f}{f^{\varepsilon,\delta}}\right|^pdxdvds\\
		&\leq C\int_{Q\cap\{f^{\varepsilon,\delta}<1\}\cap \{f^{\varepsilon,\delta}<f\}}|(f-f^{\varepsilon,\delta})\log f|^pdxdvds\\
        &\hspace{0.5cm}+C\int_{Q\cap\{f^{\varepsilon,\delta}<1\}\cap \{f^{\varepsilon,\delta}<f\}}\left|f^{\varepsilon,\delta}\log\frac{f}{f^{\varepsilon,\delta}}\right|^pdxdvds\\
		&=C\int_{Q\cap\{f^{\varepsilon,\delta}<1\}\cap \{f^{\varepsilon,\delta}<f\}\cap \{f<1\}}|(f-f^{\varepsilon,\delta})\log f|^pdxdvds\\
		&\hspace{0.5cm}+C\int_{Q\cap\{f^{\varepsilon,\delta}<1\}\cap \{f^{\varepsilon,\delta}<f\}\cap \{f\geq1\}}|(f-f^{\varepsilon,\delta})\log f|^pdxdvds\\
		&\hspace{0.5cm}+C\int_{Q\cap\{f^{\varepsilon,\delta}<1\}\cap \{f^{\varepsilon,\delta}<f\}}\left|f^{\varepsilon,\delta}\log\frac{f}{f^{\varepsilon,\delta}}\right|^pdxdvds\\
		&:=K_{21}+K_{22}+K_{23}.
	\end{align*}
	Fix $0<\eta<1-1/p$, independently of the velocity-moment exponent $\kappa$.
For \(K_{21}\), as the proof  of \(L_2\) in \eqref{l1l2}, it becomes
	\begin{align*}
		K_{21}&\leq C\int_{Q\cap\{f^{\varepsilon,\delta}<1\}\cap \{f^{\varepsilon,\delta}<f\}\cap \{f<1\}}|f-f^{\varepsilon,\delta}|^{p(1-\eta)}dxdvds\\
		&\leq C\norm{f^{\varepsilon,\delta}-f}^{p(1-\eta)}_{L^{p(1-\eta)}((0,T)\times\mathbb{R}^3_{x}\times\mathbb{R}^3_{v})}.
	\end{align*}
	For \(K_{22}\), we get
	\begin{align*}
		K_{22} &\leq C\int_{Q\cap\{f^{\varepsilon,\delta}<1\}\cap \{f^{\varepsilon,\delta}<f\}\cap \{f\geq1\}}|(f^{\varepsilon,\delta}-f)f|^pdxdvds\\
		&\leq C\norm{f}^{p}_{L^{\infty}((0,T) \times \mathbb{R}^3_x\times \mathbb{R}^3_v)}\norm{f^{\varepsilon,\delta}-f}^{p}_{L^{p}((0,T)\times\mathbb{R}^3_{x}\times\mathbb{R}^3_{v})}.
	\end{align*}
	For \(K_{23}\), as the proof  in \(K_{12}\), there holds
	\begin{align*}
		K_{23} &\leq C\int_{Q\cap\{f^{\varepsilon,\delta}<1\}\cap \{f^{\varepsilon,\delta}<f\}}|f-f^{\varepsilon,\delta}|^pdxdvds\\
		&\leq C\norm{f^{\varepsilon,\delta}-f}^{p}_{L^{p}((0,T)\times\mathbb{R}^3_{x}\times\mathbb{R}^3_{v})}.
	\end{align*}
For $K_3$, the following estimates hold:
\begin{align*}
	K_{3} &\leq C\int_{Q\cap\{f^{\varepsilon,\delta}\geq1\}}|(f^{\varepsilon,\delta}-f)f^{\varepsilon,\delta}|^pdxdvds\\
	&\leq C\norm{f^{\varepsilon,\delta}}^{p}_{L^{\infty}((0,T) \times \mathbb{R}^3_x\times \mathbb{R}^3_v)}\norm{f^{\varepsilon,\delta}-f}^{p}_{L^{p}((0,T)\times\mathbb{R}^3_{x}\times\mathbb{R}^3_{v})}.
\end{align*}
For $K_4$, we have
\begin{align*}
	K_4 &= \int_{Q\cap\{f^{\varepsilon,\delta}\geq1\}\cap\{f^{\varepsilon,\delta}>f\}}\left|f\log\left(1+\frac{f^{\varepsilon,\delta}-f}{f}\right)\right|^p \,dx\,dv\,ds\\
    &\hspace{0.5cm}+ \int_{Q\cap\{f^{\varepsilon,\delta}\geq1\}\cap\{f^{\varepsilon,\delta}<f\}}\left|f\log\left(1+\frac{f-f^{\varepsilon,\delta}}{f^{\varepsilon,\delta}}\right)\right|^p \,dx\,dv\,ds\\
	&\leq \int_{Q\cap\{f^{\varepsilon,\delta}\geq1\}\cap\{f^{\varepsilon,\delta}>f\}}|f^{\varepsilon,\delta}-f|^p \,dx\,dv\,ds+\int_{Q\cap\{f^{\varepsilon,\delta}\geq1\}\cap\{f^{\varepsilon,\delta}<f\}}\left|f\left(\frac{f-f^{\varepsilon,\delta}}{f^{\varepsilon,\delta}}\right)\right|^p \,dx\,dv\,ds\\
	&\leq C\|f^{\varepsilon,\delta}-f\|^{p}_{L^{p}(Q)}.
\end{align*}
	To sum up, by Lemma \ref{lem:fcompact}, this lemma is proven.
	\end{proof}

\begin{lemma}\label{lem:flogf2bound}
	Let $\Omega_{x} \subset \mathbb{R}^3_x$ be a bounded domain. The following uniform bound holds:
	\begin{equation*}
		\|f^{\varepsilon,\delta}(\log f^{\varepsilon,\delta})^2\|_{L^1((0,T)\times\Omega_{x}\times\mathbb{R}^3_v)} \leq C,
	\end{equation*}
	where $C$ is independent of $\varepsilon$ and $\delta$.
\end{lemma}
	\begin{proof}
		By \eqref{8}, it shows that
	\begin{align*}
		\int_{(0,T) \times \Omega_x \times \mathbb{R}^3_v}&f^{\varepsilon,\delta}(\log f^{\varepsilon,\delta})^2dvdxds \\
		&\leq \left(1+\frac{4}{3}\pi\norms{f^{\varepsilon,\delta}(\log f^{\varepsilon,\delta})^2}_{L^{\infty}_{txv}}\right)\int_{(0,T) \times \Omega_x}\left(\int_{  \mathbb{R}^3_v}|v|f^{\varepsilon,\delta}(\log f^{\varepsilon,\delta})^2dv\right)^\frac34dxds.
	\end{align*}
	For \(\left(\int_{  \mathbb{R}^3_v}|v|f^{\varepsilon,\delta}(\log f^{\varepsilon,\delta})^2dv\right)^\frac34\), due to
    \begin{align}\label{38}
        \left|f^{\frac18}\log f\right|\leq \max\{\frac8e,e^{\frac18}\}\quad(0\le f\le e),
    \end{align}
	% \begin{align}
	% 	\label{38}
	% 	\lim_{f^{\varepsilon,\delta}\rightarrow 0} (f^{\varepsilon,\delta})^{\frac18}\log f^{\varepsilon,\delta}=0,
	% \end{align}
	we find
	\begin{align*}
		&\hspace{0.5cm}\int_{(0,T) \times \Omega_x}\left(\int_{  \mathbb{R}^3_v}|v|f^{\varepsilon,\delta}(\log f^{\varepsilon,\delta})^2dv\right)^\frac34dxds\\
		 &= 	\int_{(0,T) \times \Omega_x}\left(\int_{  \mathbb{R}^3_v \cap \{f^{\varepsilon,\delta}>e\}}|v|f^{\varepsilon,\delta}(\log f^{\varepsilon,\delta})^2dv+\int_{  \mathbb{R}^3_v \cap \{f^{\varepsilon,\delta} \leq e\}}|v|f^{\varepsilon,\delta}(\log f^{\varepsilon,\delta})^2dv\right)^\frac34dxds\\
		% & \leq C\int_{(0,T) \times \Omega_x}\left(\int_{  \mathbb{R}^3_v \cap \{f^{\varepsilon,\delta}>e\}}|v|(f^{\varepsilon,\delta})^3dv+\int_{  \mathbb{R}^3_v \cap \{f^{\varepsilon,\delta} \leq e\}}|v|(f^{\varepsilon,\delta})^\frac34dv\right)^\frac34dxds\\
		&\leq C\int_{(0,T) \times \Omega_x}\left(\int_{  \mathbb{R}^3_v}|v|(f^{\varepsilon,\delta})^3dv\right)^\frac34dxds+C\int_{(0,T) \times \Omega_x}\left(\int_{  \mathbb{R}^3_v}|v|(f^{\varepsilon,\delta})^\frac34dv\right)^\frac34dxds\\
		&:= k_1+k_2.
	\end{align*}
	For \(k_1\), by H\"{o}lder's inequality, it becomes that
	\begin{align*}
		k_1 \leq C\norms{f^{\varepsilon,\delta}}_{L^{\infty}((0,T)\times\Omega_{x}\times\mathbb{R}^3_v)}^\frac32\norms{|v|f^{\varepsilon,\delta}}_{L^1((0,T)\times\Omega_{x}\times\mathbb{R}^3_v)}^\frac34.
	\end{align*}
	For \(k_2\), we have
	\begin{align*}
		k_2\leq \int_{(0,T) \times \Omega_x}\left(\int_{  \mathbb{R}^3_v}\frac{1}{1+|v|}(|v|+1)^2(f^{\varepsilon,\delta})^\frac34dv\right)^\frac34dxds,
	\end{align*}
	and by H\"{o}lder's inequality about \(v\), it becomes
	\begin{align*}
		&\leq \int_{(0,T) \times \Omega_x}\left(\int_{  \mathbb{R}^3_v}\frac{1}{(1+|v|)^4}dv\right)^{\frac{3}{16}}\left(\int_{  \mathbb{R}^3_v}(1+|v|)^{\frac83}f^{\varepsilon,\delta}dv\right)^{\frac{9}{16}}dxds\\
		&\leq C\left(\int_{(0,T) \times \Omega_x \times \mathbb{R}^3_v}(1+|v|)^{\frac83}f^{\varepsilon,\delta}dvdxds\right)^{\frac{9}{16}}.
	\end{align*}
	By \eqref{5} and \eqref{10}, we derive
	\begin{align*}
		\norms{f^{\varepsilon,\delta}(\log f^{\varepsilon,\delta})^2}_{L^1((0,T)\times\Omega_{x}\times\mathbb{R}^3_v)}\leq C.
	\end{align*}
	\end{proof}

	\begin{lemma}\label{lem:flogfcompact1}
		\(f^{\varepsilon,\delta}\log f^{\varepsilon,\delta} \rightarrow f\log f\) strongly in \(L^1((0,T)\times\Omega_{x}\times\mathbb{R}^3_v)\) as \(\varepsilon,\delta \rightarrow 0\).
	\end{lemma}
	 \begin{proof}
	 	Firstly, by H\"{o}lder's inequality, it follows that
	\begin{align*}
		\int_{(0,T) \times \Omega_x}&\int_{|v|\geq R} f^{\varepsilon,\delta}|\log f^{\varepsilon,\delta}|dvdxds \leq\\ &\left(\int_{(0,T) \times \Omega_x}\int_{|v|\geq R} f^{\varepsilon,\delta}dvdxds\right)^\frac12\left(\int_{(0,T) \times \Omega_x}\int_{|v|\geq R} f^{\varepsilon,\delta}(\log f^{\varepsilon,\delta})^2dvdxds\right)^\frac12.
	\end{align*}
	Under the result of  \eqref{32}, we derive
	\begin{align}
		\label{35}
		\sup_{\varepsilon,\delta}\int_{0}^{T}\int_{\Omega_{x}}\int_{|v|\geq R}f^{\varepsilon,\delta}|\log f^{\varepsilon,\delta}|dvdxds\rightarrow 0,~\text{as}~R\rightarrow \infty.
	\end{align}	
	Then, by Lemma \ref{lem:flogfcompact}, one has
	\begin{align*}
		f^{\varepsilon,\delta}\log f^{\varepsilon,\delta} \rightarrow f\log f~\text{a.e. on}~(0,T)\times\Omega_{x}\times\mathbb{R}^3_v.
	\end{align*}
    
	Finally, we prove \(\int_{(0,T) \times \Omega_x \times \mathbb{R}^{3}_{v}}f^{\varepsilon,\delta}\log f^{\varepsilon,\delta}dvdxds\) is equi-integrable. Due to \eqref{35}, this yields the followings: For \(\forall \varepsilon_1>0, \exists R>0\), while \(|v|>R\), for \(\forall \varepsilon,\delta\), \( s.t.\)
	\begin{align}
		\label{36}
		\sup_{\varepsilon,\delta}\int_{0}^{T}\int_{\Omega_x}\int_{|v|\geq R}f^{\varepsilon,\delta}\left|\log f^{\varepsilon,\delta}\right|dvdxds < \varepsilon_1
	\end{align} holds.
	For any \(E \subset (0,T)\times\Omega_{x}\times\mathbb{R}^3_v\), if \(E \subset (0,T)\times\Omega_{x}\times B_R\), 
    \(\mu(E) < \frac{\varepsilon_1}{\norms{f^{\varepsilon,\delta}\log f^{\varepsilon,\delta}}_{L^{\infty}_{txv}}}\), then
	\begin{align}
		\label{37}
		\sup_{\varepsilon,\delta}\int_{E}f^{\varepsilon,\delta}\left|\log f^{\varepsilon,\delta}\right|dvdxds \leq \norms{f^{\varepsilon,\delta}\log f^{\varepsilon,\delta}}_{L^{\infty}_{txv}} \mu(E) < \varepsilon_1
	\end{align}
	holds. 
%     If \begin{align*}E \cap (0,T)\times\Omega_{x}\times B^c_R =\varnothing,\end{align*} by \eqref{36}, it is obvious that \eqref{37} holds. If \begin{align*}E \cap \big((0,T)\times\Omega_{x}\times B^c_R\big) \neq \varnothing,\end{align*}
% we write \(E = E_1 \cup E_2\) with \(E_1 \subset (0,T)\times\Omega_{x}\times B_R\) 
% and \(E_2 \cap \big((0,T)\times\Omega_{x}\times B^c_R\big) = \varnothing\), 
% which implies that~\eqref{37} still holds.
 By the Lebesgue-Vitali theorem (Lemma \ref{A.6}), we prove this lemma.
	 \end{proof}
	
% \begin{revision}
% \paragraph{Global entropy convergence:}
% The uniform $L^\infty$ bound and the spatial and velocity
% second-moment bounds imply, with
% $w=(1+|x|^2+|v|^2)^{1/2}$,
% \[
% \operatorname*{ess\,sup}_{0<t<T}
% \int_{\{|(x,v)|>R\}}
% \bigl|f^{\varepsilon,\delta}
% \log f^{\varepsilon,\delta}\bigr|\,dx\,dv
% \leq \frac{C_T}{R}
% +C\int_{\{|(x,v)|>R\}}e^{-w/2}\,dx\,dv
% \longrightarrow 0
% \]
% as $R\to\infty$, uniformly in $\varepsilon,\delta$.
% Together with the local convergence, this yields
% \begin{equation}\label{min:globalentropy}
% f^{\varepsilon,\delta}\log f^{\varepsilon,\delta}
% \longrightarrow f\log f
% \quad\text{in }L^1((0,T)\times\mathbb R^6).
% \end{equation}
% Combining this with Lemma~\ref{lem:vfcompact}
% and taking a diagonal subsequence, the entropy densities
% and second velocity-moment densities converge in
% $L^1(\mathbb R^6)$ for almost every $t>0$.
% \end{revision}	

\begin{revision}
\begin{lemma}[Global entropy convergence]\label{cor:globalentropy}
Under the assumptions and uniform estimates established above, we have, along the subsequence extracted above,
\begin{equation}\label{min:globalentropy}
f^{\varepsilon,\delta}\log f^{\varepsilon,\delta}\longrightarrow f\log f
\quad\text{strongly in }L^1((0,T)\times\mathbb R^6).
\end{equation}
\end{lemma}

\begin{proof}
Set $z=(x,v)$ and $w(z)=(1+|x|^2+|v|^2)^{1/2}$. By the uniform $L^\infty$ estimate, there exists $M_T\ge1$, independent of $\varepsilon,\delta$, such that $0\le f^{\varepsilon,\delta}\le M_T$. Since it holds
\[
a\log a+wa\ge-e^{-w-1},\quad (a\log a)_+\le(\log M_T)a,\quad 0\le a\le M_T,
\]
we obtain
\[
|a\log a|\le C_Twa+Ce^{-w/2},\qquad 0\le a\le M_T.
\]
Consequently, the uniform spatial and velocity second-moment bounds yield, for $R\ge1$,
\begin{align*}
&\operatorname*{ess\,sup}_{0<t<T}\int_{|z|>R}|f^{\varepsilon,\delta}\log f^{\varepsilon,\delta}|\,dz\\
&\quad\le\frac{C_T}{R}\operatorname*{ess\,sup}_{0<t<T}\int_{\mathbb R^6}(1+|z|^2)f^{\varepsilon,\delta}\,dz+C\int_{|z|>R}e^{-w/2}\,dz\\
&\quad\le\frac{C_T}{R}+C\int_{|z|>R}e^{-w/2}\,dz\longrightarrow0,
\end{align*}
as $R\to\infty$, uniformly in $\varepsilon,\delta$. 
Similarly, by the weak lower semicontinuity of the uniform bounds, the same tail estimate holds for $f\log f$.
By the local strong convergence established above, for every $R>0$,
\[
\|f^{\varepsilon,\delta}\log f^{\varepsilon,\delta}-f\log f\|_{L^1((0,T)\times B_R)}\longrightarrow0.
\]
Therefore,
\begin{align*}
&\|f^{\varepsilon,\delta}\log f^{\varepsilon,\delta}-f\log f\|_{L^1((0,T)\times\mathbb R^6)}\\
&\quad\le\|f^{\varepsilon,\delta}\log f^{\varepsilon,\delta}-f\log f\|_{L^1((0,T)\times B_R)}\\
&\qquad+\|f^{\varepsilon,\delta}\log f^{\varepsilon,\delta}\|_{L^1((0,T)\times B_R^c)}
+\|f\log f\|_{L^1((0,T)\times B_R^c)}.
\end{align*}
Letting first $\varepsilon,\delta\to0$ and then $R\to\infty$, we obtain \eqref{min:globalentropy}.
\end{proof}
\end{revision}

\begin{lemma}\label{lem:vflogf compact}
	Under the assumptions stated in Theorem \ref{thm:solution}, 
    it holds
    \begin{align}
        \norms{|v|^{q}f^{\varepsilon,\delta}\log f^{\varepsilon,\delta}}_{L^1((0,T)\times\Omega_{x}\times\mathbb{R}^3_v)} \leq C
    \end{align}
    with \(q\in [0,\frac{\kappa}{2}]\). Moreover, the following strong convergence holds:
	\begin{equation*}
		|v|^p f^{\varepsilon,\delta}\log f^{\varepsilon,\delta} \to |v|^p f\log f \quad \text{strongly in } L^1((0,T)\times\Omega_x\times\mathbb{R}^3_v)
	\end{equation*}
	for all exponents $p \in [0, \frac{\kappa}{2})$.
\end{lemma}
	\begin{proof}
	The first step is to establish the bound
	\begin{equation*}
		\left\||v|^{q} f^{\varepsilon,\delta} \log f^{\varepsilon,\delta} \right\|_{L^1((0,T)\times\Omega_{x}\times\mathbb{R}^3_v)} \leq C.
	\end{equation*}
	 Using H\"{o}lder's inequality, it becomes
	\begin{align*}
		&\hspace{0.5cm}\norms{|v|^{q}f^{\varepsilon,\delta}\log f^{\varepsilon,\delta}}_{L^1((0,T)\times\Omega_{x}\times\mathbb{R}^3_v)} \\
        &\leq 	\norms{|v|^{2q}f^{\varepsilon,\delta}}^\frac12_{L^1((0,T)\times\Omega_{x}\times\mathbb{R}^3_v)}	\norms{f^{\varepsilon,\delta}(\log f^{\varepsilon,\delta})^2}^\frac12_{L^1((0,T)\times\Omega_{x}\times\mathbb{R}^3_v)}.
	\end{align*}
	By Lemma \ref{lem:mf} and Lemma \ref{lem:flogf2bound}, it follows that
	\begin{align*}
		&\norms{|v|^p(f^{\varepsilon,\delta}\log f^{\varepsilon,\delta}-f\log f)}_{L^1_{txv}}\\
		&\leq \norms{|v|^{\frac{\kappa}{2}}\left|f^{\varepsilon,\delta}\log f^{\varepsilon,\delta}-f\log f\right|}^{\frac{2p}{\kappa}}_{L^1_{txv}}\norms{f^{\varepsilon,\delta}\log f^{\varepsilon,\delta}-f\log f}^{\frac{\kappa-2p}{\kappa}}_{L^1_{txv}}\\
		&\leq C\norms{f^{\varepsilon,\delta}\log f^{\varepsilon,\delta}-f\log f}^{\frac{\kappa-2p}{\kappa}}_{L^1_{txv}}.
	\end{align*}
	By Lemma~\ref{lem:flogfcompact1}, the lemma is proved.
	\end{proof}

	\section{Verification of the conditions in \texorpdfstring{$i)-iv)$}{i)-iv)} of Definition \ref{def:weak}.}
    
	\par In this section, we will verify the conditions in \texorpdfstring{$i)-iv)$}{i)-iv)} in Definition \ref{def:weak} and show the global energy inequality that is satisfied for \(\varepsilon,\delta\rightarrow 0\). Firstly, let us review some necessary estimates in Section 3. By Proposition \ref{pro:estimate}, Lemma \ref{lem:deltaf}, Lemma \ref{lem:3f}, Proposition \ref{pro:xf} and Lemma \ref{le:p}, we have
    \begin{equation}\label{eq:uniform norms of f}
        \begin{aligned}
            &\norm{f^{\varepsilon,\delta}(t,x,v)}_{L^{\infty}(0,t;L^{1}(\mathbb{R}^{3}_x\times \mathbb{R}^{3}_v))}+\norm{f^{\varepsilon,\delta}(t,x,v)}_{L^{\infty}((0,t)\times\mathbb{R}^{3}_x\times \mathbb{R}^{3}_v)}\\
        &+\norm{\nabla_{v}f^{\varepsilon,\delta}}^2_{L^{2}((0,t)\times \mathbb{R}^3_{x}\times \mathbb{R}^3_{v})}+\norms{|v|^2f^{\varepsilon,\delta}}_{L^{\infty}(0,t;{L^{1}(\mathbb{R}^{3}_x\times\mathbb{R}^{3}_v)})}+\norms{|v|^3f^{\varepsilon,\delta}}_{L^{\infty}(0,t;L^1(\mathbb{R}^3_x\times\mathbb{R}^3_v))}\\
        &+\norms{f^{\varepsilon,\delta}\log f^{\varepsilon,\delta}}_{L^{\infty}\left(0,t;L^1\left(\mathbb{R}^3_x\times\mathbb{R}^3_v\right)\right)}+\norm{\nabla u^{\varepsilon,\delta}}^{2}_{L^{2}(0,t;{L^{2}(\mathbb{R}^{3}_x)})}+\norm{u^{\varepsilon,\delta}}^{2}_{L^{\infty}(0,t;{L^{2}(\mathbb{R}^{3}_x)})}\\
        &\hspace{2cm}+\norms{P^{\varepsilon,\delta}}_{L^\frac53((0,t)\times \rev{\mathbb{R}^3_x)}}\leq C\left(t+1\right)^3 \exp\left[C(t+1)^3e^{4t}\right],
        \end{aligned}
    \end{equation}
    where
    \begin{align*}
        C=C\left(\norm{f_{in}}_{L^{\infty}_{xv}}, \norm{f_{in}}_{L^{1}_{xv}},\norms{|v|^3f_{in}}_{L^{1}_{xv}},\norm{u_{in}}^{2}_{L^{2}_{x}},\norm{f_{in}\log f_{in}}_{L^{1}_{xv}},\norms{|x|^2f_{in}}_{L^{1}_{xv}} \right).
    \end{align*}
	Furthermore, applying Tao's method   to the a priori estimates of the Navier-Stokes equations, the following estimate is obtained:
	\ben\label{eq:uniform norms of f1}
		\|u^{\varepsilon,\delta}\|_{L^1(0,t; L^{\kappa+3}(\mathbb{R}^3_x))} + \left\||v|^\kappa f^{\varepsilon,\delta}\right\|_{L^{\infty}(0,t; L^1(\mathbb{R}^3_x \times \mathbb{R}^3_v))} \leq C(t+1)^{4(\kappa+3)}e^{{C}(t+1)^3e^{4t}},
	\een
	where \(\kappa \in [3, +\infty)\) and
    \begin{align*}
C=C\left(\kappa,\norm{f_{in}}_{L^{1}_{xv}},\norm{f_{in}}_{L^{\infty}_{xv}},\norms{|v|^\kappa f_{in}}_{L^{1}_{xv}},\norm{u_{in}}_{L^2_x}\right),
    \end{align*} as a consequence of Lemmas~\ref{lem:ulp} and~\ref{lem:mf}.

{\bf \underline{Verification of a priori norms in $i)$ of Definition \ref{def:weak}.} } Due to Lemma \ref{lem:ul2},  \eqref{eq:uniform norms of f} and weak  lower semi-continuity, there hold
\begin{align}
\label{ul2weak}
    \norm{u}^{2}_{L^{\infty}(0,t;{L^{2}_{x}(\mathbb{R}^{3})})}&\leq\liminf_{\varepsilon,\delta\rightarrow0}\norm{u^{\varepsilon,\delta}}^{2}_{L^{\infty}(0,t;{L^{2}_{x}(\mathbb{R}^{3})})};
\end{align}
% Due to \eqref{eq:uniform norms of f} and weak lower semi-continuity, there hold
\begin{align}
    \label{uh2weak}
    \norm{u}^{2}_{L^{2}(0,t;{\dot{H}^{1}_{x}(\mathbb{R}^{3})})}&\leq\liminf_{\varepsilon,\delta\rightarrow0}\norm{u^{\varepsilon,\delta}}^{2}_{L^{2}(0,t;{\dot{H}^{1}_{x}(\mathbb{R}^{3})})};\\
    \norms{P}_{L^\frac53((0,t)\times \rev{\mathbb{R}^3_x)}}&\leq\liminf_{\varepsilon,\delta\rightarrow0}\norms{P^{\varepsilon,\delta}}_{L^\frac53((0,t)\times \rev{\mathbb{R}^3_x)}}\nonumber
\end{align}
(see also  the arguments by constructing an auxiliary function in the $t$ direction,   Theorem 4.2 in \cite{bedrossianMathematicalAnalysisIncompressible2022}).
Moreover, it follows from Theorem \ref{thm:solution} and Lemma \ref{lem:fcompact} that 
  \(f\geq0\), which along with Fatou's lemma in \(xv\), weak* lower semi-continuity in \(t\) and Lemma \ref{lem:flogfcompact} implies that
  \begin{align}
      \norms{f}_{L^{\infty}\left(0,t;L^1\left(\mathbb{R}^3_v\times\mathbb{R}^3_x\right)\right)}&\leq\liminf_{\varepsilon,\delta\rightarrow0}\norms{f^{\varepsilon,\delta}}_{L^{\infty}\left(0,t;L^1\left(\mathbb{R}^3_v\times\mathbb{R}^3_x\right)\right)};\nonumber\\
      \label{eq:flog weaksemi}
      \norms{f\log f}_{L^{\infty}\left(0,t;L^1\left(\mathbb{R}^3_v\times\mathbb{R}^3_x\right)\right)}&\leq\liminf_{\varepsilon,\delta\rightarrow0}\norms{f^{\varepsilon,\delta}\log f^{\varepsilon,\delta}}_{L^{\infty}\left(0,t;L^1\left(\mathbb{R}^3_v\times\mathbb{R}^3_x\right)\right)}.
  \end{align}
% \ben\label{eq:flog weaksemi}
% \norms{f\log f}_{L^{\infty}\left(0,t;L^1\left(\mathbb{R}^3_v\times\mathbb{R}^3_x\right)\right)}&\leq\liminf_{\varepsilon,\delta\rightarrow0}\norms{f^{\varepsilon,\delta}\log f^{\varepsilon,\delta}}_{L^{\infty}\left(0,t;L^1\left(\mathbb{R}^3_v\times\mathbb{R}^3_x\right)\right)}.
% \een
% By Theorem 4.3 in \cite{giustiDirectMethodsCalculus2003}, it follows that
% \begin{align*}
%     \norms{f\log f}_{L^{\infty}\left(0,t;L^1\left(\mathbb{R}^3_v\times\mathbb{R}^3_x\right)\right)}&\leq\liminf_{\varepsilon,\delta\rightarrow0}\norms{f^{\varepsilon,\delta}\log f^{\varepsilon,\delta}}_{L^{\infty}\left(0,t;L^1\left(\mathbb{R}^3_v\times\mathbb{R}^3_x\right)\right)}.
% \end{align*}
Due to \eqref{17} and \eqref{eq:uniform norms of f}, also by Fatou's lemma in \(xv\) and weak* lower semi-continuity in \(t\), we have
\begin{align}
\label{xvfweak}
\norms{\left(|x|^2+|v|^2\right)f}_{L^{\infty}\left(0,t;L^1\left(\mathbb{R}^3_v\times\mathbb{R}^3_x\right)\right)}\leq\liminf_{\varepsilon,\delta\rightarrow0}\norms{\left(|x|^2+|v|^2\right)f^{\varepsilon,\delta}}_{L^{\infty}\left(0,t;L^1\left(\mathbb{R}^3_v\times\mathbb{R}^3_x\right)\right)}.
\end{align}
 For \eqref{flp}, it follows that
\begin{align*}
    	\norm{f}_{L^{\infty}_tL^{p}_{xv}((0,T)\times\mathbb{R}^{6})} \leq\liminf_{\varepsilon,\delta\rightarrow0}\norm{f^{\varepsilon,\delta}}_{{L^{\infty}_tL^{p}_{xv}((0,T)\times\mathbb{R}^{6})}} \leq e^{\frac{3(p-1)}{p}T}\norm{f_{in}}_{L^{p}_{xv}(\mathbb{R}^{6})}
    \end{align*}
    for arbitrary \(p \in [1,+\infty)\), which yields
    \begin{align*}
        \norm{f}_{L^{\infty}((0,T)\times\mathbb{R}^{6}_{xv})} \leq e^{3T}\norm{f_{in}}_{L^{\infty}_{xv}(\mathbb{R}^{6})}.
    \end{align*}
    % The behavior of \(\delta\rightarrow0\) is the same as that of \(\varepsilon\rightarrow0\). We omit it.

{\bf \underline{Verification of a priori norms in $ii)$ of Definition \ref{def:weak}.}} From Lemma \ref{lem:ulp}, Lemma \ref{lem:mf}, Lemma \ref{lem:ul2} and Lemma \ref{lem:vflogf compact}, using Fatou's lemma, we have 
\begin{align*}
    \norm{u}_{L^1(0,t; L^{p}(\mathbb{R}^3_x))}&\leq\liminf_{\varepsilon,\delta\rightarrow0}\norm{u^{\varepsilon,\delta}}_{L^1(0,t; L^{p}(\mathbb{R}^3_x))};\\
    \norms{|v|^{\kappa}f}_{L^{\infty}(0,t;L^1(\mathbb{R}^3_x\times\mathbb{R}^3_v))}&\leq\liminf_{\varepsilon,\delta\rightarrow0}\norms{|v|^{\kappa}f^{\varepsilon,\delta}}_{L^{\infty}(0,t;L^1(\mathbb{R}^3_x\times\mathbb{R}^3_v))};\\
    \left\||v|^{\frac{\bar{\kappa}}{2}} f \log f \right\|_{L^1((0,T)\times\Omega_{x}\times\mathbb{R}^3_v)}&\leq\liminf_{\varepsilon,\delta\rightarrow0}\left\||v|^{\frac{\bar{\kappa}}{2}} f^{\varepsilon,\delta} \log f^{\varepsilon,\delta} \right\|_{L^1((0,T)\times\Omega_{x}\times\mathbb{R}^3_v)},
\end{align*}
where the second inequality uses weak* lower semi-continuity in \(t\).
% As \(\delta \rightarrow0\), the following inequalities hold: 
% \begin{align*}
%     \norm{u}_{L^1(0,t; L^{\kappa}(\mathbb{R}^3_x))}&\leq\liminf_{\delta\rightarrow0}\norm{u}_{L^1(0,t; L^{\kappa}(\mathbb{R}^3_x))};\\
%     \norms{|v|^{\kappa}f}_{L^{\infty}(0,t;L^1(\mathbb{R}^3_x\times\mathbb{R}^3_v))}&\leq\liminf_{\delta\rightarrow0}\norms{|v|^{\kappa}f}_{L^{\infty}(0,t;L^1(\mathbb{R}^3_x\times\mathbb{R}^3_v))};\\
%     \left\||v|^{\frac{\bar{\kappa}}{2}} f \log f \right\|_{L^1((0,T)\times\Omega_{x}\times\mathbb{R}^3_v)}&\leq\liminf_{\delta\rightarrow0}\left\||v|^{\frac{\bar{\kappa}}{2}} f \log f \right\|_{L^1((0,T)\times\Omega_{x}\times\mathbb{R}^3_v)}.
% \end{align*}
Based on the above inequalities, we derive
\begin{align*}
    \norms{|v|^{\bar{\kappa}}f\log f}&_{L^{\infty}\left(0,t;L^{\eta}(\mathbb{R}_x^3\times\mathbb{R}_v^3)\right)}\\
    &\leq\norms{|v|^{\bar{\kappa}}f\log f\mathbf{1}_{\{f>1\}}}_{L^{\infty}\left(0,t;L^{\eta}(\mathbb{R}_x^3\times\mathbb{R}_v^3)\right)}+\norms{|v|^{\bar{\kappa}}f\log f\mathbf{1}_{\{f\leq1\}}}_{L^{\infty}\left(0,t;L^{\eta}(\mathbb{R}_x^3\times\mathbb{R}_v^3)\right)}\\
    &\leq C_\eta\left(\norms{|v|^{\bar{\kappa}}f^2\mathbf{1}_{\{f>1\}}}_{L^{\infty}\left(0,t;L^{\eta}(\mathbb{R}_x^3\times\mathbb{R}_v^3)\right)}+\norms{|v|^{\bar{\kappa}}f^{1/\eta}\mathbf{1}_{\{f\leq1\}}}_{L^{\infty}\left(0,t;L^{\eta}(\mathbb{R}_x^3\times\mathbb{R}_v^3)\right)}\right)\\
    &\leq C_\eta\operatorname*{ess\,sup}_{0<s<t}\left(\left[\int_{\mathbb{R}^6_{xv}}(1+|v|^\kappa)f^{2\eta}(s,x,v)\,dx\,dv\right]^{\frac{1}{\eta}}+\left[\int_{\mathbb{R}^6_{xv}}(1+|v|^\kappa)f(s,x,v)\,dx\,dv\right]^{\frac{1}{\eta}}\right).
\end{align*}
Here \(\eta>1\) and \(\eta\bar{\kappa}\leq\kappa\). Due to the bounds of \(\norms{(1+|v|^{\kappa})f}_{L^{\infty}(0,t;L^1(\mathbb{R}^3_x\times\mathbb{R}^3_v))}\) and \(\norm{f}_{L^{\infty}((0,t)\times\mathbb{R}^{6}_{xv})}\), we get
\begin{align*}
    &\norms{|v|^{\bar{\kappa}}f\log f}_{L^{\infty}\left(0,t;L^{\eta}(\mathbb{R}_x^3\times\mathbb{R}_v^3)\right)}\\
    &\leq C\left((t+1)^{4(\kappa+3)}\exp[C(t+1)^3e^{4t}]\right)^2+C\left((t+1)^{4(\kappa+3)}\exp[C(t+1)^3e^{4t}]\right)^{1/\eta}\\
    &\leq C(t+1)^{8(\kappa+3)}\exp[C(t+1)^3e^{4t}].
\end{align*}
Here \(C\) depends on \(\eta,\kappa\) and the corresponding initial data norms.

 \par {\bf \underline{Verification of weak solutions in $iii)$ of Definition \ref{def:weak}.}}$ $

 For \((\theta_\varepsilon \ast u^{\varepsilon,\delta}) n^{\varepsilon,\delta}\), by Corollary \ref{jthetau} we get
 \begin{align*}
     \int_{(0,t) \times \mathbb{R}^3_x}n^{\varepsilon,\delta}(\theta_{\varepsilon}\ast u^{\varepsilon,\delta})\cdot\Psi dxds \rightarrow \int_{(0,t) \times \mathbb{R}^3_x}nu\cdot\Psi dxds
 \end{align*}
 as \(\varepsilon,\delta \rightarrow 0.\)
 For \((\theta_\varepsilon \ast u^{\varepsilon,\delta}) \otimes u^{\varepsilon,\delta}\), by Lemma \ref{lem:ul2} we get
 \begin{align*}
     \int_{(0,t) \times \mathbb{R}^3_x}u^{\varepsilon,\delta}\otimes(\theta_{\varepsilon}\ast u^{\varepsilon,\delta}):\nabla\Psi dxds \rightarrow \int_{(0,t) \times \mathbb{R}^3_x}u\otimes u:\nabla\Psi dxds
 \end{align*}
 as \(\varepsilon,\delta \rightarrow 0.\)
 Combining \eqref{eq:uniform norms of f},  the functions \( u \) and \( f \) solve the Navier-Stokes-Vlasov-Fokker-Planck equations \eqref{eq:NSVFK} in the sense of distributions:

\begin{align*}
% \sup_{s\in [0,T]}\int_{\mathbb{R}^3} (u \cdot \Psi)(s, x)  dx 
% + 
\int_0^T \int_{\mathbb{R}^3} \big( \nabla u : \nabla \Psi - u \otimes u : \nabla \Psi - u \cdot \partial_t \Psi \big)(s, x)  dx  ds
\end{align*}
\begin{align*}
= \int_{\mathbb{R}^3} u_{in}(x) \Psi(0, x)  dx 
+ \int_0^T \int_{\mathbb{R}^3} \int_{\mathbb{R}^3} (v - u) f  dv  \Psi(s, x)  dx  ds,
\end{align*}
with
\(
\nabla u : \nabla \Psi = \sum_{j,k=1}^d \partial_j u^k \partial_j \Psi^k
~ \text{and} ~
u \otimes u : \nabla \Psi = \sum_{j,k=1}^d u^j u^k \partial_j \Psi^k\).

\begin{align*}
\int_{0}^{T} \int_{\mathbb{R}^3 \times \mathbb{R}^3} f(\partial_t \Phi + v \cdot \nabla_x \Phi + (u - v) \cdot \nabla_v \Phi + \Delta_v \Phi)  dx  dv  dt 
= -\int_{\mathbb{R}^3 \times \mathbb{R}^3} f_{in} \Phi(0, x, v)  dx  dv,
\end{align*}
for any \(\Psi \in C^1_0([0,T); (C_c^\infty (\mathbb{R}^3_x))^3)\) with \(\nabla \cdot \Psi = 0\) and 
\(\Phi \in C^1_0([0,T); C_c^\infty (\mathbb{R}^3_x \times \mathbb{R}^3_v))\).

    \par {\bf \underline{Verification of global energy inequality in $iv)$ of Definition \ref{def:weak}.}}
    % MINREV ORIGINAL BEGIN: signed terminal entropy
% Next, we prove the convergence of global energy inequality. For the left-hand side of \eqref{eq:13}, by Fatou's lemma in \(xv\) and weak* lower semi-continuity in \(t\), it shows that
% 	  %    \label{eq: u estimates}
%    %          ? \int_{\mathbb{R}^3_x} |u|^2 dx 
% 	 	% \leq& \liminf_{\varepsilon,\delta \rightarrow 0} \int_{\mathbb{R}^3_x} |u^{\varepsilon,\delta}|^2 dx;\\
%    %      \int_{(0,t) \times \mathbb{R}^3_x} |\nabla u|^2 dxds 
% 	 	% \leq& \liminf_{\varepsilon,\delta \rightarrow 0} \int_{(0,t) \times \mathbb{R}^3_x} |\nabla u^{\varepsilon,\delta}|^2 dxds;\\
% 	 	% \int_{\mathbb{R}^6_{xv}} \frac{1}{2} |v|^2 f dxdv
% 	 	% \leq& \liminf_{\varepsilon,\delta \rightarrow 0} \int_{\mathbb{R}^6_{xv}} \frac{1}{2} |v|^2 f^{\varepsilon,\delta} dxdv;  \\
%          \begin{align*}
% 	 	\sup_{s\in [0,t]}\int_{\mathbb{R}^3_x \times \mathbb{R}^3_v} f \log f dxdv
% 	 	\leq& \liminf_{\varepsilon,\delta \rightarrow 0} \sup_{s\in [0,t]}\int_{\mathbb{R}^3_x \times \mathbb{R}^3_v} f^{\varepsilon,\delta} \log f^{\varepsilon,\delta} dxdv,
%          \end{align*}
%      where it is similar with \eqref{eq:flog weaksemi}. Others are the same as in \eqref{ul2weak}, \eqref{uh2weak} and \eqref{xvfweak}.
% 
%      
% MINREV ORIGINAL END
In this section, we prove
     	\ben\label{eq:global inequality}
		&&\frac{1}{2}\int_{\mathbb{R}^3_x}|u(t,\cdot)|^2dx+\int_{\mathbb{R}^6_{xv}}\left(\frac{1}{2}|v|^2f+f\log f\right)(t,\cdot)dxdv+\int_{(0,t)\times \mathbb{R}_x^3}|\nabla u|^2dxds\nonumber\\
		&&+\int_{(0,t)\times \mathbb{R}_{xv}^6}\frac{|(u-v)f-\nabla_{v}f|^2}{f}dxdvds\\
		&&\hspace{4cm}\leq \int_{\mathbb{R}^6_{xv}}\left(\frac{1}{2}|v|^2f_{in}+f_{in}\log f_{in}\right)dxdv+\frac{1}{2}\int_{\mathbb{R}^3_x}|u_{in}|^2dx.\nonumber
	\een
\begin{revision}
By Lemmas~\ref{lem:ul2} and \ref{lem:vfcompact} together with
\eqref{min:globalentropy}, after passing to a common
subsequence, for almost every $t\in(0,T)$,
\begin{align*}
 &u^{\varepsilon,\delta}(t)\longrightarrow u(t)
 ~\text{in }L^2(B_R),\\
 &\int_{\mathbb R^3_x\times\mathbb R^3_v}
 |v|^2|f^{\varepsilon,\delta}(t,x,v)-f(t,x,v)|
 \,dxdv\longrightarrow0,\\
 &\int_{\mathbb R^3_x\times\mathbb R^3_v}
 \left|
 f^{\varepsilon,\delta}\log f^{\varepsilon,\delta}
 -f\log f
 \right|(t,x,v)\,dxdv\longrightarrow0.
\end{align*}
Together with the uniform
$L^\infty(0,T;L^2(\mathbb R^3_x))$ bound, the first
convergence implies
\begin{align*}
 u^{\varepsilon,\delta}(t)
 &\rightharpoonup u(t)
 \quad\text{in }L^2(\mathbb R^3_x),\\
 \frac12\int_{\mathbb R^3_x}|u(t,x)|^2\,dx
 &\leq
 \liminf_{\varepsilon,\delta\to0}
 \frac12\int_{\mathbb R^3_x}
 |u^{\varepsilon,\delta}(t,x)|^2\,dx.
\end{align*}
Consequently, for almost every $t\in(0,T)$,
\begin{align*}
 &\frac12\int_{\mathbb R^3_x}|u(t,x)|^2\,dx
 +\int_{\mathbb R^3_x\times\mathbb R^3_v}
 \left(\frac12|v|^2f+f\log f\right)(t,x,v)\,dxdv\\
 &\quad\leq
 \liminf_{\varepsilon,\delta\to0}
 \left[
 \frac12\int_{\mathbb R^3_x}
 |u^{\varepsilon,\delta}(t,x)|^2\,dx
 +\int_{\mathbb R^3_x\times\mathbb R^3_v}
 \left(
 \frac12|v|^2f^{\varepsilon,\delta}
 +f^{\varepsilon,\delta}\log f^{\varepsilon,\delta}
 \right)(t,x,v)\,dxdv
 \right].
\end{align*}

By \eqref{udelta}, \eqref{fdelta}, and the convention that 
$0\log0:=0$, we have
\begin{align*}
 u_{in}^{\delta}\longrightarrow u_{in}
 \quad\text{in }L^2(\mathbb R^3_x),\quad
 f_{in}^{\delta}\log f_{in}^{\delta}
 =\mathbf{1}_{\{|v|\leq\delta^{-1}\}}f_{in}\log f_{in}.
\end{align*}
Dominated convergence therefore gives
\begin{align*}
 &\lim_{\delta\to0}
 \left[
 \frac12\int_{\mathbb R^3_x}|u_{in}^{\delta}|^2dx
 +\int_{\mathbb R^3_x\times\mathbb R^3_v}
 \left(
 \frac12|v|^2f_{in}^{\delta}
 +f_{in}^{\delta}\log f_{in}^{\delta}
 \right)dxdv
 \right]\\
 &\quad=
 \frac12\int_{\mathbb R^3_x}|u_{in}|^2dx
 +\int_{\mathbb R^3_x\times\mathbb R^3_v}
 \left(\frac12|v|^2f_{in}+f_{in}\log f_{in}\right)dxdv.
\end{align*}
The dissipation terms are treated below.
\end{revision}

% {\bf \underline{Verification of a priori norms in $i)$ of Definition \ref{def:weak}.} }

% {\bf Verification of a priori norms in $ii)$ of Definition \ref{def:weak}.}

% {\bf Verification of $iii)$ of Definition \ref{def:weak}.}

	For \(\int_{(0,t)\times \mathbb{R}_{xv}^6}\frac{|(\theta_{\varepsilon}\ast u^{\varepsilon,\delta}-v)f^{\varepsilon,\delta}-\nabla_{v}f^{\varepsilon,\delta}|^2}{f^{\varepsilon,\delta}}dxdvds\), it follows that
	\begin{align*}
		&\int_{(0,t)\times \mathbb{R}_{xv}^6}\frac{|(\theta_{\varepsilon}\ast u^{\varepsilon,\delta}-v)f^{\varepsilon,\delta}-\nabla_{v}f^{\varepsilon,\delta}|^2}{f^{\varepsilon,\delta}}dxdvds\\
		&=\int_{(0,t)\times \mathbb{R}_{xv}^6}|\theta_{\varepsilon}\ast u^{\varepsilon,\delta}-v|^2f^{\varepsilon,\delta}dxdvds+4\int_{(0,t)\times \mathbb{R}_{xv}^6}\left|\nabla_{v}\sqrt{f^{\varepsilon,\delta}}\right|^2dxdvds\\
		&\hspace{6cm}-2\int_{(0,t)\times \mathbb{R}_{xv}^6}(\theta_{\varepsilon}\ast u^{\varepsilon,\delta}-v)\cdot\nabla_{v}f^{\varepsilon,\delta}dxdvds\\
		&=\int_{(0,t)\times \mathbb{R}_{xv}^6}|\theta_{\varepsilon}\ast u^{\varepsilon,\delta}-v|^2f^{\varepsilon,\delta}dxdvds+4\int_{(0,t)\times \mathbb{R}_{xv}^6}\left|\nabla_{v}\sqrt{f^{\varepsilon,\delta}}\right|^2dxdvds-6\int_{(0,t)\times \mathbb{R}_{xv}^6}f^{\varepsilon,\delta}dxdvds\\
		&=\int_{(0,t)\times \mathbb{R}_{xv}^6}|\theta_{\varepsilon}\ast u^{\varepsilon,\delta}|^2f^{\varepsilon,\delta}dxdvds-2\int_{(0,t)\times \mathbb{R}_{xv}^6} v\cdot (\theta_{\varepsilon}\ast u^{\varepsilon,\delta})f^{\varepsilon,\delta}dxdvds+\int_{(0,t)\times \mathbb{R}_{xv}^6}|v|^2f^{\varepsilon,\delta}dxdvds\\
		&\hspace{5cm}+4\int_{(0,t)\times \mathbb{R}_{xv}^6}\left|\nabla_{v}\sqrt{f^{\varepsilon,\delta}}\right|^2dxdvds-6\int_{(0,t)\times \mathbb{R}_{xv}^6}f^{\varepsilon,\delta}dxdvds\\
		&:= i_1+i_2+i_3+i_4+i_5.
	\end{align*}
	For \(i_1\), we derive
    \begin{align*}
        &\int_{(0,t)\times \mathbb{R}_{xv}^6}|u|^2f-|\theta_{\varepsilon}\ast u^{\varepsilon,\delta}|^2f^{\varepsilon,\delta}dxdvds\\
        &=\int_{(0,t)\times \mathbb{R}_{xv}^6}|u|^2f-|u|^2f^{\varepsilon,\delta}dxdvds+\int_{(0,t)\times \mathbb{R}_{xv}^6}|u|^2f^{\varepsilon,\delta}-|\theta_\varepsilon \ast u^{\varepsilon,\delta}|^2f^{\varepsilon,\delta}dxdvds\\
        &:=i_{11}+i_{12}.
    \end{align*}
    By \eqref{8}, \eqref{utp26} and \eqref{eq:uniform norms of f}, \(i_{11}\) becomes
    \begin{align*}
        |i_{11}|&\leq C\left(1+\frac43 \pi\norms{f-f^{\varepsilon,\delta}}_{L^{\infty}_{txv}}\right)\int_{(0,t)\times\mathbb{R}^3_x}|u|^2\left(\int_{\mathbb{R}^3_v}|v|^2|f-f^{\varepsilon,\delta}|dv\right)^{\frac35}dxds\\
        &\leq C\left(1+\frac43 \pi\norms{f-f^{\varepsilon,\delta}}_{L^{\infty}_{txv}}\right)\int^t_0\norms{u}_{L^5_{x}}^2\norms{|v|^2(f-f^{\varepsilon,\delta})}_{L^1_{xv}}^\frac35ds\\
        &\leq C\left(1+\frac43 \pi\norms{f-f^{\varepsilon,\delta}}_{L^{\infty}_{txv}}\right)\norms{u}^2_{L^\frac{20}{9}_tL^5_{x}}\norms{|v|^2(f-f^{\varepsilon,\delta})}_{L^6_t L^1_{xv}}^\frac{3}{5}\\
        &\leq C\left(1+\frac43 \pi\norms{f-f^{\varepsilon,\delta}}_{L^{\infty}_{txv}}\right)\left(\norms{u}_{L^\infty_tL^2_{x}}+\norms{u}_{L^2_t\dot{H}^1_{x}}\right)^2\norms{|v|^2(f-f^{\varepsilon,\delta})}_{L^6_tL^1_{xv}}^\frac{3}{5}\\
        &\leq C\norms{|v|^2(f-f^{\varepsilon,\delta})}_{L^6_tL^1_{xv}}^\frac{3}{5}.
        % &\leq Ct^{(\frac14-\frac{1}{2p})}\norms{|v|^3(f-f^{\varepsilon,\delta})}_{L^p_tL^1_{xv}}^\frac12.
    \end{align*}
     Due to Corollary \ref{lem:tpcompact}, we have 
    \begin{align*}i_{11}\rightarrow 0\end{align*} as \(\varepsilon,\delta\rightarrow 0.\) For \(i_{12}\), we derive
    \begin{align*}
        |i_{12}|&\leq\int_{(0,t)\times \mathbb{R}_{xv}^6}\left|(u-\theta_\varepsilon \ast u^{\varepsilon,\delta})\cdot (u+\theta_\varepsilon \ast u^{\varepsilon,\delta})f^{\varepsilon,\delta}\right|dxdvds\\
        &\leq \int_{(0,t)\times \mathbb{R}_{v}^3\times\{|x|\leq R\}}\left|(u-\theta_\varepsilon \ast u^{\varepsilon,\delta})\cdot (u+\theta_\varepsilon \ast u^{\varepsilon,\delta})f^{\varepsilon,\delta}\right|dxdvds\\
        &\hspace{0.5cm}+\int_{(0,t)\times \mathbb{R}_{v}^3\times\{|x|\geq R\}}\left|(u-\theta_\varepsilon \ast u^{\varepsilon,\delta})\cdot (u+\theta_\varepsilon \ast u^{\varepsilon,\delta})f^{\varepsilon,\delta}\right|dxdvds\\
        &:=i_{12-1}+i_{12-2}.
    \end{align*}
    By \eqref{8}, \eqref{utp26} and \eqref{eq:uniform norms of f}, it becomes
    \begin{align*}
          i_{12-1}&\leq C\left(1+\frac43 \pi\norms{f^{\varepsilon,\delta}}_{L^{\infty}_{txv}}\right)\int_0^t\int_{|x|\leq R} \left|(u-\theta_\varepsilon \ast u^{\varepsilon,\delta})\cdot (u+\theta_\varepsilon \ast u^{\varepsilon,\delta})\right|\left(\int_{\mathbb{R}^3_v}|v|^2f^{\varepsilon,\delta}dv\right)^{\frac35}dxds\\
          &\leq Ct^{\frac{1}{20}}\left(1+\frac43 \pi\norms{f^{\varepsilon,\delta}}_{L^{\infty}_{txv}}\right)\norms{|v|^2f^{\varepsilon,\delta}}_{L^\infty_t L^1_{xv}}^{\frac35}\norms{u+\theta_\varepsilon \ast u^{\varepsilon,\delta}}_{L^\frac{20}{9}_tL^5_x}\norms{u-\theta_\varepsilon \ast u^{\varepsilon,\delta}}_{L^2(0,t;L^5(|x|\leq R))}\\
          &\leq C(t)\norms{u-\theta_\varepsilon \ast u^{\varepsilon,\delta}}_{L^2(0,t;L^5(|x|\leq R))},
    \end{align*}
and
    \begin{align*}
        i_{12-2}&\leq C\left(1+\frac43 \pi\norms{f^{\varepsilon,\delta}}_{L^{\infty}_{txv}}\right)\int_0^t\int_{|x|\geq R} \left|(u-\theta_\varepsilon \ast u^{\varepsilon,\delta})\cdot (u+\theta_\varepsilon \ast u^{\varepsilon,\delta})\right|\left(\int_{\mathbb{R}^3_v}|v|^2f^{\varepsilon,\delta}dv\right)^{\frac35}dxds\\
        &\leq Ct^\frac{1}{10}\left(1+\frac43 \pi\norms{f^{\varepsilon,\delta}}_{L^{\infty}_{txv}}\right)\norms{u+\theta_\varepsilon \ast u^{\varepsilon,\delta}}_{L^\frac{20}{9}_tL^5_x}\norms{u-\theta_\varepsilon \ast u^{\varepsilon,\delta}}_{L^\frac{20}{9}_tL^5_x}\\
        &\hspace{1cm}\times\left(\sup_{s\in [0,T]}\int_{\{|x|\geq R\}\times \mathbb{R}^3_v}|v|^2f^{\varepsilon,\delta}dxdv\right)^\frac35\\
        &\leq C(t)\left(\sup_{s\in [0,T]}\int_{\{|x|\geq R\}\times \mathbb{R}^3_v}|x|^{-\frac23}|x|^{\frac23}|v|^2f^{\varepsilon,\delta}dxdv\right)^\frac35\\
        &\leq C(t)R^{-\frac25}\norms{|x|^2f^{\varepsilon,\delta}}_{L^\infty_tL^1_{xv}}^\frac15 \norms{|v|^3f^{\varepsilon,\delta}}_{L^\infty_tL^1_{xv}}^\frac25.
    \end{align*}
    Due to Lemma \ref{lem:ul2} and the uniform boundedness  property \eqref{eq:uniform norms of f} and \eqref{xvfweak}  with respect to \(\varepsilon\) and \(\delta,\)  it is deduced that
    \begin{align*} i_{12}\rightarrow 0\end{align*}
    as \(\varepsilon,\delta \rightarrow 0\).
    % For \(\int_{\mathbb{R}^3_v}(u+\theta_\varepsilon \ast u^{\varepsilon,\delta})f^{\varepsilon,\delta}dv\), we have
    % \begin{align*}
    %     &\hspace{0.5cm}\norms{(u+\theta_\varepsilon \ast u^{\varepsilon,\delta})\int_{\mathbb{R}^3_v}f^{\varepsilon,\delta}dv}_{L^2_tL^{\frac{6}{5}}_{x}}\\
    %     &\leq C\left(1+\frac43 \pi\norms{f^{\varepsilon,\delta}}_{L^{\infty}_{txv}}\right)\norms{(u+\theta_\varepsilon \ast u^{\varepsilon,\delta})\left(\int_{\mathbb{R}^3_v}|v|^3f^{\varepsilon,\delta}dv\right)^{\frac12}}_{L^2_tL^{\frac{6}{5}}_{x}}\\
    %     &\leq C\left(1+\frac43 \pi\norms{f^{\varepsilon,\delta}}_{L^{\infty}_{txv}}\right)\norms{u+\theta_\varepsilon \ast u^{\varepsilon,\delta}}_{L^2_tL^3_{x}}\norms{|v|^3f^{\varepsilon,\delta}}_{L^\infty_tL^1_{xv}}^\frac12\\
    %     &\leq Ct^\frac14\left(1+\frac43 \pi\norms{f^{\varepsilon,\delta}}_{L^{\infty}_{txv}}\right)\norms{u+\theta_\varepsilon \ast u^{\varepsilon,\delta}}_{L^4_tL^3_{x}}\norms{|v|^3f^{\varepsilon,\delta}}_{L^\infty_tL^1_{xv}}^\frac{1}{2}\\
    %     &\leq Ct^\frac14\left(1+\frac43 \pi\norms{f^{\varepsilon,\delta}}_{L^{\infty}_{txv}}\right)\left(\norms{u+\theta_\varepsilon \ast u^{\varepsilon,\delta}}_{L^\infty_tL^2_{x}}+\norms{u+\theta_\varepsilon \ast u^{\varepsilon,\delta}}_{L^2_t\dot{H}^1_{x}}\right)\norms{|v|^3f^{\varepsilon,\delta}}_{L^\infty_tL^1_{xv}}^\frac{1}{2}.
    % \end{align*}
    % Due to \eqref{02}, we deduce
    % \begin{align*}i_{12}\rightarrow0\end{align*} as \(\varepsilon,\delta\rightarrow0\).
    To sum up, this leads to
    \begin{align*}
        \lim_{\varepsilon,\delta\rightarrow0}\int_{(0,t)\times \mathbb{R}_{xv}^6}|\theta_{\varepsilon}\ast u^{\varepsilon,\delta}|^2f^{\varepsilon,\delta}dxdvds=\int_{(0,t)\times \mathbb{R}_{xv}^6}|u|^2fdxdvds.
    \end{align*}
	% \begin{align*}
	% 	\int_{(0,t)\times \mathbb{R}_{xv}^6}|u|^2fdxdvds&=\int_{(0,t)\times \mathbb{R}_{xv}^6}\liminf_{\varepsilon,\delta \rightarrow 0}|\theta_{\varepsilon}\ast u^{\varepsilon,\delta}|^2f^{\varepsilon,\delta}dxdvds\\
	% 	&\leq\liminf_{\varepsilon,\delta \rightarrow 0}\int_{(0,t)\times \mathbb{R}_{xv}^6}|\theta_{\varepsilon}\ast u^{\varepsilon,\delta}|^2f^{\varepsilon,\delta}dxdvds.
	% \end{align*}
    The limit result in \(i_2\) as \(\varepsilon,\delta\rightarrow0\) is the same as \(i_1\). In fact,
    \begin{align*}
        &\left|i_2+2\int_{(0,t)\times\mathbb R^6}u\cdot vf\,dxdvds\right|\\
        &\leq2\int_{(0,t)\times \mathbb{R}_{xv}^6}\left|u\cdot vf-(\theta_\varepsilon \ast u^{\varepsilon,\delta})\cdot vf^{\varepsilon,\delta}\right|dxdvds\\
        &\leq2\int_{(0,t)\times \mathbb{R}_{xv}^6}\left|u\cdot v(f-f^{\varepsilon,\delta})\right|dxdvds+\auditmath{2}\int_{(0,t)\times \mathbb{R}_{xv}^6}\left|(u-\theta_{\varepsilon}\ast u^{\varepsilon,\delta})\cdot vf^{\varepsilon,\delta}\right|dxdvds\\
        &:=i_{21}+i_{22}.
    \end{align*}
    For \(i_{21}\), by \eqref{8}, \eqref{utp26} and \eqref{eq:uniform norms of f}, we have
    \begin{align*}
        |i_{21}|&\leq C\left(1+\frac43 \pi\norms{f-f^{\varepsilon,\delta}}_{L^{\infty}_{txv}}\right)\int_{(0,t)\times\mathbb{R}^3_x}|u|\left(\int_{\mathbb{R}^3_v}|v|^\frac95|f-f^{\varepsilon,\delta}|dv\right)^{\frac56}dxds\\
        &\leq C\left(1+\frac43 \pi\norms{f-f^{\varepsilon,\delta}}_{L^{\infty}_{txv}}\right)\int^t_0\norms{u}_{L^6_{x}}\norms{|v|^\frac95(f-f^{\varepsilon,\delta})}_{L^1_{xv}}^\frac56ds\\
        &\leq C\left(1+\frac43 \pi\norms{f-f^{\varepsilon,\delta}}_{L^{\infty}_{txv}}\right)\norms{u}_{L^2_tL^6_{x}}\norms{|v|^\frac95(f-f^{\varepsilon,\delta})}_{L^\frac53_tL^1_{xv}}^\frac56\\
        &\leq C\left(1+\frac43 \pi\norms{f-f^{\varepsilon,\delta}}_{L^{\infty}_{txv}}\right)\left(\norms{u}_{L^\infty_tL^2_{x}}+\norms{u}_{L^2_t\dot{H}^1_{x}}\right)\norms{|v|^\frac95(f-f^{\varepsilon,\delta})}_{L^\frac53_tL^1_{xv}}^\frac56\\
        &\leq C\norms{|v|^\frac95(f-f^{\varepsilon,\delta})}_{L^\frac53_tL^1_{xv}}^\frac56.
    \end{align*}
    Due to Corollary \ref{lem:tpcompact}, it is deduced that 
    \begin{align*}i_{21}\rightarrow 0\end{align*} as \(\varepsilon,\delta\rightarrow 0.\)
    For \(i_{22}\), we derive
    \begin{align*}
        |i_{22}|&\leq 2\int_{(0,t)\times \mathbb{R}_{xv}^6}\left|(u-\theta_\varepsilon \ast u^{\varepsilon,\delta})\cdot vf^{\varepsilon,\delta}\right|dxdvds\\
        &\leq 2\int_{(0,t)\times \mathbb{R}_{v}^3\times\{|x|\leq R\}}\left|(u-\theta_\varepsilon \ast u^{\varepsilon,\delta})\cdot vf^{\varepsilon,\delta}\right|dxdvds\\
        &\hspace{0.5cm}+2\int_{(0,t)\times \mathbb{R}_{v}^3\times\{|x|\geq R\}}\left|(u-\theta_\varepsilon \ast u^{\varepsilon,\delta})\cdot vf^{\varepsilon,\delta}\right|dxdvds\\
        &:=i_{22-1}+i_{22-2}.
    \end{align*}
    By \eqref{8}, \eqref{utp26} and \eqref{eq:uniform norms of f}, it becomes
    \begin{align*}
          i_{22-1}&\leq C\left(1+\frac43 \pi\norms{f^{\varepsilon,\delta}}_{L^{\infty}_{txv}}\right)\int_0^t\int_{|x|\leq R} \left|u-\theta_\varepsilon \ast u^{\varepsilon,\delta}\right|\left(\int_{\mathbb{R}^3_v}|v|^2f^{\varepsilon,\delta}dv\right)^{\frac45}dxds\\
          &\leq Ct^{\frac{1}{2}}\left(1+\frac43 \pi\norms{f^{\varepsilon,\delta}}_{L^{\infty}_{txv}}\right)\norms{|v|^2f^{\varepsilon,\delta}}_{L^\infty_t L^1_{xv}}^{\frac45}\norms{u-\theta_\varepsilon \ast u^{\varepsilon,\delta}}_{L^2(0,t;L^5(|x|\leq R))}\\
          &\leq C(t)\norms{u-\theta_\varepsilon \ast u^{\varepsilon,\delta}}_{L^2(0,t;L^5(|x|\leq R))},
    \end{align*}
and
    \begin{align*}
        i_{22-2}&\leq C\left(1+\frac43 \pi\norms{f^{\varepsilon,\delta}}_{L^{\infty}_{txv}}\right)\int_0^t\int_{|x|\geq R} \left|u-\theta_\varepsilon \ast u^{\varepsilon,\delta}\right|\left(\int_{\mathbb{R}^3_v}|v|^2f^{\varepsilon,\delta}dv\right)^{\frac45}dxds\\
        &\leq Ct^\frac{11}{20}\left(1+\frac43 \pi\norms{f^{\varepsilon,\delta}}_{L^{\infty}_{txv}}\right)\norms{u-\theta_\varepsilon \ast u^{\varepsilon,\delta}}_{L^\frac{20}{9}_tL^5_x}\left(\sup_{s\in [0,T]}\int_{\{|x|\geq R\}\times \mathbb{R}^3_v}|v|^2f^{\varepsilon,\delta}dxdv\right)^\frac45\\
        &\leq C(t)\left(\sup_{s\in [0,T]}\int_{\{|x|\geq R\}\times \mathbb{R}^3_v}|x|^{-\frac23}|x|^{\frac23}|v|^2f^{\varepsilon,\delta}dxdv\right)^\frac45\\
        &\leq C(t)R^{-\frac{8}{15}}\norms{|x|^2f^{\varepsilon,\delta}}_{L^\infty_tL^1_{xv}}^\frac{4}{15}\norms{|v|^3f^{\varepsilon,\delta}}_{L^\infty_tL^1_{xv}}^\frac{8}{15}.
    \end{align*}
    Due to Lemma \ref{lem:ul2} and the uniform smallness property of \(i_{22-2}\) at infinity with respect to \(\varepsilon\) and \(\delta,\) it is deduced that
    \begin{align*} i_{22}\rightarrow 0\end{align*}
    as \(\varepsilon,\delta \rightarrow 0\).
    % \begin{align*}
    %     i_{22}=\int_{(0,t)\times \mathbb{R}_{xv}^6}(u-\theta_\varepsilon \ast u^{\varepsilon,\delta})\cdot vf^{\varepsilon,\delta}dxdvds.
    % \end{align*}
    % For \(\int_{\mathbb{R}^3_v}vf^{\varepsilon,\delta}dv\), we have
    % \begin{align*}
    %     &\hspace{0.5cm}\norms{\int_{\mathbb{R}^3_v}vf^{\varepsilon,\delta}dv}_{L^{\frac{3}{2}}_{tx}}\\
    %     &\leq \left(\int^t_0\left(\|f^{\varepsilon,\delta}\|_{L^\infty_{xv}}^\frac13\||v|^3f\|^\frac23_{L^1_{xv}}\right)^\frac32ds\right)^\frac23\\
    %     &\leq Ct^\frac23\norms{f^{\varepsilon,\delta}}_{L^{\infty}_{txv}}^\frac13\norms{|v|^3f^{\varepsilon,\delta}}_{L^\infty_tL^{1}_{xv}}^{\frac{2}{3}}.
    % \end{align*}
    % Due to \eqref{28} and \eqref{upweakly}, we deduce
    % \begin{align*}i_{22}\rightarrow0\end{align*} as \(\varepsilon,\delta\rightarrow0\).
    To sum up, this leads to
    \begin{align*}
        \lim_{\varepsilon,\delta\rightarrow0}\int_{(0,t)\times\mathbb{R}^6_{xv}}v\cdot (\theta_\varepsilon \ast u^{\varepsilon,\delta})f^{\varepsilon,\delta}dxdvds= \int_{(0,t)\times\mathbb{R}^6_{xv}}v\cdot ufdxdvds.
    \end{align*}
    For \(i_3\), it follows from Fatou's lemma that
	\begin{align*}
		\int_{(0,t)\times \mathbb{R}_{xv}^6} |v|^2 f dxdvds 
		&\leq \liminf_{\varepsilon,\delta \rightarrow 0} \int_{(0,t)\times \mathbb{R}_{xv}^6} |v|^2 f^{\varepsilon,\delta} dxdvds.
	\end{align*}
	For \(i_4\), the strong \(L^1\) convergence of \(f^{\varepsilon,\delta}\) implies the strong \(L^2\) convergence of \(\sqrt{f^{\varepsilon,\delta}}\) to \(\sqrt f\). Together with the uniform \(L^2\) bound on \(\nabla_v\sqrt{f^{\varepsilon,\delta}}\), the weak lower semicontinuity of the \(L^2\) norm gives
	\begin{align*}
		4 \int_{(0,t)\times \mathbb{R}_{xv}^6} \left|\nabla_{v} \sqrt{f}\,\right|^2 dxdvds 
		&\leq \liminf_{\varepsilon,\delta \rightarrow 0} 4 \int_{(0,t)\times \mathbb{R}_{xv}^6} \left|\nabla_{v} \sqrt{f^{\varepsilon,\delta}}\,\right|^2 dxdvds.
	\end{align*}
    For \(i_5\), due to the strong convergence of \(f^{\varepsilon,\delta}\) in \(L^1((0,t)\times\mathbb{R}^3_x\times\mathbb{R}^3_v)\), we have
    \begin{align*}
       - 6 \int_{(0,t)\times \mathbb{R}_{xv}^6} f^{\varepsilon,\delta} dxdvds 
		\rightarrow - 6 \int_{(0,t)\times \mathbb{R}_{xv}^6} f dxdvds,
    \end{align*}
    as \(\varepsilon,\delta\rightarrow0\).
	Consequently, the following inequality holds:
	\begin{align*}
		&\int_{(0,t)\times \mathbb{R}_{xv}^6} 
		\frac{\left| (u - v) f - \nabla_{v} f \right|^2}{f} dxdvds \\
		&\hspace{4cm}\leq \liminf_{\varepsilon,\delta \rightarrow 0} \int_{(0,t)\times \mathbb{R}_{xv}^6} 
		\frac{\left| (\theta_{\varepsilon} \ast u^{\varepsilon,\delta} - v) f^{\varepsilon,\delta} - \nabla_{v} f^{\varepsilon,\delta} \right|^2}{f^{\varepsilon,\delta}} dxdvds.
	\end{align*}
	
	For the right-hand side of \eqref{eq:13}, by \eqref{10}, Lemma \ref{lem:vfcompact} and the dominated convergence theorem we have
 %    it follows that
	% \begin{align*}
	% 	\int_{(0,t)\times \mathbb{R}_{xv}^6}|v|^2f^{\varepsilon,\delta}(1-\gamma_{\delta}(v))dxdvds &\leq C\sup_{t\in(0,T)}\int_{ \mathbb{R}_{xv}^6}|v|^2f^{\varepsilon,\delta}(1-\gamma_{\delta}(v))dxdv\\
	% 	&\leq C\norms{|v|^3f^{\varepsilon,\delta}}_{L^{1}((0,T)\times\mathbb{R}^3_x\times\mathbb{R}^3_v)}\delta.
	% \end{align*}
	% Due to \eqref{10}, we have the following convergence holds:
	\begin{align}\label{convergencevf}
		&\int_{(0,t)\times \mathbb{R}_{xv}^6}|v|^2f^{\varepsilon,\delta}(1-\gamma_{\delta}(v))dxdvds \nonumber\\
        &\leq\int_{(0,t)\times \mathbb{R}_{xv}^6}|v|^2|f^{\varepsilon,\delta}-f|(1-\gamma_{\delta}(v))dxdvds+\int_{(0,t)\times \mathbb{R}_{xv}^6}|v|^2f(1-\gamma_{\delta}(v))dxdvds\\
        &\rightarrow 0~\text{as}~\varepsilon,\delta\rightarrow 0.\nonumber
	\end{align}
	To sum up, when \(\varepsilon,\delta \rightarrow 0\), the global energy inequality \eqref{eq:global inequality} is obtained.
	% \begin{align}
	% 	\label{7}
	% 	&\frac{1}{2} \int_{\mathbb{R}^3_x} |u|^2 dx 
	% 	+ \int_{\mathbb{R}^6_{xv}} \left( \frac{1}{2} |v|^2 f + f \log f \right) dxdv
	% 	+ \int_{(0,t)\times \mathbb{R}^3_x} |\nabla u|^2 dxds \nonumber \\
	% 	&\quad + \int_{(0,t)\times \mathbb{R}^6_{xv}} \frac{\left| (u - v) f - \nabla_{v} f \right|^2}{f} dxdvds \nonumber \\
	% 	&\hspace{2cm} \leq C \delta 
	% 	+ \int_{\mathbb{R}^6_{xv}} \left( \frac{1}{2} |v|^2 f_{\mathrm{in}} + f_{\mathrm{in}} \log f_{\mathrm{in}} \right) dxdv
	% 	+ \frac{1}{2} \int_{\mathbb{R}^3_x} (u_{\mathrm{in}})^2 dx.
	% \end{align}
	% \par Then, we will show the global energy inequality that is satisfied for \(\delta\rightarrow 0\).
	% \par The left-hand side of \eqref{7} can be obtained by employing the same arguments used in the derivation of the left-hand side of (iii). By the property of the initial conditions, when \(\delta \rightarrow 0\), we get the global energy inequality:
	
\section{Verification of \texorpdfstring{$v)$}{v)} in Definition \ref{def:weak}: the convergence of local energy inequality of the first type}
	 \par {\bf \underline{Verification of local energy inequality in $v)$ of Definition \ref{def:weak}.}} In the following, we will verify the local energy inequality in $v)$ of Definition \ref{def:weak}. For the left-hand side of \eqref{eq:14}, the proof follows from the same procedure as in Section 6, and is therefore omitted.
 For the right-hand side of \eqref{eq:14}, the weak convergence of \(P^{\varepsilon,\delta}\) is first required. By the equation of \(u^{\varepsilon,\delta}\), we have
	\begin{align*}
		-\Delta(P^{\varepsilon,\delta}-\bar{P}^{\varepsilon,\delta})=\partial_{j}(\theta_{\varepsilon}\ast u^{\varepsilon,\delta}_i)\partial_{i}u^{\varepsilon,\delta}_j\rev{-}\nabla\cdot\left[\theta_\varepsilon\ast\int_{\mathbb{R}^3_v}(v-\theta_{\varepsilon}\ast u^{\varepsilon,\delta})f^{\varepsilon,\delta}dv\right].
	\end{align*}
Using the same methods as Lemma \ref{le:p}, we have  
\begin{align*}
	\int_{(0,t)\times \Omega_x} \left| P^{\varepsilon,\delta} - \bar{P}^{\varepsilon,\delta} \right|^\frac53 dxdt \leq C,
\end{align*}
which implies
\begin{align}\label{pweak}
	P^{\varepsilon,\delta} - \bar{P}^{\varepsilon,\delta} \rightharpoonup P - \bar{P} \quad \text{in}~ L^\frac32((0,t)\times \Omega_x) ~\text{as}~ \varepsilon,\delta\rightarrow 0.
\end{align}
	\par Next, we prove the convergence of the right-hand terms of the local energy inequality \eqref{eq:14}. Let us write the right terms of the local energy inequality \eqref{eq:14} as {$M_1,\ldots,M_{12}$}
	\par \textbf{The estimate of \(M_1\):} Since \(\psi \in C^{\infty}_c((0,\infty)\times\Omega_x)\), we have \((\partial_t\psi+\Delta_x\psi) \in C^{\infty}_c((0,\infty)\times\Omega_x)\). Therefore, due to Lemma \ref{lem:ul2}, it shows that
	\begin{align*}
		\frac{1}{2}\int_{(0,t)\times\Omega_{x}}|u^{\varepsilon,\delta}|^2(\partial_t\psi+\Delta_x\psi) dxds \rightarrow \frac{1}{2}\int_{(0,t)\times\Omega_{x}}|u|^2(\partial_t\psi+\Delta_x\psi) dxds, ~\text{as}~ \varepsilon,\delta\rightarrow 0.
	\end{align*}
	\par \textbf{The estimate of \(M_2\):}
	\begin{align*}
		&\frac{1}{2}\int_{(0,t)\times\Omega_{x}}|u^{\varepsilon,\delta}|^2(\theta_{\varepsilon}\ast u^{\varepsilon,\delta})\cdot \nabla_{x}\psi dxds-\frac{1}{2}\int_{(0,t)\times\Omega_{x}}|u|^2u\cdot \nabla_{x}\psi dxds\\
		&=\frac{1}{2}\int_{(0,t)\times\Omega_{x}}(|u^{\varepsilon,\delta}|^2-|u|^2)(\theta_{\varepsilon}\ast u^{\varepsilon,\delta})\cdot \nabla_{x}\psi dxds+\frac{1}{2}\int_{(0,t)\times\Omega_{x}}|u|^2(\theta_{\varepsilon}\ast u^{\varepsilon,\delta}-u)\cdot \nabla_{x}\psi dxds\\
		&\leq \frac{1}{2}\norm{\nabla_{x}\psi}_{L^{\infty}((0,t)\times\Omega_{x})}\norm{u^{\varepsilon,\delta}+u}_{L^{2}(0,t;L^4(\Omega_x))}\norm{u^{\varepsilon,\delta}}_{L^{\infty}(0,t;L^2(\mathbb R^3_x))}\norm{u^{\varepsilon,\delta}-u}_{L^{2}(0,t;L^4(\Omega_x))}\\
		&\hspace{0.5cm}+\frac{1}{2}\norm{\nabla_{x}\psi}_{L^{\infty}((0,t)\times\Omega_{x})}\norm{\theta_{\varepsilon}\ast u^{\varepsilon,\delta}-u}_{L^{3}((0,t)\times\Omega_x)}\norm{u}^2_{L^{3}((0,t)\times\Omega_x)}\\
        &\leq \frac{1}{2}\norm{\nabla_{x}\psi}_{L^{\infty}((0,t)\times\Omega_{x})}\norm{u^{\varepsilon,\delta}+u}_{L^{2}(0,t;L^4(\Omega_x))}\norm{u^{\varepsilon,\delta}}_{L^{\infty}(0,t;L^2(\mathbb R^3_x))}\norm{u^{\varepsilon,\delta}-u}_{L^{2}(0,t;L^4(\Omega_x))}\\
		&\hspace{0.5cm}+\frac{1}{2}\norm{\nabla_{x}\psi}_{L^{\infty}((0,t)\times\Omega_{x})}\left(\norm{\theta_{\varepsilon}\ast u^{\varepsilon,\delta}-\theta_\varepsilon \ast u}_{L^{3}((0,t)\times\Omega_x)}+\norm{\theta_{\varepsilon}\ast u-u}_{L^{3}((0,t)\times\Omega_x)}\right)\norm{u}^2_{L^{3}((0,t)\times\Omega_x)}.
	\end{align*}
	By Lemma \ref{lem:ul2}, it follows that
	\begin{align*}
		M_2 \to \frac{1}{2} \int_{(0,t)\times \Omega_{x}} |u|^2 u \cdot \nabla_{x} \psi dxds ~\text{as}~ \varepsilon,\delta\rightarrow 0. 
	\end{align*}
	\par \textbf{The estimate of \(M_3\):}  Since \(\Omega_v\) is a bounded domain, for any \(v\), there exists a ball \(B(0,R)\) such that \(v \in B(0,R)\). Therefore, we obtain
	\begin{align*}
		&\int_{(0,t)\times \Omega_x \times \Omega_v} \left[ \left( \frac{|v|^2}{2} + \log f^{\varepsilon,\delta} \right) f^{\varepsilon,\delta} - \left( \frac{|v|^2}{2} + \log f \right) f \right] (\partial_t \phi + \Delta_v \phi) dxdvds \\
		&\leq \frac{R^2}{2} \int_{(0,t)\times \Omega_x \times \Omega_v} \left| (f^{\varepsilon,\delta} - f) (\partial_t \phi + \Delta_v \phi) \right| dxdvds \\
		&\hspace{3cm} + \int_{(0,t)\times \Omega_x \times \Omega_v} \left( f^{\varepsilon,\delta} \log f^{\varepsilon,\delta} - f \log f \right) (\partial_t \phi + \Delta_v \phi) dxdvds.
	\end{align*}
	By Lemma \ref{lem:fcompact} and Lemma \ref{lem:flogfcompact}, this yields
	\begin{align*}
		&\int_{(0,t)\times \Omega_x \times \Omega_v} \left( \frac{|v|^2}{2} + \log f^{\varepsilon,\delta} \right) f^{\varepsilon,\delta} (\partial_t \phi + \Delta_v \phi) dxdvds \\
		&\hspace{3cm} \to \int_{(0,t)\times \Omega_x \times \Omega_v} \left( \frac{|v|^2}{2} + \log f \right) f (\partial_t \phi + \Delta_v \phi) dxdvds, ~\text{as}~ \varepsilon,\delta\rightarrow 0.
	\end{align*}
	
	\par \textbf{The estimate of \(M_4\):}
	\begin{align*}
		&\int_{(0,t) \times \Omega_x \times \Omega_v} \left( \frac{|v|^2}{2} + \log f^{\varepsilon,\delta} \right) f^{\varepsilon,\delta}v\cdot \left( \nabla_{x}\phi - \nabla_{v}\phi \right) dxdvds\\
		&\hspace{4cm}-\int_{(0,t) \times \Omega_x \times \Omega_v} \left( \frac{|v|^2}{2} + \log f \right)f v\cdot\left( \nabla_{x}\phi - \nabla_{v}\phi \right) dxdvds\\
		&\leq \frac{R^3}{2}\int_{(0,t)\times \Omega_x \times \Omega_v}\left|(f^{\varepsilon,\delta}-f)\left( \nabla_{x}\phi - \nabla_{v}\phi \right)\right|dxdvds\\
		&\hspace{4cm}+R\int_{(0,t)\times \Omega_x \times \Omega_v}\left|(f^{\varepsilon,\delta}\log f^{\varepsilon,\delta}-f\log f)\left( \nabla_{x}\phi - \nabla_{v}\phi \right)\right|dxdvds
	\end{align*}
	Since Lemma \ref{lem:fcompact} and Lemma \ref{lem:flogfcompact}, we have
	\begin{align*}
		\int_{(0,t) \times \Omega_x \times \Omega_v} \left( \frac{|v|^2}{2} + \log f^{\varepsilon,\delta} \right)&  f^{\varepsilon,\delta} v\cdot\left( \nabla_{x}\phi - \nabla_{v}\phi \right) dxdvds \rightarrow\\
		&\int_{(0,t) \times \Omega_x \times \Omega_v} \left( \frac{|v|^2}{2} + \log f\right) fv\cdot \left( \nabla_{x}\phi - \nabla_{v}\phi \right) dxdvds, ~\text{as}~ \varepsilon,\delta\rightarrow 0.
	\end{align*}
	\par \textbf{The estimate of \(M_5\):}
	\begin{align*}
		&\int_{(0,t)\times\Omega_{x}}(P^{\varepsilon,\delta}-\bar{P}^{\varepsilon,\delta})u^{\varepsilon,\delta}\cdot\nabla_{x}\psi dxds-\int_{(0,t)\times\Omega_{x}}(P-\bar{P})u\cdot\nabla_{x}\psi dxds\\
		&=\int_{(0,t)\times\Omega_{x}}(P^{\varepsilon,\delta}-\bar{P}^{\varepsilon,\delta})(u^{\varepsilon,\delta}-u)\cdot\nabla_{x}\psi dxds+\int_{(0,t)\times\Omega_{x}}[(P^{\varepsilon,\delta}-\bar{P}^{\varepsilon,\delta})-(P-\bar{P})]u\cdot\nabla_{x}\psi dxds\\
		&\leq \norm{P^{\varepsilon,\delta}-\bar{P}^{\varepsilon,\delta}}_{L^\frac32((0,t)\times\Omega_{x})}\norm{u^{\varepsilon,\delta}-u}_{L^3((0,t)\times\Omega_{x})}\norm{\nabla_{x}\psi}_{L^{\infty}((0,t)\times\Omega_{x})}\\
		&\hspace{5cm}+\int_{(0,t)\times\Omega_{x}}\left[(P^{\varepsilon,\delta}-\bar{P}^{\varepsilon,\delta})-(P-\bar{P})\right]u\cdot\nabla_{x}\psi dxds.
	\end{align*}
From \eqref{40}, it follows that
\begin{align*}
	\norm{P^{\varepsilon,\delta} - \bar{P}^{\varepsilon,\delta}}_{L^\frac32((0,t)\times \Omega_{x})} \norm{u^{\varepsilon,\delta} - u}_{L^3((0,t)\times \Omega_{x})} \norm{\nabla_{x} \psi}_{L^{\infty}((0,t)\times \Omega_{x})} \to 0, ~\text{as}~ \varepsilon,\delta\rightarrow 0.
\end{align*}
Since \(|u \cdot \nabla_{x} \psi| \leq C |u| |\nabla_{x} \psi| \in L^3((0,t) \times \Omega_{x})\), by the definition of weak convergence and \eqref{pweak}, it holds that
\begin{align*}
	\int_{(0,t) \times \Omega_{x}} \left[ (P^{\varepsilon,\delta} - \bar{P}^{\varepsilon,\delta}) - (P - \bar{P}) \right] u \cdot \nabla_{x} \psi dxds \to 0, ~\text{as}~ \varepsilon,\delta\rightarrow 0.
\end{align*}
	\par \textbf{The estimate of \(M_6\):}
\begin{align*}
	&\int_{(0,t) \times \Omega_x \times \Omega_v}fv\cdot u\phi dxdvds-\int_{(0,t) \times \Omega_x \times \Omega_v}f^{\varepsilon,\delta}v\cdot (\theta_{\varepsilon}\ast u^{\varepsilon,\delta})\phi dxdvds\\
	&= \int_{(0,t) \times \Omega_x \times \Omega_v}(f-f^{\varepsilon,\delta})v\cdot u\phi dxdvds+\int_{(0,t) \times \Omega_x \times \Omega_v}f^{\varepsilon,\delta}v\cdot (u-\theta_{\varepsilon}\ast u^{\varepsilon,\delta})\phi dxdvds\\
	&:=M_{61}+M_{62}.
\end{align*}
Here $R>0$ is chosen so that $|v|\le R$ on $\Omega_v$.
For the estimate of \(M_{61}\), it holds that
\begin{align*}
	|M_{61}| &\leq R\norm{\phi}_{L^{\infty}((0,t)\times\Omega_{x}\times\Omega_{v})}\int_{(0,t) \times \Omega_x \times \Omega_v}|(f-f^{\varepsilon,\delta})u|dxdvds \\
	&\leq C R\norm{\phi}_{L^{\infty}((0,t)\times\Omega_{x}\times\Omega_{v})}\|u\|_{L^3((0,t)\times\Omega_x)}\left\|\int_{\Omega_v}|f-f^{\varepsilon,\delta}|dv\right\|_{L^\frac32((0,t)\times\Omega_x)}\\
	&\leq C\|u\|_{L^3((0,t)\times\Omega_x)}\|f-f^{\varepsilon,\delta}\|_{L^\frac32((0,t)\times\Omega_x\times\Omega_v)}\\
	&\leq Ct^\frac{1}{12} \norm{f^{\varepsilon,\delta} - f}_{L^\frac32\left((0,t)\times \Omega_{x} \times \Omega_{v}\right)}.
\end{align*}
By Lemma \ref{lem:fcompact} and the uniform $L^\infty$ bound on $f^{\varepsilon,\delta}$, it is concluded that
\begin{align*}
	M_{61} \to 0, ~\text{as}~ \varepsilon,\delta\rightarrow 0.
\end{align*}
For the estimate of \(M_{62}\), we have
\begin{align*}
	|M_{62}| &\leq \norm{\phi}_{L^{\infty}((0,t)\times\Omega_{x}\times\Omega_{v})}\int_{(0,t) \times \Omega_x \times \Omega_v} |v| f^{\varepsilon,\delta} |u - \theta_{\varepsilon}\ast u^{\varepsilon,\delta}| dxdvds \\
	&\leq C\norm{\phi}_{L^{\infty}((0,t)\times\Omega_{x}\times\Omega_{v})}\int_{(0,t) \times \Omega_x} \left(\int_{\Omega_v}f^{\varepsilon,\delta}dv\right)\times|u - \theta_{\varepsilon}\ast u^{\varepsilon,\delta}| dxds \\
	&\leq C \norm{\phi}_{L^{\infty}((0,t)\times\Omega_{x}\times\Omega_{v})} \left\|\int_{\Omega_v}f^{\varepsilon,\delta}dv\right\|_{L^\frac32((0,t)\times\Omega_x)}\times \norm{u - \theta_{\varepsilon}\ast u^{\varepsilon,\delta}}_{L^{3}((0,t)\times(\Omega_x))}\\
	&\leq C\|f^{\varepsilon,\delta}\|_{L^\frac32((0,t)\times\Omega_x\times\Omega_v)}\norm{u - \theta_{\varepsilon}\ast u^{\varepsilon,\delta}}_{L^{3}((0,t)\times(\Omega_x))}\\
	&\leq C\|f^{\varepsilon,\delta}\|_{L^\frac32((0,t)\times\Omega_x\times\Omega_v)}\left(\norm{u - \theta_{\varepsilon}\ast u}_{L^{3}((0,t)\times(\Omega_x))}+\norm{\theta_\varepsilon \ast u - \theta_{\varepsilon}\ast u^{\varepsilon,\delta}}_{L^{3}((0,t)\times(\Omega_x))}\right).
\end{align*}
By Lemma~\ref{lem:ul2}, we obtain
\begin{align*}
	M_{62} \to 0, ~\text{as}~ \varepsilon,\delta\rightarrow 0.
\end{align*}
In summary, we obtain
\begin{align*}
	\int_{(0,t) \times \Omega_x \times \Omega_v} f^{\varepsilon,\delta} v \cdot (\theta_\varepsilon\ast u^{\varepsilon,\delta}) \phi dxdvds \to \int_{(0,t) \times \Omega_x \times \Omega_v} f v \cdot u \phi dxdvds, ~\text{as}~ \varepsilon,\delta\rightarrow 0.
\end{align*}

	\par \textbf{The estimate of \(M_7\):}
	\begin{align*}
		&\int_{(0,t) \times \rev{\mathbb R^3_x} \times \mathbb{R}^3_v}f^{\varepsilon,\delta}v\cdot \left(\theta_{\varepsilon}\ast (u^{\varepsilon,\delta}\psi)\right) dxdvds-\int_{(0,t) \times \rev{\mathbb R^3_x} \times \mathbb{R}^3_v}fv\cdot u\psi dxdvds\\
		&=\int_{(0,t) \times \rev{\mathbb R^3_x} }j^{\varepsilon,\delta}u^{\varepsilon,\delta}\psi -ju\psi dxds+\int_{(0,t) \times \rev{\mathbb R^3_x} \times \mathbb{R}^{3}_{v}}f^{\varepsilon,\delta}v\cdot (\theta_{\varepsilon}\ast (u^{\varepsilon,\delta}\psi)-u^{\varepsilon,\delta}\psi) dvdxds\\
        &=M_{7-1}+M_{7-2}.
	\end{align*}
    For \(M_{7-1}\), Since \eqref{eq:17}, we derive
	\begin{align*}
		\int_{(0,t) \times \Omega_x \times \mathbb{R}^3_v}f^{\varepsilon,\delta}v\cdot u^{\varepsilon,\delta}\psi dxdvds \rightarrow \int_{(0,t) \times \Omega_x \times \mathbb{R}^3_v}fv\cdot u\psi dxdvds, ~\text{as}~ \varepsilon,\delta\rightarrow 0.
	\end{align*}
    For \(M_{7-2}\), it shows that
\begin{align*}
		&\int_{(0,t) \times \rev{\mathbb R^3_x} \times \mathbb{R}^{3}_{v}}f^{\varepsilon,\delta}v\cdot (\theta_{\varepsilon}\ast (u^{\varepsilon,\delta}\psi)-u^{\varepsilon,\delta}\psi) dvdxds\\
		\leq&\int_{(0,t) \times \rev{\mathbb R^3_x} \times \mathbb{R}^{3}_{v}}|\theta_\varepsilon \ast (u^{\varepsilon,\delta}\psi)-u^{\varepsilon,\delta}\psi||v|f^{\varepsilon,\delta} dxdvds\\
		\leq&\left(\frac{4}{3} \pi \|f^{\varepsilon,\delta}(t,x,v)\|_{L^{\infty}} +1\right)\int_{(0,t) \times \rev{\mathbb R^3_x} }|\theta_{\varepsilon}\ast (u^{\varepsilon,\delta}\psi)-u^{\varepsilon,\delta}\psi|\left(\int_{  \mathbb{R}^3_v}|v|^2f^{\varepsilon,\delta}dv\right)^\frac45dxds\\
		\leq&Ct^\frac12\left(\frac{4}{3} \pi \|f^{\varepsilon,\delta}(t,x,v)\|_{L^{\infty}} +1\right)\norm{u^{\varepsilon,\delta}\psi-\theta_{\varepsilon}\ast (u^{\varepsilon,\delta}\psi)}_{L^2(0,t;L^5(\auditmath{\mathbb R_x^3}))}\norms{|v|^2f^{\varepsilon,\delta}}_{L^{\infty}(0,t;L^1(\mathbb{R}^6_{xv}))}^\frac45\\
        \leq&C(t)\bigg(\norm{\theta_{\varepsilon}\ast (u\psi)-\theta_{\varepsilon}\ast (u^{\varepsilon,\delta}\psi)}_{L^2(0,t;L^5(\auditmath{\mathbb R_x^3}))}+\norm{u\psi-\theta_{\varepsilon}\ast (u\psi)}_{L^2(0,t;L^5(\auditmath{\mathbb R_x^3}))}\\
        &+\norm{u^{\varepsilon,\delta}\psi-u\psi)}_{L^2(0,t;L^5(\auditmath{\mathbb R_x^3}))}\bigg)\rightarrow 0, ~\text{as}~ \varepsilon,\delta\rightarrow 0.
	\end{align*}
    To sum up, we have
    \begin{align*}
        \int_{(0,t) \times \rev{\mathbb R^3_x} \times \mathbb{R}^3_v}f^{\varepsilon,\delta}v\cdot \left(\theta_{\varepsilon}\ast (u^{\varepsilon,\delta}\psi)\right) dxdvds \rightarrow \int_{(0,t) \times \rev{\mathbb R^3_x} \times \mathbb{R}^3_v}fv\cdot u\psi dxdvds, ~\text{as}~ \varepsilon,\delta\rightarrow 0.
    \end{align*}

	\par \textbf{The estimate of \(M_8\):}
\begin{align*}
	&\int_{(0,t) \times \Omega_x \times \Omega_v} \left( 2 + \frac{|v|^2}{2} + \log f^{\varepsilon,\delta} \right) f^{\varepsilon,\delta}  (\theta_\varepsilon \ast u^{\varepsilon,\delta})\cdot \nabla_{v}\phi dxdvds \\
	&\hspace{5cm}- \int_{(0,t) \times \Omega_x \times \Omega_v} \left( 2 + \frac{|v|^2}{2} + \log f \right) f u \cdot \nabla_{v}\phi dxdvds \\
	&= \int_{(0,t) \times \Omega_x \times \Omega_v} \left( 2 + \frac{|v|^2}{2} \right) f^{\varepsilon,\delta} \left(\theta_\varepsilon \ast u^{\varepsilon,\delta} - u \right) \cdot \nabla_{v} \phi dxdvds \\
	&\hspace{0.5cm} + \int_{(0,t) \times \Omega_x \times \Omega_v} \left( 2 + \frac{|v|^2}{2} \right) \left( f^{\varepsilon,\delta} - f \right) u \cdot \nabla_{v} \phi dxdvds \\
	&\hspace{0.5cm} + \int_{(0,t) \times \Omega_x \times \Omega_v} \left( f^{\varepsilon,\delta}\log f^{\varepsilon,\delta}  - f\log f  \right) (\theta_\varepsilon \ast u^{\varepsilon,\delta}) \cdot \nabla_{v} \phi dxdvds \\
	&\hspace{0.5cm} + \int_{(0,t) \times \Omega_x \times \Omega_v} f\log f  \left( \theta_\varepsilon \ast u^{\varepsilon,\delta} - u \right) \cdot \nabla_{v} \phi dxdvds \\
	&:= M_{81} + M_{82} + M_{83} + M_{84}.
\end{align*}
For the estimate of \(M_{81}\), by H\"{o}lder's inequality, \eqref{eq:thebound-n} and the boundedness of the velocity \(v\), we obtain
\begin{align*}
	M_{81} &= \int_{(0,t) \times \Omega_x \times \Omega_v} \left( 2 + \frac{|v|^2}{2} \right) f^{\varepsilon,\delta} \left(\theta_\varepsilon \ast u^{\varepsilon,\delta} - u \right) \cdot \nabla_{v} \phi dxdvds \\
	&\leq \left( 2 + \frac{R^2}{2} \right) \norm{\nabla_{v} \phi}_{L^{\infty}((0,t) \times \Omega_x \times \Omega_v)} \int_{(0,t) \times \Omega_x \times \Omega_v} |\theta_\varepsilon \ast u^{\varepsilon,\delta} - u| f^{\varepsilon,\delta} dxdvds \\
	&\leq Ct^\frac23 \left( 2 + \frac{R^2}{2} \right) \norm{\nabla_{v} \phi}_{L^{\infty}((0,t) \times \Omega_x \times \Omega_v)} \|n^{\varepsilon,\delta}(t,x)\|_{L^{\infty}(0,t;L^\frac{3}{2}(\Omega_x) )} \\
	&\hspace{2cm} \times \norm{\theta_\varepsilon \ast u^{\varepsilon,\delta} - u}_{L^3((0,t)\times\Omega_x)}\\
    &\leq Ct^\frac23 \left( 2 + \frac{R^2}{2} \right) \norm{\nabla_{v} \phi}_{L^{\infty}((0,t) \times \Omega_x \times \Omega_v)} \|n^{\varepsilon,\delta}(t,x)\|_{L^{\infty}(0,t;L^\frac{3}{2}(\Omega_x) )} \\
	&\hspace{2cm} \times \left(\norm{\theta_\varepsilon \ast u^{\varepsilon,\delta} - \theta_\varepsilon \ast u}_{L^3((0,t)\times\Omega_x)}+\norm{\theta_\varepsilon \ast u - u}_{L^3((0,t)\times\Omega_x)}\right).
\end{align*}
By Lemma \ref{lem:ul2}, it follows that
\begin{align*}
	M_{81} \to 0, ~\text{as}~ \varepsilon,\delta\rightarrow 0.
\end{align*}
For the estimate of \(M_{82}\), applying H\"{o}lder's inequality and the boundedness of the velocity \(v\), we have
\begin{align*}
	M_{82} &\leq \left( 2 + \frac{R^2}{2} \right) \norm{\nabla_{v} \phi}_{L^{\infty}((0,t) \times \Omega_x \times \Omega_v)} \int_{(0,t) \times \Omega_x \times \Omega_v} |(f^{\varepsilon,\delta} - f) u| dxdvds \\
	&\leq C \left( 2 + \frac{R^2}{2} \right) \norm{\nabla_{v} \phi}_{L^{\infty}((0,t) \times \Omega_x \times \Omega_v)} \norm{u}_{L^3((0,t) \times \Omega_x)} \\
	&\hspace{2cm} \times \left\|\int_{\Omega_v}|f^{\varepsilon,\delta} - f|dv\right\|_{L^\frac32((0,t)\times\Omega_x)} \\
	&\leq C \left( 2 + \frac{R^2}{2} \right) \norm{\nabla_{v} \phi}_{L^{\infty}((0,t) \times \Omega_x \times \Omega_v)} \norm{u}_{L^3((0,t) \times \Omega_x)} \\
	&\hspace{2cm} \times \norm{f^{\varepsilon,\delta} - f}_{L^\frac32((0,t) \times \Omega_x \times \Omega_v)}.
\end{align*}
By Lemma \ref{lem:fcompact}, it follows that
\begin{align*}
	M_{82} \to 0, ~\text{as}~ \varepsilon,\delta\rightarrow 0.
\end{align*}
For the estimate of \(M_{83}\), using H\"{o}lder's inequality, we obtain
\begin{align*}
	M_{83} &\leq \norm{\nabla_{v} \phi}_{L^{\infty}((0,t) \times \Omega_x \times \Omega_v)} \norm{\theta_\varepsilon \ast u^{\varepsilon,\delta}}_{L^2((0,t) \times \Omega_x)} \\
	&\hspace{1cm} \times \norm{f^{\varepsilon,\delta}\log f^{\varepsilon,\delta}  - f\log f }_{L^1_v L^2_{xt}} \\
	&\leq C \norm{\nabla_{v} \phi}_{L^{\infty}((0,t) \times \Omega_x \times \Omega_v)} \norm{u^{\varepsilon,\delta}}_{L^2((0,t) \times \mathbb R_x^3}) \\
	&\hspace{1cm} \times \norm{f^{\varepsilon,\delta}\log f^{\varepsilon,\delta}  - f\log f }_{L^2((0,t) \times \Omega_x \times \Omega_v)}.
\end{align*}
By Lemma \ref{lem:flogfcompact}, it follows that
\begin{align*}
	M_{83} \to 0, ~\text{as}~ \varepsilon,\delta\rightarrow 0.
\end{align*}
For the estimate of \(M_{84}\), applying H\"{o}lder's inequality, we get
\begin{align*}
	M_{84} &\leq C \norm{\nabla_{v} \phi}_{L^{\infty}((0,t) \times \Omega_x \times \Omega_v)} \norm{f \log f}_{L^{\frac32}((0,t) \times \Omega_x \times \Omega_v)} \\
	&\hspace{1cm} \times \norm{u - \theta_\varepsilon \ast u^{\varepsilon,\delta}}_{L^3((0,t) \times \Omega_x)} \\
	&\leq C \norm{\nabla_{v} \phi}_{L^{\infty}((0,t) \times \Omega_x \times \Omega_v)} \norm{f \log f}_{L^\frac32((0,t) \times \Omega_x \times \Omega_v)} \\
	&\hspace{1cm} \times \left(\norm{\theta_\varepsilon \ast u - \theta_\varepsilon \ast u^{\varepsilon,\delta}}_{L^3((0,t) \times \Omega_x)} +\norm{u - \theta_\varepsilon \ast u}_{L^3((0,t) \times \Omega_x)} \right).
\end{align*}
By Lemma \ref{lem:ul2}, it follows that
\begin{align*}
	M_{84} \to 0, ~\text{as}~ \varepsilon,\delta\rightarrow 0.
\end{align*}
To sum up, we obtain
	\begin{align*}
		&\int_{(0,t) \times \Omega_x \times \Omega_v} \left( 2 + \frac{|v|^2}{2} + \log f^{\varepsilon,\delta} \right)  f^{\varepsilon,\delta}(\theta_\varepsilon*u^{\varepsilon,\delta})\cdot \nabla_{v}\phi dxdvds\\ &\hspace{3cm}\rightarrow\int_{(0,t) \times \Omega_x \times \Omega_v} \left( 2 + \frac{|v|^2}{2} + \log f \right) f  u\cdot\nabla_{v}\phi dxdvds, ~\text{as}~ \varepsilon,\delta\rightarrow 0.
	\end{align*}
	\par \textbf{The estimate of \(M_9\):}
	\begin{align*}
		&\int_{(0,t) \times {\rev{\mathbb{R}^3}} \times \mathbb{R}^3_v}|\theta_{\varepsilon}\ast u^{\varepsilon,\delta}|^2f^{\varepsilon,\delta}\psi-(\theta_{\varepsilon}\ast u^{\varepsilon,\delta})\cdot \left(\theta_{\varepsilon}\ast (u^{\varepsilon,\delta}\psi)\right)f^{\varepsilon,\delta} dxdvds\\
		=&\int_{(0,t) \times \rev{\mathbb{R}^3} \times \mathbb{R}^3_v}(\theta_{\varepsilon}\ast u^{\varepsilon,\delta})\cdot (\theta_{\varepsilon}\ast u^{\varepsilon,\delta}-u^{\varepsilon,\delta})f^{\varepsilon,\delta}\psi dxdvds\\
        &+\int_{(0,t) \times \rev{\mathbb{R}^3} \times \mathbb{R}^3_v}(\theta_{\varepsilon}\ast u^{\varepsilon,\delta})\cdot \left(u^{\varepsilon,\delta}\psi-\theta_{\varepsilon}
        \ast (u^{\varepsilon,\delta}\psi)\right)f^{\varepsilon,\delta}dxdvds\\
        :=&M_{9-1}+M_{9-2},
	\end{align*}
   % Here \(\rev{\mathbb{R}^3}\) is a fixed bounded spatial region containing \(\Omega_x\) and the spatial support of \(\theta_\varepsilon*(u^{\varepsilon,\delta}\psi)\) for all sufficiently small \(\varepsilon\). Such an enlargement is needed because convolution may extend the support beyond \(\Omega_x\). Since Lemma~\ref{lem:ul2} holds on any bounded spatial region, the strong convergence established there also holds on \(K\) and can therefore be used in the following estimates.
   
    The following bounds also hold for $|M_{9-1}|$ and $|M_{9-2}|$. For \(M_{9-1}\), it becomes
    \begin{align*}
    &\int_{(0,t) \times \rev{\mathbb{R}^3} \times \mathbb{R}^3_v}(\theta_{\varepsilon}\ast u^{\varepsilon,\delta})\cdot (\theta_{\varepsilon}\ast u^{\varepsilon,\delta}-u^{\varepsilon,\delta})f^{\varepsilon,\delta}\psi dxdvds\\
		\leq&\int_{(0,t) \times \rev{\mathbb{R}^3} \times \mathbb{R}^3_v}|(u^{\varepsilon,\delta}-\theta_{\varepsilon}\ast u^{\varepsilon,\delta})\cdot (\theta_{\varepsilon}\ast u^{\varepsilon,\delta})|f^{\varepsilon,\delta}\psi dxdvds\\
		\leq& C\left(\frac{4}{3} \pi \|f^{\varepsilon,\delta}(t,x,v)\|_{L^{\infty}} +1\right)\int_{(0,t) \times \rev{\mathbb{R}^3}}|\psi(u^{\varepsilon,\delta}-\theta_{\varepsilon}\ast u^{\varepsilon,\delta}
        )\cdot (\theta_{\varepsilon}\ast u^{\varepsilon,\delta})|\left(\int_{ \mathbb{R}_v^3}|v|^3f^{\varepsilon,\delta}dv\right)^\frac12dxds\\
		\leq& C\left(\frac{4}{3} \pi \|f^{\varepsilon,\delta}(t,x,v)\|_{L^{\infty}} +1\right)\norm{u^{\varepsilon,\delta}}_{L^2(0,t;L^4(\auditmath{\mathbb R_x^3}))}\\
        &\quad\times\norms{\psi\left(u^{\varepsilon,\delta}-\theta_{\varepsilon}\ast u^{\varepsilon,\delta}\right)}_{L^2(0,t;L^4(\rev{\mathbb{R}^3}))}\norms{|v|^3f^{\varepsilon,\delta}}_{L^{\infty}(0,t;L^1(\mathbb{R}^6_{xv}))}^\frac12\\
        \leq& C\biggl( \norms{\psi\left(u^{\varepsilon,\delta}-u\right)}_{L^2(0,t;L^4(\rev{\mathbb{R}^3}))}+\norms{\psi\left(u-\theta_{\varepsilon}\ast u\right)}_{L^2(0,t;L^4(\rev{\mathbb{R}^3}))}\\
        &\hspace{5cm}+\norms{\psi\left(\theta_{\varepsilon}\ast u-\theta_{\varepsilon}\ast u^{\varepsilon,\delta}\right)}_{L^2(0,t;L^4(\rev{\mathbb{R}^3}))} \biggr)\rightarrow 0.
	\end{align*}

For \(M_{9-2}\), we have
\begin{align*}
		&\int_{(0,t) \times \rev{\mathbb{R}^3} \times \mathbb{R}^{3}_{v}}f^{\varepsilon,\delta}(\theta_{\varepsilon}\ast u^{\varepsilon,\delta})\cdot (u^{\varepsilon,\delta}\psi-\theta_{\varepsilon}\ast (u^{\varepsilon,\delta}\psi)) dvdxds\\
		\leq&\int_{(0,t) \times \rev{\mathbb{R}^3} \times \mathbb{R}^{3}_{v}}\left|\left(u^{\varepsilon,\delta}\psi-\theta_{\varepsilon}\ast (u^{\varepsilon,\delta}\psi)\right)\cdot (\theta_{\varepsilon}\ast u^{\varepsilon,\delta})\right|f^{\varepsilon,\delta} dxdvds\\
		\leq&\left(\frac{4}{3} \pi \|f^{\varepsilon,\delta}(t,x,v)\|_{L^{\infty}} +1\right)\int_{(0,t) \times \rev{\mathbb{R}^3} }|\theta_{\varepsilon}\ast (u^{\varepsilon,\delta}\psi)-u^{\varepsilon,\delta}\psi||\theta_{\varepsilon}\ast u^{\varepsilon,\delta}|\left(\int_{  \mathbb{R}^3_v}|v|^3f^{\varepsilon,\delta}dv\right)^\frac12dxds\\
		\leq&C\left(\frac{4}{3} \pi \|f^{\varepsilon,\delta}(t,x,v)\|_{L^{\infty}} +1\right)\norm{u^{\varepsilon,\delta}\psi-\theta_{\varepsilon}\ast (u^{\varepsilon,\delta}\psi)}_{L^2(0,t;L^4(\rev{\mathbb{R}^3}))}\times\\
		&\norm{u^{\varepsilon,\delta}}_{L^2(0,t;L^4(\auditmath{\mathbb R_x^3}))}\norms{|v|^3f^{\varepsilon,\delta}}_{L^{\infty}(0,t;L^1(\mathbb{R}^6_{xv}))}^\frac12\\
        \leq& C\bigg(\norm{\theta_{\varepsilon}\ast (u\psi)-\theta_{\varepsilon}\ast (u^{\varepsilon,\delta}\psi)}_{L^2(0,t;L^4(\rev{\mathbb{R}^3}))}+\norm{u\psi-\theta_{\varepsilon}\ast (u\psi)}_{L^2(0,t;L^4(\rev{\mathbb{R}^3}))}\\
        &+\norm{u^{\varepsilon,\delta}\psi-u\psi}_{L^2(0,t;L^4(\rev{\mathbb{R}^3}))}\bigg)\rightarrow 0, ~\text{as}~ \varepsilon,\delta\rightarrow 0.
	\end{align*}
    To sum up, we derive
    \begin{align*}
        \int_{(0,t) \times \rev{\mathbb{R}^3} \times \mathbb{R}^3_v}|\theta_{\varepsilon}\ast u^{\varepsilon,\delta}|^2f^{\varepsilon,\delta}\psi-(\theta_{\varepsilon}\ast u^{\varepsilon,\delta})\cdot \left(\theta_{\varepsilon}\ast (u^{\varepsilon,\delta}\psi)\right)f^{\varepsilon,\delta} dxdvds \rightarrow 0, ~\text{as}~ \varepsilon,\delta\rightarrow 0.
    \end{align*}

    \par \textbf{The estimate of \(M_{10}\):} Similary with \eqref{convergencevf}, we have
	\begin{align*}
		\int_{(0,t) \times \Omega_x \times \Omega_v}f^{\varepsilon,\delta}|v|^2(1-\gamma_{\delta}(v))\phi dxdvds\rightarrow 0, ~\text{as}~ \varepsilon,\delta\rightarrow 0.
	\end{align*}

	\par \textbf{The estimate of \(M_{11}\):}
	\begin{align*}
		M_{11}\leq&\frac{1}{2}\int_{(0,t) \times \Omega_x \times \Omega_v}|v|^3f^{\varepsilon,\delta}(1-\gamma_{\delta}(v))|\nabla_{v}\phi| dxdvds\\
		&+\int_{(0,t) \times \Omega_x \times \Omega_v}|vf^{\varepsilon,\delta}\log f^{\varepsilon,\delta}|(1-\gamma_{\delta}(v))|\nabla_{v}\phi| dxdvds\\
        \leq &Ct\norms{|v|^3f^{\varepsilon,\delta}}_{L^{\infty}(0,t;L^1(\mathbb{R}^6_{xv}))}\norms{\nabla_{v}\phi}_{L^{\infty}((0,t)\times\Omega_{x}\times\Omega_{v})}\norms{1-\gamma_\delta(v)}_{L^{\infty}(\Omega_{v})}\\
        &+\norms{|v|f^{\varepsilon,\delta}\log f^{\varepsilon,\delta}}_{L^1((0,T)\times\Omega_{x}\times\Omega_v)} \norms{\nabla_{v}\phi}_{L^{\infty}((0,t)\times\Omega_{x}\times\Omega_{v})}\norms{1-\gamma_\delta(v)}_{L^{\infty}(\Omega_{v})}.
		%:=&M_{11-1}+M_{11-2}.
	\end{align*}
Due to the boundedness of \(\Omega_v\), as \(\delta \rightarrow 0\), it follows that \(\gamma_\delta (v)=1\) in \(\Omega_v\). It follows that
\begin{align*}
		M_{11} \rightarrow 0, ~\text{as}~ \varepsilon,\delta\rightarrow 0.
	\end{align*}
\begin{revision}
\par \textbf{The estimate of \(M_{12}\):}
\begin{align*}
&M_{12}
-\int_{(0,t)\times\Omega_x\times\mathbb R^3_v}
|u|^2f(\phi-\psi)dxdvds\\
=&
\int_{(0,t)\times\Omega_x\times\mathbb R^3_v}
\bigl(|\theta_{\varepsilon} \ast u^{\varepsilon,\delta}|^2-|u|^2\bigr)
f^{\varepsilon,\delta}(\phi-\psi)dxdvds+
\int_{(0,t)\times\Omega_x\times\mathbb R^3_v}
|u|^2(f^{\varepsilon,\delta}-f)(\phi-\psi)dxdvds\\
:=&M_{12,1}+M_{12,2}.
\end{align*}
By H\"older's inequality, Lemma~\ref{lem:ul2}, and the definition of $n^{\varepsilon,\delta}$ (see Proposition \ref{pro:jnweak}), it holds
\begin{align*}
|M_{12,1}|
&\leq
C\int_{(0,t)\times\Omega_x}
n^{\varepsilon,\delta}
|\theta_{\varepsilon} \ast u^{\varepsilon,\delta}-u|
\bigl(|\theta_{\varepsilon} \ast u^{\varepsilon,\delta}|+|u|\bigr)dxds\\
\leq&
C\|n^{\varepsilon,\delta}\|_{L^\infty(0,t;L^2(\Omega_x))}
\|\theta_{\varepsilon} \ast u^{\varepsilon,\delta}-u\|_{L^2(0,t;L^4(\Omega_x))}
\left(
\|\theta_{\varepsilon} \ast u^{\varepsilon,\delta}\|_{L^2(0,t;L^4(\Omega_x))}
+\|u\|_{L^2(0,t;L^4(\Omega_x))}
\right)\\
&\times(\|\psi\|_{L^\infty((0,t)\times\Omega_x\times\mathbb{R}^3_v)}+\|\phi\|_{L^\infty((0,t)\times\Omega_x\times\mathbb{R}^3_v)})\longrightarrow 0,
\qquad\text{as}~\varepsilon,\delta\to0,
\end{align*}
where \(C\) is independent of \(\varepsilon,\delta\).

% Since
% \begin{align*}
% \big\||u|^2\big\|_{L^1(0,t;L^2(\Omega_x))}
% &=\|u\|_{L^2(0,t;L^4(\Omega_x))}^2<\infty,\\
% \big\||u|^2\psi\big\|_{L^1(0,t;L^2(\Omega_x))}
% &\leq
% \|\psi\|_{L^\infty((0,t)\times\Omega_x)}
% \|u\|_{L^2(0,t;L^4(\Omega_x))}^2<\infty,
% \end{align*}
For $M_{12,2}$, applying the same moment interpolation as for $i_{11}$ in Section 6, the uniform $L^\infty$ bound gives
\begin{align*}
|M_{12,2}|&\le C_{\phi,\psi,T}\int_0^t\|u(s)\|_{L^5_x}^2\left(\int_{\mathbb R^6}|v|^2|f^{\varepsilon,\delta}-f|\,dxdv\right)^{3/5}ds\\
&\le C_{\phi,\psi,T}\|u\|_{L^{20/9}(0,t;L^5_x)}^2\bigl\||v|^2(f^{\varepsilon,\delta}-f)\bigr\|_{L^6(0,t;L^1_{xv})}^{3/5}\longrightarrow0,
\end{align*}
by Corollary~\ref{lem:tpcompact}. All norms of $u$ here involve only $(t,x)$.
Consequently,
\begin{align*}
M_{12}\longrightarrow
\int_{(0,t)\times\Omega_x\times\mathbb R^3_v}
|u|^2f(\phi-\psi)\,dxdvds,
\qquad\text{as }\varepsilon,\delta\to0.
\end{align*}
\end{revision}
Finally, as \(\varepsilon,\delta\rightarrow 0\), one can get the energy inequality as follows:
	% \begin{align}
	% 	\label{eq:19}
	% 	\frac{1}{2}&\int_{\Omega_x}|u|^2\psi dx+\int_{\Omega_x \times \Omega_v}\left(\frac{1}{2}|v|^2f+f\log f\right)\phi dxdv+\int_{(0,t)\times \Omega_x}|\nabla u|^2\psi dxds\nonumber\\
	% 	&+\int_{(0,t)\times \Omega_x \times \Omega_v}\frac{|(u-v)f-\nabla_{v}f|^2}{f}\phi dxdvds\leq \frac{1}{2}\int_{(0,t)\times\Omega_{x}}|u|^2(\partial_t\psi+\Delta_x\psi) dxds\nonumber\\
	% 	&+\frac{1}{2}\int_{(0,t)\times\Omega_{x}}|u|^2u\cdot \nabla_{x}\psi dx
	% 	+\int_{(0,t)\times \Omega_x \times \Omega_v}\left(\frac{|v|^2}{2}+\log f\right)f(\partial_t\phi+\Delta_v\phi)dxdvds\nonumber\\
	% 	&+ \int_{(0,t) \times \Omega_x \times \Omega_v} \left( \frac{|v|^2}{2} + \log f \right) v f \left( \nabla_{x}\phi - \nabla_{v}\phi \right) dxdvds
	% 	+\int_{(0,t)\times\Omega_{x}}(P-\bar{P})u\cdot\nabla_{x}\psi dxds\nonumber \\
	% 	&-\int_{(0,t) \times \Omega_x \times \Omega_v}vfu\phi dxdvds +\int_{(0,t) \times \Omega_x \times \mathbb{R}^3_v}vfu\psi dxdvds\nonumber\\
	% 	&+ \int_{(0,t) \times \Omega_x \times \Omega_v} \left( 2 + \frac{|v|^2}{2} +\log f \right) u f \nabla_{v}\phi dxdvds+C\norms{|v|^3f^{\varepsilon,\delta}}_{L^{\infty}(0,T;L^1(\mathbb{R}^6_{xv}))}\delta\nonumber \\
	% 	&+C\norm{\nabla_{v}\phi}_{L^{\infty}((0,t)\times\Omega_{x}\times\Omega_{v})}\delta.
	% \end{align}
	% \par By completely similar arguments as the convergence of \(\left(u^{\varepsilon,\delta}, f^{\varepsilon,\delta}\right)\) in the subsection, we can get the convergence of \(\left(u, f\right)\) and the details are omitted for brevity. Thus, \(\left(u, f\right)\) satisfies the following local energy inequality:
	\begin{equation}
		\label{eq:20}
	    \begin{aligned}
		\frac{1}{2}&\int_{\Omega_x}\left(|u|^2\psi\right)(t,\cdot) dx+\int_{\Omega_x \times \Omega_v}\left(\left(\frac{1}{2}|v|^2f+f\log f\right)\phi\right)(t,\cdot) dxdv+\int_{(0,t)\times \Omega_x}|\nabla u|^2\psi dxds\\
		&+\int_{(0,t)\times \Omega_x \times \Omega_v}\frac{|(u-v)f-\nabla_{v}f|^2}{f}\phi dxdvds\leq \frac{1}{2}\int_{(0,t)\times\Omega_{x}}|u|^2(\partial_t\psi+\Delta_x\psi) dxds\\
		&+\frac{1}{2}\int_{(0,t)\times\Omega_{x}}|u|^2u\cdot \nabla_{x}\psi dxds
		+\int_{(0,t)\times \Omega_x \times \Omega_v}\left(\frac{|v|^2}{2}+\log f\right)f(\partial_t\phi+\Delta_v\phi)dxdvds\\
		&+ \int_{(0,t) \times \Omega_x \times \Omega_v} \left( \frac{|v|^2}{2} + \log f \right) fv \cdot \left( \nabla_{x}\phi - \nabla_{v}\phi \right) dxdvds
		+\int_{(0,t)\times\Omega_{x}}(P-\bar{P})u\cdot\nabla_{x}\psi dxds \\
		&-\int_{(0,t) \times \Omega_x \times \Omega_v}f
        v\cdot u\phi dxdvds +\int_{(0,t) \times \Omega_x \times \mathbb{R}^3_v}f
        v\cdot u\psi dxdvds\\
		&+ \int_{(0,t) \times \Omega_x \times \Omega_v} \left( 2 + \frac{|v|^2}{2} +\log f \right) f
        u \cdot \nabla_{v}\phi dxdv\rev{ds}\\
        &\rev{+\int_{(0,t)\times\Omega_x\times\mathbb R^3_v}|u|^2f(\phi-\psi)dxdvds.}
	\end{aligned}
	\end{equation}
	
\begin{revision}
Now we consider the local energy inequality of fluid.
Let $0\leq \psi\in C_c^\infty((0,T)\times\mathbb R^3_x)$,
and choose a bounded domain $\Omega_x$ containing the spatial
support of $\psi$ compactly. Taking $\phi\equiv0$ in \eqref{eq:20}, all terms involving
$\phi$ or its derivatives vanish. The remaining coupling terms satisfy
\begin{align*}
&\int_0^t\int_{\Omega_x}\int_{\mathbb R^3_v}
f\,v\cdot u\,\psi\,dvdxds
-\int_0^t\int_{\Omega_x}\int_{\mathbb R^3_v}
|u|^2f\,\psi\,dvdxds
\\
=&
\int_0^t\int_{\Omega_x}
\left[\left(\int_{\mathbb{R}_v^3}fv\,dv\right)\cdot u-\left(\int_{\mathbb{R}_v^3}f\,dv\right)|u|^2\right]
\psi\,dxds
\\
=&
\int_0^t\int_{\Omega_x}
\left[\left(\int_{\mathbb{R}_v^3}fv\,dv\right)-\left(\int_{\mathbb{R}_v^3}f\,dv\,\right)u\right]\cdot u
\,\psi\,dxds.
\end{align*}
Here the second term on the left-hand side comes from
\[
\left.
\int_0^t\int_{\Omega_x}\int_{\mathbb R^3_v}
|u|^2f(\phi-\psi)\,dvdxds
\right|_{\phi=0}
=
-\int_0^t\int_{\Omega_x}\left(\int_{\mathbb{R}_v^3}f\,dv\,\right)|u|^2\psi\,dxds.
\]
Utilizing integration by parts and $\nabla\cdot u=0$, $P-\bar P$ can be replaced by $P$ in the pressure term.
Using the compact spatial support of $\psi$, we obtain,
for almost every $t\in(0,T)$,
\begin{equation}\label{min:fluid-lei}
\begin{aligned}
&\frac12\int_{\mathbb R^3_x}
|u|^2\psi\,(\cdot,t)dx+\int_0^t\int_{\mathbb R^3_x}
|\nabla_xu|^2\psi\,dxds\\ \leq&
\frac12\int_0^t\int_{\mathbb R^3_x}
|u|^2
\bigl(\partial_s\psi+\Delta_x\psi\bigr)
\,dxds+
\int_0^t\int_{\mathbb R^3_x}
\left(\frac{|u|^2}{2}+P\right)
u\cdot\nabla_x\psi\,dxds
\\
&+
\int_0^t\int_{\Omega_x}
\left[\left(\int_{\mathbb{R}_v^3}fv\,dv\right)-\left(\int_{\mathbb{R}_v^3}f\,dv\,\right)u\right]\cdot u
\,\psi\,dxds.
\end{aligned}
\end{equation}
There is no initial energy term because $\psi$ vanishes near
$s=0$; no condition $\psi(t,\cdot)=0$ is imposed.
\end{revision}

\section{Verification of \texorpdfstring{$vi)$}{vi)} in Definition \ref{def:weak}: the convergence of local energy inequality of the second type}
In this section we assume $\kappa>3$, as required in Definition~\ref{def:weak}(vi), so that the third velocity moment converges strongly by Lemma~\ref{lem:vfcompact}.
	\par {\bf \underline{Verification of local energy inequality in $vi)$ of Definition \ref{def:weak}.}} In the following, we will prove the convergence of the local energy inequality in \eqref{eq:23}. For the left-hand terms of it, the convergence is proved in Section 6, and we omit it. Let us write the right terms of the local energy inequality \eqref{eq:23} as \(N_1, N_2, \cdots, N_{8}\) term by term. Since \(M_1\), \(M_2\), \(M_5\) and \(M_9\) are the same as \(N_1\), \(N_2\), \(N_3\) and \(N_4\), we also omit it.
    \par \textbf{For \(N_{5}\):}
	\begin{align*}
		\int_{(0,t) \times \Omega_x \times \mathbb{R}^{3}_{v}}f^{\varepsilon,\delta}|v|^2(1-\gamma_{\delta}(v))\psi dxdvds&= \int_{(0,t) \times \Omega_x \times \mathbb{R}^3_v}f^{\varepsilon,\delta}|v|^3\frac{1-\gamma_{\delta}(v)}{|v|}\psi dxdvds\\
        &\leq t\norms{|v|^3f^{\varepsilon,\delta}}_{L^{\infty}(0,t;L^1(\mathbb{R}^6_{xv}))}\norms{\frac{1-\gamma_{\delta}(v)}{|v|}}_{L^\infty_v}\norms{\psi}_{L^\infty_{tx}}\\ 
        &\leq C\delta\norms{|v|^3f^{\varepsilon,\delta}}_{L^{\infty}(0,t;L^1(\mathbb{R}^6_{xv}))}\norm{\psi}_{L^{\infty}((0,t)\times\Omega_{x})}\rightarrow 0.
	\end{align*}

 %    \par \textbf{For \(N_{5}\):} Similary with \(M_{10}\), we obtain
	% \begin{align*}
	% 	\int_{(0,t) \times \Omega_x \times \mathbb{R}^{3}_{v}}f^{\varepsilon,\delta}|v|^2(1-\gamma_{\delta}(v))\psi dxdvds \rightarrow 0, \text{as}~ \varepsilon,\delta \rightarrow 0.
	% \end{align*}

	\par \textbf{For \(N_{6}\):}
	\begin{align*}
		&\int_{(0,t)\times \Omega_x}\bigg|\int_{  \mathbb{R}^3_v}\big((\frac{|v|^2}{2}+\log f^{\varepsilon,\delta})f^{\varepsilon,\delta}-(\frac{|v|^2}{2}+\log f)f\big) dv\partial_t\psi\bigg| dxds\\
		\leq&\int_{(0,t)\times \Omega_x}\bigg|\int_{  \mathbb{R}^3_v}\frac{|v|^2}{2}\big(f^{\varepsilon,\delta}-f\big)dv\partial_t\psi \bigg|dxds+\int_{(0,t)\times \Omega_x }\bigg|\int_{  \mathbb{R}^3_v}\big(f^{\varepsilon,\delta}\log f^{\varepsilon,\delta}-f\log f\big)dv\partial_t\psi\bigg| dxds\\
		:=&N_{61}+N_{62}.
	\end{align*}
	For \(N_{61}\), due to Lemma \ref{lem:vfcompact}, there holds
	\begin{align*}
		N_{61}&\leq C \norms{|v|^2\big(f^{\varepsilon,\delta}-f\big)}_{L^{1}((0,t)\times\mathbb{R}^3_x\times\mathbb{R}^3_v)}\norm{\partial_t\psi}_{L^{\infty}((0,t)\times\Omega_{x})}\rightarrow 0, ~\text{as}~ \varepsilon,\delta\rightarrow 0.
	\end{align*}
	For \(N_{62}\), due to Lemma \ref{lem:flogfcompact1}, we have
	\begin{align*}
		N_{62}&\leq C\norms{f^{\varepsilon,\delta}\log f^{\varepsilon,\delta}-f\log f}_{L^{1}((0,t)\times\Omega_{x}\times\mathbb{R}^3_v)}\norm{\partial_t\psi}_{L^{\infty}((0,t)\times\Omega_{x})}\rightarrow 0, ~\text{as}~ \varepsilon,\delta\rightarrow 0.
	\end{align*}
	In summary, we get
	\begin{align*}
		\int_{(0,t)\times \Omega_x \times \mathbb{R}^{3}_{v}}\left(\frac{|v|^2}{2}+\log f^{\varepsilon,\delta}\right)f^{\varepsilon,\delta}\partial_t\psi dxdvds \rightarrow \int_{(0,t)\times \Omega_x \times \mathbb{R}^{3}_{v}}\left(\frac{|v|^2}{2}+\log f\right)f\partial_t\psi dxdvds.
	\end{align*}
	\par \textbf{For \(N_{7}\):}
	\begin{align*}
		&\int_{(0,t)\times \Omega_x}\bigg|\int_{  \mathbb{R}^3_v}\big((\frac{|v|^2}{2}+\log f^{\varepsilon,\delta})f^{\varepsilon,\delta}-(\frac{|v|^2}{2}+\log f)f\big) vdv\cdot\nabla_x\psi\bigg| dxds\\
		\leq&\frac12\int_{(0,t)\times\Omega_x}\int_{\mathbb R^3_v}|v|^3|f^{\varepsilon,\delta}-f|\,dv\,|\nabla_x\psi|\,dxds\\
		&\hspace{4cm}+\int_{(0,t)\times \Omega_x }\bigg|\int_{  \mathbb{R}^3_v}\big(f^{\varepsilon,\delta}\log f^{\varepsilon,\delta}-f\log f\big)vdv\cdot\nabla_x\psi\bigg| dxds\\
		:=&N_{71}+N_{72}.
	\end{align*}
	For \(N_{71}\), due to Lemma \ref{lem:vfcompact}, we have
	\begin{align*}
			N_{71}&\leq C \norms{|v|^3\big(f^{\varepsilon,\delta}-f\big)}_{L^{1}((0,t)\times\mathbb{R}^3_x\times\mathbb{R}^3_v)}\norm{\nabla_x\psi}_{L^{\infty}((0,t)\times\Omega_{x})}\rightarrow 0, ~\text{as}~ \varepsilon,\delta\rightarrow 0.
	\end{align*}
	For \(N_{72}\), due to Lemma \ref{lem:vflogf compact}, we have
	\begin{align*}
		N_{72}&\leq C\norms{|v|\left|f^{\varepsilon,\delta}\log f^{\varepsilon,\delta}-f\log f\right|}_{L^{1}((0,t)\times\Omega_{x}\times\mathbb{R}^3_v)}\norm{\nabla_x\psi}_{L^{\infty}((0,t)\times\Omega_{x})}\rightarrow 0, ~\text{as}~ \varepsilon,\delta\rightarrow 0.
	\end{align*}
	To sum up, we  get
	\begin{align*}
		&\int_{(0,t)\times \Omega_x \times \mathbb{R}^{3}_{v}}\left(\frac{|v|^2}{2}+\log f^{\varepsilon,\delta}\right)f^{\varepsilon,\delta}v\cdot\nabla_x\psi dxdvds \\
		&\hspace{4cm}\rightarrow \int_{(0,t)\times \Omega_x \times \mathbb{R}^{3}_{v}}\left(\frac{|v|^2}{2}+\log f\right)fv\cdot\nabla_x\psi dxdvds, ~\text{as}~ \varepsilon,\delta\rightarrow 0.
	\end{align*}
\par \textbf{For \(N_8\):}
\begin{align*}
    &\int_{(0,t) \times \rev{\mathbb R^3_x} \times \mathbb{R}^3_v}f^{\varepsilon,\delta}v\cdot \left(\theta_{\varepsilon}\ast (u^{\varepsilon,\delta}\psi)\right)-f^{\varepsilon,\delta}v\cdot \left(\theta_{\varepsilon}\ast u^{\varepsilon,\delta}\right)\psi dxdvds\\
    =&\int_{(0,t) \times \Omega_x \times \mathbb{R}^3_v}f^{\varepsilon,\delta}v\cdot \left(u^{\varepsilon,\delta}-\theta_{\varepsilon}\ast u^{\varepsilon,\delta}\right)\psi dxdvds\\
    &+\int_{(0,t) \times \rev{\mathbb R^3_x} \times \mathbb{R}^3_v}f^{\varepsilon,\delta}v\cdot \left(\theta_{\varepsilon}\ast (u^{\varepsilon,\delta}\psi)-u^{\varepsilon,\delta}\psi\right) dxdvds\\
    :=& N_{8-1}+N_{8-2}.
\end{align*}
For \(N_{8-1}\), it becomes
\begin{align*}
&\int_{(0,t) \times \Omega_x \times \mathbb{R}^3_v}f^{\varepsilon,\delta}v\cdot \left(u^{\varepsilon,\delta}-\theta_{\varepsilon}\ast u^{\varepsilon,\delta}\right)\psi dxdvds\\
		\leq&\int_{(0,t) \times \Omega_x \times \mathbb{R}^{3}_{v}}|\theta_\varepsilon \ast u^{\varepsilon,\delta}-u^{\varepsilon,\delta}||v|f^{\varepsilon,\delta}\psi dxdvds\\
		\leq&\left(\frac{4}{3} \pi \|f^{\varepsilon,\delta}(t,x,v)\|_{L^{\infty}} +1\right)\int_{(0,t) \times \Omega_x }|\theta_{\varepsilon}\ast u^{\varepsilon,\delta}-u^{\varepsilon,\delta}|\psi\left(\int_{  \mathbb{R}^3_v}|v|^2f^{\varepsilon,\delta}dv\right)^\frac45dxds\\
		\leq&Ct^\frac12\left(\frac{4}{3} \pi \|f^{\varepsilon,\delta}(t,x,v)\|_{L^{\infty}} +1\right)\norm{u^{\varepsilon,\delta}-\theta_{\varepsilon}\ast u^{\varepsilon,\delta}}_{L^2(0,t;L^5(\Omega_x))}\norms{|v|^2f^{\varepsilon,\delta}}_{L^{\infty}(0,t;L^1(\mathbb{R}^6_{xv}))}^\frac45\norm{\psi}_{L^{\infty}((0,t)\times\Omega_{x})}\\
        \leq& C(t)\bigg(\norm{u^{\varepsilon,\delta}-u}_{L^2(0,t;L^5(\Omega_x))}+\norm{u-\theta_{\varepsilon}\ast u}_{L^2(0,t;L^5(\Omega_x))}\\
        &\hspace{1cm}+\norm{\theta_{\varepsilon}\ast u- \theta_{\varepsilon}\ast u^{\varepsilon,\delta}}_{L^2(0,t;L^5(\Omega_x))}\bigg) \rightarrow0, ~\text{as}~ \varepsilon,\delta\rightarrow 0.
	\end{align*}
For \(N_{8-2}\), it shows that
\begin{align*}
		&\int_{(0,t) \times \rev{\mathbb R^3_x} \times \mathbb{R}^{3}_{v}}f^{\varepsilon,\delta}v\cdot (\theta_{\varepsilon}\ast (u^{\varepsilon,\delta}\psi)-u^{\varepsilon,\delta}\psi) dvdxds\\
		\leq&\int_{(0,t) \times \rev{\mathbb R^3_x} \times \mathbb{R}^{3}_{v}}|\theta_\varepsilon \ast (u^{\varepsilon,\delta}\psi)-u^{\varepsilon,\delta}\psi||v|f^{\varepsilon,\delta} dxdvds\\
		\leq&\left(\frac{4}{3} \pi \|f^{\varepsilon,\delta}(t,x,v)\|_{L^{\infty}} +1\right)\int_{(0,t) \times \rev{\mathbb R^3_x} }|\theta_{\varepsilon}\ast (u^{\varepsilon,\delta}\psi)-u^{\varepsilon,\delta}\psi|\left(\int_{  \mathbb{R}^3_v}|v|^2f^{\varepsilon,\delta}dv\right)^\frac45dxds\\
		\leq&Ct^\frac12\left(\frac{4}{3} \pi \|f^{\varepsilon,\delta}(t,x,v)\|_{L^{\infty}} +1\right)\norm{u^{\varepsilon,\delta}\psi-\theta_{\varepsilon}\ast (u^{\varepsilon,\delta}\psi)}_{L^2(0,t;L^5(\auditmath{\mathbb R_x^3}))}\norms{|v|^2f^{\varepsilon,\delta}}_{L^{\infty}(0,t;L^1(\mathbb{R}^6_{xv}))}^\frac45\\
        \leq&C(t)\bigg(\norm{\theta_{\varepsilon}\ast (u\psi)-\theta_{\varepsilon}\ast (u^{\varepsilon,\delta}\psi)}_{L^2(0,t;L^5(\auditmath{\mathbb R_x^3}))}+\norm{u\psi-\theta_{\varepsilon}\ast (u\psi)}_{L^2(0,t;L^5(\auditmath{\mathbb R_x^3}))}\\
        &+\norm{u^{\varepsilon,\delta}\psi-u\psi}_{L^2(0,t;L^5(\auditmath{\mathbb R_x^3}))}\bigg)\rightarrow 0, ~\text{as}~ \varepsilon,\delta\rightarrow 0.
	\end{align*}

	\par Finally, as \(\varepsilon,\delta\rightarrow 0\), we get the energy inequality as
	% \begin{align*}
	% 	\frac{1}{2}&\int_{\Omega_x}|u|^2\psi dx+\int_{\Omega_x \times \mathbb{R}^{3}_{v}}\left(\frac{1}{2}|v|^2f+f\log f\right)\psi dxdv+\int_{(0,t)\times \Omega_x}|\nabla u|^2\psi dxds\nonumber\\
	% 	&+\int_{(0,t)\times \Omega_x \times \mathbb{R}^{3}_{v}}\frac{|(u-v)f-\nabla_{v}f|^2}{f}\psi dxdvds\leq \frac{1}{2}\int_{(0,t)\times\Omega_{x}}|u|^2(\partial_t\psi+\Delta_x\psi) dxds\nonumber\\
	% 	&+\frac{1}{2}\int_{(0,t)\times\Omega_{x}}|u|^2u\cdot \nabla_{x}\psi dx
	% 	+\int_{(0,t)\times \Omega_x \times \mathbb{R}^{3}_{v}}\left(\frac{|v|^2}{2}+\log f\right)f\partial_t\psi dxdvds\nonumber\\
	% 	&+ \int_{(0,t) \times \Omega_x \times \mathbb{R}^{3}_{v}} \left( \frac{|v|^2}{2} + \log f \right)f v\cdot \nabla_{x}\psi dxdvds +\int_{(0,t)\times\Omega_{x}}(P-\bar{P})u\cdot\nabla_{x}\psi dxds\nonumber\\
	% 	&+C\delta\norms{|v|^3f^{\varepsilon,\delta}}_{L^{\infty}(0,T;L^1(\mathbb{R}^6_{xv}))}\norm{\psi}_{L^{\infty}((0,T)\times\Omega_{x})}.
	% \end{align*}
	% \par By completely similar arguments as the convergence of \(\left(u^{\varepsilon,\delta}, f^{\varepsilon,\delta}\right)\) in the subsection, we can get the convergence of \(\left(u, f\right)\).  Thus, \(\left(u, f\right)\) satisfies the following local energy inequality:
	\begin{align*}
		\frac{1}{2}&\int_{\Omega_x}\left(|u|^2\psi\right)(t,\cdot) dx+\int_{\Omega_x \times \mathbb{R}^3_v}\left(\left(\frac{1}{2}|v|^2f+f\log f\right)\psi\right)(t,\cdot) dxdv+\int_{(0,t)\times \Omega_x}|\nabla u|^2\psi dxds\nonumber\\
		&+\int_{(0,t)\times \Omega_x \times \mathbb{R}^3_v}\frac{|(u-v)f-\nabla_{v}f|^2}{f}\psi dxdvds\leq \frac{1}{2}\int_{(0,t)\times\Omega_{x}}|u|^2(\partial_t\psi+\Delta_x\psi) dxds\nonumber\\
		&+\frac{1}{2}\int_{(0,t)\times\Omega_{x}}|u|^2u\cdot \nabla_{x}\psi dxds
		+\int_{(0,t)\times \Omega_x \times \mathbb{R}^3_v}\left(\frac{|v|^2}{2}+\log f\right)f\partial_t\psi dxdvds\nonumber\\
		&+ \int_{(0,t) \times \Omega_x \times \mathbb{R}^3_v} \left( \frac{|v|^2}{2} + \log f \right) fv \cdot \nabla_{x}\psi dxdvds
		+\int_{(0,t)\times\Omega_{x}}(P-\bar{P})u\cdot\nabla_{x}\psi dxds.
	\end{align*}

	\section{Partial regularity}
	Before we start the main content, recall \eqref{eq:uniform norms of f} and by the properties of weak convergence, we have the following:
	\begin{equation}
		\label{9.1}
	    \begin{aligned}
		&\norm{f(t,x,v)}_{L^{\infty}(0,t;L^{1}_{xv}(\mathbb{R}^{3}\times \mathbb{R}^{3}))}+\norm{f(t,x,v)}_{L^{\infty}((0,t)\times\mathbb{R}^{3}\times \mathbb{R}^{3})}+\norm{u}^{2}_{L^{\infty}(0,t;L^{2}_{x}(\mathbb{R}^{3}))} \\
        &\hspace{0.5cm}+\norms{|v|^{\kappa}f}_{L^{\infty}(0,t;{L^{1}_{xv}(\mathbb{R}^{6})})}+\norm{\nabla u}^{2}_{L^{2}(0,t;{L^{2}_{x}(\mathbb{R}^{3})})}\\
		&\leq C\left(\norm{f_{in}}_{L^{\infty}_{xv}( \mathbb{R}^{3}\times \mathbb{R}^{3})}, \norm{f_{in}}_{L^{1}_{xv}(\mathbb{R}^{6})},\norms{|v|^{\kappa}f_{in}}_{L^{1}_{xv}(\mathbb{R}^{6})},\norm{u_{in}}^{2}_{L^{2}_{x}(\mathbb{R}^{3})} \right)(t+1)^{\frac{3\kappa+12}{2}}e^{{C}(t+1)^3e^{6t}}.
	\end{aligned}
	\end{equation}
	Next we prove Theorem \ref{thm:regular} and Theorem \ref{thm:holder}.
	\subsection{The proof of Theorem \ref{thm:regular}}$ $
	\par	% MINREV ORIGINAL BEGIN: high moment approximation route
% As in the proof of Lemma \ref{lem:mf},  multiply \(|v|^{15+\varkappa}\) \(\left(\varkappa>0\right)\) on both sides of \eqref{eq:NSVFK} and integrate with \(xv\). By Gr\"{o}nwall's inequality, \eqref{9.2} and \eqref{9.1}, we derive
% 	
% MINREV ORIGINAL END
Apply Lemma~\ref{lem:mf} with exponent $15+\varkappa$ to the regularized solutions, using the finite initial moment in \eqref{9.2}. The bounds are uniform in $\varepsilon,\delta$. Fatou\textquotesingle s lemma then gives
	\begin{align}
		\label{15quan}
		\norms{|v|^{15+\varkappa}f}_{L^{\infty}(0,T;L^1(\mathbb{R}^3_x\times\mathbb{R}^3_v))}\leq C.
	\end{align}
	For \(\eqref{eq:NSVFK}_1\), the external force term of the Navier-Stokes equations is \(\int_{\mathbb{R}^3_v}(v-u)fdv\). Here we define $\varepsilon_1=\frac{5\varkappa}{144+6\varkappa}$. By \eqref{8} we derive
	\begin{align*}
			\left\| u \int_{\mathbb{R}_v^3} fdv\right\|_{L^{\frac{5}{2}+\varepsilon_1}_{tx}}
			\leq C\left(1 + \frac{4}{3} \pi \| f \|_{L^\infty_{txv}} \right) \left\| u \left( \int f |v|^{15 + \frac{144\varepsilon_1}{5-6\varepsilon_1}} \, dv \right)^{\frac{5-6\varepsilon_1}{30+12\varepsilon_1}} \right\|_{L^{\frac{5}{2}+\varepsilon_1}_{tx}},
	\end{align*}
	and by H\"{o}lder's inequality, it is controlled by
	\begin{align*}
		 \left\| u \int_{\mathbb{R}_v^3} fdv\right\|_{L^{\frac{5}{2}+\varepsilon_1}_{tx}}\leq C\left(1 + \frac{4}{3} \pi \| f \|_{L^\infty_{txv}} \right)\left\| f|v|^{15 + \frac{144\varepsilon_1}{5-6\varepsilon_1}}\right\|^{\frac{5-6\varepsilon_1}{30+12\varepsilon_1}}_{L^{\infty}_tL^1_{xv}} \|u\|_{L_t^{\frac{5}{2}+\varepsilon_1}L_x^{\frac{16}{7+6\varepsilon_1}+2}}.
	\end{align*}
	By \eqref{utp26},
	% \begin{align*}
	% 	\|u\|_{L_t^q L_x^p}^2 \leq C\left(\|u\|_{L_t^\infty L_x^2}^2 + \|\nabla u\|_{L_t^2 L_x^2}^2\right), \quad \text{for} \quad \frac{2}{q} + \frac{3}{p} = \frac{3}{2}, ~ 2 \leq p \leq 6,
	% \end{align*}
	 \eqref{9.1} and \eqref{15quan}, we derive
	\begin{align*}
		\left\| u \int_{\mathbb{R}_v^3} fdv\right\|_{L^{\frac{5}{2}+\varepsilon_1}_{tx}}\leq C.
	\end{align*}
	For the term of \(\int vfdv\), by \eqref{8} and H\"{o}lder's inequality, we have
	\begin{align*}
		\norms{\int_{  \mathbb{R}^3_v}fvdv}_{L^{\frac52+\varepsilon_1}_{tx}} &\leq C\left(1 + \frac{4}{3} \pi \| f \|_{L^\infty_{txv}} \right)\norms{\left(\int_{  \mathbb{R}^3_v}f|v|^{7+4\varepsilon_1}dv\right)^\frac{2}{5+2\varepsilon_1}}_{L^\frac{5+2\varepsilon_1}{2}_{tx}}\\
		&\leq Ct^{\frac{2}{5+2\varepsilon_1}}\left(1 + \frac{4}{3} \pi \| f \|_{L^\infty_{txv}} \right)\norms{f|v|^{7+4\varepsilon_1}}^{\frac{2}{5+2\varepsilon_1}}_{L^{\infty}_{t}L^1_{xv}}\\
		&\leq Ct^{\frac{2}{5+2\varepsilon_1}}\left(1 + \frac{4}{3} \pi \| f \|_{L^\infty_{txv}} \right)\left(\norms{f|v|^{15 + \frac{144\varepsilon_1}{5-6\varepsilon_1}}}^{\frac 1p}_{L^{\infty}_{t}L^1_{xv}}\norms{f}^{1-\frac1p}_{L^{\infty}_{t}L^1_{xv}}\right)^{\frac{2}{5+2\varepsilon_1}}\\
		&\leq C(t),
	\end{align*}
    where \(p=\frac{75+54\varepsilon_1}{(5-6\varepsilon_1)(7+4\varepsilon_1)}\).
	In summary, we have
	\begin{align*}
		\norms{\int_{\mathbb{R}_v^3}(v-u)fdv}_{L^{\frac{5}{2}+\varepsilon_1}_{tx}} \leq C.
	\end{align*}
	\rev{By \eqref{min:fluid-lei}, $(u,P)$ is a suitable weak solution of the forced Navier-Stokes equations with 
    \begin{align*}
        F=\int_{\mathbb{R}^3_v}fv dv-u\int_{\mathbb{R}^3_v}f dv=\int_{\mathbb{R}_v^3}(v-u)fdv.
    \end{align*}
    The energy and pressure classes follow from Section 6, and the force estimate above verifies the remaining hypothesis of Lemmas~\ref{singular} and \ref{regular}. These lemmas therefore prove Theorem~\ref{thm:regular}.} \qed

\subsection{Proof of Theorem \ref{thm:holder}} \Red{To the end we mainly need to verify the hypotheses of \cite[Theorem 1.4]{golseHarnackInequalityKinetic2019} on $Q_{\mathrm{ext}}$; see also \cite[Theorem 1.2]{WZ2024}.}

\par\noindent\textbf{Step 1. Rewrite the equation and bound the coefficients.}
\par It holds that
\begin{align*}
    \nabla_v\cdot((v-u)f)=(v-u)\cdot\nabla_v f+\nabla_v\cdot(v-u)f=(v-u)\cdot\nabla_v f+3f.
\end{align*}
Then for \(\eqref{eq:NSVFK}_2\), we transform it into
\begin{align*}
    \partial_t f+(v\cdot\nabla_x)f-\Delta_v f=(v-u)\cdot\nabla_v f+3f.
\end{align*}
% Similarly to (1.6) and (1.7) of \cite{golseHarnackInequalityKinetic2019}, by applying
\begin{revision}
Therefore, in terms of notations of  \cite[Theorem 1.4]{golseHarnackInequalityKinetic2019}, the diffusion matrix is $A=I$, the drift is $B=v-u$, and the source is $s=3f$. Now, we define
\begin{align*}
    \zeta_0=(t_0,x_0),\quad z_0=(t_0,x_0,v_0).
\end{align*}
If $\zeta_0$ is a regularity point, then \textcolor{black}{there exists} $0<R<\sqrt{t_0}$ such that
\begin{align*}
    M:=\|u\|_{L^\infty(P_R(\zeta_0))}<\infty;\quad P_R(\zeta_0):=(t_0-R^2,t_0\textcolor{black}{]}\times B_R(x_0).
\end{align*}
For $Q_{r_0}(z_0)$, we know
\begin{align*}
    |x-x_0|\leq |x-x_0-(t-t_0)v_0|+|t-t_0||v_0|\leq r_0^3+r_0^2|v_0|.
\end{align*}
So, if $r_0^2<R^2$ and $r_0^3+|v_0|r_0^2<R$, \textcolor{black}{we} obtain
\begin{align*}
    \{(t,x);\ t_0-r_0^2<t\leq t_0,\quad |x-x_0-(t-t_0)v_0|<r_0^3\}\subset P_R(\zeta_0).
\end{align*}
We choose
\begin{align*}
    0<r_0<\textcolor{black}{\min}\left\{R,\left(\frac{R}{2}\right)^\frac13,\left(\frac{R}{2(1+|v_0|\textcolor{black}{)}}\right)^\frac12\right\},\quad Q_{\text{ext}}=Q_{r_0}(z_0),\quad Q_{\text{int}}=Q_{\frac{r_0}{2}}(z_0).
\end{align*}
Note that the equation with $A=I$ is uniformly elliptic, then
\begin{align*}
    &\|B\|_{L^\infty(Q_{\text{ext}})}\leq \|v\|_{L^\infty(Q_{\text{ext}})}+\|u\|_{L^\infty(Q_{\text{ext}})}\leq |v_0|+r_0+M,\\
    &\|s\|_{L^\infty(Q_{\text{ext}})}=3\|f\|_{L^\infty(Q_{\text{ext}})}\leq 3\|f\|_{\textcolor{black}{L^\infty((0,T)\times\mathbb R^6)}}\leq \textcolor{black}{3}e^{3T}\|f_{in}\|_{L^\infty(\mathbb R^6)}.
\end{align*}
\par\noindent\textbf{Step 2. Integrability of $f$ and $\nabla_v f$.}

Now, we verify $f\in L^\infty_tL^2_{xv}$, $\nabla_v f\in L^2$ and $f\in L_{tx}^2H_v^1$. In fact, by Definition~\ref{def:weak}(i) and interpolation inequality,
\begin{align*}
    \|f\|_{L^\infty(0,T;L^2(\mathbb R_x^3\times\mathbb R_v^3))}\leq \|f\|^\frac12_{L^\infty(0,T;L^1(\mathbb R_x^3\times\mathbb R_v^3))}\|f\|^\frac12_{L^\infty((0,T)\times\mathbb R_x^3\times\mathbb R_v^3)}<\infty.
\end{align*}
\textcolor{black}{Passing to a subsequence if necessary,} Lemma~\ref{lem:deltaf} and the fact that $f^{\varepsilon,\delta}\to f$ in $L^1((0,T)\times\mathbb R^6)$ give
\begin{align*}
    \|\nabla_v f^{\varepsilon,\delta}\|_{L^2((0,T)\times\mathbb R^6)}\leq C(T),
\end{align*}
which further implies
\[
\nabla_v f^{\varepsilon,\delta}\rightharpoonup\nabla_v f~\text{in}~L^2((0,T)\times\mathbb R^6).
\]
\textcolor{black}{Here the strong $L^1$ convergence identifies the weak limit of the gradients with $\nabla_v f$ in the distributional sense.} By the weak lower \textcolor{black}{semicontinuity} of the norm, we obtain
\begin{align*}
    &\int_0^T\int_{\mathbb R^6}|f|^2\,dxdvdt\leq \|f\|_{L^\infty((0,T)\times\mathbb R_x^3\times\mathbb R_v^3)}\|f\|_{L^\infty(0,T;L^1(\mathbb R_x^3\times\mathbb R_v^3))}T\leq C(T),\\
    &\int_0^T\int_{\mathbb R^6}|\nabla_v f|^2\,dxdvdt\leq \liminf_{\varepsilon,\delta\rightarrow0}\int_0^T\int_{\mathbb R^6}|\nabla_v f^{\varepsilon,\delta}|^2\,dxdvdt\leq C(T).
\end{align*}
In conclusion, $f\in L^2_{tx}H_v^1$.
\par\noindent\textbf{Step 3. Regularity of the transport derivative.}

Finally, we verify $\partial_t f+v\cdot\nabla_x f\in \textcolor{black}{L^2_{t,x}H_v^{-1}}$.
\textcolor{black}{By the distributional equation established in Section~6,} we have
\begin{align*}
    (\partial_t+v\cdot\nabla_x)f=\nabla_v\cdot\big(\nabla_v f+(v-u)f\big).
\end{align*}
\textcolor{black}{Set $D=(t_0-r_0^2,t_0]\times B_R(x_0)\times B_{r_0}(v_0)$; then $Q_{\text{ext}}\subset D$.} Because $|v-u|\leq\textcolor{black}{|v_0|+r_0+M}$ on $D$ \textcolor{black}{almost everywhere}, we have
\begin{align*}
    \|\nabla_v f+(v-u)f\|_{L^2(D)}\leq \|\nabla_v f\|_{L^2(D)}+C\|f\|_{L^2(D)}<\infty.
\end{align*}
\textcolor{black}{For almost every $(t,x)\in(t_0-r_0^2,t_0)\times B_R(x_0)$ and every $\eta\in H_0^1(B_{r_0}(v_0))$,}
\begin{align*}
    |\langle\nabla_v\cdot\left(\nabla_v f+(v-u)f\right),\textcolor{black}{\eta}\rangle|
    &=\left|-\int_{B_{r_0}(v_0)}\left(\nabla_v f+(v-u)f\right)\cdot\nabla_v\eta\,dv\right|\\
    &\leq \|\nabla_v f+(v-u)f\|_{L^2(B_{r_0}(v_0))}\|\eta\|_{H_0^1(B_{r_0}(v_0))}.
\end{align*}
Then it holds
\begin{equation*}
    \|\partial_t f+v\cdot\nabla_x f\|_{H^{-1}(B_{r_0}(v_0))}
    \leq\|\nabla_v f+(v-u)f\|_{L^2(B_{r_0}(v_0))},
\end{equation*}
implying that
\begin{align*}
   &\int_{t_0-r_0^2}^{t_0}\int_{B_R(x_0)}\|\partial_t f+v\cdot\nabla_x f\|_{H^{-1}(B_{r_0}(v_0))}^2\,dxdt\\
   & \leq C\left(\int_D|\nabla_v f|^2\,dxdvdt+\int_D|f|^2\,dxdvdt\right)<\infty.
\end{align*}
This is the weak-solution energy class for the application of \cite[Theorem 1.4]{golseHarnackInequalityKinetic2019}.

\par\noindent\textbf{Step 4. Application of the H\"older regularity theorem.}

Applying \cite[Theorem 1.4]{golseHarnackInequalityKinetic2019} on $Q_{\text{ext}}$, we obtain
\begin{align*}
\|f\|_{C^\alpha(Q_{\text{int}})}\leq C\left(\|f\|_{L^2(Q_{\text{ext}})}+\|s\|_{L^\infty(Q_{\text{ext}})}\right)\leq C\left(\|f\|_{L^2(Q_{\text{ext}})}+\|f\|_{L^\infty(Q_{\text{ext}})}\right).
\end{align*}
\textcolor{black}{Here $\alpha\in(0,1)$, and $C$ depends on the local coefficient bounds and the chosen cylinders; in particular, these quantities may depend on $|v_0|$. We then complete the proof of Theorem \ref{thm:holder}.}
\end{revision}
\qed

		\appendix
	\section{}
	\begin{lemma}\textup{(Lemma 3.7 in \cite{tsaiLecturesNavierStokesEquations2018})}
		Let \(X\subset\subset Y\subset Z\) be reflexive Banach spaces, if \(0<T<\infty, 1\le p<\infty,\ r>1,\ (f_j)\subset L^p(0,T;X)\) satisfy
	\begin{align*}
		\sup_j\|f_j\|_{L^p(0,T;X)}\le C_1,\quad\sup_j\|\partial_t f_j\|_{L^r(0,T;Z)}\le C_2
	\end{align*}
	for some constants \(C_1, C_2 <\infty\), then \(f_j, j\in \mathbb{N}\) is relatively compact in \(L^{p}(0,T;Y)\).
	\end{lemma}
\begin{lemma}\label{lem:semigroup}
		\textup{(Proposition A.16 in \cite{bedrossianMathematicalAnalysisIncompressible2022})} The following  a priori estimate holds for \(1\leq q\leq p \leq \infty\) and \(0\leq s\leq r <\infty\):
	\begin{align*}
		\norms{e^{\nu t \bigtriangleup}f}_{\dot{W}^{r,p}(\mathbb{R}^n)}\leq C(r,s)(\nu t)^{\frac{n}{2}(\frac{1}{p}-\frac{1}{q})-\frac{r-s}{2}}\norms{f}_{\dot{W}^{s,q}(\mathbb{R}^n)},
	\end{align*}
	as soon as \(t>0\).
	\end{lemma}
	\begin{lemma}
		\textup{(Theorem A.1 in \cite{bedrossianMathematicalAnalysisIncompressible2022})} Let \(X\) be a Banach space, let \(B \subset X\) be a closed bounded subset, and let \(\Phi:B\rightarrow B\) be Lipschitz continuous in the norm topology with Lipschitz constant \(L<1\); \(i.e.\), \(\norms{\Phi(x)-\Phi(y)}\leq L\norms{x-y}\) and \(L<1\). Then \(\Phi\) has a unique fixed point in \(B\).
	\end{lemma}
	\begin{lemma}\label{gronwall} 
		\textup{(Appendix B.2 in \cite{evansPartialDifferentialEquations2010})} Let \(\eta(\cdot)\) be a non-negative absolutely continuous function on \([0,T]\), which satisfies for a.e. \(t\) the differential inequality
	\begin{align*}
		\eta^{'}(t) \leq \phi(t)\eta(t)+\psi(t),
	\end{align*}
	where \(\phi(t)\) and \(\psi(t)\) are non-negative summable functions on \([0,T]\). Then
	\begin{align*}
		\eta(t)\leq e^{\int_{0}^{t}\phi(s)ds}\left[\eta(0)+\int_{0}^{t}\psi(s)ds\right]
	\end{align*}
	for all \(0\leq t \leq T\).
	\end{lemma}
	\begin{lemma}\label{gronwall2}
		\textup{(Theorem 21 in \cite{dragomirGronwallTypeInequalities2002})} Let \(u(t)\) be a nonnegative function that satisfies the integral inequality
	\begin{align*}
		u(t) \leq c +\int_{t_0}^{t}(a(s)u(s)+b(s)u^{\alpha}(s))ds, c\geq 0, \alpha\geq 0,
	\end{align*}
	where \(a(t)\) and \(b(t)\) are continuous nonnegative functions for \(t\geq t_0\). For \(0\leq \alpha <1\) we have
	\begin{align*}
		u(t)\leq \left\{c^{1-\alpha}\exp\left[(1-\alpha)\int_{t_0}^{t}a(s)ds\right]+(1-\alpha)\int_{t_0}^{t}b(s)\exp\left[(1-\alpha)\int_{s}^{t}a(r)dr\right]ds\right\}^{\frac{1}{1-\alpha}};
	\end{align*}
	for \(\alpha=1\),
	\begin{align*}
		u(t) \leq c\exp\left\{\int_{t_0}^{t}\left[a(s)+b(s)\right]ds\right\}.
	\end{align*}
	\end{lemma}
	
	\begin{lemma}\label{A.6}
		\textup{( Corollary 4.5.5 in \cite{bogachevMeasureTheory2007})} Let \(\mu\) be a measure with values in \([0,+\infty]\) and let functions \(f_n, f \in \mathcal{L}^1(\mu)\) be such that \(f_n(x)\rightarrow f(x)\) a.e. Then convergence of \(\{f_n\}\) to \(f\) in \(L^1(\mu)\) is equivalent to the following:
	\begin{align*}
		\lim_{\mu(E)\rightarrow0} \sup_{n}\int_{E} |f_n|d\mu =0
	\end{align*}
	and, for every \(\varepsilon>0\), there exists a measurable set \(X_{\varepsilon}\) such that \(\mu(X_{\varepsilon})<\infty\) and
	\begin{align*}
		\sup_{n}\int_{X/X_{\varepsilon}}|f_n|d\mu<\varepsilon.
	\end{align*}
	\end{lemma}
	
	\begin{lemma}\label{singular}
		\textup{(Theorem 6.2 in \cite{tsaiLecturesNavierStokesEquations2018})} Let \(\Omega \subset \mathbb{R}^3\) be any domain and let \(0 < T \leq \infty\). Let \((v, p)\) be a suitable weak solution of (NS) in \(\Omega_T = \Omega \times (0, T)\) with force \(f \in L^q(\Omega \times (0, T))\), \(q > \frac52\). Let the set of singular points of \(v\) be
	\begin{align*}S = \{z_0 \in \Omega \times (0, T] : v \text{ is singular at } z_0\}.\end{align*}
	Then \(P^1(S) = 0\).
	\end{lemma}

	\begin{lemma}\label{regular}
		\textup{(Theorem 6.3 in \cite{tsaiLecturesNavierStokesEquations2018})} Let \((v, p)\) be a suitable weak solution of (NS) in \(Q(z_0, r_0)\subset\mathbb{R}^{3+1}\), \(r_0>0\), with force \(f\in L^q(Q(z_0, r_0))\), \(q>\frac52\). There is a small \(\varepsilon_0=\varepsilon_0(q)>0\) such that \(v\) is regular at \(z_0\) if \begin{align*}\limsup_{r\to 0+}\frac{1}{r}\iint_{Q(z_0,r)}|\nabla v|^2<\varepsilon_0.\end{align*}
	\end{lemma}

    \begin{lemma}\label{Bernstein}
    \textup{(Bernstein inequalities see Appendix A in \cite{taoNonlinearDispersiveEquations2006})}
Let $d \geq 1$, $s \geq 0$, and $1 \leq p \leq q \leq \infty$. Let $P_{\leq N}$, $P_{\geq N}$, and $P_N$ denote the Littlewood--Paley projection operators as defined above. Then the following inequalities hold for all tempered distributions $f$ for which the right-hand sides are finite:

\begin{enumerate}
    \item[(i)] (High frequency decay) For any $s \geq 0$,
        \begin{align*}
        \| P_{\geq N} f \|_{L_x^p(\mathbb{R}^d)} \lesssim_{p,s,d} N^{-s} \| \, |\nabla|^s P_{\geq N} f \|_{L_x^p(\mathbb{R}^d)}.
        \end{align*}

    \item[(ii)] (Derivative of low frequencies)
        \begin{align*}
        \| P_{\leq N} |\nabla|^s f \|_{L_x^p(\mathbb{R}^d)} \lesssim_{p,s,d} N^{s} \| P_{\leq N} f \|_{L_x^p(\mathbb{R}^d)}.
        \end{align*}

    \item[(iii)] (Characterization of homogeneous Sobolev norms on frequencies $\sim N$)
        \begin{align*}
        \| P_N |\nabla|^{\pm s} f \|_{L_x^p(\mathbb{R}^d)} \sim_{p,s,d} N^{\pm s} \| P_N f \|_{L_x^p(\mathbb{R}^d)}.
        \end{align*}

    \item[(iv)] (Low frequency Sobolev embedding)
        \begin{align*}
        \| P_{\leq N} f \|_{L_x^q(\mathbb{R}^d)} \lesssim_{p,q,d} N^{\frac{d}{p} - \frac{d}{q}} \| P_{\leq N} f \|_{L_x^p(\mathbb{R}^d)}.
        \end{align*}

    \item[(v)] (Single frequency Sobolev embedding)
        \begin{align*}
        \| P_N f \|_{L_x^q(\mathbb{R}^d)} \lesssim_{p,q,d} N^{\frac{d}{p} - \frac{d}{q}} \| P_N f \|_{L_x^p(\mathbb{R}^d)}.
        \end{align*}
\end{enumerate}

The implicit constants in the inequalities $\lesssim$ and $\sim$ depend only on the subscripts.
\end{lemma}

\begin{lemma}\label{riesz}
    \textup{(Riesz transform, Corollary A.12 in \cite{bedrossianMathematicalAnalysisIncompressible2022})}
   Let \( R_i = \partial_{x_i}(-\Delta)^{-1/2} \) be the \( i \)th Riesz transform, or equivalently the Fourier multiplier with symbol \( m(\xi) = i\xi_i |\xi|^{-1} \). Then \( R_i \) extends to a bounded operator \( L^p \to L^p \)
\begin{align*}
\|R_i f\|_{L^p} \lesssim \|f\|_{L^p}
\end{align*}
for \( 1 < p < \infty \).
\end{lemma}

\begin{lemma}\label{hardy littlewood sobolev}
    \textup{( Theorem 1 of Chapter 5 in \cite{stein1970})} Assume \(0 < \alpha < n\), \(1 \leq p < \frac{n}{\alpha}\) and,
\begin{align*}
I_{\alpha}f(x) = \int_{\mathbb{R}^n} \frac{f(y)}{|x-y|^{n-\alpha}}  dy.
\end{align*}
When \(p > 1\) and \(\frac{1}{q} = \frac{1}{p} - \frac{\alpha}{n}\), there holds
\begin{align*}
\|I_{\alpha}f\|_{L^q} \leq C\|f\|_{L^p}.
\end{align*}
\end{lemma}

\medskip
\noindent {\bf Acknowledgment:}\,
R. Duan was partially supported by the General Research Fund (Project No. 14303321) from RGC of Hong Kong
and also partially supported by the grant from the National Natural Science Foundation of China (Project
No. 12425109). W. Wang was supported by National Key R\&D Program of China (No. 2023YFA1009200) and NSFC under grant 12471219 and 12071054. \rev{The authors would thank the anonymous referee for valuable and helpful comments on the manuscript.}

\medskip
\noindent\textbf{Data Availability Statement:}
Data sharing is not applicable to this article as no datasets were generated or analysed during the current study.

\noindent\textbf{Conflict of Interest:}
The authors declare that they have no conflict of interest.

\end{document}